\documentclass{amsart}
\usepackage[normalem]{ulem}
\usepackage{amsmath,amsthm,amsfonts,amscd,amssymb, tikz-cd} 
\usepackage{verbatim}      	
\usepackage{mathtools}
\usepackage{epsfig}         	
\usepackage{url}		
\usepackage{epic,eepic,latexsym}
\usepackage{graphicx}
\usepackage{color}
\usepackage{lcd}
\usepackage{mathrsfs}
\usepackage[centering,text={15.5cm,22cm},
        marginparwidth=20mm,dvips]{geometry}
\usepackage[linktocpage]{hyperref}
\usepackage{lscape}
\usepackage{tabularx}
\usepackage{marginnote}
\usepackage[all]{xy}
\usepackage{enumitem}
\usepackage{stmaryrd}
\usepackage{tipa}
\usepackage{upgreek}

\newcommand{\lp}[1]{\dot{e}_{#1}}
\DeclareMathOperator{\moLQ}{J}
\DeclareMathOperator{\moKr}{Kr}

\newcommand{\Kr}[1]{\moKr^{(#1)}}
\newcommand{\LQ}[1]{\moLQ^{(#1)}}
\newcommand{\CKr}[1]{\mathbb{C}\!\moKr^{(#1)}}
\newcommand{\CLQ}[1]{\mathbb{C}\!\moLQ^{(#1)}}
\DeclareMathOperator{\Vect}{Vect}
\DeclareMathOperator{\tw}{tw}
\DeclareMathOperator{\sst}{-sst}
\DeclareMathOperator{\KK}{K_0}
\DeclareMathOperator{\IC}{IC}
\DeclareMathOperator{\moBQ}{Q}
\newcommand{\BQ}[1]{\moBQ(#1)}
\DeclareMathOperator{\pt}{pt}
\newcommand{\Msp}{\mathcal{M}}
\newcommand{\Mst}{\mathfrak{M}}
\newcommand{\LL}{\mathbb{L}}
\newcommand{\QQ}{\mathbb{Q}}

\DeclareMathOperator{\opp}{op}
\DeclareMathOperator{\Hom}{Hom}

\DeclareMathOperator{\pl}{\mathscr{P}\!}
\DeclareMathOperator{\kss}{-Kss}

\DeclareMathOperator{\Sym}{\mathbb{S}ym}
\DeclareMathOperator{\CSym}{Sym}
\DeclareMathOperator{\even}{even}
\DeclareMathOperator{\odd}{odd}

\DeclareMathOperator{\codim}{codim}

\DeclareMathOperator{\Spec}{Spec}
\DeclareMathOperator{\id}{Id}
\DeclareMathOperator{\Id}{Id}
\DeclareMathOperator{\MHS}{MHS}
\DeclareMathOperator{\cl}{cl}
\DeclareMathOperator{\MMHS}{MMHS}

\DeclareMathOperator{\Wt}{wt}

\DeclareMathOperator{\parit}{par}

\DeclareMathOperator{\HO}{H}
\DeclareMathOperator{\sstable}{-sst}
\DeclareMathOperator{\sign}{sgn}

\DeclareMathOperator{\mMHS}{\overline{MHS}}
\DeclareMathOperator{\trop}{trop}
\DeclareMathOperator{\qtrop}{qtr}
\DeclareMathOperator{\prin}{prin}

\DeclareMathOperator{\In}{in}

\DeclareMathOperator{\BPS}{BPS}

\DeclareMathOperator{\stable}{-st}
\DeclareMathOperator{\vir}{vir}

\DeclareMathOperator{\Ends}{Ends}

\DeclareMathOperator{\Joints}{Joints}

\DeclareMathOperator{\sat}{sat}

\DeclareMathOperator{\EE}{\mathbb{E}}

\newcommand{\lcbs}{(\!(}
\newcommand{\rcbs}{)\!)}

\DeclareMathOperator{\GL}{GL}

\DeclareMathOperator{\Supp}{Supp}

\DeclareMathOperator{\Li}{Li}

\DeclareMathOperator{\Ad}{Ad}
\DeclareMathOperator{\ad}{ad}
\DeclareMathOperator{\as}{as}

\DeclareMathOperator{\up}{up}
\DeclareMathOperator{\udim}{\underline{\dim}}
\DeclareMathOperator{\midd}{mid}

\DeclareMathOperator{\ord}{ord}
\DeclareMathOperator{\can}{can}

\DeclareMathOperator{\ob}{ob}

\DeclareMathOperator{\Hall}{Hall}
\DeclareMathOperator{\Stab}{Stab}

\DeclareMathOperator{\scat}{Scat}

\DeclareMathOperator{\uf}{uf}

\DeclareMathOperator{\Exp}{Exp}
\DeclareMathOperator{\Log}{Log}
\DeclareMathOperator{\Hodge}{Hdg}
\DeclareMathOperator{\barr}{bar}
\DeclareMathOperator{\bad}{bad}

\def\!{\mskip-\thinmuskip}
\let\?=\overline
\let\:=\colon
\let\llb=\llbracket
\let\rrb=\rrbracket
\let\bb=\mathbb

\let\rar=\rightarrow
\let\f=\mathfrak
\let\s=\mathcal

\let\wh=\widehat
\let\wt=\widetilde

\let\mr=\mathring

\def\risom{\buildrel\sim\over{\smashedlongrightarrow}}
 \def\smashedlongrightarrow{\setbox0=\hbox{$\longrightarrow$}\ht0=1.25pt\box0}

\newcommand {\kk} {\Bbbk}
\newcommand {\kt} {\kk_t}
\newcommand {\Zt} {\Z_t}
\newcommand {\sQ} {\mathcal{Q}}

\newcommand {\jj} {\vec{\jmath}}
\newcommand {\kkk} {\vec{k}}

\newcommand{\g}[1]{1_{#1}}

\renewcommand{\theta}{\uptheta}
\renewcommand{\iota}{\upiota}
\renewcommand{\alpha}{\upalpha}
\renewcommand{\beta}{\upbeta}
\renewcommand{\gamma}{\upgamma}
\renewcommand{\delta}{\updelta}
\renewcommand{\zeta}{\upzeta}
\renewcommand{\pi}{\uppi\hspace{0.05em}}
\renewcommand{\xi}{\upxi}
\renewcommand{\chi}{\upchi}
\renewcommand{\sigma}{\upsigma}
\renewcommand{\Lambda}{\Uplambda}
\renewcommand{\Gamma}{\Upgamma}
\renewcommand{\phi}{\upphi}
\renewcommand{\nu}{\upnu}
\renewcommand{\tau}{\uptau}
\renewcommand{\mu}{\upmu}
\renewcommand{\eta}{\upeta}

\theoremstyle{plain}
 \newtheorem{thm}{Theorem}[section]
 
 \newtheorem{rem}[thm]{Remark}
 \newtheorem{lem}[thm]{Lemma}
  
  \newtheorem{prop}[thm]{Proposition}
  \newtheorem{conj}[thm]{Conjecture}
   \newtheorem{cor}[thm]{Corollary}

\newcommand{\Z}{{\mathbb Z}}

\newcommand{\Gr}{\operatorname{Gr}}

\theoremstyle{definition}
 \newtheorem{dfn}[thm]{Definition}

 \newtheorem{eg}[thm]{Example}
 \newtheorem{egs}[thm]{Examples}
 
\theoremstyle{remark} 
 \newtheorem{rmk}[thm]{Remark}

\newenvironment{myproof}[1][\proofname]{\proof[#1]\mbox{}}{\endproof}

\title{Strong positivity for quantum theta bases of quantum cluster algebras}
\author{Ben Davison}
\address{School of Mathematics\\
University of Edinburgh\\
Edinburgh EH9 3FD\\
UK}
\email{Ben.Davison{\char'100}ed.ac.uk }
\author{Travis Mandel}
\address{Department of Mathematics \\ University of Oklahoma \\ Norman, OK 73019 \\ USA}
\email{tmandel{\char'100}ou.edu}

\begin{document}

\begin{abstract}
We construct ``quantum theta bases,'' extending the set of quantum cluster monomials, for various versions of skew-symmetric quantum cluster algebras.  These bases consist precisely of the indecomposable universally positive elements of the algebras they generate, and the structure constants for their multiplication are Laurent polynomials in the quantum parameter with non-negative integer coefficients, proving the quantum strong cluster positivity conjecture for these algebras.  The classical limits recover the theta bases considered by Gross--Hacking--Keel--Kontsevich \cite{GHKK}.  Our approach combines the scattering diagram techniques used in loc. cit. with the Donaldson--Thomas theory of quivers.
\end{abstract}

\maketitle

\setcounter{tocdepth}{1}
\tableofcontents

\section{Introduction}\label{intro}

Cluster algebras, now a topic of significant, wide-ranging interest, were originally defined by Fomin and Zelevinsky \cite{FZ} with the goal of better understanding Lusztig's dual canonical bases \cite{LusDualCan} and total positivity \cite{LusTP}. It is therefore no surprise that questions about canonical bases for cluster algebras and their positivity properties are fundamental to the theory, cf. \cite[\S 4]{FG1}.  These questions were largely answered in the classical setup by Gross--Hacking--Keel--Kontsevich \cite{GHKK}.  However, the quantum setup presents new challenges, particularly in the proof of positivity.  The present paper uses ideas from Donaldson--Thomas (DT) theory \cite{KSmotivic,Keller,Nagao,DavMei,Bridge,MeiRei} in concert with scattering diagram technology \cite{KS,GrP2,GS11,CPS,FS,WCS,GHK1,GHKK,Man3} to prove positivity in the quantum setting.

We construct bases for various flavors of (skew-symmetric) quantum cluster $\s{A}$- and $\s{X}$-algebras, extending the set of quantum cluster variables, such that the structure constants for the multiplication are Laurent polynomials in the quantum parameter with non-negative integer coefficients, proving the ``quantum strong positivity conjecture.''  Moreover we prove that these bases consist precisely of the universally positive indecomposable elements of these algebras (in particular reproving the ``quantum positivity conjecture'' originally proved in \cite{DavPos}).  Furthermore, we use our constructions and results for quantum scattering diagrams to prove quantum analogues of all the main results of \cite{GHKK} for skew-symmetric quantum cluster algebras.

\subsection{Main results}\label{MainResultsSection}

In order to more fully state our main results, we briefly sketch the various constructions of (skew-symmetric) quantum cluster algebras, cf. \S \ref{qclusterdefs} for details.   
Let $L$ be a lattice with a rational skew-symmetric bilinear form $\omega$, and pick $D\in \bb{Z}_{\geq 1}$ such that $D\omega$ is $\bb{Z}$-valued.  This data determines a noncommutative algebra $\Z_t^{\omega}[L]$ over $\Zt \coloneqq\Z[t^{\pm 1/D}]$, generated by elements $z^v$ for $v\in L$, with relations determined by $z^{a}z^{b}=t^{\omega(a,b)}z^{a+b}$ for $a,b\in L$.  We call elements of $\Zt$ \textbf{positive} if their coefficients are in $\bb{Z}_{\geq 0}$.  A choice of strongly convex rational polyhedral cone $\sigma$ in $L_{\bb{R}}\coloneqq L\otimes \bb{R}$ determines a monoid $L^{\oplus}\coloneqq L\cap \sigma$ and a formal completion $\Z_{\sigma,t}^{\omega}\llb L\rrb\coloneqq \Z_t^{\omega}[L]\otimes_{\Z_t^{\omega}[L^{\oplus}]}\Z_t^{\omega}\llb L^{\oplus}\rrb$ consisting of quantum Laurent series.

Now, consider the data of a seed $S$, i.e., the data $\{N,I,E\coloneqq \{e_{i}\}_{i\in I},F,B\}$, where $N$ is a finite-rank lattice with basis $E$ indexed by $I$, $F$ is a subset of $I$, and $B(\cdot,\cdot)$ is a skew-symmetric $\bb{Q}$-valued bilinear form on $N$ such that $B(e_i,e_j)\in \bb{Z}$ whenever $i$ and $j$ are not both in $F$.  The seed $S$ determines a quantum torus algebra $\s{X}_q^S\coloneqq \Z^B_t[N]$, along with a completion $\wh{\s{X}}_q^S\coloneqq \Z_{\sigma_{S,\s{X}},t}^B\llb N\rrb$, where $\sigma_{S,\s{X}}$ denotes the cone spanned by $e_i$ for $i\in I\setminus F$.  Let $\mr{\s{X}}_q^S$ denote the skew-field of fractions of $\s{X}_q^S$.  Each sequence $\jj$ of elements of $I\setminus F$ determines a sequence of ``mutations,'' giving a new seed $S_{\jj}$ (equal to $S$ except for $E$), along with a $\Zt$-algebra isomorphism $\mu_{S,\jj}^{\s{X}}:\mr{\s{X}}_q^S\risom \mr{\s{X}}_q^{\mu_{\jj}(S)}$, cf \eqref{Xmut}.  One then defines
\begin{align*}
    \s{X}_q^{\up} \coloneqq \bigcap_{\jj} (\mu_{S,\jj}^{\s{X}})^{-1}(\s{X}_q^{S_{\jj}})\subset \mr{\s{X}}_q^S.
\end{align*}
A (quantum) global monomial is defined to be an element of $\s{X}_q^{\up}$ which, for some $\jj$, is equal to $(\mu_{S,\jj}^{\s{X}})^{-1}(z^v)$ for $z^v$ a monomial in $\s{X}_q^{S_{\jj}}$. One defines $\s{X}_q^{\ord}$ to be the subalgebra of $\s{X}_q^{\up}$ generated by the quantum global monomials.

Let $M\coloneqq \Hom(N,\bb{Z})$.  Suppose we have the additional data of a compatible pair, i.e., a rational skew-symmetric bilinear form $\Lambda$ on $M$ satisfying a certain compatibility condition \eqref{Lambda} with $B$ (the existence of such a $\Lambda$ being equivalent to the injectivity assumption of \cite{GHKK}).  Let $\sigma_{S,\s{A}}$ be the cone generated by $\{B(e_i,\cdot)\in M\colon i\in I\setminus F\}$.  Then one can similarly define $\s{A}_q^S$, $\wh{\s{A}}_q^S$, $\mr{\s{A}}_q^S$, $\s{A}_q^{\up}$, and $\s{A}_q^{\ord}$, replacing $\Z_t^B[N]$ and $\Z_{\sigma_{S,\s{X}},t}^B\llb N\rrb$ with $\Z_t^{\Lambda}[M]$ and $\Z_{\sigma_{S,\s{A}},t}^{\Lambda}\llb M\rrb$, respectively, and replacing $\mu_{S,\jj}^{\s{X}}$ with a certain isomorphism $\mu_{S,\jj}^{\s{A}}:\mr{\s{A}}_q^S\rar \mr{\s{A}}_q^{\mu_j(S)}$, cf. \eqref{clust_transf}.  The algebra $\s{A}_q^{\ord}$ is generated by the usual quantum cluster variables as in \cite{BZ}, or equivalently, by the (quantum) global monomials defined analogously to the construction of $\s{X}_q^{\ord}$.

We denote $L_{\s{X}}\coloneqq N$ and $L_{\s{A}}\coloneqq M$.  Let $\omega_{\s{X}}\coloneqq B$ and $\omega_{\s{A}}\coloneqq \Lambda$.  For convenience, we now begin using $\s{V}$ as a stand-in for either $\s{A}$ or $\s{X}$ so that we can describe both cases at once.

Our construction of theta bases for the $\s{V}$-algebras will involve the construction of a scattering diagram $\f{D}^{\s{V}_q}$ in $L_{\s{V},\bb{R}}\coloneqq L_{\s{V}}\otimes\bb{R}$. Each generic point $\sQ$ in $L_{\s{V},\bb{R}}$ then determines an isomorphism \begin{align}\label{iotaQintro}
    \iota_{\sQ}\colon \wh{\s{V}}^S_{q} \risom \Z_{\sigma_{S,\s{V}},t}^{\omega_{\s{V}}}\llb L_{\s{V}}\rrb,
\end{align}
cf. \S \ref{iotaQrmk}.  We refer to the set of isomorphisms $\iota_{\sQ}$ for generic $\sQ\in L_{\s{V},\bb{R}}$ as the scattering atlas for $\wh{\s{V}}^S_q$.  So for each $f\in \wh{\s{V}}_q^S$ and each generic $\sQ\in L_{\s{V},\bb{R}}$, we can use $\iota_{\sQ}$ to obtain an expansion
\begin{align*}
    \iota_{\sQ}(f)=\sum_{v\in L_{\s{V}}} a_{v,\sQ}z^v\in \Z_{\sigma_{S,\s{V}},t}^{\omega_{\s{V}}}\llb L_{\s{V}}\rrb,
\end{align*}
where each $a_{v,\sQ}$ is an element of $\Zt$.  We say that a nonzero element $f\in \wh{\s{V}}^{S}_{q}$ is universally positive with respect to the scattering atlas if each $a_{v,\sQ}$ (for each $v$ and each generic $\sQ$) is positive.  A universally positive element is called atomic (or indecomposable) if it cannot be written as a sum of two other universally positive elements.

In fact, for $\s{V}=\s{X}$, we can define an isomorphism $\iota_{\sQ}\colon\wh{\s{X}}^S_{q} \risom \Z_{\sigma_{S,\s{X}},t}^{B}\llb L_{\s{X}}\rrb$ for any generic $\sQ\in L_{\s{A}^{\prin},\bb{R}}\supset L_{\s{X},\bb{R}}$, cf. \S \ref{iotaQrmk}. Here, $L_{\s{A}^{\prin}}=M^{\prin}=M\oplus N$ is the analogue of $L_{\s{A}}$ associated to $S^{\prin}$, i.e., the seed with principal coefficients associated to $S$.  We refer to this set of isomorphisms $\iota_{\sQ}$ as the principal coefficients scattering atlas.  Note that the notions of universally positive and atomic elements make sense for this atlas as well.

\begin{thm}\label{MainThm}
Let $\s{V}$ denote either $\s{A}$ or $\s{X}$.  There is a topological\footnote{See \S \ref{TopStr} for background on topological bases.  Roughly, this means that we may need to allow infinite linear combinations of the basis elements.} $\Zt$-module basis $\{\vartheta_p\colon p\in L_{\s{V}}\}$ for $\wh{\s{V}}_q^S$, elements of which are called (quantum) theta functions, which contains all of the global monomials.  The basis $\{\vartheta_p\}_{p\in L_{\s{V}}}$ is uniquely characterised as the set of elements of $\wh{\s{V}}^{S}_q$ which are atomic with respect to the scattering atlas.  For $\s{V}=\s{X}$, this is the same as the set of elements of $\wh{\s{X}}^{S}_q$ which are atomic with respect to the principal coefficient scattering atlas.

For generic $\sQ$ in $ L_{\s{A}^{\prin},\bb{R}}$ (for the $\s{X}$ case) or $L_{\s{A},\bb{R}}$ (for the $\s{A}$ case), each $\iota_{\sQ}(\vartheta_p)$ is positive and $p$-pointed, i.e. admits an expansion
    \begin{equation}
        \label{iota_expansion}
    \iota_{\sQ}(\vartheta_p)=\sum_{v\in L_{\s{V}}^{\oplus}}c_{p,v,\sQ}(t)z^{p+v}
    \end{equation}
where each $c_{p,v,\sQ}(t)\in\mathbb{Z}_{\geq 0}[t^{\pm 1}]\subset \Zt$ is a positive Laurent polynomial in $t$ (not just $t^{1/D})$ and $c_{p,0,\sQ}(t)=1$.  Each $c_{p,v,\sQ}(t)$ is moreover pre-Lefschetz (cf. \S \ref{Section-Notation}), bar-invariant (i.e. $c_{p,v,\sQ}(t^{-1})=c_{p,v,\sQ}(t)$), and satisfies $c_{p,v,\sQ}(-t)=\pm c_{p,v,\sQ}(t)$, where the sign depends only on $v$ and $p$ (cf. Proposition \ref{ParitProp}).

For $p_1,p_2,p\in L_{\s{V}}$, let $\alpha(p_1,p_2;p)\in \Zt$ denote the corresponding structure constant, i.e., 
    \begin{equation}
    \label{c_expansion}
        \vartheta_{p_1}\vartheta_{p_2}=\sum_{p\in L_{\s{V}}} \alpha(p_1,p_2;p) \vartheta_p.
    \end{equation} 
    Then each $\alpha(p_1,p_2;p)\in\bb{Z}_t$ is positive (we thus call the theta basis ``strongly positive'').  Furthermore, $\alpha(p_1,p_2;p_1+p_2)=t^{\omega(p_1,p_2)}$, and $\alpha(p_1,p_2;p)=0$ whenever $p\notin p_1+p_2+L_{\s{V}}^{\oplus}$. 
    \end{thm}

Using the above quantum theta bases, we produce quantum analogues of all of the main results of \cite{GHKK} regarding the algebras associated to theta bases, in particular the $\s{A}$ and $\s{X}$ cluster algebras.  Here we collect together all of the main features and our results on these bases and algebras for the convenience of the reader.

Let $K_{\s{X}}\coloneqq \bb{Z}\langle e_i\colon i\in F\rangle\subset N$, and let $\kappa_{\s{X}} \colon N\rar K_{\s{X}}$ denote the projection with kernel $\bb{Z}\langle e_i\colon i\in I\setminus F\rangle$.  Consider the maps $B_1,B_2\colon N\rar M$, $B_1(n)\coloneqq  B(n,\cdot)$ and $B_2(n)\coloneqq B(\cdot,n)$.  Let $K_i\coloneqq \ker(B_i)$ (this is of course independent of $i$ since $B$ is skew-symmetric --- cf. Remark \ref{nonskew} for comments on the non skew-symmetric setup).  Let $\kappa_{\s{A}}$ denote the projection $M\rar M/(B_1(N)^{\sat})\cong K_2^*=:K_{\s{A}}$.

Finally, let $H_{\s{V}}$ denote the set of $u\in L_{\s{V}}$ such that $\omega_{\s{V}}(u,v)=0$ for all $v\in \sigma_{S,\s{V}}$, and let $\tau_u$ denote the $\Z_t$-module automorphism of $\Z_{\sigma_{S,\s{V}},t}^{\omega_{\s{V}}}\llb L_{\s{V}}\rrb$ taking $z^p$ to $z^{p+u}$ for each $p\in L_{\s{V}}$.  E.g., if $\omega_{\s{V}}(u,p)=0$, then $\tau_u(z^p)=z^uz^p$.

\begin{thm}
    \label{MainThm2}
    \begin{enumerate}
    \item The classical limits of the quantum theta functions are the classical theta functions of \cite[Thm. 0.3]{GHKK}. 
    \item Let $\Theta_{\s{V}}^{\midd}\subset L_{\s{V}}$ be the set of $p$ such that $\vartheta_p$ is a Laurent polynomial (as opposed to merely a formal Laurent series). Then $\vartheta_p\in\s{V}_q^{\up}$ for $p\in \Theta_{\s{V}}^{\midd}$.  The set $\Theta_{\s{V}}^{\midd}$ is the same as in the corresponding classical cluster algebra setup of \cite[Thm. 0.3]{GHKK}.  In particular, $\Theta_{\s{V}}^{\midd}=\Theta^{\midd}_{\s{V},\bb{R}}\cap L_{\s{V}}$ for some cone $\Theta^{\midd}_{\s{V},\bb{R}}\subset L_{\s{V},\mathbb{R}}$ which is convex (but not necessarily strongly convex, rational polyhedral, or closed) and remains convex under the application of any of the piecewise-linear automorphisms $T_{\jj}^{\s{V}}$ of $L_{\s{V},\mathbb{R}}$ as defined in \eqref{mujj} and \eqref{TjX}.
    \item The theta functions $\{\vartheta_p\colon p\in \Theta_{\s{V}}^{\midd}\}$ form a $\Zt$-module basis for the $\Z_t$-subalgebra $\s{V}_q^{\midd}$ of $\s{V}_q^{\up}$ which they generate.  Every global monomial is a theta function corresponding to some $p\in \Theta_{\s{V}}^{\midd}$, so we have the inclusions
    \begin{align*}
        \s{V}_q^{\ord} \subset \s{V}_q^{\midd} \subset \s{V}_q^{\up} 
        \subset \wh{\s{V}}_q^S.
    \end{align*}
    \item Let $\s{V}_q^{\can}$ denote the sub $\Zt$-algebra of $\wh{\s{V}}_q^S$ generated by the theta functions.  The theta functions form a topological $\Zt$-module basis for $\s{V}_q^{\can}$, and they form an ordinary $\Zt$-module basis if and only if the corresponding classical theta functions form a $\Z$-module basis for the $\Z$-algebra $\s{V}^{\can}$ which they generate.  If $\s{V}_q$ has enough global monomials (cf. Definition \ref{EGMdef} --- this is equivalent to the corresponding condition in the classical setup up as in \cite[Def. 0.11]{GHKK}), then $\s{V}_q^{\can}$ is finitely generated over $\Zt$, and each element of $\s{V}_q^{\up}$ is a finite $\bb{Z}_t$-linear combination of theta functions.  In particular, we then have $\s{V}_q^{\up} \subset \s{V}_q^{\can}$. 
    \item We have the equality $\s{V}_q^{\midd}=\s{V}_q^{\can}$, i.e., $\Theta_{\s{V}}^{\midd}=L_{\s{V}}$, if and only if the corresponding equality holds in the classical setting.  So if we furthermore have that $\s{V}_q$ has enough global monomials, then the full quantum Fock--Goncharov conjecture holds, by which we mean that $\s{V}_q^{\midd}=\s{V}_q^{\up}=\s{V}_q^{\can}$.
    \item Up to canonical isomorphism $\psi_{\jj}^{\s{V}}$, the $\Z_t$-algebras $\s{V}_q^{\can}$ and $\s{V}_q^{\midd}$ depend only on the mutation equivalence class of the initial seed, as do their theta functions (cf. Corollary \ref{TjCor} for the precise statement).  In addition, $\s{V}_q^{\can} \subset \wh{\s{V}}_q^{\up}$, the formal quantum upper cluster algebra defined in \S \ref{FormalUp}. 
    \item  The mutations $\mu_{S,\jj}^{\s{V}}$ respect the $K_{\s{V}}$-gradings of $\wh{\s{V}}_q^S$ and $\wh{\s{V}}_q^{S_{\jj}}$ induced by $\kappa_{\s{V}}$, thus implying that we have a $K_{\s{V}}$-grading on $\s{V}_q^{\up}$.
    More generally, the adjoint action of any path-ordered product associated to $\f{D}^{\s{V}_q}$ (or $\s{D}^{\s{A}_q^{\prin}}$ when $\s{V}=\s{X}$) respects the $K_{\s{V}}$-grading on $\wh{\s{V}}_q^S$.  Furthermore, the theta functions are homogeneous with respect to the $K_{\s{V}}$-grading on $\wh{\s{V}}_q^S$.  
\item Suppose that $\Lambda$ satisfies the compatibility condition \eqref{Lambda} for all $i\in I$ (not just $i\in I\setminus F$).  Then the map $B_1$ induces homomorphisms $B_1\colon \wh{\s{X}}_q^S\rar \wh{\s{A}}_q^S$ and $B_1\colon \s{X}_q^*\rar \s{A}_q^*$ for $*=\ord, \midd, \up$, or $\can$.  
Furthermore, for any $p\in L_{\s{X}}$, we have $B_1(\vartheta_p)=\vartheta_{B_1(p)}$. 
\item For $u\in H_{\s{V}}$, $p\in L_{\s{V}}$, and generic $\sQ \in L_{\s{V},\bb{R}}$, we have $\iota_{\sQ}(\vartheta_{u+p})=\tau_u(\iota_{\sQ}(\vartheta_p))$. In particular, for $u\in H_{\s{V}}$, $\iota_{\sQ}(\vartheta_u)=z^u$  for all generic $\sQ$, and so $\vartheta_u\in \s{V}^{\ord}_q$.  
\end{enumerate}
\end{thm}

\subsection{The cluster atlas}\label{cluster_atlas}

There is a fan $\s{C}$ forming a sub cone-complex of $\f{D}^{\s{A}_q}$ (alternatively, of $\f{D}^{\s{A}_q^{\prin}}$, the scattering diagram associated to $S^{\prin}$), called the cluster complex, whose elements naturally correspond to the seeds $S_{\jj}$ (cf. Proposition \ref{Chambers}).  For $\sQ_{\jj}$ a generic point in the chamber corresponding to $S_{\jj}$, denote $\iota_{\jj}\coloneqq \iota_{\sQ_{\jj}}$. This set of charts $\iota_{\jj}$ is what we call the cluster atlas.  In \eqref{psij}, we also define a certain linear automorphism $\psi_{\jj}$ of $M$ which induces automorphisms of $\kk_t^{\Lambda}[M]$, also denoted $\psi_{\jj}$.  Proposition \ref{Chambers}(4) says that $\psi_{\jj}\circ \iota_{\jj}$ and $\mu^{\s{V}}_{S,\jj}$ agree on $\s{V}_q^{\up}$, so the restriction of the cluster atlas to $\s{V}_q^{\up}$ is the usual set of clusters on the quantum upper cluster algebra.  See \S \ref{iotaQrmk} for more details.

In \cite{FG1}, the notion of being universally positive and atomic (called ``extremal'' there) is defined using the cluster atlas.  The scattering atlas or principal coefficients scattering atlas typically includes more charts than the cluster atlas, so universal positivity with respect to the scattering atlas is a stronger condition.  Indeed, it is often strictly stronger --- \cite[Thm 1.2]{LLZ} shows that the atomic elements would  often be linearly dependent (at least in the classical setup) if we used the cluster atlas.  Atomicity with respect to the scattering atlas was suggested and proven for the classical setting in \cite{ManAtomic}.

The quantum global monomials generating $\s{A}_q^{\ord}$ are precisely those $\vartheta_p$ for $p\in \s{C}$, with the quantum cluster variables corresponding to the primitive generators for the rays of $\s{C}$, and we have $\s{C}\subset \Theta_{\s{A},\bb{R}}^{\midd}$, cf. Theorem \ref{MainThm2}(3) above. 
Applying the statement from Theorem \ref{MainThm} that $c_{p,v,\sQ}(t^{-1})=c_{p,v,\sQ}(t)\in\mathbb{Z}_{\geq 0}[t^{\pm 1}]$ in the special case in which both $p$ and $\sQ$ are in $\s{C}$, we thus deduce the main result of \cite{DavPos}:
\begin{cor}[Quantum positivity conjecture]\label{posLaurent}
The expression of any quantum cluster variable in any cluster is a bar-invariant Laurent polynomial in $t$ with non-negative integer coefficients. 
\end{cor}

Taking the classical limit of Corollary \ref{posLaurent} recovers the main result of \cite{LS}, i.e. the positivity conjecture for skew-symmetric cluster algebras.

In fact, \cite{DavPos} shows something more than Corollary \ref{posLaurent} --- it shows that the coefficients have Lefschetz type (as defined in \S \ref{Section-Notation}).  This motivates our main conjecture:
\begin{conj}\label{MainConj}
For any $p\in L_{\s{V}}$ and generic $\sQ\in L_{\s{V},\bb{R}}$, the coefficients of $\iota_{\sQ}(\vartheta_p)$ have Lefschetz type.  
\end{conj}
The main theorem of \cite{DavPos} ensures that Conjecture \ref{MainConj} holds for $\s{V}=\s{A}$ with $\sQ$ and $p$ both in the cluster complex.  We show that the coefficients of $\iota_{\sQ}(\vartheta_p)$ always at least satisfy a weaker condition which we call pre-Lefschetz.  Also, using a result of \cite{MeiRei}, we prove Conjecture \ref{MainConj} in the acyclic cases (i.e., for $\s{A}_q^{\can}$ and $\s{X}_q^{\can}$   associated to seeds whose corresponding quiver has acyclic unfrozen part), cf. \S \ref{MainProof}.

\begin{thm}\label{AcyclicLefschetz}
Conjcture \ref{MainConj} holds in acyclic cases.
\end{thm}

\subsection{Dequantization}

As an application of our results, we generalize the main result of \cite{GLSflat}, which we shall now recall.  Note that setting $t^{1/D}$ equal to $1$ determines a homomorphism 
$$\lim_{t\rar 1}\colon\Z_t^{\omega}[L]\rar \Z[L]$$
(for convenience, we write $\lim_{t\rar 1}$ instead of $\lim_{t^{1/D}\rar 1}$).  This extends to the formal versions, giving a homomorphism
\begin{align*}
    \lim_{t\rar 1} \colon \Z_{\sigma,t}^{\omega}\llb L\rrb \rar \Z_{\sigma}\llb L\rrb\coloneqq \Z[L]\otimes_{\Z[L^{\oplus}]} \Z\llb L^{\oplus} \rrb.
\end{align*}
The images under this homomorphism are what we refer to as ``classical limits'' in Theorem \ref{MainThm2}(1).  

For $*$ denoting $\ord,\midd,\up$, or $\can$, the algebras $\s{V}^*_q$ are $\mathbb{Z}_t$-algebras, and torsion-free as $\mathbb{Z}_t$-modules.  As such it is natural to consider the ``classical limits'' 
\[
\cl(\s{V}_q^*)\coloneqq \s{V}^*_q/(t^{1/D}-1)\s{V}^*_q
\]
as dequantizations of the quantum cluster algebras that we are considering in this paper.  On the other hand, there are already well-defined classical cluster algebras $\s{V}^*\subset \bb{Z}_{\sigma_{S,\s{V}}}\llb L_{\s{V}}\rrb$, and it is natural to ask whether these two dequantizations are the same.  

To approach this question we consider the restriction of $\lim_{t \rightarrow 1}$ to $\s{V}_q^*$, which for each choice of $*$ above has image contained in $\s{V}^*$.  The kernels of these homomorphisms of course always contain the ideal generated by $(t^{1/D}-1)$, but the reverse containment is not so obvious, i.e. it is not clear a priori that the induced map $\cl(\s{V}^*_q)\rightarrow \s{V}^*$ is injective.  In fact, the main result of \cite{GLSflat} is that the kernel of $\lim_{t\rar 1}:\s{A}_q^{\ord}\rar \s{A}^{\ord}$ really is equal to the ideal $( t^{1/D}-1 ) \s{A}_q^{\ord}$, and the natural map $\cl(\s{A}_q^{\ord})\rightarrow \s{A}^{\ord}$ is an isomorphism, assuming that $\s{A}^{\ord}=\s{A}^{\up}$, and further assuming that $\s{A}^{\ord}$ has a $\bb{Z}$-grading with finite-dimensional homogeneous components. As an application of our main results, we obtain a new proof of their result without the grading condition.  More generally, we prove the following:

\begin{thm}\label{flat}
Let $\s{V}$ denote $\s{A}$ or $\s{X}$.  Then $\lim_{t\rar 1}$ induces an injective morphism $\cl(\s{V}_q^*)\rightarrow \s{V}^*$ if $*=\midd, \up$ or $\can$.  The same holds for $*=\ord$ whenever $\s{V}^{\ord}=\s{V}^{\midd}$ (which at least holds when $\s{V}^{\ord}=\s{V}^{\up}$).  Moreover, the natural map $\cl(\s{V}_q^*)\rightarrow \s{V}^*$ is surjective if $*$=$\ord,\midd$ or $\can$ (and so it is surjective for $*=\up$ if $\s{V}^{\up}=\s{V}^{\midd}$, which at least holds if $\s{V}^{\ord}=\s{V}^{\up}$).
\end{thm}

Geiss--Leclerc--Schr{\"o}er had suggested this application of quantum theta functions to the second author while \cite{GLSflat} was in progress. 
Thanks to our positivity result, we are now able to show that their suggestion was correct. 
See \S \ref{MainProof} for the proof.

\begin{rmk}
The statement that $\cl(\s{V}_q^*)\rar \s{V}^*$ is an isomorphism means that $\s{V}_q^*$ is a deformation quantization of $\s{V}^*$.  By construction, the quantum cluster $\s{X}$-variety of \cite[\S 3]{FG1} (and the quantum cluster $\s{A}$-variety that could be analogously defined) is, in the sense of \cite[Def. 5]{Bou3}, a deformation quantization of the corresponding classical cluster variety.  The surjectivity of $\cl(\s{V}_q^{\up})\rar \s{V}^{\up}$ then follows if the cluster variety is affine (at least up to codimension $2$), cf. \cite[Lem. 7]{Bou3}.
\end{rmk}

\subsection{Additional properties and context for theta functions}\label{PosIntro}
It is clear that being atomic as in  Theorem \ref{MainThm} uniquely determines the theta functions, at least once the scattering atlas is given.  It would be interesting to have additional characterizations of the theta bases which could be checked using only the cluster algebra structure.  E.g., \cite[Conj. 14]{LLRZpnas}  suggests that the quantum theta basis $\{\vartheta_p\colon p\in \s{L}_{\s{V}}\}$ may be the unique minimal topological $\bb{Z}_t$-module basis amongst the set $\s{S}$ of all pointed, bar-invariant, strongly positive topological $\Z_t$-module bases for $\s{V}_q^{\can}$ including all global monomials.  Here, by minimal one means that $f_p-\vartheta_p\in \wh{\s{V}}_q^S$ has positive coefficients whenever $f_p$ is the $p$-pointed element in another such basis.

A standard argument (cf. \cite[\S 5]{LLRZpnas}) shows that for $\s{V}=\s{A}$, any (not necessarily minimal) basis in $\s{S}$ is universally positive with respect to the cluster atlas.  The conjecture of \cite{LLRZpnas} would follow if every basis in $\s{S}$ was universally positive with respect to the scattering atlas.  On the other hand, if the theta functions are atomic with respect to the cluster atlas, then it would follow that they are a minimal basis in $\s{S}$, though possibly not the unique minimal basis in $\s{S}$.  
In fact, \cite[Conj. 14]{LLRZpnas} is stated for quantum greedy bases (which are defined for rank $2$ quantum cluster algebras) instead of  quantum theta bases, but we expect the two to be equal whenever both are defined --- the classical analog of this equivalence was proved in \cite{CGMMRSW}.  

We also expect our bases to agree with those defined for quantum cluster algebras from surfaces as in \cite{AK} --- quantum Laurent positivity for these was proved in \cite[Thm 5.7]{Al} for disks and then in \cite[Thm. 1.1]{CKKO} in general.  We further expect agreement with a quantum version of the bracelets basis of \cite{MSW} --- this would prove \cite[Conj. 4.20]{Thurst}.   On the other hand, the ``triangular bases,'' as considered for some quantum cluster algebras in \cite{Qin}, seem to be different from ours --- \cite[Conj. 13]{LLRZpnas} suggests that they might be the \textit{maximal} bases in $\s{S}$.

We note that our construction of the quantum theta functions was previously given in \cite{Man3}.  However, all results on positivity properties and results stating that these theta functions really are related to the cluster algebras (as opposed to just the formal version $\wh{\s{V}}_q^S$ without the cluster structure) are new.  Since our quantum theta functions agree with those of \cite{Man3}, we know that they satisfy the following:
\begin{enumerate}
    \item Fock and Goncharov's quantum Frobenius map conjecture \cite[Conj. 4.8.6]{FG1}, cf. \cite[Thm. 4.3]{Man3};
    \item The structure constants of the form $\alpha(p_1,p_2;0)$ and $\alpha(p_1,p_2,p_3;0)$ are sufficient to uniquely determine all the structure constants \cite[Thm. 2.16]{Man3};
    \item All the structure constants and the scattering diagrams are determined by certain refined counts of tropical disks (which may have negative multiplicities).
\end{enumerate}
The classical versions of (2) and (3) above are used in \cite{ManFrob} to show that the classical theta functions are determined by certain descendant log Gromov--Witten numbers --- in fact, all the classical structure constants admit Gromov--Witten theoretic descriptions \cite{KY,GSInt2}, and as \cite{KY} notes, the classical strong positivity follows from the unobstructedness of the relevant moduli of curves.

The anticipation of such mirror symmetry type statements was the motivation for the development of the techniques used in \cite{GHKK}, cf. \cite{KS,GPS,GS11,CPS,WCS,GHK1,GHK3,GHS}.  It would be interesting to find similar geometric descriptions in the quantum setting.  Indeed, it is known that the tropical disk counts relevant to rank $2$ scattering diagrams are related to certain real/open Gromov--Witten invariants \cite{Mikq} or higher-genus Gromov--Witten invariants \cite{Bou,Bou2}, and we expect versions of such invariants to describe the quantum theta function multiplication more generally.

On the other hand, from the perspective of representation theory, it would be desirable to understand positivity in terms of a monoidal categorification of the (quantum) cluster algebra, cf. \cite{HernLec,Nak,KKKO}.  In this perspective, the (quantum) cluster algebra is identified with the Grothendieck ring of a monoidal category, and we expect the theta functions to correspond to simple objects of the category --- the atomicity statement of Theorem \ref{MainThm} is strong evidence for this.  Coefficients of the multiplication should then arise as graded dimensions of vector spaces decomposing the tensor products of these simple objects, cf. \cite[\S 1.5.4]{AlBPS}. Our arguments proceed by considering categorifications of the functions appearing on walls in scattering diagrams.  We leave for the future the project of categorifying the theta functions themselves (and hence the cluster algebra).

\subsection{Positivity for scattering diagrams from DT theory}
As in \cite{GHKK}, we deduce our results regarding (quantum) cluster algebras from positivity results for scattering diagrams, which in turn are derived from positivity results in DT theory.  We now explain how all of these positivity results fit together.

First, note that the scattering diagram $\f{D}^{\s{V}_{q}}$ comes about in a particular way.  Namely, we start with a simple, inconsistent scattering diagram $\f{D}_{\In}^{\s{V}_{q}}$ for which the only walls are full hyperplanes $v_i^{\omega\perp}$, and the functions on these walls are $\EE(-z^{v_i})$.  See \S \ref{pleth_sec} for the definition of $\EE$ in terms of plethystic exponentials/quantum dilogarithms.  To this initial scattering diagram of incoming walls, we then add outgoing walls  to obtain the consistent scattering diagram $\f{D}^{\s{V}_q}$ (see Theorem \ref{KSGS}).  

The important thing to note is that if all new walls carry functions of the form $\EE(-p_{\f{d}}(t)z^{v_{\f{d}}})$ for $p_{\f{d}}(t)\in \Zt$ positive, then the coefficients $c_{p,v,\sQ}(t)$ and $\alpha(p_1,p_2;p)$ appearing in Theorem \ref{MainThm} are positive elements of $\Zt$ as well.  By the lemmas of \S \ref{AdPres}, the same statement is true if we replace the property ``positive'' by ``bar-invariant,'' ``pre-Lefschetz,''  or ``Lefschetz.''  The constant polynomial $p(t)=1$ enjoys all four of these properties, and so the functions on the initial walls satisfy all of these properties.  So whichever of ``positive,'' ``bar-invariant,'' ``pre-Lefschetz,'' or ``Lefschetz'' we have chosen as the property we want to prove that theta functions possess, our job becomes to prove the corresponding \textit{preservation} theorem for scattering diagrams, stating that when we start adding walls as in Theorem \ref{KSGS} in order to make $\f{D}_{\In}^{\s{V}_q}$ consistent, we can proceed by only adding walls with functions that are positive/bar-invariant/pre-Lefschetz/Lefschetz.  So far we are just recalling the strategy of \cite{GHKK} for proving positivity of classical theta functions (the other three properties have no analog in the non-quantum setting).  

The precise form of our preservation theorem (regarding positive, bar-invariant, and pre-Lefschetz incoming functions) is given in Theorem \ref{PosScat}.  For now we are concentrating on the positivity result, since that is the most fundamental.  Similarly to \cite[\S C.3]{GHKK}, via perturbation arguments and various other tricks, the preservation theorem reduces to the case in which $\f{D}^{\s{V}_q}_{\In}$ consists of only two walls, with attached functions $\EE(-t^{m_1}z^{v_1})$ and $\EE(-t^{m_2}z^{v_2})$.  In the classical case, there are no powers of $t$ to worry about, and as in \cite[\S C.3, Step IV]{GHKK}, one can further reduce to the case in which $\omega(v_1,v_2)=1$, and so there is just one two-wall base case.  It is straightforward to prove by hand that the consistent scattering diagram built out of this base case is positive --- one only needs to add a single outgoing wall (by the pentagon identity), and its attached function is positive.  

In the quantum case, we have to allow $\omega(v_1,v_2)$ to take any value in $\mathbb{Z}_{\geq 1}$, and furthermore the problem turns out to depend fundamentally on the parities of $m_1$ and $m_2$.  So we are confronted by four \textit{infinite} families of base cases, almost all of which are extremely complicated (with infinitely many outgoing walls, almost always dense in some region of nonzero measure --- see for instance Examples \ref{dense_example} and \ref{Kronecker_example} in \S \ref{denseSec}).

Our approach to proving positivity in all of these base cases is to translate the problem into the DT theory of quiver representations.  The links between scattering diagrams, quantum dilogarithm identities and DT theory is well established (see \cite{Keller}, \cite{WCS} or \cite{Bridge} for background).  We show that the positivity problem in the base cases involves the DT invariants of a quiver depending on $n=\omega(v_1,v_2)$, and the parities of $m_1$ and $m_2$.  In the case in which $m_1$ and $m_2$ are both even, this turns out to be a problem that has already been solved: positivity amounts to proving positivity of the DT invariants for the $n$-Kronecker quiver, which has been established as a very special case of a positivity result of Meinhardt and Reineke regarding acyclic quivers \cite{MeiRei} using refined DT theory.

In all three of the other infinite families, the situation is not so simple: the quiver for which we have to establish positivity of DT invariants involves loops, i.e. it is non-acyclic.  To deal with these cases, we prove a generalization of the result of \cite{MeiRei}, namely the positivity of refined DT invariants for arbitrary quivers.  This is the content of our Theorem \ref{DT_pos_thm}, which we restate here:
\begin{thm}[Positivity for refined DT invariants]
\label{DTpos_restate}
Let $\zeta\in \mathbb{R}^{Q_0}$ be a generic stability condition.  Let $S_{\theta}^{\zeta}\subset \bb{Z}_{\geq 0}^{Q_0}$ be the submonoid of dimension vectors of slope $\theta$ with respect to $\zeta$.  There is an equality of generating series 
\begin{equation*}
\sum_{v\in \mathbb{Z}_{\geq 0}^{Q_0}}\chi_{t}(\HO(\Mst_v(Q),\mathbb{Q}))x^vt^{\chi_Q(v,v)}=\prod_{\infty\xrightarrow{\theta} -\infty}\prod_{0\neq v\in S_{\theta}^{\zeta}}\EE(f_v(t)x^v)
\end{equation*}
where the $f_v(t)\in\mathbb{Z}_{\geq 0}[t^{\pm 1}]$ have positive coefficients and are even or odd Laurent polynomials, depending on whether $\chi_Q(v,v)$ is odd or even, respectively, where $\chi_Q(\cdot,\cdot)$ is the Euler form for $Q$.
\end{thm}
This theorem turns out to be an easy consequence of the cohomological wall crossing and integrality theorems of \cite{DavMei}.  In fact, it follows from a very special case since we do not have to consider potentials (equivalently, we set $W=0$).  We hope that for the audience coming from the theory of cluster algebras, the absence of potentials and the resulting absence of arguments involving vanishing cycles and monodromic motives/mixed Hodge modules from e.g. \cite{Nagao, Efi, DavPos} will make this paper more approachable.  Using Theorem \ref{DTpos_restate}, we can prove preservation of positivity for all four infinite families of 2-wall base cases (cf. Corollary \ref{2WallPos}) and deduce all of our other positivity results.

By combining the parity statement of Theorem \ref{DTpos_restate} with our positivity result for scattering diagrams and the consistency of Bridgeland's Hall algebra scattering diagrams \cite{Bridge}, we prove a general parity statement for scattering diagrams, Theorem \ref{parity_cor}.  This then implies the parity statement $c_{p,v,\sQ}(-t)=\pm c_{p,v,\sQ}(t)$ from Theorem \ref{MainThm}, cf. Proposition \ref{ParitProp}.

The strategy for proving preservation of the pre-Lefschetz property is the same as for the preservation of positivity.  The main difference in the proof is that the preservation result for the basic two wall case (Proposition \ref{2WallpL}) can be proved by translating it into the DT theory of acyclic quivers, allowing us to use the Meinhardt--Reineke theorem to settle it.  The pre-Lefschetz property implies positivity, and so this provides a second proof of all of the positivity results in this paper. 
\subsection{Logical structure of the paper}
We derive our main results as consequences of results on scattering diagrams and DT theory.  On the other hand, we have deferred all of the DT theory until after the proofs of Theorems \ref{MainThm} and \ref{MainThm2} at the end of Section \S \ref{QCA_sec}, in an effort to not over-burden the reader who is primarily interested in cluster algebras.  The result is that the logical structure of the paper is somewhat nonlinear, and so for the reader's benefit we have drawn a map of it below.  Since we have pared the paper down to its essentials in this diagram, we have ignored all results and parts of results regarding Lefschetz type/bar invariance/parity while drawing it.
\bigskip{\center 
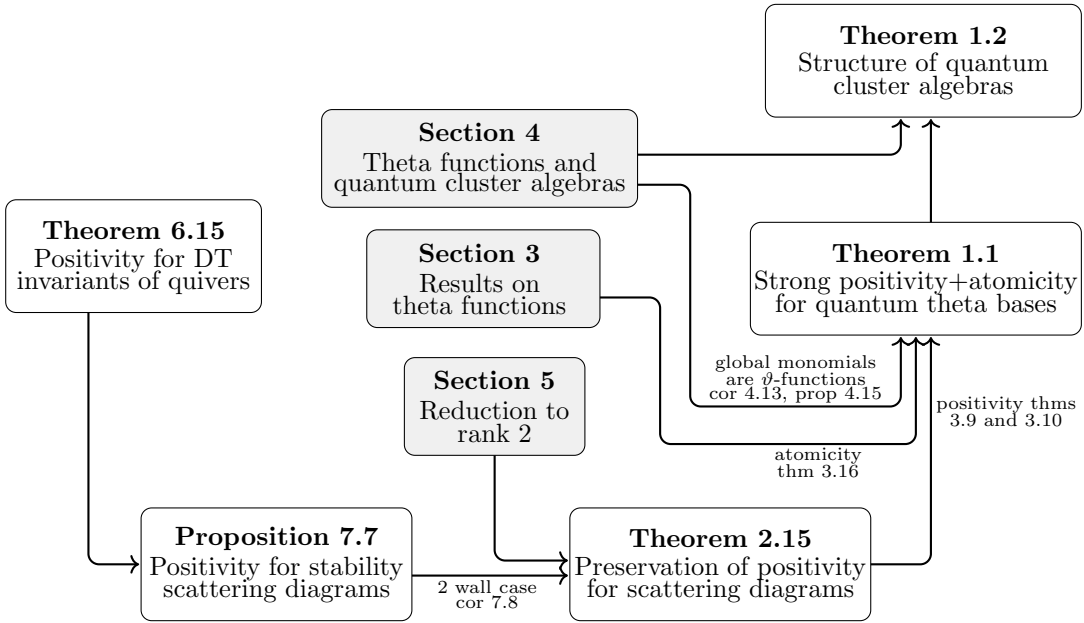
\begin{tikzpicture}
\definecolor{tempcolor}{RGB}{240,240,240};

\draw [rounded corners, fill=tempcolor] (4.9,5.5) rectangle (9.1,6.8);
\node at (7,6.5) {\textbf{Section \ref{QCA_sec}}};
\node at (7,6.1) {Theta functions and};
\node at (7,5.8) {quantum cluster algebras};
\draw [rounded corners,thick, ->](9.1,6.2) -- (12.6,6.2) -- (12.6,6.675);
\draw [rounded corners,thick, ->](9.1,5.8) -- (9.8, 5.8)--(9.8, 2.85)--(12.6,2.85)--(12.6,3.775);
\node at (11.2,3) {\scriptsize cor \ref{UpTheta}, prop \ref{midX}};
\node at (11.2,3.425) {\scriptsize global monomials};
\node at (11.2,3.225) {\scriptsize are $\vartheta$-functions};

\draw [rounded corners, fill=tempcolor] (5.5,3.9) rectangle (8.6,5.2);
\node at (7,4.9) {\textbf{Section \ref{theta_sec}}};
\node at (7,4.5) {Results on};
\node at (7,4.2) {theta functions};
\draw [rounded corners,thick, ->](8.6,4.3) -- (9.4,4.3) -- (9.4,2.35) -- (12.8,2.35) -- (12.8,3.775);
\node at (11.5,2.2) {\scriptsize atomicity};
\node at (11.5,2) {\scriptsize thm \ref{qAtomic}};

\draw [rounded corners, fill=tempcolor] (6,2.2) rectangle (8.4,3.5);
\node at (7.2,3.2) {\textbf{Section \ref{PosSect}}};
\node at (7.2,2.8) {Reduction to};
\node at (7.2,2.5) {rank 2};
\draw [rounded corners,thick, ->](7.2,2.2) -- (7.2,.8) -- (8.175,.8);

\draw [rounded corners] (10.8,6.7) rectangle (15,8.2);
\node at (12.9,7.8) {\textbf{Theorem \ref{MainThm2}}};
\node at (12.9,7.4) {Structure of quantum};
\node at (12.9,7.1) {cluster algebras};

\draw [rounded corners] (10.6,3.8) rectangle (15,5.3);
\node at (12.8,4.9) {\textbf{Theorem \ref{MainThm}}};
\node at (12.8,4.5) {Strong positivity+atomicity};
\node at (12.8,4.2) {for quantum theta bases};
\draw [thick, ->](13,5.3) -- (13,6.675);

\draw [rounded corners] (8.2,0) rectangle (12.2,1.5);
\node at (10.2,1.1) {\textbf{Theorem \ref{PosScat}}};
\node at (10.2,0.7) {Preservation of positivity};
\node at (10.2,0.4) {for scattering diagrams};
\draw [thick, rounded corners, ->] (12.2,.75) -- (13,.75) -- (13,3.775);
\node at (14,2.88) {\scriptsize positivity thms};
\node at (14,2.68) {\scriptsize \ref{qPosLoc} and \ref{qPosStrong}};

\draw [rounded corners] (2.5,0) rectangle (6.1,1.5);
\node at (4.3,1.1) {\textbf{Proposition \ref{Bridge_scat}}};
\node at (4.3,0.7) {Positivity for stability};
\node at (4.3,0.4) {scattering diagrams};
\draw [thick, ->](6.1,.6) -- (8.175,.6);
\node at (7.1,.25) {\scriptsize cor \ref{2WallPos}};
\node at (7.1,.45) {\scriptsize 2 wall case};

\draw [rounded corners] (.7,4.1) rectangle (4.1,5.6);
\node at (2.4,5.2) {\textbf{Theorem \ref{DT_pos_thm}}};
\node at (2.4,4.8) {Positivity for DT};
\node at (2.4,4.5) {invariants of quivers};
\draw [thick, rounded corners, ->] (1.8,4.1) -- (1.8,.75) -- (2.475,.75);
\end{tikzpicture}}
\medbreak

\subsection{Acknowledgements} During the writing of the paper, both authors were supported by the starter grant ``Categorified Donaldson--Thomas theory'' No. 759967 of the European Research Council. BD was also supported by a Royal Society university research fellowship.  We would like to thank Dylan Allegretti, Pierrick Bousseau, Geiss--Leclerc--Schr{\"o}er, Sean Keel, Bernhard Keller, Lang Mou, Fan Qin, and Dylan Rupel for useful discussions.  BD is especially grateful to Sean Keel for the opportunity to visit UT Austin and for his patient explanations there.  Finally, we thank the anonymous referee for their careful reading of the manuscript and many helpful suggestions.

\section{Quantum scattering diagrams}
\label{QSD_sec}

\subsection{Laurent polynomials}\label{Section-Notation}

\subsubsection{Notation}
Let $\kk$ be an arbitrary characteristic $0$ field, or more generally, any commutative ring containing $\bb{Q}$.  Denote $\kt\coloneqq \kk[t^{\pm 1/D}]$ for a fixed positive integer $D$ (we do not worry about the ambiguity of $D$ in this notation because one could instead just take $\kt$ to include all fractional powers of $t$). 

For $a\in \bb{Z}$, we denote
\begin{align}\label{at}
    (a)_t\coloneqq t^a-t^{-a} \in \bb{Z}[t^{\pm 1}],
\end{align}
and when $a\in \bb{Z}_{\geq 1}$ we denote
\begin{align}\label{at2}
    [a]_t\coloneqq \frac{(a)_t}{(1)_t}=\frac{t^a-t^{-a}}{t-t^{-1}}=t^{-a+1}+t^{-a+3}+\ldots+t^{a-3}+t^{a-1}\in \bb{Z}[t^{\pm 1}].
\end{align}
We will use $t$ throughout to denote the quantum deformation parameter, avoiding $q^{1/2}$, as the meaning of this parameter is ambiguous\footnote{See e.g. \cite[Rem 4.25]{DMSS} for a discussion of this point.} up to a sign. 
We say $p(t)$ is even if $p(t)=p(-t)$ and odd if $p(t)=-p(-t)$.  We define the parity function on the set of polynomials satisfying $p(-t)=\pm p(t)$ by
  \[
  \parit(p(t))=\begin{cases}0& \textrm{if }p(t) \textrm{ is even}\\
  1& \textrm{if }p(t) \textrm{ is odd.}\end{cases}
  \]
Note that $\parit([a]_t)\equiv a+1$ (mod $2$).

We say that a polynomial $p(t)\in\mathbb{Z}[t^{\pm 1}]$ is of \textbf{Lefschetz type} if it is a sum of the polynomials $[n]_t$ for $n\in\mathbb{Z}_{\geq 1}$.  For $n\in\mathbb{Z}$ we define
\begin{equation}\label{plPolys}
\pl_n(t)=\begin{cases}t^n&\textrm{if }n\textrm{ is even}\\ t^{n-1}[2]_t=t^{n-2}+t^n &\textrm{if }n \textrm{ is odd.} \end{cases}
\end{equation}
and we say that $p(t)$ is \textbf{pre-Lefschetz}, or just \textbf{pL}, if it is a sum of polynomials $\pl_n(t)$ for $n\in\mathbb{Z}$.   A Lefschetz type polynomial is pL.  We say that $p(t)$ is \textbf{bar-invariant} if it is invariant under the transformation $t\mapsto t^{-1}$.

\subsubsection{The Lefschetz property}
The origin of Lefschetz polynomials in geometry is as follows.  Let $X$ be an irreducible variety, and assume that for some complex of constructible sheaves $\mathcal{F}$ on $X$ there are isomorphisms
\[
\HO^i(X,\mathcal{F})\cong \HO^{2\dim(X)-i}(X,\mathcal{F})^*
\]
for all $i\in\mathbb{Z}$.  The classical example would be given by setting $\mathcal{F}$ to be the constant sheaf $\mathbb{Q}_X$, where $X$ is smooth projective --- then the above isomorphism is given by Poincar\'e duality.  We normalise the cohomology of $X$ with coefficients $\mathcal{F}$ by writing
\[
V=\HO(X,\mathcal{F})_{\vir}\coloneqq \HO(X,\mathcal{F})[\dim(X)].
\]
Then the Poincar\'e polynomial
\begin{align}\label{chit}
\chi_t(V)\coloneqq\sum_{i\in \mathbb{Z}}\dim(V^i )t^{i}
\end{align} is bar-invariant.  Assume that, moreover, $V$ carries an operator $\mathscr{L}\colon V\rightarrow V[-2]$, i.e. a linear map $\mathscr{L}$ of cohomological degree 2, such that for every $i\geq 1$ the linear map $\mathscr{L}^i\colon V^{-i}\rightarrow V^i$ is an isomorphism.  Then $\chi_t(V)$ is of Lefschetz type.  In the classical example mentioned above, this isomorphism is provided by the hard Lefschetz theorem, and the operator $\mathcal{L}$ is the Lefschetz operator, defined by capping with a hyperplane section.

\subsection{Quantum tori}\label{qtori}
\subsubsection{The quantum torus algebra}\label{qtoralg}
Given any finite-rank lattice $L$, we write $L_{\bb{Q}}\coloneqq L\otimes \bb{Q}$ and $L_{\bb{R}} \coloneqq  L\otimes \bb{R}$.  We denote the dual pairing between $L$ and its dual lattice $L^*\coloneqq \Hom(L,\bb{Z})$ by $\langle \cdot,\cdot\rangle$.  We call a nonzero vector $n\in L$ \textbf{primitive} if it is not a positive integer multiple of any other element of $L$.  If $m=km'$ for $k\in \bb{Z}_{\geq 0}$ and $m'\in L$ primitive, we call $k$ the \textbf{index} of $m$, denoted $|m|$.

Let $L$ denote a lattice of finite rank $r$, equipped with a $\frac{1}{D}\bb{Z}$-valued skew-symmetric bilinear form $\omega$ for some $D\in \bb{Z}_{\geq 1}$ (we take $\kk_t\coloneqq \kk[t^{\pm 1/D}]$ for this $D$).  We consider the \textbf{quantum torus algebra}
\begin{align}\label{qtor}
    \kk_t[L]\coloneqq \kt\langle z^v\colon v\in L\rangle /\langle z^{a}z^{b}=t^{\omega(a,b)}z^{a+b} \colon a,b\in L\rangle.
\end{align}
This is a unital associative algebra, which is noncommutative if $\omega$ is nonzero.  Note that $\kk_t[L]$ is $L$-graded.  We may write $\kk_t[L]$ as $\kk^{\omega}_t[L]$ if $\omega$ is not clear from context.

Now fix a strongly convex rational polyhedral cone $\sigma\subset L_{\bb{R}}$.  Let $L^{\oplus}\coloneqq \sigma\cap L$, and $L^+\coloneqq L^{\oplus} \setminus \{0\}$. 
 For each $k\in \bb{Z}_{\geq 1}$, let
\begin{align}\label{kNplus}
kL^+\coloneqq \{v_1+\ldots+v_k\in L\colon v_i\in L^+ \mbox{ for each } i=1,\ldots,k\}.
\end{align} 
We consider the formal Laurent series ring $\kk_{\sigma,t}\llb L\rrb\coloneqq \kk_t[L]\otimes_{\kk_t[L^{\oplus}]} \kk_t\llb L^{\oplus}\rrb$, i.e., $\kk_{\sigma,t}\llb L\rrb$ is spanned by formal sums of the form $\sum_{v\in L^{\oplus}} a_vz^{v+w}$ for $w\in L$ and coefficients $a_v\in \kt$. Where the cone $\sigma$ is clear from the context or not explicitly important, we will omit it from the notation, writing $\kk_t\llb L\rrb\coloneqq \kk_{\sigma,t}\llb L\rrb$.

For $p\in L$, we say $f\in \kk_{\sigma,t}\llb L\rrb$ is $p$-pointed if it has the form \begin{align*}
    f=z^p\left(1+\sum_{v\in L^+} a_vz^v\right)
\end{align*} 
with $a_v\in \kt$.

\subsubsection{Topological structure}\label{TopStr}

In general, the theta functions will only provide topological bases, roughly meaning that we may need infinite linear combinations of the theta functions in order to reach every element of the algebras we consider.  We carefully review this notion here.

For the purposes of the following, we endow $\kk_t$ with the discrete topology.  We make $\kk_t[L]$ into a topological $\kk_t$-module as follows: let $\f{m}$ denote the unique maximal monomial ideal of $\kk_t[L^{\oplus}]$, i.e., $\f{m}$ is spanned by monomials $z^v$ for $v\in L^+$, and so $\f{m}^k$ is spanned by monomials $z^v$ for $v\in kL^+$.  Let $\f{U}_0$ denote the set of $\kk_t$-submodules $U\subset \kk_t[L]$ which are closed under left (equivalently, right) multiplication by elements of $\f{m}$ and such that for every $f\in \kk_t[L]$, there exists some $k$ such that $\f{m}^kf\subset U$.  Let $\f{U}$ denote the set of cosets of these submodules, i.e., $\f{U}\coloneqq \{a+U\colon a\in \kk_t[L],\medspace U\in \f{U}_0\}$.  We consider the topology on $\kk_t[L]$ with basis of open sets $\f{U}$.  Note that $\f{U}_0$ is a neighborhood basis for the point $0$, and the intersection of the sets in $\f{U}_0$ is zero, so the topology is Hausdorff.
 
For each $S\subset L$, denote
\begin{align*}
    U_S\coloneqq \bigoplus_{v\in S} \kt\cdot z^v\subset \kk_t[L].
\end{align*}
Then the $\kk_t$-submodule $U_S\subset \kk_t[L]$ is open if $S$ satisfies the following two conditions:
\begin{enumerate}
    \item $S$ is closed under addition by elements of $L^{\oplus}$;
    \item
    For all $v\in L$ there exists some $k$ such that $v+kL^+\subset S$.
\end{enumerate}
It will sometimes be useful to let $L^{\pm}$ denote the $\bb{Z}$-span of $L^{\oplus}$ in $L$, and let $\pi\colon L\rar\?{L}\coloneqq L/L^{\pm}$ denote the projection.

\begin{dfn}\label{CauchyDef}
Let $R$ be a commutative ring, which we endow with the discrete topology, and let $A$ be a topological $R$-module with a basis of open neighborhoods of $0$ consisting of $R$-submodules.  By a \textbf{Cauchy sequence} in $A$, we mean a sequence $(f_k)_k$ of elements of $A$, indexed by $k\in\bb{Z}_{\geq 1}$, such that for any open neighborhood $U$ of $0\in A$, we have $f_{k_1}-f_{k_2}\in U$ for all sufficiently large $k_1,k_2$.  Two Cauchy sequences $(f_k)_k$ and $(g_k)_k$ are then equivalent if, for each open neighborhood $V$ of $0$, we have $f_k-g_k\in V$ for all sufficiently large $k$.  The assumption that $0$ has a basis of neighborhoods that are modules ensures that this is an equivalence relation. 
\end{dfn}

Note that the Cauchy completion $\wh{A}$ of $A$, i.e., the set of all equivalence classes of Cauchy sequences in $A$, naturally inherits a topological $R$-module structure from $A$. More precisely, $(f_{k})_k+(g_k)_k=(f_k+g_k)_k$, and $r\cdot (g_k)_k=(r\cdot g_k)_k$; a subset $\wh{U}\subset \wh{A}$ is open if it is the set of all classes of Cauchy sequences equivalent to sequences of elements of some open set $U\subset A$.

\begin{prop}\label{CauchyProp}
With the above topology, $\kk_t\llb L\rrb$ is the Cauchy completion of $\kk_t[L]$.
\end{prop}
\begin{proof}
Given an element $f=\sum_{v\in L^{\oplus}}a_{v}z^{v+w}\in z^w\kk_t\llb L^{\oplus}\rrb\subset \kk_t\llb L\rrb$, $a_v$ denoting coefficients in $\kt$, we form the sequence of Laurent polynomials $f_k=\sum_{v\in L^{\oplus}}a_{v,k}z^{v+w}\in \kk_t[L]$ via the rule  
\[
a_{v,k}=\begin{cases} a_{v}&\textrm{if }v\notin w+ kL^+\\ 0 &\textrm{otherwise.}\end{cases}
\]
Let $U\subset \kk_t[L]$ be an open set containing 0.  Then there exists a $k_U\in \bb{Z}_{\geq 1}$ such that $z^{w}\mathfrak{m}^{k_U}\subset U$, from which it follows that $f_{k_1}-f_{k_2}\in U$ for all $k_1,k_2\geq k_U$, i.e., $(f_k)_k$ is indeed Cauchy.  Since every element of $\kk_t\llb L\rrb$ is a finite sum of such elements $f$, we see that $\kk_t\llb L\rrb$ is contained in the Cauchy completion of $\kk_t[L]$.

For the reverse implication we break the proof into two steps, firstly reducing to the special case in which $\sigma$ is a top dimensional cone.  Let $(g_k)_k$ be a sequence in $\kk_t[L]$, and let $S\subset L$ be defined as the set of $v$ such that the $z^v$-coefficient of $g_k$ is nonzero for some $k$.  Let $\?{S}\coloneqq \pi(S)$, let $s\colon \?{S}\rar S$ be a section of $\pi|_S$, and let $U\coloneqq U_{L\setminus (s(\?{S})-L^{\oplus})}$.  Then $U$ is open.  If $\?{S}$ is infinite, then since each $g_k$ is a Laurent polynomial, for all $k$ we must be able to find a $k'>k$ such that 
$g_k-g_{k'}\not\in U$.  In particular, if $(g_k)_k$ is Cauchy, then the set $S$ must be finite.  So since $\kk_t\llb L\rrb$ is closed under finite sums, it suffices to deal with the case $|\?{S}|=1$.  Then by restricting to $U_{S+L^{\pm}}$, we can assume that $\sigma$ is a top-dimensional cone.

So we make this assumption.  For $v\in L$ we define 
\[
\lvert v\lvert_-=\min(k\in\mathbb{Z}_{\geq0} \colon v=a-b \textrm{ with }a\in L^{\oplus}\textrm{ and }b\in kL^+).
\]
For instance $v\in L^{\oplus}$ if and only if $\lvert v\lvert_-=0$.

Now let $(g_k)_k$ be a Cauchy sequence, and assume for a contradiction that there does not exist a $w$ with $g_{k}\in z^{w}\kk_t[L^{\oplus}]$ for each $k$.  Then for all $N\in \mathbb{Z}_{\geq 0}$, there exists a $v_N\in L$ with $\lvert v_N\lvert_- > N$ such that the $z^{v_N}$ coefficient of $g_k$ is nonzero for some $k$ 
--- otherwise we could take 
\[
w=\sum_{v\in NL^+\setminus (N+1)L^+}v,
\]
and then we would have $g_k\in z^{-w}\kk_t[L^{\oplus}]$ for all $k$.  
After fixing such a choice of $v_N$ for each $N$, let 
\[
S=\bigcap_{N\geq 0} (L\setminus (v_N-L^{\oplus})). 
\]
We claim that $U_S$ is open.  The fact that $S+L^+\subset S$ is clear, so to prove the claim, we fix $v\in L$ and show that there exists a $k$ such that $v+kL^+\subset S$.

If $N\geq |v|_-$, then $|v_N|_->|v|_-$, and so $v_N\notin v+L^+$.  Since there are only finitely many $N\in \bb{Z}_{\geq 0}$ with $N<|v|_-$, and since $\bigcap_k (v+kL^+)=\emptyset$, there exists some fixed $k$ such that $v_N\notin v+kL^+$ for any $N$.  Then for this $k$, we have $(v+kL^+)\cap (v_N-L^{\oplus})=\emptyset$ for all $N$, and so $v+kL^+\subset S$, as desired.  On the other hand, by construction, there is no $N$ such that $g_n-g_{n'}\in U_S$ for all $n,n'\geq N$, contradicting the assumption that $(g_k)_k$ is Cauchy. 
\end{proof}

For $A$ a topological $R$-module, we say that $S$ \textbf{topologically spans} $A$ over $R$ if, for every $f\in A$ and each open neighborhood $U$ of $0$, there exists a finite linear combination $g=\sum_{s\in S} a_s s$ with $a_s\in R$ such that $f-g\in U$.  Equivalently, $A$ is the closure of the space spanned by $S$ in the usual sense.  One says that $S$ forms a \textbf{topological $R$-module basis} for $A$ if no proper subset of $S$ topologically spans $A$.  In the case in which $\{0\}$ is open, this is just the usual notion of a basis.

One easily sees the following:
\begin{lem}\label{zbas}
The set $\{z^v\}_{v\in L}$ is a topological $\kt$-module basis for $\kk_t\llb L\rrb$.
\end{lem}

More generally, we have the following:
\begin{lem}\label{PointedBasis}
Let $\{f_p\}_{p\in L}$ be a subset of $\kk_t\llb L\rrb$, indexed by $L$, such that $f_p$ is $p$-pointed for each $p$.  Then $\{f_p\}_{p\in L}$ is a topological $\kk_t$-module basis for $\kk_t\llb L\rrb$.
\end{lem}
\begin{proof}
Let $$f=\sum_{v\in L^{\oplus}} a_vz^{w+v}\in z^w\kk_t\llb L^{\oplus}\rrb \subset \kk_t\llb L\rrb.$$
Modulo $U_{w+L^+}$, the pointedness implies that $f_w\equiv z^w$, and so $f\equiv g_0\coloneqq a_0f_w$.  Now suppose that there exist constants $b_v\in \kk_t$ for $v\in L^{\oplus}\setminus kL^+$ such that $$f\equiv g_k\coloneqq \sum_{v\in L^{\oplus}\setminus kL^+} b_vf_{w+v}\quad\quad \mbox{(modulo $U_{w+kL^+}$).}$$  Then there exist additional constants $b_v\in \kk_t$ for $v\in kL^+\setminus (k+1)L^+$ such that
\begin{align*}
    f-g_k&\equiv \sum_{v\in w+(kL^+\setminus(k+1)L^+)} b_vz^{w+v} \quad\quad \mbox{(modulo $U_{w+(k+1)L^+}$)}\\
    &\equiv \sum_{v\in w+(kL^+\setminus(k+1)L^+)} b_vf_{w+v} \quad\quad \mbox{(modulo $U_{w+(k+1)L^+}$)},
\end{align*} 
the last equivalence using that $f_{w+v}$ is $(w+v)$-pointed.  Thus, $$f\equiv \sum_{v\in L^{\oplus}\setminus (k+1)L^+} b_vf_{w+v}\quad\quad \mbox{(modulo $U_{w+(k+1)L^+}$)}.$$
Proceeding in this way we recursively define $b_v\in\kk_t$ for all $v\in L^{\oplus}$.  For every open neighborhood $U$ of $0$, we have that $U_{w+kL^+}\subset U$ for sufficiently large $k$, hence $f\equiv \sum_{v\in L^{\oplus}\setminus kL^+} b_vf_{w+v}$ modulo $U$.  Since every element of $\kk_t\llb L\rrb$ is a finite sum of elements like $f$, we see that $\{f_p\}_{p\in L}$ topologically spans.

We now show that no proper subset of $\{f_p\}_{p\in L}$ topologically spans $\kk_t\llb L\rrb$.  Say for a contradiction that there is such a proper subset $\{f_p\}_{p\in L'}$ for some $L'\subset L$.  Let $p\in L\setminus L'$.  We define $S=L\setminus (p-L^{\oplus})$.  Then $U_S$ is open.  Assume that we can pick a finite subset $S'\subset L'$ such that $(\sum_{s\in S'}a_sf_s)-z^p\in U_S$ for some set of nonzero $a_s\in\kk_t$.  If $s\notin p-L^{\oplus}$, then $f_s\in U_S$ by definition and $s$-pointedness, and so we can assume that $S'\subset p-L^{\oplus}$.  In fact, $S'\subset p-L^+$ since $p\notin L'$.  Now pick $s\in S'$ minimal with respect to the partial ordering: $l\leq l'$ if $l'\in l+L^{\oplus}$.  Then the $z^s$ coefficient of $(\sum_{s\in S'}a_sf_s)-z^p$ is $a_s\neq 0$, and since the $z^s$ coefficient of every element of $U_S$ is zero, this gives a contradiction. 
\end{proof}
We note that the same argument applies in the classical limit.

\begin{rem}\label{formal_mult}
Let $\{f_p\}_{p\in L}$ be as in Lemma \ref{PointedBasis}.  By the lemma, for $p,p'\in L$ we may write 
\[
f_pf_{p'}=\sum_{v\in p+p'+L^{\oplus}}a(p,p';v)f_v
\]
for certain $a(p,p';v)\in \kk_t$.  Let $A$ denote the $\kk_t$-algebra spanned by formal sums $\sum_{v\in w+L^{\oplus}}b_v\mathscr{F}_v$, for various $w\in L$ and some collection of symbols $\{\mathscr{F}_v\}_{v\in L}$, with structure constants given by $a(p,p';v)$, i.e. for coefficients $b_v,c_v\in \kk_t$ we define
\[
\left(\sum_{v\in w+L^{\oplus}}b_v\mathscr{F}_v \right)\left( \sum_{v\in w'+L^{\oplus}}c_v\mathscr{F}_v\right)=\sum_{v\in w+w'+L^{\oplus}}\sum_{\substack{u\in w+L^{\oplus}\\u'\in w'+L^{\oplus}}}a(u,u';v)b_u c_{u'}\mathscr{F}_{v}
\]
where the sum is well defined since unless $v-(u+u')\in L^{\oplus}$ the structure constant $a(u,u';v)$ is zero.  Then as in the first part of the proof of Proposition \ref{CauchyProp}, a formal power series $\sum_{v\in w+L^{\oplus}}b_v\mathscr{F}_v$ defines a Cauchy convergent series in $\kk_t\llb L\rrb$ by replacing each $\mathscr{F}_v$ with $f_v$, and the induced map $A\rightarrow\kk_t\llb L\rrb$ is an isomorphism of algebras via the proof of Lemma \ref{PointedBasis}.  
\end{rem}

\subsubsection{The quantum torus Lie algebra}\label{qTorLie}

Let $L_0\subset L$ be a sublattice such that $\omega(u,v)\in \bb{Z}$ whenever $u\in L_0$ and $v\in L$.  Define $L_0^+\coloneqq L_0\cap L^+$. 
Consider the Lie algebra over $\kk[t^{\pm 1}]$
\begin{align*}
    \f{g}'\coloneqq \bigoplus_{v\in L_0^+} \kk[t^{\pm 1}]\cdot z^v
\end{align*}
with bracket $[z^a,z^b]=(\omega(a,b))_t z^{a+b}$, using the notation of \eqref{at}.  Note that this is a Lie subalgebra of the commutator algebra of $\kk_t[L]$.  Now define the \textbf{quantum torus Lie algebra} $\f{g}=\f{g}_{L_0^+,\omega}$ to be the Lie subalgebra of $\f{g}'\otimes_{\kk[t^{\pm 1}]} \kk(t)$ generated over $\kk[t^{\pm 1}]$ by the elements of the form 
\begin{align}\label{hatz}
    \hat{z}^v\coloneqq \frac{z^v}{(|v|)_{t}}=\frac{z^v}{t^{|v|}-t^{-|v|}}
\end{align} 
for $v\in L_0^+$, recalling that $|v|$ denotes the index of $v$.  We define an $L_0^+$-grading on $\f{g}$ by letting the degree $v$ part be
\begin{align*}
    \f{g}_v\coloneqq \f{g}\cap (\kk(t)\cdot z^v) \subset \f{g}.
\end{align*}
Here, being $L_0^+$-graded means that for all $a,b\in L_0^+$, we have $[\f{g}_a,\f{g}_b]\subset \f{g}_{a+b}$.  
Note that $\f{g}$ and $\omega$ satisfy the following compatibility condition:
\begin{align}\label{skewCondition}
\mbox{if~}\omega(a,b)=0, \mbox{~then~} [\f{g}_{a},\f{g}_{b}]=0.    
\end{align}

The adjoint action of $\f{g}'$ induces an action of the Lie algebra $\f{g}$ on the noncommutative torus $\kk_t[L]$ by $L$-graded $\kt$-algebra derivations:
\begin{align*}
    \ad_{\hat{z}^a}(z^b) = \frac{(\omega(a,b))_t}{(|a|)_{t}} z^{a+b} \in \kk_t[L].
\end{align*}
Here, we use our assumption that $\omega(a,b)\in \bb{Z}$, as well as the observation that if $a$ and $c$ are positive integers with $a|c$, then 
\begin{align}\label{tquotient}
    \frac{(c)_{t}}{(a)_{t}} = \sum_{k=1}^{c/a} t^{-c+(2k-1)a}\in \kk[t^{\pm 1}]\subset \kt.
\end{align}

Let $\f{g}^{\geq k}\subset \f{g}$ denote the Lie algebra ideal spanned by the summands $\f{g}_v$ with $v\in kL_0^+$.  We define $\f{g}_k\coloneqq \f{g}/\f{g}^{\geq (k+1)}$, and
\begin{align*}
    \hat{\f{g}}\coloneqq \varprojlim_k \f{g}_k.
\end{align*}

By the Baker--Campbell--Hausdorff formula, we can apply $\exp$ to $\f{g}_k$ and $\hat{\f{g}}$ to obtain unipotent (respectively, pro-unipotent) algebraic groups $G_k$ and $\hat{G}=\varprojlim_k G_k$, respectively.  Note that $\hat{G}$ comes with a projection to $G_k$ for each $k\in \bb{Z}_{\geq 1}$, and that we naturally have $$\hat{G}\subset 1+\f{m}\kk[t^{\pm 1}]\llb L^{\oplus}\rrb\subset \kk[t^{\pm 1}]\llb L^{\oplus}\rrb\subset \kk_t\llb L^{\oplus}\rrb \subset \kk_t\llb L\rrb$$ for $\f{m}$ the maximal monomial ideal of $\kk_t[L^{\oplus}]$.

The (adjoint) action of $\f{g}$ on $\kk_t[L]$ by $L$-graded $\kt$-derivations induces an action of $\hat{\f{g}}$ on $\kk_t\llb L\rrb$ by (topologically)\footnote{A topological $L$-grading on a topological $R$-module $A$ is an $L$-grading for an $R$-submodule which topologically spans $A$.  For convenience, we may just refer to this as an $L$-grading.} $L$-graded $\kt$-derivations, hence an (Adjoint) action of $\hat{G}$ on $\kk_t\llb L\rrb$ by $\kt$-algebra automorphisms: for $g=\exp(b)\in \hat{G}$ and $a\in \kk_t\llb L\rrb$, the action is $\Ad_g(a)=gag^{-1}=\exp([b,\cdot])(a)$.

\begin{lem}\label{AdHomeo}
For each $g\in \hat{G}$, $\Ad_g\colon \kk_t\llb L\rrb\rar \kk_t\llb L\rrb$ is a homeomorphism. 
\end{lem}
\begin{proof}
Let $\hat{\f{m}}$ be the subspace of $\kk_t\llb L\rrb$  topologically spanned by the elements $z^n$ for $n\in L^+$.   Since $\f{g}$ is graded by $L_0^+\subset L^+$ and the action respects the grading, the action of $\hat{G}$ must preserve $\hat{\f{m}}$, i.e., $\Ad_g(\hat{\f{m}})\subset \hat{\f{m}}$. Since the same holds for $g^{-1}$, we have $\Ad_{g^{-1}}(\hat{\f{m}})\subset \hat{\f{m}}$, so $\hat{\f{m}}\subset \Ad_{g}(\hat{\f{m}})$. Hence, $\Ad_g(\hat{\f{m}})=\hat{\f{m}}$.

Let $\f{S}$ be the set of open sets obtained by taking closures of images of the sets in $\f{U}_0$ under the inclusion $\kk_t[L]\rightarrow \kk_t\llb L\rrb$.  So $\f{S}$ is a basis of open neighborhoods of $0\in\kk_t\llb L\rrb$, and translates of sets in $\f{S}$ form a basis for all open sets in $\kk_t\llb L\rrb$.  Now, a set $U\ni 0$ is in $\f{S}$ if and only if it is a $\kk_t$-submodule of $\kk_t\llb L\rrb$ such that for all $k\in \mathbb{Z}_{\geq 0}$ we have $\wh{\f{m}}^kU\subset U$, and, for all $f\in \kk_t\llb L\rrb$, there exists a $k_0\in \bb{N}$ such that $k>k_0$ implies $\wh{\f{m}}^kf\subset U$.  Applying $\Ad_g$ to both sides, the first containment is equivalent to $\wh{\f{m}}\Ad_g(U)\subset \Ad_g(U)$ and so $\Ad_g(U)$ satisfies the first condition for being in $\f{S}$.  The second containment is equivalent to $\wh{\f{m}}^k\Ad_g(f)\subset \Ad_g(U)$.  Then since $\Ad_g$ is a bijection, this is equivalent to the condition that for all $f\in \kk_t\llb L\rrb$, there exists some $k_0\in \bb{N}$ such that $k>k_0$ implies $\wh{\f{m}}^k f\subset \Ad_g(U)$.  But this is equivalent to the second condition for $\Ad_g(U)$ to be in $\f{S}$.  So $\Ad_g$ takes sets in $\f{S}$ to sets in $\f{S}$.  Additionally, $\Ad_g(U+f)=\Ad_g(U)+\Ad_g(f)$ and so $\Ad_g$ takes translates (under addition) of sets in $\f{S}$ to other translates of sets in $\f{S}$.  Applying the same argument to $\Ad_{g^{-1}}=\Ad_g^{-1}$ yields the claim.
\end{proof}

\subsubsection{Classical limits}\label{scS}
One can take \textbf{classical limits} of everything above.  In place of $\kk_t[L]$, one considers the usual torus algebra $\kk[L]$.  Taking $t^{1/D}\mapsto 1$ defines a surjective graded algebra homomorphism
\begin{align*}
    \lim_{t\rar 1}\colon  \kk_t[L] \rar \kk[L].
\end{align*}

Similarly, we can define the (classical) torus Lie algebra,
\begin{align*}
 \f{g}^{\cl}\coloneqq \bigoplus_{v\in L_0^+} \kk\cdot z^v,
\end{align*}
equipped with the bracket
\begin{align*}
    [z^a,z^b]\coloneqq \omega(a,b) z^{a+b}.
\end{align*}
Then $t\mapsto 1$ and $\hat{z}^v\mapsto \frac{1}{|v|}z^v$ determines a graded Lie algebra homomorphism
\begin{align*}
    \lim_{t\rar 1}\colon \f{g}\rar \f{g}^{\cl}.
\end{align*}
Indeed,
\begin{align*}
    \lim_{t\rar 1} [\hat{z}^a,\hat{z}^b]_{\f{g}}=\lim_{t\rar 1}\left(\frac{(\omega(a,b))_{t}(|a+b|)_{t}}{(|a|)_{t}(|b|)_{t}} \hat{z}^{a+b}\right) = \frac{\omega(a,b)}{|a||b|}z^{a+b}=[z^a/|a|,z^b/|b|]_{\f{g}^{\cl}}.
\end{align*}
We similarly see that these maps $\lim_{t\rar 1}$ respect the action of $\f{g}$ on $\kk_t[L]$: 
\begin{align*}
        \lim_{t\rar 1} \ad_{\hat{z}^a}(z^b)=\lim_{t\rar 1}\left(\frac{(\omega(a,b))_{t}}{(|a|)_{t}} z^{a+b}\right) = \frac{\omega(a,b)}{|a|}z^{a+b}=[z^a/|a|,z^b]_{\f{g}^{\cl}}.
\end{align*}
Since the morphisms $\lim_{t\rar 1}$ respect the $L$-gradings, they induce morphisms $\lim_{t\rar 1}$ on the respective completed Lie algebras, giving $\lim_{t\rar 1} \colon  \kk_t\llb L\rrb \rar \kk\llb L\rrb$, $\lim_{t\rar 1} \colon \hat{\f{g}}\rar \hat{\f{g}}^{\cl}$, and $\lim_{t\rar 1} \hat{G}\rar \hat{G}^{\cl}$, each intertwining the corresponding actions on $\kt\llb L\rrb$ and $\kk\llb L\rrb$.

\subsection{Scattering diagrams}\label{WallScat}

\subsubsection{Walls and scattering diagrams}

We assume as above that $L$ is a lattice equipped with a skew-symmetric rational bilinear form $\omega$, $L_0\subset L$ is a sublattice satisfying $\omega(u,v)\in \bb{Z}$ whenever $u\in L_0$ and $v\in L$, and $\sigma\subset L_{\mathbb{R}}$ is a strongly convex rational polyhedral cone.  We let $\mathfrak{g}$ be any $L_0^+$-graded Lie algebra satisfying \eqref{skewCondition}.  One defines $\f{g}_k$, $\hat{\f{g}}$, $G_k$, and $\hat{G}$ analogously to the case in which $\f{g}=\f{g}_{L_0^+,\omega}$.

Given a primitive $v\in L_0^+$ (primitive as an element of $L_0$), we define a Lie subalgebra
\begin{align*}
    \f{g}_v^{\parallel}\coloneqq \prod_{k\in \bb{Z}_{>0}} \f{g}_{kv}\subset \hat{\f{g}},
\end{align*}
and a pro-unipotent algebraic subgroup $G_v^{\parallel}\coloneqq \exp(\f{g}_v^{\parallel})\subset \hat{G}$.  Note that $\f{g}$ satisfying \eqref{skewCondition} implies that $\f{g}_v^{\parallel}$ and $G_v^{\parallel}$ are Abelian.  Consider the map $\omega_1\colon L\rar L^*$, $v\mapsto \omega(v,\cdot)$.  We denote 
\begin{align*}
    v^{\omega\perp}\coloneqq \omega_1(v)^{\perp}=\{a\in L_{\bb{R}}\colon \omega(v,a)=0\} \subset L_{\bb{R}}.
\end{align*}

A \textbf{wall} in $L_{\bb{R}}$ over $\f{g}$ is a pair $(\f{d},f_{\f{d}})$, where\footnote{Our definition of walls is consistent with that of \cite[Def. 2.2]{Man3} (for $m_{\f{d}}$ there taken to be $\omega(v_{\f{d}},\cdot)\in L^*$ in our setup) but slightly different from that of \cite[Def. 1.4]{GHKK}.  In the convention used by \cite{GHKK}, the support $\f{d}$ would live in $L^*_{\bb{R}}$ and would be parallel to $v_{\f{d}}^{\perp}$.  We find that our approach is more natural for the quantum setup, especially for defining the scattering diagrams $\f{D}^{\s{X}_q}$ of \S \ref{Initial}.   See \cite[Rmk. 2.3]{Man3} or \S \ref{ComparisonTricks} for an explanation of how to pass between the different conventions.\label{ConventionsFootnote}}
\begin{itemize}
    \item $f_{\f{d}}\in G_{v_{\f{d}}}^{\parallel}$ for some primitive $v_{\f{d}}\in L_0^+\setminus \ker(\omega_1)$, and
    \item $\f{d}$ is a closed, convex (but not necessarily strongly convex), $(r-1)$-dimensional, rational-polyhedral affine cone in $L_{\bb{R}}$ parallel to $v_{\f{d}}^{\omega\perp}$.
\end{itemize} 
Here, by a closed, convex, rational polyhedral affine cone in $L_{\bb{R}}$, we mean an intersection of sets of the form $\{v\in L_{\bb{R}}\colon \langle u_i,v\rangle \geq \langle u_i,v_0\rangle\}$ for some fixed (but possibly not uniquely determined) $v_0\in L_{\bb{R}}$ called the apex of the cone, and some finite collection $\{u_i\}$ of elements of $L^*$.  Note that the primitive vector $v_{\f{d}}\in L_0^+$ is uniquely determined by $f_{\f{d}}$ (assuming $f_{\f{d}}\neq 1$).  We call $-v_{\f{d}}$ the \textbf{direction} of the wall.  A wall is called \textbf{incoming} if $v+\bb{R}_{\geq 0}v_{\f{d}}\subset \f{d}$ for all $v\in \f{d}$, and it is called \textbf{outgoing} otherwise.

A \textbf{scattering diagram} $\f{D}$ in $L_{\bb{R}}$ over $\f{g}$ 
is a set of walls in $L_{\bb{R}}$ over $\f{g}$ such that for each $k >0$, there are only finitely many $(\f{d},f_{\f{d}})\in \f{D}$ with $f_{\f{d}}$ not projecting to $1$ in $G_k$.  For each $k\in \bb{Z}_{>0}$ and scattering diagram $\f{D}$ over $\f{g}$, let $\f{D}_k\subset \f{D}$ denote the finite scattering diagram over $\f{g}$ consisting of the $(\f{d},f_{\f{d}})\in \f{D}$ for which $f_{\f{d}}$ is nontrivial in the projection to $G_k$.  We may also view $\f{D}_k$ as a scattering diagram over $\f{g}_k$ by taking the projections of the functions $f_{\f{d}}$.

\subsubsection{Path-ordered products}
\label{pop_sec}
We will sometimes denote a wall $(\f{d},f_{\f{d}})$ by just $\f{d}$.  Denote $\Supp(\f{D})\coloneqq  \bigcup_{\f{d}\in \f{D}} \f{d}$, and \begin{align}\label{joints}
\Joints(\f{D})\coloneqq  \bigcup_{\f{d}\in \f{D}} \partial \f{d} \cup \bigcup_{\substack{\f{d}_1,\f{d}_2\in \f{D}\\
		               \dim \f{d}_1\cap \f{d}_2 = r-2}} \f{d}_1\cap \f{d}_2. 
\end{align}

Consider a smooth immersion $\gamma\colon [0,1]\rar L_{\bb{R}}\setminus \Joints(\f{D})$ with endpoints not in $\Supp(\f{D})$ which is transverse to each wall of $\f{D}$ it crosses.  Let $(\f{d}_i,f_{\f{d}_i})$, $i=1,\ldots, s$, denote the walls of $\f{D}_k$ crossed by $\gamma$, and say they are crossed at times $0<t_1\leq \ldots \leq t_s<1$, respectively (if $t_i=t_{i+1}$, then the fact that each $G_v^{\parallel}$ is Abelian will imply that the ordering of these two walls does not affect \eqref{WallCross2} and therefore does not matter).  Define 
\begin{align}\label{WallCross}
\theta_{\f{d}_i}\coloneqq f_{\f{d}_i}^{\sign \omega(v_{\f{d}_i},-\gamma'(t_i))} \in G_k.
\end{align}
Let 
\begin{align}\label{WallCross2}
    \theta_{\gamma,\f{D}}^k\coloneqq \theta_{\f{d}_s} \cdots \theta_{\f{d}_1}\in G_k,
\end{align} and define the path-ordered product:
\begin{align}\label{pathprod}
\theta_{\gamma,\f{D}}\coloneqq  \varprojlim_k \theta_{\gamma,\f{D}}^k \in \hat{G}.
\end{align}

\begin{dfn}
Two scattering diagrams $\f{D}$ and $\f{D}'$ are  \textbf{equivalent} if $\theta_{\gamma,\f{D}} = \theta_{\gamma,\f{D}'}$ for each smooth immersion $\gamma$ as above (assuming transversality with the walls of both $\f{D}$ and $\f{D}'$).  A scattering diagram $\f{D}$ is \textbf{consistent} if each $\theta_{\gamma,\f{D}}$ depends only on the endpoints of $\gamma$, or equivalently, if $\theta_{\gamma,\f{D}}=\id$ whenever $\gamma$ is a closed path.
\end{dfn}

\begin{egs}\label{EquivEx}
~

\begin{enumerate}
    \item If $(\f{d},f_1)$ and $(\f{d},f_2)$ are two walls of a scattering diagram with the same support and direction, then replacing these two walls with the single wall $(\f{d},f_1f_2)$ results in an equivalent scattering diagram.
    \item Replacing a wall $(\f{d},f_{\f{d}})\in \f{D}$ with a pair of walls $(\f{d}_i,f_{\f{d}})$, $i=1,2$, such that $\f{d}_1\cup \f{d}_2=\f{d}$ and $\codim(\f{d}_1\cap \f{d}_2)=2$ produces an equivalent scattering diagram.
    \item One says that $x\in L_{\bb{R}}$ is \textbf{general} if it is contained in at most one hyperplane of the form $u^{\perp}$ for $u\in L^*$.  For $\f{D}$ a scattering diagram in $L_{\bb{R}}$ over $\f{g}$ and $x\in L_{\bb{R}}$ general, denote $f_{x,\f{D}}\coloneqq \prod_{\f{d}\ni x} f_{\f{d}}\in \hat{G},$ where the product is over all walls $(\f{d},f_{\f{d}})\in \f{D}$ with $\f{d}\ni x$.  Note that $\f{g}$ satisfying \eqref{skewCondition} and $x$ being general ensures that these $f_{\f{d}}$ commute, so $f_{x,\f{D}}$ is well-defined.  One easily sees (cf. \cite[Lem. 1.9]{GHKK}) that two scattering diagrams $\f{D}$ and $\f{D}'$ in $L_{\bb{R}}$ over $\f{g}$ are equivalent if and only if $f_{x,\f{D}}=f_{x,\f{D}'}$ for all general $x\in L_{\bb{R}}$.
\end{enumerate}
\end{egs}

\subsubsection{Operations on scattering diagrams}
\label{ScatOperations}
We introduce a few tools that will prove useful throughout the paper in passing between (consistent) scattering diagrams.  As above, we retain our standing assumptions on $L,L_0,\omega,\sigma$ and allow $\f{g}$ to be any $L_0^+$-graded Lie algebra satisfying \eqref{skewCondition}.

Let $\pi\colon \mathfrak{g}\rightarrow \mathfrak{g}'$ (or $\pi\colon \hat{\f{g}}\rar \hat{\f{g}}'$) be a morphism of $L_0^+$-graded Lie algebras satisfying \eqref{skewCondition}.  We denote by $\exp\pi\colon \hat{G}\rightarrow\hat{G}'$ the induced morphism of pro-unipotent algebraic groups.  Let $\f{D}=\{(\f{d}_j,f_j)\}_{j\in J}$ be a scattering diagram in $L_{\mathbb{R}}$ over $\f{g}$.  Then we define $\pi_*\f{D}=\{  (\f{d}_j,\exp\pi(f_j))\}_{j\in J}$.  This gives a scattering diagram in $L_\mathbb{R}$ over $\mathfrak{g}'$, as we have not altered the gradings of the functions or the underlying cones for the walls.  The scattering diagram $\pi_*\f{D}$ is consistent if $\f{D}$ is.

\begin{prop}
Let $\f{D}=\{(\f{d}_j,f_j)\}_{j\in J}$ be a scattering diagram in $L_{\mathbb{R}}$ over $\f{g}$.  We define the \textbf{opposite} scattering diagram $\f{D}^{\opp}=\{(\f{d}_j,f_j)\}_{j\in J}$, a scattering diagram in $L^-_{\mathbb{R}}$ over $\f{g}^{\opp}$ (i.e., $\f{g}$ with its bracket negated), where $L^-_{\bb{R}}$ is identified with $L_{\bb{R}}$ and is endowed with the form $-\omega$.  Then $\f{D}$ is consistent if and only if $\f{D}^{\opp}$ is, and the subset $J_{\In}\subset J$ parameterizing incoming walls is the same for both scattering diagrams.
\end{prop}
\begin{proof}

The fact that the above pairs really define walls is immediate from the definitions, as is the statement regarding incoming walls, since the underlying cones of the walls, as well as their directions, are identified.

Now let $\gamma$ be a smooth immersion as in \S \ref{pop_sec}.  Then 
\begin{align*}
\theta^k_{\gamma,\f{D}^{\opp}}=&\theta^{-1}_{\f{d}_s}\cdot_{\opp}  \theta^{-1}_{\f{d}_{s-1}}\cdots_{\opp} \theta^{-1}_{\f{d}_{1}}\\=&
\theta^{-1}_{\f{d}_{1}}\cdots \theta^{-1}_{\f{d}_s}\\=
&(\theta_{\f{d}_{s}}\cdots \theta_{\f{d}_1})^{-1}
\end{align*}
and the statement regarding consistency follows.
\end{proof}

\begin{cor}
\label{reverse_cor}
Let $\f{D}=\{(\f{d}_j,f_j)\}_{j\in J}$ be a scattering diagram in $L_{\mathbb{R}}$ over $\f{g}$.  Then $\f{D}^-\coloneqq \{(\f{d}_j,f_j^{-1})\}_{j\in J}$ is a scattering diagram in $L^-_{\bb{R}}$ over $\mathfrak{g}$, with the same incoming walls as $\f{D}$, which is consistent if and only if $\f{D}$ is.
\end{cor}
\begin{proof}
There is an isomorphism
\begin{align*}
    s\colon &\f{g}^{\opp}\rightarrow \f{g}\\
    &a\mapsto -a
\end{align*}
and $\f{D}^-=s_*(\f{D}^{\opp})$.
\end{proof}
It will prove extremely useful to be able to take direct images of scattering diagrams along morphisms of lattices as well as graded Lie algebras, and so we generalize the definition of $\pi_*$ given above.  We say that a scattering diagram is \textbf{saturated} if for every wall $(\f{d},f)$ in $\f{D}$, the cone $\f{d}$ is closed under addition of elements in $\ker(\omega_1)$.  It is easy to check that every consistent scattering diagram is equivalent to a saturated scattering diagram.
\begin{lem}\label{QuotientScat}
Consider $L,L_0,\omega$ as above, and assume moreover that we have a lattice $\wt{L}$ with a sublattice $\wt{L}_0$ and a $\bb{Q}$-valued skew-symmetric form $\wt{\omega}$ on $\wt{L}$. Let $\pi\colon \wt{L}\rar L$ be a map of lattices such that $\wt{\omega}=\pi^* \omega$ and $\pi_{\bb{R}}\colon \wt{L}_{\bb{R}}\rightarrow L_{\bb{R}}$ is surjective.  Let $\wt{\sigma}\subset \wt{L}_{\bb{R}}$ be a strongly convex rational polyhedral cone such that $\pi(\wt{\sigma})\subset \sigma$.  Set $\wt{L}_0^+=(\wt{L}_0\cap \wt{\sigma})\setminus \{0\}$.

Let $\f{g}$ and $\wt{\f{g}}$ be $L^+$ and $\tilde{L}^+$-graded Lie algebras satisfying \eqref{skewCondition} with respect to $\omega$ and $\wt{\omega}$ respectively.  We assume that we are given a morphism of Lie algebras $\varpi\colon \wt{\f{g}}\rightarrow \f{g}$, such that $\varpi(\wt{\f{g}}_v)\subset \f{g}_{\pi(v)}$ for all $v\in \wt{L}_0^+$ (where we write $\f{g}_0=0$).  We denote by $\exp\varpi$ the induced morphism of the pro-unipotent algebraic groups associated to the completions of $\wt{\f{g}}$ and $\f{g}$. 

If 
$\wt{\f{D}}\coloneqq \{(\f{d}_j,f_{j})\}_{j\in J}$ is a scattering diagram in $\wt{L}_{\bb{R}}$ over $\wt{\f{g}}$, then 
\begin{align}\label{piD}
    \f{D}\coloneqq (\pi,\varpi)_*(\wt{\f{D}})\coloneqq \{(\pi(\f{d}_j),\exp\varpi(f_{j}))\}_{j\in J}
\end{align}
is a scattering diagram in $L_{\bb{R}}$ over $\f{g}$.  If $\wt{\f{D}}$ is consistent and saturated, then so is $\f{D}$. 
If $(\f{d}_j, f_j)$ is incoming then so is the corresponding wall in $\f{D}$, and if $\wt{\f{D}}$ is saturated, then the reverse implication also holds.
\end{lem}

\begin{proof}
Firstly, since $\pi(k\wt{L}_0^+)\subset kL_0^+$, there is a morphism of quotient groups $\exp\varpi\colon\wt{G}_k\rightarrow G_k$ (this is also why the induced map on the completions is well-defined) and the finiteness condition for $\f{D}$ to be a scattering diagram is satisfied.  

Let $\f{d}_j$ be a wall of $\wt{\f{D}}$.  Since $\pi_{\bb{R}}$ is surjective, the codimension of the span of $\pi(\f{d}_k)$ is at most one.  On the other hand, $\pi(v_{\f{d}_j})\notin \ker(\omega_1)$ since $v_{\f{d}_j}\notin\ker(\wt{\omega}_1)$, and so $\pi(\f{d}_j)\subset \pi(v_{\f{d}_j})^{\omega\perp}$ has codimension at least one, i.e. it has the required codimension: one.

Now, let $\gamma\subset L_{\bb{R}}$ be a path for which $\theta_{\gamma,\f{D}}$ is defined.  Saturation of the walls of $\wt{\f{D}}$ implies that they are invariant under translation by elements of $\ker(\pi)\otimes \bb{R}$, and so any smooth lift $\wt{\gamma}'$ of $\gamma'$ to $\wt{L}_{\bb{R}}$ will satisfy $\exp \varpi(\theta_{\wt{\gamma}',\wt{\f{D}}})=\theta_{\gamma',\f{D}}$.  It follows that consistency of $\wt{\f{D}}$ implies consistency of $\f{D}$, as desired.

Let $(\f{d}_j,f_j)$ be a wall in $\wt{\f{D}}$.  Then there is a (possibly non-unique) $\alpha_j\in \wt{L}_{\bb{R}}$ such that $\f{d}_j+\alpha_j$ has the origin as an apex.  Then the cone $\pi(\f{d}_j)+\pi(\alpha_j) \subset L_{\bb{R}}$ has the origin as an apex as well.  If $(\f{d}_j,f_j)$ is incoming, then $v_{\f{d}_j}\in \f{d}_j+\alpha_j$, and so $\pi(v_{\f{d}_j})\in \pi(\f{d}_j + \alpha_j) =  \pi(\f{d}_j)+\pi(\alpha_j)$, meaning that the corresponding wall in $\f{D}$ is incoming.

Conversely, suppose that the wall in $\f{D}$ indexed by $j$ is incoming, so $\pi(v_{\f{d}_j})\in \pi(\f{d}_j)+\pi(\alpha_j)$. Then there exists a $\tilde{v}\in \f{d}_j+\alpha_j$ such that $\pi(\tilde{v})-\pi(v_{\f{d}_j})=0$, and so $\tilde{v}-v_{\f{d}_j}\in \ker(\wt{\omega}_1)$.  Now if $\wt{\f{D}}$ is saturated, then $v_{\f{d}_j}\in\f{d}_j+\alpha_j$, so $\f{d}_j$ is incoming, as desired.
\end{proof}
Finally, as well as being able to take the above direct image of scattering diagrams, it will be useful to be able to take inverse images, at least along inclusions of lattices.
\begin{lem}
\label{RestrictScat}
Let $L,L_0,\omega,\sigma$ be as above.  Let $i\colon \wt{L}\hookrightarrow L$ be a saturated sublattice such that the projection of $\sigma$ along $L_{\bb{R}}\rightarrow L_{\bb{R}}/\wt{L}_{\bb{R}}$ is strongly convex, and define $\wt{\sigma}\coloneqq \sigma\cap \wt{L}_{\bb{R}}$, $\wt{L}_0\coloneqq L_0\cap \wt{L}$, $\wt{L}_0^+\coloneqq (\wt{L}_0\cap \wt{\sigma})\setminus \{0\}$, $\wt{\f{g}}\coloneqq \bigoplus_{v\in \wt{L}_0^{+}}\f{g}_v$,  and $\wt{\omega}=\omega\lvert_{\wt{L}}$.  We denote by
\[
p\colon \bigoplus_{v\in L_0^+}\f{g}_v\rightarrow \bigoplus_{v\in \wt{L}_0^+}\f{g}_v.
\]
the retraction of Lie algebras obtained by taking the quotient by the factors $\f{g}_v$ for $v\notin \wt{L}_0^+$.

Let $\f{D}\coloneqq \{(\f{d}_j,f_{j})\}_{j\in J}$ be a scattering diagram in $L_{\bb{R}}$ over $\f{g}$.  Denote by $J'\subset J$ those walls such that $v_{\f{d}_j}\in \wt{L}_0$ and $\codim_{\wt{L}_{\bb{R}}}(\wt{L}_{\bb{R}}\cap \f{d}_j) = 1$ (so in particular, $v_{\f{d}_j}\notin  \ker(\wt{\omega})$). We assume that $\codim_{\wt{L}_{\bb{R}}}(\wt{L}_{\bb{R}}\cap \partial \f{d}_j) \geq 2$ for each $j\in J'$ --- up to equivalence, this always holds for consistent $\f{D}$. Then
\[
i^*\f{D}\coloneqq \{(\f{d}_j\cap \wt{L}_{\bb{R}},\exp(p)(f_{j}))\}_{j\in J'}
\]
is a scattering diagram in $\wt{L}_{\bb{R}}$ over $\wt{\f{g}}$ which is consistent if $\f{D}$ is.  The incoming walls of $i^*\f{D}$ are given by those $j\in J'$ such that $(\f{d}_j,f_j)$ is an incoming wall of $\f{D}$.
\end{lem}
\begin{proof}
The proof of finiteness is as in the proof of Lemma \ref{QuotientScat}, and the fact that the cones in $i^*\f{D}$ have the right dimension follows from the assumptions, so $i^*\f{D}$ is indeed a scattering diagram.  The claim regarding incoming walls is immediate from the definition.

Let $\gamma$ be a smooth closed path $[0,1]\rar \wt{L}_{\bb{R}}\setminus \Joints(i^*\f{D})$ with base-points not in $\Supp(i^*\f{D})$ which is transverse to each wall of $i^*\f{D}$ it crosses, which we may moreover consider as a path in $L_{\bb{R}}$.  Fix a $k\geq 1$.  First, we may perturb $\gamma$ so that it avoids any walls $\f{d}$ of $\f{D}$ satisfying $\codim(\wt{L}_{\bb{R}}\cap\f{d})\geq 2$, without affecting the associated path ordered product $\theta^k_{\gamma,\f{D}}$.  We then perturb $\gamma$ in $L_{\bb{R}}$ to a smooth closed path $\gamma'$ which is transverse to all walls of $\f{D}_{k}$, has base-point in the complement of $\Supp(\f{D}_{k})$, crosses the same walls of $\{(\f{d}_j,f_j)\colon j\in J'\}\cap \f{D}_k$ as before in the same order as before, and crosses no walls $\f{d}_j$ with $v_{\f{d}_j}\in\ker(\wt{\omega}_1)$.  Then $\theta_{\gamma,i^*\f{D}}^k = \exp(p)(\theta_{\gamma',\f{D}}^k)=\exp(p)(1)=1$.  The consistency claim follows. 

We now consider the claim that, for consistent $\f{D}$ considered up to equivalence, the condition $\codim_{\wt{L}_{\bb{R}}}(\wt{L}_{\bb{R}}\cap \partial \f{d}_j) \geq 2$ is always satisfied for $j\in J'$.  Suppose $\f{d}_j$ is a wall failing this condition.  Then consistency around the corresponding joint of $\f{d}_j$ forces there to be additional such walls.  But for any two such walls $\f{d}_j,\f{d}_{j'}$, we have $\wt{\omega}(v_{\f{d}_j},v_{\f{d}_{j'}})=0$, hence $[f_j,f_{j'}]=0$ by \eqref{skewCondition}.  The only possibility then is that the collection of such walls can be glued (as in Example \ref{EquivEx}(2)) to remove the existence of such joints.
\end{proof}

\subsubsection{Existence and uniqueness of consistent scattering diagrams}\label{existence_sssec}\label{ComparisonTricks}

The following theorem is a fundamental result on scattering diagrams.  The two-dimensional case over $\f{g}^{\cl}$ was proved in \cite{KS}, while the generalization to higher dimensions (and more general affine spaces than just $L_{\bb{R}}$) was proved in \cite{GS11}.  An extension to more general Lie algebras, including the quantum torus Lie algebra $\f{g}_{L_0^+,\omega}$, is given by \cite[Prop. 3.2.6, 3.3.2]{WCS}, (cf. \cite[Thm. 1.21]{GHKK} for a review of this argument in the cluster setup).

\begin{thm}\label{KSGS}
Let $\f{D}_{\In}$ be a finite scattering diagram in $L_{\bb{R}}$ over $\f{g}$ whose only walls are incoming.  Then there is a unique-up-to-equivalence scattering diagram $\f{D}$, also denoted $\scat(\f{D}_{\In})$, such that $\f{D}$ is consistent, $\f{D} \supset \f{D}_{\In}$, and $\f{D}\setminus \f{D}_{\In}$ consists only of outgoing walls.
\end{thm}

As noted in Footnote \ref{ConventionsFootnote}, our definition of a wall is slightly different from that used in \cite{WCS,GHKK}, and our quantum torus Lie algebras are different from the Lie algebras used in \cite{GS11}, so we briefly explain how to recover Theorem \ref{KSGS} in our setup.  First, we note that the uniqueness statement follows from exactly the same argument used in the proof of \cite[Thm. 3.5]{CM} (just replace the $\Lambda_{\bb{R}}^{\vee}$ there with $\Lambda_{\bb{R}}$, which in our notation is $L_{\bb{R}}$ --- also, $\f{g}_k$ there is $\f{g}_{k-1}$ here).   On the other hand, the existence follows from taking the asymptotic scattering diagram of the perturbed scattering diagram $\f{D}_k^{\infty}$ constructed in \cite[\S 3.2.2]{Man3}.  We note that the assumption in \cite[Thm. 3.5]{CM} that walls of $\f{D}_{\In}$ are full hyperplanes is not needed for the uniqueness proof, and since incoming walls can always be extended to full affine hyperplanes by adding outgoing walls and then applying the equivalence of Example \ref{EquivEx}(2), we can assume this property holds when applying the existence proof.

Alternatively, one may use the operations of \S \ref{ScatOperations} to recover the existence result in our setup directly from that of \cite{WCS}: 

We apply what we call the principal coefficients trick, which involves essentially the same procedure used for defining cluster algebras with principal coefficients, cf. \S \ref{DinS}.  Let $v_1,\ldots,v_r$ be a basis for $L$, and consider the lattice  $L^{\circ}\coloneqq \bb{Z}\langle v_1,\ldots,v_r,f_1,\ldots,f_r\rangle$ for $f_1,\ldots,f_r$ some new formal generators.  Denote by $i_{\bb{R}}\colon L_{\bb{R}}\hookrightarrow L^{\circ}_{\bb{R}}$, the natural inclusion.  We extend $\omega$ to a skew-symmetric form $\omega^{\circ}$ on $L^{\circ}$ by defining $\omega^{\circ}(v_i,f_j)=\delta_{ij}$ and $\omega^{\circ}(f_i,f_j)=0$ for all $i,j$.  We set $\sigma^{\circ}=i_{\bb{R}}(\sigma)$.  The matrix for $\omega^{\circ}$ in this basis is $\left(\begin{matrix}A & B \\ C & D \end{matrix}\right)$, where $A,B,C,D$ are $r\times r$ matrices, $A$ being the matrix for $\omega$, $B$ the identity matrix, $C$ negative the identity, and $D$ the $0$-matrix.  Such a matrix has determinant $1$, so the map $\omega_1^{\circ}\colon L^{\circ}\rar (L^{\circ})^*$, $v\mapsto \omega^{\circ}(v,\cdot)$, is an isomorphism.  

Let $\f{D}_{\In}$ be as in Theorem \ref{KSGS}.  As noted above, we may assume that the initial walls are full affine hyperplanes, i.e. all the walls are of the form $(v_j^{\omega\perp},f_{\f{d}_j})$ for $v_j\in L^+$ indexed by $j\in J$.  We define $\f{D}_{\In}^{\circ}$ to be the scattering diagram in $L^{\circ}_{\bb{R}}$ over $\f{g}$ with walls given by $(i(v_j)^{\omega^{\circ}\!\perp},f_{\f{d}_{v_j}})$ for $j\in J$.  The scattering diagram $\f{D}^{\circ}_{\In}$ has only incoming walls.  Let $\scat(\f{D}_{\In}^{\circ})$ be a consistent scattering diagram in $L^{\circ}_{\bb{R}}$ over $\f{g}$ such that $\scat(\f{D}_{\In}^{\circ})\setminus \f{D}^{\circ}_{\In}$ only contains outgoing walls.  Then by Lemma \ref{RestrictScat} the scattering diagram $i^* \scat(\f{D}_{\In}^{\circ})$ is consistent, with incoming walls equal to 
\[
i^*\f{D}^{\circ}_{\In}=\f{D}_{\In}.
\]

So now for the existence proof, it suffices to deal with the cases where $\omega$ is unimodular, i.e., $\omega_1\colon L\rar L^*$ is an isomorphism.  For such cases, the two viewpoints on scattering diagrams can easily be seen to be equivalent as follows: given a wall parallel to $v_{\f{d}}^{\perp}$ and having direction $-\omega_1(v_{\f{d}})$ as in the \cite{GHKK} perspective, applying $\omega_1^{-1}$ to the support and direction of the wall yields a wall in our perspective.  Paths and the corresponding path-ordered products are similarly related using $\omega_1$ and $\omega_1^{-1}$ to pass between the two perspectives.  Thus, existence follows from the usual formulation of the result.

\subsection{The quantum and plethystic exponentials}\label{pleth_sec}

\subsubsection{The quantum exponential}
 For $\f{g}=\f{g}_{L_0^+,\omega}$ and $v\in L_0^+$, define the quantum dilogarithm
\begin{align}\label{Li}
    -\Li(-z^v;t)\coloneqq \sum_{k=1}^{\infty} \frac{(-1)^{k-1}}{k(t^k-t^{-k})} z^{kv}\in \hat{\f{g}}.
\end{align}
Note that $k$ divides the index $|kv|$, so $\frac{z^{kv}}{t^k-t^{-k}}=\frac{(|kv|)_t}{(k)_t}\hat{z}^{kv}$ has coefficient in $\kk[t^{\pm 1}]$ by \eqref{tquotient}, hence $-\Li(-z^v;t)$ is indeed in $\hat{\f{g}}$.  Consider the \textbf{quantum exponential} 
\begin{align*}
    \Psi_t(z^v)\coloneqq \exp(-\Li(-z^v;t))\in \hat{G}.
\end{align*}
$\Psi_t(z^v)$ is the plethystic exponential (defined in \S \ref{Sec:plethystic} below) of $\frac{z^vt}{1-t^2}$, implying that
\begin{align}\label{Psit}
    \Psi_t(z^v)=\prod_{k=1}^{\infty} \frac{1}{1+t^{(2k-1)}z^{v}}
\end{align}
(this equality is well-known; cf. \cite[Lem. 8]{Kir}).
 
Let $p\in L$, and let $\epsilon\coloneqq \sign \omega(v,p)$.  Using \eqref{Psit}, one easily checks that the adjoint action of $\Psi_t(z^v)^{\epsilon}$ on $z^p\in \kk_t[L]$ is given by
\begin{align}\label{AdPsi}
\Ad_{\Psi_t(z^v)^{\epsilon}}(z^p)=\Psi_t(z^v)^{\epsilon}z^p \Psi_t(z^v)^{-\epsilon} = z^p \prod_{k=1}^{|\omega(v,p)|} (1+t^{\epsilon(2k-1)}z^{v}).
\end{align}

We use the notation $(0)_t!\coloneqq 1$ and
\begin{align*}
    (k)_t!\coloneqq (k)_t(k-1)_t\cdots(2)_t(1)_t
\end{align*}
for $k\in \bb{Z}_{\geq 1}$.  For $a,k\in \bb{Z}_{\geq 0}$ with $a>k$, we consider the bar-invariant quantum binomial coefficient
\begin{align}\label{qbinom}
    \binom{a}{k}_t\coloneqq \frac{(a)_t!}{(k)_t!(a-k)_t!} \in \bb{Z}_{\geq 0}[t^{\pm 1}].
\end{align}
With this notation, the quantum analog of the binomial theorem gives
\begin{align}\label{qbinomthm}
   \prod_{k=1}^{a} (1+t^{\epsilon(2k-1)}x)=\sum_{k=0}^a t^{\epsilon ka}\binom{a}{k}_t x^k.
\end{align}
Combining \eqref{AdPsi} with \eqref{qbinomthm}, one obtains
\begin{align*}
    \Ad_{\Psi_t(z^v)^{\epsilon}}(z^p) = \sum_{k=0}^{|\omega(v,p)|} \binom{|\omega(v,p)|}{k}_t z^{kv+p}.
\end{align*}
One easily extends this to obtain
\begin{align}\label{tAdBinomial}
    \Ad_{\Psi_t(t^{a}z^v)^{\epsilon}}(z^p) = \sum_{k=0}^{|\omega(v,p)|} \binom{|\omega(v,p)|}{k}_t t^{ka} z^{kv+p}
\end{align}
for any $a\in \bb{Z}$.

\subsubsection{The plethystic exponential}\label{Sec:plethystic}

We generalize the above via the \textbf{plethystic exponential}.  Let $R=\mathbb{Z}\lcbs t\rcbs \llb x_1,\ldots x_n\rrb$ and let $\mathscr{I}\subset R$ be the maximal ideal generated by $x_1,\ldots,x_n$.  We write $K=\mathbb{Z}_{\geq 0}^n$ and for $v\in K$ we define $x^v=\prod x_i^{v_i}$ as usual.  The plethystic exponential is a group isomorphism from $\mathscr{I}$  (viewed as a group under addition) to $1+\mathscr{I}$, defined by 
\begin{align*}
    \Exp_{-t,x_1,\ldots,x_n}\left(\sum_{\substack{0\neq v\in K\\r\in \bb{Z}}} a_{vr} x^v(-t)^r\right) = \prod_{\substack{0\neq v\in K\\r\in \bb{Z}}} (1-x^v(-t)^r)^{-a_{vr}}.
\end{align*}
We denote by
\begin{equation}
    \label{Log_def}
    \Log_{-t,x_1,\ldots,x_n}\colon 1+\mathscr{I} \rightarrow \mathscr{I}
\end{equation}
the inverse isomorphism of groups.  The plethystic exponential can be equivalently defined (see \cite{Get96}) in terms of the Adams operators $\psi_k(f(x_1,\ldots x_n,-t))=f(x_1^k,\ldots x_n^k,(-t)^k)$ via
\[
\Exp_{-t,x_1,\ldots,x_n}(f(x_1,\ldots x_n,-t))=\exp\left(\sum_{k\geq 1}\psi_k(f(x_1,\ldots x_n,-t))/k\right).
\]

Given $f(x_1,\ldots x_n,t)\in \mathscr{I}$ we define 
\[\EE(f(x_1,\ldots x_n,t))\coloneqq \Exp_{-t,x_1,\ldots,x_n}\left(f(x_1,\ldots x_n,t)(t+t^3+t^5+\ldots)\right).\]

It follows from the definition that
\begin{align}\label{plus-prod}
    \EE(f(z^v,t)+g(z^v,t))=&\EE(f(z^v,t))\cdot \EE(g(z^v,t))\\
\label{EPsi}
\EE(-z^v)=&\EE(z^v)^{-1}=\Psi_t(z^v)\\
    \label{bosvfer}
\EE(f(z^v,t))=&\EE(f(t^{-j}z^v,t))_{z^{v}\mapsto t^{j}z^{v}}&\textrm{for }j\in 2\cdot\mathbb{Z}.
\end{align}
Note that \eqref{bosvfer} fails if we allow odd $j$.

Similarly, it is straightforward to compute that for any $a\in \bb{Z}$, we have
\begin{align}\label{EEeven}
    \EE(-t^{2a}z^v)=\EE(t^{2a}z^v)^{-1}=\Psi_t(t^{2a}z^v),
\end{align}
and
\begin{align}\label{EEodd}
    \EE(t^{2a+1}z^v)=\prod_{k=1}^{\infty} \frac{1}{1-t^{2a+2k}z^{v}}.
\end{align}
Using the identity $\frac{1}{1-x}=\prod_{r=0}^{\infty} (1+x^{2^r})$, \eqref{EEodd} becomes
\begin{align}\nonumber
    \EE(t^{2a+1}z^v) &= \prod_{k=1}^{\infty}\prod_{r=0}^{\infty} [1+(t^{2a+2k}z^{v})^{2^r}] \\
    &=\prod_{r=0}^{\infty}\prod_{k=1}^{\infty} [1+(t^{2^r})^{2k-1} (t^{2^r(2a+1)})z^{2^rv})] \nonumber \\
    &=\prod_{r=0}^{\infty}\Psi_{t^{2^r}}(t^{2^r(2a+1)}z^{2^rv})^{-1},\nonumber
\end{align}
i.e.,
\begin{align}\label{EEoddPsi}
    \EE(-t^{2a+1}z^v)=\prod_{r=0}^{\infty}\Psi_{t^{2^r}}(t^{2^r(2a+1)}z^{2^rv}).
\end{align}

\begin{lem}\label{log_barinv}
Let $p(t)\in\mathbb{Z}[t^{\pm 1}]$.  Then $\log(\EE(-p(t)z^v))$ is bar-invariant if and only if $p(t)$ is.
\end{lem}
\begin{proof}
First, assume that $p(t)$ is bar-invariant.  We can write $p(t)=\alpha(t)+\beta(t)$ where $\alpha(t)$ is odd and $\beta(t)$ is even, and both are bar-invariant.  Since 
\begin{align*}
\log\left[\EE(-\alpha(t)-\beta(t))\right]=&\log\left[\EE(-\alpha(t))\EE(-\beta(t))\right]\\
=&\log\left[\EE(-\alpha(t))\right]+\log\left[\EE(-\beta(t))\right],
\end{align*}
it suffices to prove the claim for both $\alpha(t)$ and $\beta(t)$ separately.  Write $\alpha(t)=\sum_i \alpha_it^i$.  Then by \eqref{EEoddPsi} we have
\begin{align*}
    \log[\EE(\alpha(t))]=&\log\left(\prod_{r=0}^{\infty}\prod_{i\in\mathbb{Z}\lvert   \alpha_i\neq 0}\Psi_{t^{2^r}}(t^{2^ri}z^{2^rv})^{\alpha_i}\right)\\
    =&-\sum_{r=0}^{\infty}\sum_{i\in\mathbb{Z}\lvert \alpha_i\neq 0}\alpha_i\Li(-t^{2^ri}z^{2^rv};t^{2^r})
\end{align*}
and the statement follows from
\[
\Li(-t^{2^ri}z^{2^rv};t^{2^r})_{t\mapsto t^{-1}}=\Li(-t^{-2^ri}z^{2^rv};t^{2^r})
\]
and $\alpha_i=\alpha_{-i}$.  The proof for $\beta(t)$ is similar, but simpler.

The reverse implication follows from the fact that the $z^v$ coefficient of $\log(\EE(-p(t)z^v))$ is precisely $p(t)/(1)_t$.
\end{proof}

\subsubsection{Characteristic function interpretation of the plethystic exponential}

The operations $\Exp$ and $\EE$ have a natural interpretation in terms of generating functions of graded vector spaces, which we now explain.  Let $K=\mathbb{Z}_{\geq 0}^n$ be as above.  Let $\mathcal{C}$ denote the Abelian category of $(K\oplus \mathbb{Z})$-graded vector spaces.  We consider the $\mathbb{Z}$-grading as a cohomological grading, and for $V\in \ob(\mathcal{C})$ and $v\in K$ we denote by $V_v$ the cohomologically graded vector space obtained by taking the $v$th graded piece with respect to the $K$-grading.  We denote by $\Vect_{K}^+$ the full subcategory of $\mathcal{C}$ containing those objects such that for every $i\in \mathbb{Z}$ and $v\in K$, the vector space $V_v^i$ is finite-dimensional, and for every $v\in K$ there is an $N_v\in \bb{Z}$ such that $V_v^i=0$ for all $i\leq N_v$.  We define the \textbf{characteristic function} of an object of $\Vect_K^+$ by 
\begin{equation}
    \label{char_fn_def}
    \chi_{K,t}(V)=\sum_{v\in K}\sum_{i\in \mathbb{Z}}\dim(V_v^i)x^vt^i\in R.
\end{equation}
Given an object $V\in \Vect^+_K$ we can form the free supercommutative algebra $\Sym(V)$ generated by $V$, i.e. the quotient of the free tensor algebra of $V$ by the two-sided ideal spanned by elements of the form
\begin{align*}
    a\otimes b -(-1)^{\alpha \beta}b\otimes a
\end{align*}
for $a$ and $b$ homogeneous elements with respect to the cohomological grading, of degrees $\alpha$ and $\beta$.  Decomposing $V=V^{\odd}\oplus V^{\even}$, with $V^{\odd}$ the summand of $V$ concentrated in odd cohomological degree, and $V^{\even}$ the summand concentrated in even cohomological degree, we have
\[
\Sym(V)\cong\CSym(V^{\even})\otimes \bigwedge(V^{\odd})
\]
where $\CSym(V^{\even})$ is the free commutative algebra generated by $V^{\even}$ and $\bigwedge(V^{\odd})$ is the free exterior algebra generated by $V^{\odd}$.

The object $\Sym(V)$ will again be an object of $\Vect_K^+$ if $V_0=0$.  This condition also implies that $\chi_{K,t}(V)\in\mathscr{I}$.  The plethystic exponential is uniquely determined by the facts that it is a group homomorphism and satisfies
\begin{equation}
\label{cat_Exp}
\Exp_{-t,x_1,\ldots x_n}(\chi_{K,t}(V))=\chi_{K,t}(\Sym(V))
\end{equation}
for $V\in\ob(\Vect_K^+)$ satisfying the condition that $V_0=0$.  We write $B\coloneqq \HO(\pt/\mathbb{C}^*,\mathbb{Q})[-1]=\mathbb{Q}[u]$ where $u^r$ has cohomological degree $2r+1$.  Then if $a\in V$ has degree $(v,i)$, i.e. cohomological degree $i$ and $K$-degree $v$, the element $a\otimes u^r\in V\otimes B$ has degree $(v,i+2r+1)$.  The tensor product $V\otimes B$ is again an object of $\Vect_K^+$.  Assuming that $V_0=0$, we may write
\begin{equation}
\label{cat_Expb}
\EE(\chi_{K,t}(V))=\chi_{K,t}\big(\Sym\big(V\otimes \HO(\pt/\mathbb{C}^*,\mathbb{Q})[-1]\big)\big).
\end{equation}

Following \cite[Sec. 6.1]{CoHA} we define the \textbf{quantum tropical vertex group} $G^{\qtrop}$ to be the closure of the subgroup of $\hat{G}$ generated by the elements $\EE(-t^a z^v)$, and we define the \textbf{quantum tropical vertex algebra} $\mathfrak{g}^{\qtrop}\subset \hat{\mathfrak{g}}$ to be the image of $G^{\qtrop}$ under $\log$.  By \eqref{EEeven} and \eqref{EEoddPsi}, $\mathfrak{g}^{\qtrop}$ is a Lie subalgebra of $\hat{\f{g}}$, considered as a Lie algebra over $\mathbb{Z}$.  Precisely, it is the closure of the Lie algebra generated by quantum dilogarithms, which in turn is a Lie subalgebra of our quantum torus Lie algebra $\hat{\f{g}}_{L_0^+,\omega}$.

\subsection{Main result for scattering diagrams} 

The following theorem, for which the proof will be completed in \S \ref{PosSect}--\S\ref{2wallSection}, is the quantum analog of \cite[Thm. 1.13]{GHKK} and is the key to all our other results.

\begin{thm}\label{PosScat}
Let $\f{D}_{\In}$ be a finite scattering diagram in $L_{\bb{R}}$ over $\f{g}=\f{g}_{L_0^+,\omega}$ of the form
\begin{align}\label{DIn}
\f{D}_{\In}=\left\{(v_i^{\omega\perp},\EE(-p_i(t)z^{v_i}))\colon v_i\in L_0^+,i=1,\ldots,s\right\},
\end{align}
where $p_i(t)\in\mathbb{Z}_{\geq 0}[t^{\pm 1}]$ for all $i$.
Then there is some $\f{D}$ representing the equivalence class of $\scat(\f{D}_{\In})$ such that the incoming walls of $\f{D}$ are precisely the walls \eqref{DIn}, and all walls of $\f{D}$ are of the form 
\begin{align}\label{Dwalls}
    (\f{d},\EE(-p_{\f{d}}(t)z^{v_{\f{d}}}))
\end{align}
for various cones $\f{d}\subset L_{\bb{R}}$, vectors $v_{\f{d}}\in L_0^+$, and Laurent polynomials $p_{\f{d}}(t)$ with positive integer coefficients. If each $p_i(t)$ is of pL type, then we can choose the $p_{\f{d}}(t)$ to be of pL type as well.  If each $p_i(t)$ is bar-invariant in addition to having positive integer coefficients, then we can take each $p_{\f{d}}(t)$ to be bar-invariant as well. 

In particular, by \eqref{EEeven} and \eqref{EEoddPsi}, we could alternatively find $\f{D}$ representing $\scat(\f{D}_{\In})$ such that the set of incoming walls of this $\f{D}$ forms a scattering diagram equivalent to $\f{D}_{\In}$, and such that all the scattering functions of $\f{D}$ have the form $\Psi_t(t^{2a}z^v)$ or $\Psi_{t^{2^r}}(t^{2^r(2a+1)}z^{2^rv})$ for various $a\in \bb{Z}$, $r\in \bb{Z}_{\geq 0}$, and $v\in L_0^+$.
\end{thm}
\begin{proof}
For the statements regarding $p_i(t)$ being positive or pL, we will in \S \ref{PosSect} inductively reduce to the case of just two initial walls.  Then Corollary \ref{2WallPos} and Proposition \ref{2WallpL} will prove the positivity and pL statements, respectively, for these two-wall cases.

For bar-invariance, recall from Lemma \ref{log_barinv} that $p(t)$ is bar-invariant if and only if $\log(\EE(-p(t)z^v))$ is bar-invariant.
Note that we can view $\f{g}=\f{g}_{L_0^+,\omega}$ as a Lie algebra over $\kk$ rather than $\kk[t^{\pm 1}]$, and then we can consider the Lie subalgebra $\f{g}^{\barr}$ (over $\kk$) consisting of the bar-invariant elements.  So the bar-invariance of the Laurent polynomials $p_i(t)$ implies that we can view $\f{D}_{\In}$ as a scattering diagram over $\f{g}^{\barr}$, and then applying Theorem \ref{KSGS}  over $\f{g}^{\barr}$ yields $\f{D}$ defined over $\f{g}^{\barr}$ as well.  Then each $p_{\f{d}}(t)$ must be bar-invariant, as desired.
\end{proof}

\begin{rmk}
In \cite[\S 8]{Bou2}, the polynomials $p_{\f{d}}(t)$ of \eqref{Dwalls} for certain scattering diagrams are interpreted as (higher-genus) log BPS invariants.  Theorem \ref{PosScat} thus recovers the integrality result of \cite[Thm. 33]{Bou2}, and furthermore, proves positivity for these log BPS invariants, as well as bar-invariance and pL type (we note that positivity for acyclic cases was proved in \cite[Prop. 41]{Bou2}).  Theorem \ref{parity_cor} also yields a parity result for these log BPS invariants.  
\end{rmk}

Theorem \ref{PosScat} combined with \eqref{tAdBinomial} will imply that the coefficients of the monomials attached to our broken lines in \S \ref{BrokenLines} are Laurent polynomials in $t$ with positive integer coefficients.  This is the key to proving Theorems \ref{qPosLoc} and \ref{qPosStrong}, hence Theorem \ref{MainThm}.

\begin{conj}\label{Lefschetz_conjecture}
For $\f{D}_{\In}$ as in \eqref{DIn}, if each $p_i(t)$ is of Lefschetz type, then we can choose $\f{D}$ representing $\scat(\f{D}_{\In})$ so that all walls are as in \eqref{Dwalls} with each $p_{\f{d}}(t)$ of Lefschetz type.  
\end{conj}
For cases of interest for studying theta bases of quantum cluster algebras associated to \textit{acyclic} quivers, we can in fact prove Conjecture \ref{Lefschetz_conjecture}:

\begin{thm}
\label{acyclic_scat_thm}
Assume that $\f{D}_{\In}$ is as in \eqref{DIn}, and furthermore that each $p_i(t)=1$, and also that the $s\times s$ matrix $(\omega(v_i,v_j))_{i,j}$ is the signed adjacency matrix of an acyclic quiver $Q$.  Then Conjecture \ref{Lefschetz_conjecture} holds for $\scat(\f{D}_{\In})$.
\end{thm}
This follows by combining results of \cite{Bridge} and \cite{MeiRei}.  Since both of these papers are written in terms of DT theory, we defer the proof to \S \ref{PosPar}, where the result will appear as part of Proposition \ref{Bridge_scat}.  Then in \S \ref{parity_proof}, we will combine Proposition \ref{Bridge_scat} with Theorem \ref{PosScat} to prove the following.

\begin{thm}
\label{parity_cor}
For $\f{D}_{\In}$ as in \eqref{DIn} with each $p_i(t)$ having positive integer coefficients, assume that for $1\leq i\leq s'$ the polynomial $p_i(t)$ is even, and that for $s'<i\leq s$ the polynomial $p_i(t)$ is odd.  Define a bilinear form on $\mathbb{Z}^s=\bb{Z}\langle e_1,\ldots,e_s\rangle$ by
\begin{equation}\label{lambda_def}
\lambda(e_i,e_j)=\begin{cases} \max(0,\omega(v_i,v_j)) &\textrm{if~}i\neq j\\
1 &\textrm{if~}i=j\leq s'\\
0&\textrm{if~}i=j>s'.
\end{cases}
\end{equation}
Then we can find $\f{D}$ representing $\f{D}_{\scat}$ such that all the walls are of the form 
\begin{align}
\label{wall_parities}
(\f{d},\EE(-p_{\f{d}}(t)z^{v_{\f{d}}}))
\end{align}
for various cones $\f{d}\subset L_{\mathbb{R}}$, where the Laurent polynomials $p_{\f{d}}(t)$ are positive and are either even or odd, with $\parit(p_{\f{d}}(t))\equiv 1+\lambda(a,a)$ (mod $2$) for some $a=(a_1,\ldots,a_s)\in \bb{Z}_{\geq 0}^s$ satisfying $v_{\f{d}}=\sum_i a_i v_i$, $a_i\in \mathbb{Z}_{\geq 0}$. 
\end{thm}

\section{Broken lines and theta bases}\label{BrokenLines}
\label{theta_sec}

\subsection{Definitions and basic properties}\label{theta-def}

Fix a consistent scattering diagram $\f{D}$ in $L_{\bb{R}}$ over $\f{g}$.  Let $p\in L$, $\sQ\in L_{\bb{R}}\setminus \Supp(\f{D})$.  A \textbf{broken line} $\gamma$ with respect to $\f{D}$ with  ends $(p,\sQ)$ is the data of a continuous map $\gamma\colon (-\infty,0]\rar L_{\bb{R}}\setminus \Joints(\f{D})$, values $-\infty < \tau_0 \leq \tau_1 \leq \ldots \leq \tau_{\ell} = 0$, and for each $i=0,\ldots,\ell$, an associated monomial $c_iz^{v_i} \in  \kk_t[L]$, with $c_i\in \kk[t^{\pm 1}]$ and $v_i\in L$, such that:
\begin{enumerate}[label=(\roman*), noitemsep]
\item $\gamma(0)=\sQ$.
\item For $i=1\ldots, \ell$, $\gamma'(\tau)=-v_i$ for all $\tau\in (\tau_{i-1},\tau_{i})$.  Similarly, $\gamma'(\tau)=-v_0$ for all $\tau\in (-\infty,\tau_0)$. 
\item $c_0=1$ and $v_0=p$.
\item For $i=0,\ldots,\ell-1$, $\gamma(\tau_i)\in \Supp(\f{D})$.  Let 
\begin{align}\label{fi}
f_i\coloneqq \prod_{\substack{(\f{d},f_{\f{d}})\in \f{D} \\ \f{d}\ni \gamma(\tau_i)}} f_{\f{d}}^{\sign\omega(v_{\f{d}},v_i)} \in \hat{G}.
\end{align}
I.e., $f_i$ is the $\epsilon\rar 0$ limit of the element $\theta_{\gamma|_{(\tau_i-\epsilon,\tau_i+\epsilon)},\f{D}}$ defined in \eqref{pathprod} (using a smoothing of $\gamma$). Then $c_{i+1}z^{v_{i+1}}$ is a homogeneous (with respect to the $L$-grading) term of $\Ad_{f_i}(c_iz^{v_i})$, not equal to $c_iz^{v_i}$ (this term $c_iz^{v_i}$ is instead obtained by forgetting the time $\tau_i$ from the data).
\end{enumerate}

If every wall of $\f{D}$ has apex at the origin, then we say that $\sQ\in L_{\bb{R}}$ is \textbf{generic} if $\sQ$ is not in $v^{\perp}$ for any $v\in L^*$.  For more general $\f{D}$, we call $\sQ$ generic as long as it is not in some bad\footnote{Specifically, $\sQ\in L_{\bb{R}}\setminus \Supp(\f{D})$ is generic unless it is an endpoint for some ``degenerate broken line'' which passes through a joint, cf. \cite[Rmk. 4.5]{CPS}.} measure-zero subset of $L_{\bb{R}}$ (we shall have no need for broken lines associated to these more general $\f{D}$).  For each $p\in L$ and each generic $\sQ\in L_{\bb{R}}$, we define an element $\vartheta_{p,\sQ}\in \kk_t\llb L\rrb$ which we view as the local expression of a \textbf{theta function} in a chart corresponding to $\sQ$.  For any $p\in L$, we define
\begin{align}\label{vartheta-dfn}
\vartheta_{p,\sQ}\coloneqq \sum_{\Ends(\gamma)=(p,\sQ)} c_{\gamma}z^{v_{\gamma}} \in \kk_t\llb L\rrb
\end{align}
(the fact that $\vartheta_{p,\sQ}$ is well-defined is part of Proposition \ref{TopBasis}).  Here, the sum is over all broken lines $\gamma$ with ends $(p,\sQ)$, and $c_{\gamma}z^{v_{\gamma}}$ denotes the monomial attached to the final straight segment of $\gamma$.   In particular, we always have $\vartheta_{0,\sQ}=1$.  We define $A_{\sQ}^{\can}$ to be the $\kt$-algebra generated by the set of theta functions $\Theta_{\sQ}\coloneqq \{\vartheta_{p,\sQ}\colon p\in L\}$.

The following is essentially \cite[Prop. 2.13]{Man3}, but we include a more detailed proof here for completeness.  We note that exactly the same argument can be used in the classical setting.

\begin{prop}\label{TopBasis}
For fixed generic $\sQ \in L_{\bb{R}}$ and any $p\in L$, \eqref{vartheta-dfn} gives a well-defined element $\vartheta_{p,\sQ} \in \kk_t\llb L\rrb$ of the form \begin{align}\label{pointed}
\vartheta_{p,\sQ}=z^p\left(1+\sum_{v\in L_0^+} a_vz^v\right).
\end{align}
 Furthermore, the theta functions $\Theta_{\sQ}\coloneqq \{\vartheta_{p,\sQ}\colon p\in L\}$ form a topological $\kt$-module basis
 for $A_{\sQ}\coloneqq \kk_t\llb L \rrb$.  Hence, $\Theta_{\sQ}$ also forms a topological basis for the subalgebra $A^{\can}_{\sQ}\subset A_{\sQ}$.
\end{prop}

\begin{proof}
Given $n\in L^{+}\subset L$, let $\delta(n)\in \bb{Z}_{\geq 1}$ denote the largest number $k$ such that $n\in kL^+$ as defined in \eqref{kNplus}.  We extend this to $\delta\colon L^{\oplus}\rar \bb{Z}_{\geq 0}$ by taking $\delta(0)\coloneqq 0$.  Note that $\delta(n_1+n_2)\geq \delta(n_1)+\delta(n_2)$ for any $n_1,n_2\in L^{\oplus}$.

Since $\f{g}$ is $L_0^+$-graded and acts on $\kk_t\llb L\rrb$ by $L$-graded $\kt$-derivations, we see that for any $v\in L_0^+$, $f\in \hat{G}_v^{\parallel}$, $c\in \kk[t^{\pm 1}]$, and $p\in L$, the element $\Ad_f(cz^p)\in \kk_t\llb L\rrb$ must have the form
\begin{align*}
    \Ad_f(cz^p)=cz^p\left(1+\sum_{k\in \bb{Z}_{>0}}c_k z^{kv}\right)
\end{align*}
for some collection of coefficients $c_k\in \kk[t^{\pm 1}]$.  From this it is clear that if $\gamma$ is a broken line with ends $(p,\sQ)$, then $v_{\gamma}\in p+L^{\oplus}$.  Furthermore, we see that $v_{\gamma}=p$ if and only if $\gamma$ is the unique straight broken line with ends $(p,\sQ)$, and in this case we have $c_{\gamma}z^{\gamma}=z^p$.  In particular, once we show that $\vartheta_{p,\sQ}$ is well-defined, it will immediately follow that it has the form given in \eqref{pointed}.

Now to show that $\theta_{p,\sQ}$ is well-defined, we need to show that for each fixed $v\in p+L^+$, there are only finitely many broken lines $\gamma$ with ends $(p,\sQ)$ and with $v_{\gamma}=v$.  
Suppose $(\f{d},f_{\f{d}})\in \f{D}\setminus \f{D}_{\delta(v-p)}$.  Then if $(\f{d},f_{\f{d}})$ contributes to a bend of a broken line $\gamma$ with ends $(p,\sQ)$, we have $\delta(v_{\gamma}-p)>\delta(v-p)$.  Thus, it suffices to restrict to the finite scattering diagram $\f{D}_{\delta(v-p)}$.

By starting at $\tau=0$ and flowing backwards, we see that a broken line $\gamma$ with ends $(p,\sQ)$ and $v_{\gamma}=v$ is uniquely determined by its kinks, i.e., by the data of where it bends and the degrees $\delta$ which it bends by.  But since each kink increases $\delta(v_{\gamma}-p)$, such $\gamma$ can have at most $\delta(v-p)$ many kinks, and since $\f{D}_{\delta(v-p)}$ is finite, there are only finitely many places where $\gamma$ can cross a wall of $\f{D}_{\delta(v-p)}$.  This shows that there are only finitely many such $\gamma$, as desired, and so $\vartheta_{p,\sQ}$ is well-defined.

Finally, since \eqref{pointed} says that $\vartheta_{p,\sQ}$ is $p$-pointed, Lemma \ref{PointedBasis} tells us that $\Theta_{\sQ}$ is indeed a topological basis for $\kk_t\llb L\rrb$.
\end{proof}

The theta functions also satisfy the following important compatibility condition, due to \cite{CPS} in the classical limit and \cite[Thm 2.14]{Man3} in the quantum setup.

\begin{thm}\label{CPS}
Consider $\f{D}=\scat(\f{D}_{\In})$ as in Theorem \ref{KSGS}.  Fix two generic points $\sQ_1,\sQ_2\in L_{\bb{R}}$.  Let $\gamma$ be a smooth path in $L_{\bb{R}}\setminus \Joints(\f{D})$ from $\sQ_1$ to $\sQ_2$, transverse to each wall of $\f{D}$ it crosses.  Then for any $p\in L$,
\begin{align*}
    \vartheta_{p,\sQ_2}=\Ad_{\theta_{\gamma,\f{D}}} (\vartheta_{p,\sQ_1}).
\end{align*}
\end{thm}

As a special case of Theorem \ref{CPS}, if $\sQ_1$ and $\sQ_2$ are two generic points in $L_{\bb{R}}$ in the same path component of $L_{\bb{R}}\setminus \Supp(\f{D})$, then there is an equality $\vartheta_{p,\sQ_1}=\vartheta_{p,\sQ_2}$.

Given $p_1,\ldots,p_s,p\in L$, the structure constant $\alpha(p_1,\ldots,p_s;p)\in \kk_t$ is defined by 
\begin{align*}
    \vartheta_{p_1,\sQ}\cdots \vartheta_{p_s,\sQ} = \sum_{p\in \sum_{i=1}^s p_i+L^{\oplus}} \alpha(p_1,\ldots,p_s;p) \vartheta_{p,\sQ}.
\end{align*}
Here, Theorem \ref{CPS} and Lemma \ref{AdHomeo} ensures that these structure constants are independent of the generic choice of $\sQ$.  

Let $A$ denote the $\kk_t$-algebra defined as in Remark \ref{formal_mult}, replacing the formal symbols $\mathscr{F}_v$ there with formal symbols $\vartheta_v$, and using the structure constants $\alpha(p_1,p_2;q)$ defined above.  Let $A^{\can}$ denote the $\kt$-subalgebra generated by $\Theta=\{\vartheta_p\colon p\in L\}$.  By Remark \ref{formal_mult}, the maps $\iota_{\sQ}\colon  \vartheta_p\mapsto \vartheta_{p,\sQ}$ give $\kk_t$-algebra isomorphisms 
\[
\iota_{\sQ}\colon A\risom A_{\sQ}=\kk_t\llb L\rrb
\]
for each $\sQ$.

The maps $\iota_{\sQ}$ restrict to isomorphisms $\iota_{\sQ}\colon A^{\can}\risom A_{\sQ}^{\can}$.  Theorem \ref{CPS} and Lemma \ref{AdHomeo} ensure that the isomorphisms $\iota_{\sQ_2}\circ \iota_{\sQ_1}^{-1}\colon A_{\sQ_1}\risom A_{\sQ_2}$ or $A_{\sQ_1}^{\can}\risom A_{\sQ_2}^{\can}$, can be given directly by $\theta_{\gamma,\f{D}}$ for $\gamma$ a path from $\sQ_1$ to $\sQ_2$.

The following formula for the structure constants is \cite[Prop. 2.15]{Man3} (the classical case being \cite[Prop 6.4(3)]{GHKK}):
\begin{prop}\label{StructureConstants}
For any $p_1,\ldots,p_s,p\in L$, the corresponding structure constant is given by 
\begin{align}\label{alpha}
 \alpha(p_1,\ldots,p_s;p)z^p= \sum_{\substack{\gamma_1,\ldots,\gamma_s \\
    	               \Ends(\gamma_i)=(p_i,\sQ), \medspace i=1,\ldots,s \\
		                v_{\gamma_1}+\ldots+v_{\gamma_s} = p \\
		                }}
	                (c_{\gamma_1}z^{v_{\gamma_1}}) \cdots (c_{\gamma_s}z^{v_{\gamma_s}}) \in z^p\cdot \kt,
\end{align}
where $\sQ$ is any generic point of $L_{\bb{R}}$ which is sufficiently close to $p$.
\end{prop}

\subsection{Properties preserved by adjoint actions}\label{AdPres}

The following lemma implies that if $\f{D}$ has the form of \eqref{Dwalls} with each $p_{\f{d}}(t)$ being bar-invariant and having positive integer coefficients, then the theta functions defined with respect to $\f{D}$ will be bar-invariant as well. 

\begin{lem}\label{bar_prop}
Let $v\in L_0^+$, $p\in L$, and $\epsilon\coloneqq \sign\omega(v,p)$.  Suppose $p(t)$ is a bar-invariant positive Laurent polynomial in $t$.  Then $\Ad_{\EE\left(-\epsilon p(t)z^v\right)}(z^p)$ is bar-invariant.
\end{lem}
\begin{proof}
This can be seen using \eqref{EEeven} or \eqref{EEoddPsi} to relate $\EE(- \epsilon p(t)z^v)$ to a product of functions of the form $\Psi_{t^a}(t^{-b}z^{cv})\Psi_{t^a}(t^bz^{cv})$ for various $a,b,c$, and then applying \eqref{tAdBinomial} to compute how these act on $z^p$.  Alternatively, by Lemma \ref{log_barinv}, bar-invariance of $p(t)$ implies bar-invariance of $\log(\EE\left(- p(t)z^v\right))$, and then the claim follows from noting that the adjoint action of bar-invariant elements of $\hat{\f{g}}_{L_0^+,\omega}$ on $\kk_{\sigma,t}^{\omega}\llb L\rrb$ preserves bar-invariance, and so $\exp([\log(\EE(-p(t)z^v)),\cdot])$ does as well.
\end{proof}

The next lemma implies that if Conjecture \ref{Lefschetz_conjecture} holds, then the coefficients of the theta functions are in fact of Lefschetz type as predicted in Conjecture \ref{MainConj}, cf. Conjecture \ref{LefschetzTheta}.  In particular, combining Theorem \ref{acyclic_scat_thm} with Lemma \ref{Lefschetz_prop} will yield Theorem \ref{AcyclicLefschetz}, cf. \S \ref{MainProof}.

\begin{lem}\label{Lefschetz_prop}
Let $v\in L_0^+$, $p\in L$, and let $\epsilon\coloneqq \sign\omega(v,p)$.  Then 
\[
\Ad_{\EE\left(-\epsilon[a]_tz^v\right)}(z^p)=\sum_{r\geq 0} p_r(t)z^{p+rv}
\]
where the Laurent polynomials $p_r(t)$ are of Lefschetz type.
\end{lem}
\begin{proof}
Let $K=\mathbb{Z}_{\geq 0}$ and let $V$ be an integrable $K$-graded $\mathfrak{sl}_2$-representation.  Let $h$ be a nonzero element in the Cartan subalgebra of $\mathfrak{sl}_2$.  We give $V$ a cohomological grading by placing the $i$th weight space with respect to the action of $h$ in cohomological degree $i$.  Then $V$ acquires a $K\oplus \mathbb{Z}$ grading, and 
\begin{equation}
    \label{sl2char}
\chi_{K,t}(V)=\sum_{n\in K} f_n(t)x^n
\end{equation}
with each $f_n(t)$ a Lefschetz type polynomial.  Let $V_i$ be the irreducible $i$-dimensional $\mathfrak{sl}_2$-representation (without any $K$-grading).  Then the characteristic function $\chi_t(V)$ of $V_i$ is $[i]_t$, and so the characteristic function of a $K\oplus\mathbb{Z}$-graded vector space $V$ satisfies the condition that all $f_n(t)$ defined as in \eqref{sl2char} are of Lefschetz type if and only if $V$ can be constructed from a $K$-graded $\mathfrak{sl}_2$-representation in this way.

We calculate 
\begin{align}\label{deComp}
\EE\left(-\epsilon [a]_tz^v\right)z^p\EE\left(\epsilon[a]_tz^v\right)=&z^p\EE\left(-\epsilon[a]_tt^{2\omega(v,p)}z^v\right)\EE\left(\epsilon [a]_tz^v\right)\\ \nonumber
=&z^p\Exp_{-t,z}\left([a]_t\frac{t^{-\lvert \omega(v,p)\lvert}-t^{\lvert \omega(v,p)\lvert}}{t^{-1}-t}t^{\omega(v,p)}z^v\right)\\
=&z^p\Exp_{-t,z}\left([a]_t[|b|]_tt^{b}z^v\right)\label{unComp}
\end{align}
where $b=\omega(v,p)$.  The parity of the polynomial
\[
[a]_t[\lvert b\lvert]_tt^b=\left(t^a-t^{-a}\right)\left(t^{\lvert b\lvert}-t^{-\lvert b\lvert}\right)t^b/(t-t^{-1})^2
\]
is equal to the parity of the number $a$, or equivalently, that of $1+\parit([a]_t)$.

Now we deal with two cases.  If $a$ is odd, the $z^{nv}$ coefficient of $\Exp_{-t,z}([a]_t[|b|]_tt^bz^v)$ is equal to $t^{nb}\chi_{t}(\bigwedge^n(V_a\otimes V_{|b|}))$ by (\ref{cat_Exp}), and so the $z^{p+nv}$ coefficient of (\ref{unComp}) is $\chi_{t}(\bigwedge^n(V_a\otimes V_{|b|}))$.  In particular, this coefficient is of Lefschetz type since it arises from an integrable $\mathfrak{sl}_2$-representation.  Similarly, if $a$ is even, then by (\ref{cat_Exp}) again, the $z^{p+nv}$ coefficient of (\ref{unComp}) is equal to $\chi_{t}(\CSym^n(V_a\otimes V_{|b|}))$, and again of Lefschetz type.
\end{proof}

The next lemma plays a similar role in showing that if the functions on walls are all pre-Lefschetz, then so are the theta functions.
\begin{lem}
\label{pL_prop}
Let $v\in L_0^+$, $p\in L$, let $n\in \mathbb{Z}_{
\geq 0}$, and let $\epsilon\coloneqq \sign\omega(v,p)$.  Then 
\[
\Ad_{\EE\left(-\epsilon\!\pl_n(t)z^v\right)}(z^p)=\sum_{r\geq 0} p_r(t)z^{p+rv}
\]
where the Laurent polynomials $p_r(t)$ are of pre-Lefschetz type.
\end{lem}
\begin{proof}
Substituting $z^v\mapsto t^mz^v$ for $m$ even, we deduce that 
\[
\Ad_{\EE\left(-\epsilon t^m\!\pl_n(t)z^v\right)}(z^p)=\sum_{r\geq 0} t^{mr}p_r(t)z^{p+rv}.
\]
Since the property of being pre-Lefschetz is invariant under multiplication by even powers of $t$, we deduce that it is enough to prove the cases $n=0,1$, since all other $\pl_n(t)$ are obtained from $\pl_0(t)$ and $\pl_1(t)$ by multiplication by even powers  of $t$.  Both $\pl_0(t)$ and $\pl_1(t)$ are Lefschetz, and so in both of these cases, by Lemma \ref{Lefschetz_prop} the polynomials $p_r(t)$ are Lefschetz, and hence pre-Lefschetz.
\end{proof}

The next lemma enables us to control the parities of some theta function coefficients in $\kk[t^{\pm 1}]$.
\begin{lem}\label{parity_prop}
Let $v\in L_0^+$, $p\in L$, and let $\epsilon=\sign \omega(v,p)$.  Let $a(t),b(t)\in\bb{Z}_{\geq 0}[t^{\pm 1}]$ be positive Laurent polynomials, satisfying $\parit(a(t))=\alpha$ and $\parit(b(t))=\beta$.  Then
\[
\Ad_{\EE\left(-\epsilon a(t)z^v\right)}(b(t)z^p)=\sum_{r\geq 0} p_r(t)z^{p+rv}
\]
where $\parit(p_r(t))\equiv\beta+r(\omega(p,v) + \alpha+1)$ (mod $2$). 
\end{lem}
\begin{proof}
Define positive numbers $a_i$ by $a(t)=\sum_{i=m}^n a_it^i$ for $[m,n]$ a sufficiently large interval with $m,n\in \bb{Z}$, $n-m\in 2\cdot\bb{Z}_{\geq 0}$.  Then writing
\begin{align*}
\Ad_{\EE\left(-\epsilon a(t)z^v\right)}(b(t)z^p)=&b(t)\Ad_{\EE\left(-\epsilon a(t)z^v\right)}(z^p)\\
=&b(t)\Ad_{\EE\left(-\epsilon t^m z^v\right)}^{a_m}\left( \Ad_{\EE\left(-\epsilon t^{m+2}z^v \right)}^{a_{m+2}}\left(\cdots \Ad_{\EE\left(-\epsilon t^nz^v \right)}^{a_n}\left(z^p\right)\right)\right)
\end{align*}
we see that it is enough to prove the lemma in the special case in which both $a(t)$ and $b(t)$ are monomials.  For this the calculation is as in \eqref{unComp}.
\end{proof}

\subsection{Main results for theta functions}\label{PosTheta}

For the entirety of \S \ref{PosTheta} we assume that $\f{D}_{\In}$ is as in \eqref{DIn} with each $p_i(t)$ a Laurent polynomial in $t$ with positive integer coefficients, and $\f{D}$ is a consistent scattering diagram representing $\scat(\f{D}_{\In})$ and having one of the forms described in Theorem \ref{PosScat}.

\subsubsection{Positivity}
Equation \eqref{tAdBinomial}, together with the definition of a broken line, implies the following:

\begin{lem}\label{BrokenPos}
Let $\f{D}$ be a consistent scattering diagram representing $\scat(\f{D}_{\In})$ as in Theorem \ref{PosScat} such that all scattering functions have the form $\Psi_t(t^{2a}z^v)$ or $\Psi_{t^{2^r}}(t^{2^r(2a+1)}z^{2^r v})$ for various $a\in \bb{Z}$, $r\in \bb{Z}_{\geq 0}$.  Let $\gamma$ be a broken line with ends $(p,\sQ)$, and let $cz^{v}$ be the monomial attached to some straight segment of $\gamma$, other than the initial straight segment (whose attached monomial is $z^p$).  Then $v\in p+L^{+}\subset L$, and $c$ is a Laurent polynomial in $t$ with positive integer coefficients.
\end{lem}

The following is an immediate corollary of Lemma \ref{BrokenPos} together with the definition of $\vartheta_{p,\sQ}$.

\begin{thm}[Universal positivity]\label{qPosLoc}
For any $p\in L$ and generic $\sQ\in L_{\bb{R}}$, the function $\vartheta_{p,\sQ}\in \kk_t\llb L\rrb$ has the form
\begin{align*}
    \vartheta_{p,\sQ}=z^p+ \sum_{v\in L^+} c_vz^{p+v}
\end{align*}
where each $c_v$ is a Laurent polynomial in $t$ with positive integer coefficients.
\end{thm}

Similarly, the following is a corollary of Lemma \ref{BrokenPos} and \eqref{alpha}.

\begin{thm}[Strong positivity]\label{qPosStrong}
For any $p_1,p_2\in L$, $$\vartheta_{p_1}\vartheta_{p_2}=\vartheta_{p_1+p_2}+\sum_{v\in L^+} \alpha(p_1,p_2;p_1+p_2+v) \vartheta_{p_1+p_2+v},$$
where each $\alpha(p_1,p_2;p_1+p_2+v)\in\kk_t$ is positive.  More generally, all structure constants $\alpha(p_1,\ldots,p_s;p)\in\kk_t$ are positive.
\end{thm}

\subsubsection{Parity}
Using Theorem \ref{parity_cor} and applying Lemma \ref{parity_prop} inductively, we obtain the following result concerning the parities of the coefficients of the theta functions.
\begin{prop}\label{ParitProp}
Let $\f{D}_{\In}$ be as \eqref{PosScat}, with the polynomials $p_i(t)$ satisfying the parity assumptions of Theorem \ref{parity_cor}, and furthermore assume that the vectors $v_1,\ldots,v_s$ are linearly independent in $L_{\mathbb{R}}$.  Then in the expression 
\begin{align*}
    \vartheta_{p,\sQ}=z^p+ \sum_{v\in L^+} c_v(t)z^{p+v},
\end{align*}
each $c_v(t)$ is even or odd, with parity given by that of $\omega(v,p)+\lambda(a,a)$, where $a$ is determined by $v=\sum_{i=1}^s a_i v_i$, 
and $\lambda$ is  as in \eqref{lambda_def}. \end{prop}
\begin{proof}
Define $P$ to be the span over $\mathbb{R}$ of $v_1,\ldots,v_s$.  We define $\overline{\lambda}$ on $P$ by 
\[
\overline{\lambda}\left(\sum_{i=1}^s a_iv_i,\sum_{i=1}^s b_iv_i\right)=\lambda(a,b).
\]
This gives a well-defined bilinear form by linear independence of $v_1,\ldots, v_s$, and we moreover have
\[
\omega(v,v')=\overline{\lambda}(v,v')-\overline{\lambda}(v',v)
\]
for $v,v'\in P$.

Let $\gamma$ be a broken line with ends $(p,\sQ)$, and let $c_j(t)z^{p+w_i}$ be the monomials associated to the straight line segments of $\gamma$.  We assume for all $j\leq k$ that the parity of $c_j(t)$ is given by $\omega(p,w_j)+\overline{\lambda}(w_j,w_j)$.  Assume that $\gamma(\tau_k)$ lies on the wall
\begin{align}
(\f{d},\EE(-d(t)z^{v_{\f{d}}})).
\end{align}
Then by Theorem \ref{parity_cor}, $\parit(d(t))=\overline{\lambda}(v_{\f{d}},v_{\f{d}})+1$.

Combining Lemma \ref{parity_prop} and the inductive hypothesis, writing
\begin{align*}
    \Ad_{\EE(-\epsilon d(t)z^{v_{\f{d}}})}(c_k(t)z^{p+w_k})=\sum_{i\geq 0}g_i(t)z^{p+w_k+iv_{\f{d}}}
\end{align*}
we calculate that (mod $2$) 
\begin{align*}
\parit(g_i(t))\equiv&\parit(c_k(t))+i(\omega(v_{\f{d}},p+w_k)+\parit(d(t))+1)\\\equiv&[\omega(p,w_k)+\overline{\lambda}(w_k,w_k)]+[i\omega(v_{\f{d}},p+w_k)+i\overline{\lambda}(v_{\f{d}},v_{\f{d}})]\\
\equiv&[\omega(p,w_k)+\omega(iv_{\f{d}},p+w_k)] + [\overline{\lambda}(w_k+iv_{\f{d}},w_k+iv_{\f{d}}) - \omega(w_k,iv_{\f{d}})]\\
\equiv&\omega(p,w_k+i v_{\f{d}})+\overline{\lambda}(w_k+iv_{\f{d}},w_k+iv_{\f{d}})
\end{align*}
as required.
\end{proof}

\subsubsection{Bar-invariance, pre-Lefschetz, and Lefschetz type}

The following is an immediate consequence of Lemmas \ref{bar_prop} and \ref{pL_prop}.
\begin{prop}\label{bar-pL}
Suppose $\f{D}_{\In}$ is as in \eqref{DIn} with each $p_i(t)\in \bb{Z}_{\geq 0}[t^{\pm 1}]$ being pL (respectively, bar-invariant), and that $\f{D}$ is the resulting consistent scattering diagram with pL (respectively, bar-invariant) scattering functions as in Theorem \ref{PosScat}.  If $cz^v$ is the monomial attached to a straight segment of a broken line $\gamma$ defined with respect to $\f{D}$, then $c$ is pL (respectively, bar-invariant).  Consequently, for any $p\in L$ and generic $\sQ\in L_{\bb{R}}$, the coefficients $c_v$ as in Theorem \ref{qPosLoc} are pL (respectively, bar-invariant).
\end{prop}

As previously noted, if the first part of Conjecture \ref{Lefschetz_conjecture} is true, then by Lemma \ref{Lefschetz_prop}, we could deduce the following conjecture in the same way as Theorems \ref{qPosStrong} and \ref{qPosLoc}:

\begin{conj}\label{LefschetzTheta}
Assume that $\f{D}_{\In}$ is as in \eqref{DIn} and satisfies the stronger assumption that each of the Laurent polynomials $p_i(t)$ is of Lefschetz type.  Then for $p\in L$ and generic $\sQ\in L_{\bb{R}}$, the coefficients $c_v$ in Theorem \ref{qPosLoc} are of Lefschetz type.
\end{conj}

Conjecture \ref{MainConj} is then a restatement of special cases of Conjecture \ref{LefschetzTheta}.  For cases corresponding to \textit{acyclic} quivers, we can in fact prove Conjecture \ref{LefschetzTheta} by combining Theorem \ref{acyclic_scat_thm} with Lemma \ref{Lefschetz_prop}, yielding the following:

\begin{thm}\label{AcyclicThetas}
Assume that $\f{D}_{\In}$ is as in \eqref{DIn}, and satisfies the stronger assumption that each of the polynomials $p_i(t)$ is equal to $1$, and also that the $s\times s$ matrix $(\omega(v_i,v_j))_{i,j}$ is the signed adjacency matrix of an acyclic quiver and is nondegenerate.  Then Conjecture \ref{LefschetzTheta} holds.
\end{thm}

\subsubsection{An application of positivity}

One of the most useful consequences of positivity is that it ensures that broken lines and structure constants do not vanish when applying $\lim_{t\rar 1}$.  We illustrate this in the proof of the following lemma, which will be useful when proving Theorem \ref{flat} in \S \ref{MainProof}.
\begin{lem}\label{Thetaprime}
Consider subsets $\Theta\subset \Theta'\subset L$, and let $A_{\Theta}$ denote the subalgeba generated by $\{\vartheta_p\}_{p\in \Theta}$ (e.g., $A_{\Theta}=A^{\can}$ if $\Theta=L$).  If the classical limits of the theta functions $\{\lim_{t\rar 1} \vartheta_p\}_{p\in \Theta'}$ form a $\kk$-module basis for $\lim_{t\rar 1}(A_{\Theta})$, then the theta functions $\{\vartheta_p\}_{p\in \Theta'}$ form a $\kt$-module basis for $A_{\Theta}$ (not just a topological basis).
\end{lem}
\begin{proof}
By definition, $A_{\Theta}$ is spanned over $\kt$ by products of theta functions, so it suffices to check that each product $\prod_{i=1}^s \vartheta_{p_i}$ of theta functions (with $p_i\in \Theta$ for each $i$) can be expressed as a \textit{finite} $\kt$-linear combination of theta functions $\vartheta_p$ for $p\in \Theta'$. Note that strong positivity in the quantum setting (i.e., Theorem \ref{qPosStrong}) implies that $\alpha(p_1,\ldots,p_s;p)$ is nonzero if and only if its classical limit is nonzero.  Hence, for $p_1,\ldots,p_s\in \Theta$, we have that $\alpha(p_1,\ldots,p_s;p)$ is zero except at a finite collection of $p\in \Theta'$, as desired, simply because the same is true in the classical limit.
\end{proof}

\subsubsection{Atomicity}

Recall that for each generic $\sQ\in L_{\bb{R}}$, we have an isomorphism $\iota_{\sQ}\colon A\risom  \kk_t\llb L\rrb$.  Given $f\in A$, let $f_{\sQ}\coloneqq \iota_{\sQ}(f)\in \kk_t\llb L\rrb$.  One says that a nonzero element $f\in A$ is \textbf{universally positive} (with respect to the scattering atlas) if for each generic $\sQ\in L_{\bb{R}}$, the coefficients of $f_{\sQ}\in \kk_t\llb L\rrb$ are Laurent polynomials in $t$ with non-negative integer coefficients.  We say that a universally positive function $f$ is \textbf{atomic} (or indecomposable) if it cannot be written as a sum of two other universally positive functions.  We obtain the following by the same argument used to prove the atomicity of the classical theta basis in \cite[Thm. 1]{ManAtomic}.

\begin{thm}\label{qAtomic}
The quantum theta functions are the atomic elements of $A$. 
\end{thm}
\begin{proof}
Theorem \ref{qPosLoc} tells us that the quantum theta functions are universally positive with respect to the scattering atlas.  To show atomicity, it thus suffices to show that for any $f\in A$ universally positive with respect to the scattering atlas, the expansion $f=\sum_{p\in L} a_p \vartheta_p$ has positive integer coefficients.  For any $p\in L$, one can choose a positive integer $k$ which is large enough to ensure that that $p\notin p'+kL_0^+$ for any $p'$ with $a_{p'}\neq 0$.  Then for $\sQ$ sharing a chamber of $\f{D}_k$ with $p$, the only possible broken line with respect to $\f{D}$, with ends $(p',\sQ)$ for $a_{p'}\neq 0$, and with $v_{\gamma}=p$ is the straight broken line with attached monomial $z^p$.   
It follows that for $\sQ$ sufficiently close to $p$, $a_p$ is equal to $c_p$ in the formal Laurent series expansion $\iota_{\sQ}(f)=\sum_{p\in L} c_p z^p$.  This is indeed positive by the universal positivity assumption on $f$.
\end{proof}

\subsubsection{Homogeneity of theta functions}

The classical theta functions are known to be eigenfunctions for various torus actions, cf. \cite[Prop. 7.7 and Prop. 7.19]{GHKK} for the $\s{A}^{\prin}$ and $\s{A}$ cases, respectively, and also \cite[Thm. 5.2]{GHK1} for what can be interpreted as a special case of a torus action on the $\s{X}$-space.  In other words, the theta functions are homogeneous with respect to certain gradings.  This property is important since, for example, the weight spaces of these torus actions on $\s{A}^{\prin}$ and $\s{A}$ are related to spaces of global sections of line bundles on $\s{X}$ and on various (partial) compactifications of $\s{X}$ --- cf. \cite[\S 4]{GHK3} and \cite{ManCox}, respectively.  Applications of this to representation theory are given in \cite[\S 0.4]{GHKK} and in \cite{MageeFG,MageeGHK}.  These results on the homogeneity of theta functions for various gradings are all special cases of the following general observation (which holds in the quantum and classical settings):

\begin{prop}[Homogeneity of theta functions]\label{homprop}
Let $\kappa\colon L\rar K$ be a morphism of lattices such that the directions of all walls of $\f{D}_{\In}$ (hence all walls of $\scat(\f{D}_{\In})$) lie in $\ker(\kappa)$.  Then $\Ad_{\alpha}$ with $\alpha=\theta_{\gamma,\f{D}}$ defined as in \eqref{pathprod} respects the $K$-grading on the algebras $A_{\sQ}$, and so $\kappa$ canonically determines a $K$-grading on $A^{\can}$ and on any $A_{\Theta}$ as in Lemma \ref{Thetaprime}.  The theta functions are homogeneous with respect to this grading, and $\deg(\vartheta_p)=\kappa(p)$.  
\end{prop}

\subsubsection{Fibration over a torus}

The following will yield Theorem \ref{MainThm2}(9).  In the classical setting, it essentially says that $\Spec$ of an algebra of theta functions can be viewed as a flat family over an algebraic torus such that each fiber has a theta basis which is canonical up to a global re-scaling.

\begin{prop}\label{HV}
Let $H\coloneqq \{u\in L\colon \omega(u,v)=0 \mbox{ for all } v\in L^+\}$.  Let $\tau_u$ denote the $\kk_t$-module automorphism of $\kk_t\llb L\rrb$ taking $z^v$ to $z^{v+u}$ for each $v\in L$.  Then for $u\in H$, $p\in L$, and any generic $\sQ \in L_{\bb{R}}$, we have $\vartheta_{p+u,\sQ}=\tau_u(\vartheta_{p,\s{Q}})$.
\end{prop}
\begin{proof}
Note that $H_{\bb{R}}$ is parallel to all walls of $\f{D}$, so we can project to $L_{\bb{R}}/H_{\bb{R}}$ without changing the set of walls crossed by a broken line.  Also, $u\in H$ implies that $gz^u=z^u g$ for all $g\in \hat{G}$, hence
\begin{align*}
    \Ad_g(\tau_uz^v)=\Ad_g(t^{-\omega(u,v)}z^uz^v)=t^{-\omega(u,v)}z^u\Ad_g(z^v) = \tau_u(\Ad_g(z^v)).
\end{align*}
It follows that the projections to $L_{\bb{R}}/H_{\bb{R}}$ of the broken lines with ends $(p,\sQ)$ are precisely the same as for the projections of the broken lines with ends $(p+u,\sQ)$, except that the monomials attached to the straight segments of the latter broken lines are obtained by applying $\tau_u$ to the corresponding monomials in the former set of broken lines.  The proposition then follows.
\end{proof}

\section{Quantum cluster algebras}
\label{QCA_sec}
\subsection{Seeds}\label{DinS}

We review the definitions of seeds and cluster algebras from the perspective of \cite{FG1}.  A (skew-symmetric) \textbf{seed} is a collection of data \begin{align}\label{S}
     S=\{N,I,E\coloneqq \{e_{i}\}_{i\in I},F,B\},
\end{align} where $N$ is a finitely generated free Abelian group, $I$ is a finite index-set, $E$ is a basis for $N$ indexed by $I$, $F$ is a subset of $I$, and $B(\cdot,\cdot)$ is a skew-symmetric $\bb{Q}$-valued bilinear form on $N$ such that $B(e_i,e_j)\in \bb{Z}$ whenever $i$ and $j$ are not both in $F$.   When $i\in F$, we say $e_i$ is a \textbf{frozen vector}.  Let $N_{\uf}$ denote the span of $\{e_i\}_{i\in I\setminus F}$ in $N$.  If the seed $S$ is not clear from the context, we may write the data with subscripts $S$ to clarify, e.g., $S=\{N_S,I_S,E_S=\{e_{S,i}\},F_S,B_S\}$.

Let $M= N^*\coloneqq \Hom_{\mathbb{Z}}(N,\mathbb{Z})$.  We define elements $e^*_i\in M$ by $e^*_i(e_j)=\delta_{ij}$, i.e. the $e^*_i$ form a dual basis to the $e_i$.  We define maps $B_1$ and $B_2$ from $N$ to $M$ by $B_1\colon n\mapsto B(n,\cdot)$ and $B_2\colon n\mapsto B(\cdot,n)$.  Given a seed $S$ and a skew-symmetric $\bb{Q}$-valued bilinear form $\Lambda$ on $M$, we say (following the quantization of cluster algebras from \cite{BZ}) that $(S,\Lambda)$ forms a \textbf{compatible pair} if\footnote{Our $B$ is actually what \cite{FG1} would call $\epsilon$, and this is the transpose of what is typically called $B$ in the Fomin-Zelevinsky \cite{FZ} perspective used by \cite{BZ}.}
\begin{align}\label{Lambda}
\begin{array}{c l}
\Lambda(\cdot,B_1(e_i)) = e_i 
 &\mbox{ for each $i\in I\setminus F$.}
\end{array}
\end{align}
Note that the existence of such a $\Lambda$ is equivalent to the condition that $B_1|_{N_{\uf}}$ is injective (called the ``Injectivity Assumption'' in \cite{GHKK}).  Also note that $\Lambda(u,v)\in \bb{Z}$ for any $u\in B_1(N_{\uf})$, $v\in M$. 

Given a seed $S$ as above, one associates another seed $S^{\prin}$ defined as follows:
\begin{itemize}[noitemsep]
\item $N_{S^{\prin}}\coloneqq N\oplus M$, often written as just $N^{\prin}$.
\item $I_{S^{\prin}}$ is the disjoint union of two copies of $I$.  We will call them $I_1$ and $I_2$ to distinguish between them.
\item $E_{S^{\prin}}\coloneqq \{(e_i,0)\colon i\in I_1\} \cup \{(0,e_i^*)\colon i \in I_2\}$.
\item $F_{S^{\prin}}\coloneqq F_1\cup I_2$, where $F_1$ is $F$ viewed as a subset of $I_1$.
\item $B_{S^{\prin}}((n_1,m_1), (n_2,m_2)) \coloneqq B(n_1,n_2) + m_2(n_1) - m_1(n_2)$.
\end{itemize}
We denote the analog of $B_1$ for $B_{S^{\prin}}$ by $B^{\prin}_1\colon N^{\prin}\rar M^{\prin}\coloneqq (N^{\prin})^* = M\oplus N$, identifying $N=N^{**}$.  I.e.,
\begin{align}\label{pi1prin}
    B^{\prin}_1\colon (n,m)\mapsto B_{S^{\prin}}((n,m),\cdot) = (B_1(n)-m,n).
\end{align}
Note that the form $B_{S^{\prin}}$ is unimodular (i.e., has determinant $1$) on $N_{S^{\prin}}$ by the argument given for $\omega^{\circ}$ in \S \ref{ComparisonTricks}.  Thus, $B^{\prin}_1$ is an isomorphism, and so one can always find a skew-symmetric form $\Lambda^{\prin}$ on $M^{\prin}$ such that $(S^{\prin},\Lambda^{\prin})$ forms a compatible pair.  For example, we may take $\Lambda^{\prin}$ to be the skew-symmetric bilinear form on $M^{\prin}$ determined by 
\begin{align}\label{LambdaPrin}
    \Lambda^{\prin}(B^{\prin}_1(a),B^{\prin}_1(b))=B_{S^{\prin}}(a,b).
\end{align}

Alternatively, if $(S,\Lambda)$ is a compatible pair, we can obtain a compatible $\Lambda^{\prin}$ for $S^{\prin}$ by taking $\Lambda^{\prin}=\rho^* \Lambda$ where $\rho$ is the map $M\oplus N\rar M$, $(m,n)\mapsto m$.  That is, we can choose $\Lambda^{\prin}$ to be given by
\begin{align}\label{Lambda_rho}
\Lambda^{\prin}((m_1,n_1),(m_2,n_2))\coloneqq \Lambda(m_1,m_2)
\end{align}
for all $(m_1,n_1),(m_2,n_2)\in M^{\prin}$.

\subsection{Construction of quantum cluster algebras}\label{qclusterdefs}
Given a compatible pair $(S,\Lambda)$, we define the quantum torus algebra $$\s{A}_q^S\coloneqq \kk^{\Lambda}_t[M]$$ as in \eqref{qtor}.  For each $i\in I$, we denote
\begin{align*}
A_{S,i}\coloneqq z^{e_{S,i}^*}\in \s{A}_q^S  
\end{align*}
where $\{e_{S,i}^*\}_{i\in I}$ is the basis for $M$ dual to $E_S$.   
Let 
\begin{align}\label{CS+}
    C_S^+\coloneqq \{m\in M_{\bb{R}}\colon \langle e_i,m\rangle \geq 0 ~ \forall  i\in I\setminus F\}.
\end{align}   Let $\mr{\s{A}}_q^S$ denote the skew-field of fractions of $\s{A}_q^S$ --- cf. \cite[\S 11]{BZ} for a brief review of skew-fields of fractions including a proof of the fact that quantum torus algebras satisfy the (left) Ore condition and thus have well-defined (right) skew-fields of fractions, into which they embed.

Similarly, we obtain from $S$ (without needing the existence of a compatible $\Lambda$) another quantum torus algebra $$\s{X}_q^S\coloneqq \kk^{B}_t[N].$$  
We write $\mr{\s{X}}_q^S$ for the skew-field of fractions of $\s{X}_q^S$.

Given a seed $S=(N,I,E=\{e_i\}_{i\in I},F,B)$, the {\bf mutation} of $S$ with respect to $j\in I\setminus F$ is the seed $\mu_j(S)\coloneqq (N,I,E'=\{e_i'\}_{i\in I},F,B)$, where the vectors $e_i'$ are defined by
\begin{align}\label{eprime}
    e'_{i} \coloneqq  \mu_{j}(e_{i})\coloneqq \left\{ \begin{array}{lr}
  e_{i}+\max(0,B(e_i,e_j))e_{j} &\mbox{if $i\neq j$~} \\
    -e_{j} &\mbox{if $i=j$.} 
	\end{array} \right.
\end{align}

We can use the new seed to define another copy of the $\s{X}$ version of the quantum torus algebra, $\s{X}_q^{\mu_j(S)}=\kk_t^B[N_{\mu_j(S)}]\cong\s{X}_q^S$ --- this last isomorphism is a by-product of the fact that the construction of the quantum torus algebra is independent of $E_S$. For the $\s{A}$ version, recall that $(S,\Lambda)$ being a compatible pair implies that 
\[
B_1\colon N_{\uf,\bb{R}}\rar B_1(N_{\uf,\bb{R}})\subset M_{\bb{R}}
\]
is an isomorphism onto its image, and by \eqref{Lambda}, $m\mapsto \Lambda(\cdot,m)$ is the inverse isomorphism.  Since mutation preserves $N_{\uf}$ and $B$, it follows that for each $j\in I\setminus F$, $(\mu_j(S),\Lambda)$ is again a compatible pair.   It therefore makes sense to define $\s{A}_q^{\mu_j(S)}=\kk_t^{\Lambda}[M]$, using the same $\Lambda$ as in the definition of $\s{A}_q^S$.

In order to consider formal versions of the quantum cluster algebras, we define certain completions of $\kk_t^B[N]$ and $\kk_t^{\Lambda}[M]$.  Taking $\sigma_{S,\s{X}}$ to be the cone generated by $\{e_{S,i}\colon i\in I\setminus F\}$, we consider the completion $$\wh{\s{X}}_q^{S}\coloneqq \kk_{\sigma_{S,\s{X}},t}^B\llb N\rrb.$$  
We similarly define the formal completion $$\wh{\s{A}}_q^S\coloneqq \kk_{\sigma_{S,\s{A}},t}^{\Lambda}\llb M\rrb,$$ where $\sigma_{S,\s{A}}$ is the cone spanned by $\{B_1(e_{S,i})\colon i\in I\setminus F\}$ (the existence of $\Lambda$ satisfying \eqref{Lambda} ensures that the vectors in this set are linearly independent, thus determine a strongly convex cone).  Similarly in the principal coefficients case, we define $\wh{\s{A}}_q^{\prin,S}\coloneqq \wh{\s{A}}_q^{S^{\prin}}$.

For $j\in I\setminus F$, the quantum dilogarithm $\Psi_t(z^{e_j})$ is an element of the completion $\wh{\s{X}}^S_q=\kk^B_{\sigma_{S,\s{X}},t}\llb N\rrb$.  One may easily verify, e.g., using \eqref{tAdBinomial}, that for $p\in N$ satisfying $B(e_j,p)\leq 0$ (with $e_j=e_{S,j}$), we have
\begin{align}\label{Xmut}
\mu_{S,j}^{\s{X}}(z^p)\coloneqq \Ad_{\Psi_t(z^{e_{j}})^{-1}}(z^{p})=\sum_{n=0}^{B(p,e_{j})}\binom{B(p,e_{j})}{n}_tz^{p+ne_{j}}.
\end{align}
It follows that the automorphism $\Ad_{\Psi_t(z^{e_j})^{-1}}\colon  \wh{\s{X}}^S_q\risom \wh{\s{X}}^S_q$ induces a well-defined algebra isomorphism 
\[
\mu_{S,j}^{\s{X}}\colon \mr{\s{X}}_q^S  \risom \mr{\s{X}}^{\mu_j(S)}_q,
\]
with inverse defined similarly and induced by  $\Ad_{\Psi_t(z^{e_j})}$. 

We define 
\[
\mu_{S,j}^{\s{A}}\colon\mr{\s{A}}_q^S \risom \mr{\s{A}}^{\mu_j(S)}_q
\]
similarly using the algebra automorphism $\Ad_{\Psi_t(z^{B_1(e_j)})^{-1}}$ of $\wh{\s{A}}_q^S=\kk_{\sigma_{S,\s{A}},t}^{\Lambda}\llb M\rrb$.  
By \eqref{tAdBinomial} again, we deduce that
\begin{equation}\label{clust_transf}
\mu_{S,j}^{\s{A}}(A_{S,i})=\begin{cases}A_{S,i}& \textrm{if }i\neq j\\
A_{S,i}+z^{e_i^*+B_1(e_i)}&\textrm{if }i=j.\end{cases}
\end{equation}
One refers to $\mu_{S,j}^{\s{A}}$ and $\mu_{S,j}^{\s{X}}$ as the cluster $\s{A}$-mutations and cluster $\s{X}$-mutations, respectively.

Now consider an $s$-tuple $\jj=(j_1,\ldots,j_s)\subset (I\setminus F)^s$.  For $k=1,\ldots,s,s+1$ we let $\jj_k$ denote the tuple $(j_1,\ldots,j_{k-1})$.  The tuple $\jj$ determines a new seed 
\begin{align}\label{Sjj}
S_{\jj}\coloneqq \mu_{j_s}\circ \cdots \circ \mu_{j_1}(S),    
\end{align}
as well as an isomorphism $$\mu_{S,\jj}^{\s{A}}\coloneqq \mu_{{S_{\jj_s},}j_s}^{\s{A}}\circ \cdots \circ \mu_{{S_{\jj_1},}j_1}^{\s{A}}\colon \mr{\s{A}}_q^S \risom \mr{\s{A}}^{\mu_{\jj}(S)}_q,$$ 
and similarly, an isomorphism $\mu_{{S},\jj}^{\s{X}}\colon \mr{\s{X}}_q^S \risom \mr{\s{X}}^{\mu_{\jj}(S)}_q$.  If $\jj$ is empty, these are the identity isomorphisms.  In particular, $S_{\jj_1}=S$, and $S_{\jj_{s+1}}=S_{\jj}$.  Define
\begin{align*}
    \s{A}^{\jj}_q\coloneqq (\mu_{{S},\jj}^{\s{A}})^{-1}(\s{A}_q^{S_{\jj}}) \subset \mr{\s{A}}_q^S,
\end{align*}
and similarly,
\begin{align*}
    \s{X}^{\jj}_q\coloneqq (\mu_{{S},\jj}^{\s{X}})^{-1}(\s{X}_q^{S_{\jj}}) \subset \mr{\s{X}}_q^S.
\end{align*}
For each $i\in I$, we obtain an element $A_{\jj,i}\coloneqq (\mu_{{S},\jj}^{\s{A}})^{-1}(A_{S_{\jj},i})\in \s{A}^{\jj}_q$.  The elements $$\bigcup_{\jj} \{A_{\jj,i}\colon i\in I\}\cup \{A_{\jj,i}^{-1}\colon i\in F\}\subset \s{A}^{\jj}_q$$ are called the \textbf{quantum cluster variables}.  By a \textbf{quantum cluster monomial}, we mean an element of the form $(\mu_{{S},\jj}^{\s{A}})^{-1}(z^m)\in \s{A}^{\jj}_q$ for $m\in C_{S_{\jj}}^+$.

The \textbf{ordinary quantum cluster algebra} $\s{A}_q^{\ord}$ is defined to be the sub $\kt$-algebra of $\mr{\s{A}}_q^S$ generated by the quantum cluster variables.  The \textbf{upper quantum cluster algebra $\s{A}_q^{\up}$} is defined by
\begin{align*}
    \s{A}_q^{\up} \coloneqq  \bigcap_{\jj} \s{A}^{\jj}_q \subset \mr{\s{A}}_q^S
\end{align*}
where the intersection is over all finite tuples of elements of $I \setminus F$.  The algebras $\s{A}^{\jj}_q$ are called \textbf{clusters} of $\s{A}_q^{\up}$ (note that two different tuples $\jj$ may yield the same cluster). 
The ``quantum Laurent phenomenon'' \cite[Cor. 5.2]{BZ} says that the quantum cluster variables can be expressed as quantum Laurent polynomials in every cluster, i.e.,
\begin{align}\label{qLaurentPhen}
    \s{A}_q^{\ord}\subset \s{A}_q^{\up}.
\end{align}

The \textbf{upper quantum cluster $\s{X}$-algebra $\s{X}^{\up}_q$} is similarly defined by
\begin{align*}
    \s{X}^{\up}_q \coloneqq  \bigcap_{\jj} \s{X}^{\jj}_q \subset  \mr{\s{X}}_q^S.
\end{align*}
We say an element $f\in \s{X}^{\up}_q$ is a \textbf{(quantum) global monomial} if there is some $\jj$ such that $\mu_{{S},\jj}^{\s{X}}(f)$ is a monomial $z^{n}\in \s{X}_q^{S_{\jj}}$ for some $n\in N_{S_{\jj}}$.   Then the \textbf{ordinary quantum cluster $\s{X}$-algebra $\s{X}^{\ord}_q$} is defined to be the subalgebra of $\s{X}_q^{\up}$ generated by the quantum global monomials.

We may define global monomials in $\s{A}_q^{\up}$ analogously to the definition in $\s{X}_q^{\up}$.  In light of \eqref{qLaurentPhen}, one sees that the quantum cluster monomials are global monomials.  On the other hand, let $m\in M\setminus C^+_{S_{\mathbf{j}}}$.  Then in $\s{A}_q^{\mu_{\jj}(S)}$, we have $z^m=t^{a}z^{m'}A_{S_{\jj},j}^{-k}$ for some $j\in I\setminus F$, $m'$ in the span of $\{e_{S_{\jj},i}^*\}_{i\in I\setminus \{j\}}$,  $a\in\frac{1}{D}\mathbb{Z}$, and $k\in \bb{Z}_{\geq 1}$.  By \eqref{clust_transf} it follows that $\mu^{\mathcal{A}}_{S_{\jj},j}(z^m)$ is a genuine rational function, and not a Laurent polynomial, and so $z^m$ is not a global monomial.  In conclusion, the quantum cluster monomials are precisely the global monomials in $\s{A}_q^{\up}$.

Given a seed $S$ and a compatible $\Lambda^{\prin}$ for the corresponding seed with principal coefficients $S^{\prin}$, we write $\s{A}_q^{\prin,\up}$ to denote the upper quantum cluster algebra with principal coefficients, i.e., the upper quantum cluster algebra associated to the compatible pair $(S^{\prin},\Lambda^{\prin})$.  Similarly, $\s{A}_q^{\prin,\ord}$ denotes the ordinary quantum cluster algebra with principal coefficients.

The following lemma relates $\s{A}_q^{\prin,\up}$ to $\s{A}_q^{\up}$ and $\s{X}^{\up}_q$.  

\begin{lem}\label{xirho}
Let $S$ be a seed, and let $(S^{\prin},\Lambda^{\prin})$ be a compatible pair.   If $\Lambda^{\prin}|_{B^{\prin}_{1}((N,0))}$ is given as in \eqref{LambdaPrin}, then the map $\xi\colon N\rar M\oplus N$, $n\mapsto (B_1(n),n)$ induces injective homomorphisms $\xi\colon {\s{X}}^{S}_q\hookrightarrow {\s{A}}_q^{\prin,S}$, extending to $\xi\colon \mr{\s{X}}^{S}_q\hookrightarrow \mr{\s{A}}_q^{\prin,S}$ and $\xi\colon \wh{\s{X}}^{S}_q\hookrightarrow \wh{\s{A}}_q^{\prin,S}$, which intertwine all sequences of mutations, thus also inducing an injection $\xi\colon \s{X}^{\up}_q\hookrightarrow \s{A}_q^{\prin,\up}$. 

On the other hand, suppose that $(S,\Lambda)$ is a compatible pair and that $\Lambda^{\prin}$ is given as in \eqref{Lambda_rho}.  Then the map $\rho\colon M\oplus N\rar M$, $(m,n)\mapsto m$ induces surjective homomorphisms $\rho\colon {\s{A}}_q^{\prin,S}\rar {\s{A}}_q^{S}$ and $\rho\colon \wh{\s{A}}_q^{\prin,S}\rar \wh{\s{A}}_q^{S}$ intertwining sequences of mutations, thus also inducing $\rho\colon \s{A}_q^{\prin,\up}\rar \s{A}_q^{\up}$.
\end{lem} 
\begin{proof}
First we must check that $\xi$ and $\rho$ induce well-defined maps of the quantum torus algebras.  For $\xi$, this means checking that 
\begin{align}\label{LambdaPrinB}
    \Lambda^{\prin}(\xi(n_1),\xi(n_2))=B(n_1,n_2).
\end{align}
By the definition of $\xi$, the left-hand side of \eqref{LambdaPrinB} is $\Lambda^{\prin}((B_1(n_1),n_1),(B_1(n_2),n_2))$, which by \eqref{pi1prin} equals $\Lambda^{\prin}(B^{\prin}_1(n_1,0),B^{\prin}_1(n_2,0))$.  The equality \eqref{LambdaPrinB} is now equivalent to the condition that $\Lambda^{\prin}|_{B_{1,\prin}((N,0))}$ is given as in \eqref{LambdaPrin}.  It is clear that $\xi$ is injective and thus extends to the skew-fields of fractions.  Since $\xi(e_i)=B_1^{\prin}(e_i)$, we see that $\xi(\sigma_{S,\s{X}})= \sigma_{S^{\prin},\s{A}}$, and so $\xi$ extends to the completions as well.

Similarly, the condition for $\rho$ to induce a well-defined map of quantum torus algebras is just that $\Lambda^{\prin}((m_1,n_1),(m_2,n_2))=\Lambda(\rho(m_1,n_1),\rho(m_2,n_2))$, and this is the condition we imposed.  Since $\rho(B_1^{\prin}(e_i))=B_1(e_i)$, we have $\rho(\sigma_{S^{\prin},\s{A}})=\sigma_{S,\s{A}}$, and so $\rho$ extends to $\wh{\s{A}}_q^{\prin,S}\rar \wh{\s{A}}_q^S$.

It now suffices to check that $\xi$ and $\rho$ commute with mutations.  Since the mutations are automorphisms of the corresponding formal quantum cluster algebras, it suffices to check that $\xi$ and $\rho$ commute with the mutations associated to the initial seed (because the same formulas define $\xi$ and $\rho$ in the coordinate system associated to any other seed).  For $\xi$, using \eqref{pi1prin} we have that $\xi(e_i)=(B_1(e_i),e_i)=\pi^{\prin}_1(e_i,0)$, so $\xi$ maps $\Psi_t(z^{e_i})$ to $\Psi_t(z^{\pi^{\prin}_1(e_i)})$, and the commutativity of $\xi$ with the mutations follows.  Similarly, for $\rho$, we have $\rho(\pi^{\prin}_1(e_i)) = \rho(B_1(e_i),e_i) =  B_1(e_i)$, so $\rho$ maps $\Psi_t(z^{\pi^{\prin}_1(e_i)})$ to $\Psi_t(z^{B_1(e_i)})$, as desired. 
\end{proof}

\begin{lem}\label{pi1}
Let $(S,\Lambda)$ be a compatible pair, and suppose $\Lambda(\cdot,B_1(e_i))=e_i$ for all $i\in I$, not just $i\in I\setminus F$.  Then $B_1\colon N\rar M$ induces homomorphisms $B_1\colon \s{X}_q^{\up} \rar \s{A}_q^{\up}$ and $B_1\colon \wh{\s{X}}_q^S\rar \wh{\s{A}}_q^{S}$. 
\end{lem}
\begin{proof}
The condition is equivalent to requiring that $\Lambda(B_1(u),B_1(v))=B(u,v)$ for all $u,v\in N$, and it follows that $B_1$ induces a well-defined map of quantum torus algebras.  This extends to the formal quantum tori since $B_1(\sigma_{S,\s{X}})=\sigma_{S,\s{A}}$.   Compatibility with mutations now follows from the definitions of $\mu_j^{\s{X}}$ and $\mu_j^{\s{A}}$ in terms of conjugation by $\Psi_t(z^{e_j})$ and $\Psi_t(z^{B_1(e_j)})$, respectively, and the observation that $B_1(\Psi_t(z^{e_j})) = \Psi_t(z^{B_1(e_j)})$.  The claim for the upper quantum cluster algebras follows.
\end{proof}

\subsection{The scattering diagrams associated to a seed}\label{Initial}

Given a compatible pair $(S,\Lambda)$, we define a scattering diagram $\f{D}^{\s{A}_q}$ as follows.  Take $\f{g}\coloneqq \f{g}_{M_0^+,\Lambda}$ for $M_0\coloneqq B_1(N_{\uf})\subset M$ and $M_0^+\coloneqq (M_0\cap \sigma_{S,\s{A}})\setminus \{0\}$. 
We then take the following as our initial scattering diagram in $M_{\bb{R}}$ over $\f{g}$:
\begin{align*}
    \f{D}^{\s{A}_q}_{\In}\coloneqq \{(e_i^{\perp},\Psi_t(z^{B_1(e_i)}))\colon i\in I\setminus F\}.
\end{align*}
Note that we indeed have $B_1(e_i)^{\Lambda\perp}=e_i^{\perp}\subset M_{\bb{R}}$ by \eqref{Lambda}.  We then take $\f{D}^{\s{A}_q}\coloneqq \scat(\f{D}^{\s{A}_q}_{\In})$ as in Theorem \ref{KSGS}.

We define $\f{D}_{\In}^{\s{A}_q^{\prin}}$ and $\f{D}^{\s{A}_q^{\prin}}$ in the same way as $\f{D}_{\In}^{\s{A}_q}$ and $\f{D}^{\s{A}_q}$ above, but using a compatible pair $(S^{\prin},\Lambda^{\prin})$ in place of $(S,\Lambda)$.

Similarly, given a seed $S$ and taking $N^+\coloneqq (N\cap \sigma_{S,\s{X}})\setminus \{0\}$, we construct the $\s{X}$-scattering diagram in $N_{\bb{R}}$ over $\f{g}_{N^+,B}$ as follows:
\begin{align*}
    \f{D}^{\s{X}_q}_{\In}\coloneqq \{(B_2(e_i)^{\perp},\Psi_t(z^{e_i}))\colon i\in I\setminus F, ~B_2(e_i)\neq 0\} \quad \mbox{and} \quad \f{D}^{\s{X}_q}\coloneqq  \scat(\f{D}^{\s{X}_q}_{\In}).
\end{align*}
Note that we indeed have $B_2(e_i)^{\perp} = e_i^{B\perp}$.  Also, the condition that $B_2(e_i)\neq 0$ ensures that $e_i\notin \ker(B)$, as required for our definition of a wall.  We can always treat $e_i\in \ker(B)$ as a frozen vector without changing the construction of the $\s{X}$-spaces or $\f{D}^{\s{X}_q}_{\In}$, so for simplicity, one may assume that no unfrozen vectors in $\ker(B)$ exist.   We make the following observations:

\begin{lem}\label{cluster-scat-pos}
All hypotheses (hence conclusions) from Theorem \ref{PosScat} apply to the scattering diagrams $\f{D}^{\s{A}_q}$, $\f{D}^{\s{A}_q^{\prin}}$, and $\f{D}^{\s{X}_q}$.  Similarly, all results of \S \ref{PosTheta} (except possibly Theorem \ref{AcyclicThetas}) apply to the theta functions constructed from these scattering diagrams.
\end{lem}

For the rest of this section, we assume that all consistent scattering diagrams are in one of the forms described in the conclusions of Theorem \ref{PosScat}.

\begin{rmk}[Non skew-symmetric quantum cluster algebras]\label{nonskew}
Throughout this article, we assume that all seeds are skew-symmetric, i.e., that the bilinear form $B$ is skew-symmetric.  More generally though, $B$ might be merely ``skew-symmetrizable,'' i.e., of the form $B(e_i,e_j)=d_j\beta(e_i,e_j)$ for some skew-symmetric $\bb{Q}$-valued form $\beta$ and some numbers $d_j\in \bb{Q}_{> 0}$ such that $B(e_i,e_j)\in \bb{Z}$ whenever $i$ and $j$ are not both in $F$.  Then following \cite{BZ} for the $\s{A}$-side and \cite{FG1} for the $\s{X}$-side, the constructions of the quantum cluster algebras change in the following ways.  The maps $B_1$ and $B_2$ are defined in the same way using $B$, but the quantum torus algebra $\s{X}_q^S$ is defined using $\beta$, i.e., $\s{X}_q^S\coloneqq \kk_t^{\beta}[N]$. The compatibility condition \eqref{Lambda} is changed to $\Lambda(\cdot,B_1(e_i))=\frac{1}{d_i}e_i$.  The mutations $\mu_j^{\s{A}}$ and $\mu_j^{\s{X}}$ are defined using 
$\Ad_{\Psi_{t^{1/d_i}}(z^{B_1(e_i)})^{-1}}$ and $\Ad_{\Psi_{t^{1/d_i}}(z^{e_i})^{-1}}$, respectively.  Then as suggested in \cite[\S 4.2]{Man3}, the appropriate initial scattering diagrams seem to be $ \f{D}^{\s{A}_q}_{\In}\coloneqq \{(e_i^{\perp},\Psi_{t^{1/d_i}}(z^{e_i}))\colon i\in I\setminus F\}$ and
$\f{D}^{\s{X}_q}_{\In}\coloneqq \{(B_2(e_i)^{\perp},\Psi_{t^{1/d_i}}(z^{B_1(e_i)}))\colon i\in I\setminus F, B_2(e_i)\neq 0\}$, because crossing these walls recovers the quantum mutations associated to the initial seed as defined in \cite{BZ} and \cite{FG1}, respectively (defining these scattering diagrams in a way compatible with $\lim_{t\rar 1}$ requires some careful modifications to the definition of $\f{g}$).  However, it is suspected that positivity for quantum theta functions will fail in this general setup --- in rank $2$, the quantum theta functions are expected to agree with the quantum greedy bases, and positivity for some greedy basis elements (outside the cluster complex) is known to fail, cf. \cite[end of \S 3]{LLRZpnas}.
\end{rmk}

\subsection{Relating theta functions for different spaces}
In this subsection we explain how the theta functions for the algebras $\wh{\s{A}}_q$, $\wh{\s{A}}_q^{\prin}$, and $\wh{\s{X}}_q$ are related to each other.  We then use this to relate the quantum theta functions to the classical theta functions as defined in \cite{GHKK}.  We will use the maps $\xi\colon N\rar M\oplus N$ and $\rho\colon  M\oplus N \rar M$ from Lemma \ref{xirho}, and by abuse of notation we extend these to maps between the respective lattices tensored with $\bb{R}$.

Suppose that $(S,\Lambda)$ is a compatible pair and that $\Lambda^{\prin}$ is given by $\rho^* \Lambda$ as in \eqref{Lambda_rho}, so we can consider $\rho$ as in Lemma \ref{xirho}.  Then $\ker(\rho)=(0,N_{\bb{R}})$ is parallel to the supports of the walls in $\f{D}_{\In}^{\s{A}_q^{\prin}}$, hence also parallel to the supports of the walls in $\f{D}^{\s{A}_q^{\prin}}$.  
Furthermore, as observed in the proof of Lemma \ref{xirho}, $\rho$ maps $\Psi_t(z^{B^{\prin}_{1}(e_i)})$ to $\Psi_t(z^{B_1(e_i)})$.  Thus, using the uniqueness of the consistent scattering diagram in Theorem \ref{KSGS}, we have (up to equivalence) that
\begin{align}\label{rhoDA}
    \f{D}^{\s{A}_q}=\{(\rho(\f{d}),\rho(f))\colon (\f{d},f)\in \f{D}^{\s{A}_q^{\prin}}\},
\end{align}
and furthermore, for each $(\f{d},f)\in \f{D}^{\s{A}_q^{\prin}}$, $\f{d}=\rho^{-1}(\rho(\f{d}))$.  We thus obtain the following (cf. Lemma \ref{QuotientScat}):

\begin{lem}\label{prin2A}
Suppose that $(S,\Lambda)$ and $(S^{\prin},\Lambda^{\prin})$ are compatible pairs such that $\Lambda(m_1,m_2)=\Lambda^{\prin}((m_1,n_1),(m_2,n_2))$ for all $(m_1,n_1),(m_2,n_2)\in M^{\prin}$, so $\rho$ is well-defined as in Lemma \ref{xirho}.  Then for any $p\in M^{\prin}$ and generic $\sQ\in M_{\bb{R}}^{\prin}$, the theta function $\vartheta_{p,\sQ}$ defined with respect to $\f{D}^{\s{A}_q^{\prin}}$ and the theta functions $\vartheta_{\rho(p),\rho(\sQ)}$ defined with respect to $\f{D}^{\s{A}_q}$ are related via
\begin{align*}
    \vartheta_{\rho(p),\rho(\sQ)}=\rho(\vartheta_{p,\sQ}).
\end{align*}
\end{lem}

Now suppose that $\Lambda^{\prin}$ is given as in \eqref{LambdaPrin} so we can consider the map $\xi$ as in Lemma \ref{xirho}. We similarly saw in the proof of Lemma \ref{xirho} that $\xi$ takes the scattering functions of $\f{D}_{\In}^{\s{X}_q}$ to those of $\f{D}_{\In}^{\s{A}_q^{\prin}}$.  Furthermore, if $n\in B_2(e_i)^{\perp}$, then $\xi(n)=(B_1(n),n)\in (e_i,0)^{\perp}$, so $\xi$ maps the supports of walls of $\f{D}_{\In}^{\s{X}_q}$ to the supports of the respective walls of $\f{D}_{\In}^{\s{A}_q^{\prin}}$.  It follows that we can use $\xi$ to identify $\f{D}^{\s{X}_q}$ with the intersection of $\f{D}^{\s{A}_q^{\prin}}$ and \begin{align}\label{xiN}
    \xi(N_{\bb{R}})=\{(m,n)\in M^{\prin}_{\bb{R}}\colon m-B_1(n)=0\}.
\end{align}
I.e.,
\begin{align}\label{xiDX}
    \f{D}^{\s{X}_{q}}=\{(\xi^{-1}(\f{d}),f_{\f{d}})\colon (\f{d},f_{\f{d}})\in \f{D}^{\s{A}_q^{\prin}},\medspace \codim(\xi^{-1}(\f{d}))=1\}.
\end{align}

Note that directions of walls in $\f{D}^{\s{A}^{\prin}_q}$ lie in \begin{align}\label{Nprime}
    N'\coloneqq B_1^{\prin}(N\oplus 0) = \{(B_1(n),n)\in M\oplus N\colon n\in N\},
\end{align} 
and by \eqref{xiN}, this is the same as $\xi(N)$.  Define 
\begin{align}
\label{Mprime}
    \kappa\colon M\oplus N&\rightarrow M\\
    (m, n) &\mapsto m-B_1(n).\nonumber
\end{align}
We note that this map $\kappa$ agrees with the grading $\deg$ considered in \cite{ManCox} and the weight $w$ as in \cite[(B.3)]{GHKK}.  Then $N'=\ker(\kappa)$, and by Proposition \ref{homprop}, $\wh{\s{A}}^{\prin}_q$ is $M$-graded. 

\begin{lem}\label{prin2X}
Let $S$ be a seed, and let $(S^{\prin},\Lambda^{\prin})$ be a compatible pair such that $\Lambda^{\prin}|_{B^{\prin}_{1}((N,0))}$ is given as in \eqref{LambdaPrin}, so the map $\xi$ of Lemma \ref{xirho} is well-defined.  For $p\in N$ and generic $\sQ\in N_{\bb{R}}$, the theta function $\vartheta_{p,\sQ}$ defined with respect to $\f{D}^{\s{X}_q}$ is related to the theta function $\vartheta_{\xi(p),\xi(\sQ)}$ defined with respect to $\f{D}^{\s{A}_q^{\prin}}$ via
\begin{align*}
    \xi(\vartheta_{p,\sQ})=\vartheta_{\xi(p),\xi(\sQ)}.
\end{align*}
\end{lem}
\begin{proof}
For $\gamma$ a broken line in $\f{D}^{\s{X}_q}$, we obtain a broken line in $\f{D}^{\s{A}^{\prin}_q}$ by considering $\xi\circ \gamma$ and replacing the monomials $c_iz^{v_i}$ associated to straight sections with $c_iz^{\xi(v_i)}$.  By \eqref{xiDX} these are exactly the broken lines with ends of $M$-degree zero.
\end{proof}

We similarly find the following:
\begin{lem}\label{pi1Theta}
Suppose $\Lambda(\cdot,B_1(e_i))=e_i$ for all $i\in I$ so that Lemma \ref{pi1} applies.  Then for all $p\in N$ and generic $\sQ\in N_{\bb{R}}$, the theta function $\vartheta_{p,\sQ}$ defined with respect to $\f{D}^{\s{X}_q}$ is related to the theta function $\vartheta_{B_1(p),B_1(\sQ)}$ defined with respect to $\f{D}^{\s{A}_q^{\prin}}$ via  \begin{align*}
    B_1(\vartheta_{p,\sQ})=\vartheta_{B_1(p),B_1(\sQ)}.
\end{align*}
\end{lem}
\begin{prop}\label{q2classicalTheta}
The classical limits of the quantum theta functions in $\wh{\s{A}}_q^S$, $\wh{\s{A}}^{\prin,S}_q$ and $\wh{\s{X}}_q^S$
are the corresponding classical theta functions of \cite{GHKK}.
\end{prop}
\begin{proof}
It is straightforward to check that applying $\lim_{t\rar 1}$ to $\f{D}^{\s{A}_q}$ yields the corresponding classical scattering diagram $\f{D}_{\In,\textbf{s}}$ of \cite{GHKK}.  The claim thus follows for $\wh{\s{A}}_q^{\prin,S}$, since for these cases the \cite{GHKK} theta functions are constructed directly from the broken lines corresponding to $\lim_{t\rar 1} \f{D}_q^{\s{A}_t^{\prin}}=\f{D}_{\In,\textbf{s}^{\prin}}$.  For $\wh{\s{A}}_q^{S}$, \cite[\S 7.2]{GHKK} constructs the theta functions by applying the classical analog of our map $\rho$ from Lemma \ref{xirho}.\footnote{The reason that \cite{GHKK} does not just directly construct the theta functions for the $\s{A}$-space using $\f{D}_{\In,\textbf{s}}$ is that they allow $\s{A}$-spaces for which their ``injectivity assumption'' is not satisfied, i.e., for which there is no compatible pair $\Lambda$, hence no quantization.  In such cases, there might not exist a strictly convex cone $\sigma_{S,\s{A}}$ containing the negative directions of the initial walls.}  Lemma \ref{prin2A} states that $\rho$ takes quantum theta functions to quantum theta functions, so the classical limits of the quantum theta functions in $\wh{\s{A}}_q$ are indeed the classical theta functions of \cite{GHKK}.

In the classical limit, applying $\xi^{-1}$ to the theta functions of $M$-degree $0$ is precisely the construction of the theta functions on the $\s{X}$-space in \cite[Constr. 7.11]{GHKK}.  So by Lemma \ref{prin2X}, the classical limits of our theta functions in $\wh{\s{X}}^S_q$ are the classical theta functions of \cite{GHKK}.
\end{proof}

\subsection{The cluster complex}\label{ChamberSection}

Let $(S,\Lambda)$ be a compatible pair.  For each $j\in I\setminus F$, define an integral piecewise linear bijection $T_j\colon M_{\bb{R}}\rar M_{\bb{R}}$ by
\begin{align}\label{Tj}
    T_j(m)\coloneqq 
    \begin{cases} m+\langle e_j,m\rangle B_{1}(e_j) & \mbox{if $\langle e_j,m\rangle \geq 0$} \\ m & \mbox{otherwise.}
    \end{cases}
\end{align}
We may write $T_j$ as $T_j^S$ if $S$ is not clear from context --- the vector $e_j$ above is more precisely $e_{S,j}$.

Note that by \eqref{eprime}, we have for each $i\in I\setminus \{j\}$ that
\begin{align}\label{Tkei}
    T_j^S(B_1(e_{S,i}))=B_1(e_{\mu_j(S),i}),
\end{align}
while $T_j^S(B_1(e_{S,j}))=-B_1(e_{\mu_j(S),j})$.  We note that one may use the compatibility condition \eqref{Lambda} to express $T_j$ in terms of $\Lambda$ as follows:
\begin{align*}
    T_j(m)\coloneqq 
    \begin{cases} m+\Lambda(m,B_1(e_j)) B_{1}(e_j) & \mbox{if $\Lambda(m,B_1(e_j)) \geq 0$} \\ m & \mbox{otherwise.}
    \end{cases}
\end{align*}

Given a tuple $\jj=(j_1,\ldots,j_s)\subset (I\setminus F)^s$, we define\footnote{As noted in \cite[Prop. 2.4]{GHKK}, $T^S_{\jj}$ may be viewed as the Fock--Goncharov tropicalization of $\mu_{S,\jj}$.} 
\begin{align}\label{mujj}
T_{\jj}\coloneqq T_{\jj}^{\s{A}}\coloneqq T^{S_{\jj_s}}_{j_s}\circ \cdots \circ T^{S_{\jj_1}}_{j_1}\colon M_{\bb{R}}\rar M_{\bb{R}}.
\end{align}
Given generic $\sQ\in M_{\bb{R}}$, let 
\begin{align}\label{psijQ}
    \psi_{\jj,\sQ}\colon M\risom M
\end{align} denote the linear automorphism obtained by restricting $T_{\jj}$ to a neighborhood of $\sQ$ and then extending linearly.  Writing $T^S_{j,-}(m)=m$ and $T^S_{j,+}(m)=m+\langle e_{S,j},m\rangle B_1(e_j)$, we see that each $T^{S}_{j,\pm}$ preserves $\Lambda$.  Since $\psi_{\jj,\sQ}$ is a composition of such maps, it follows that $\psi_{\jj,\sQ}$ preserves $\Lambda$ and thus induces an automorphism of $\kk_t^{\Lambda}[M]$.

Recall that for each seed $S_{\jj}$, we have a cone $C_{S_{\jj}}^+\subset M_{\bb{R}}$ defined as in \eqref{CS+} to be the dual to the cone spanned by $\{e_{S_{\jj},i}\}_{i\in I\setminus F}$.  This may be computed explicitly using the dual to the mutation, i.e.,
\begin{align*}
    (e_{\mu_j(S),i}^*)\coloneqq  \begin{cases}
     e_{S,i}^* \quad &\mbox{for $i\neq j$}\\
   -e_{S,j}^*+\sum_{k\neq j} \max(0,B(e_{S,k},e_{S,j}))e_{S,k}^*\quad  & \mbox{for $i=j$.}
    \end{cases}
\end{align*}

The following property of $\f{D}^{\s{A}_q}$ is key for relating the theta functions to the cluster algebras.
\begin{prop}[Chamber structure of $\f{D}^{\s{A}_q}$]\label{Chambers}
There exists a fan\footnote{Cones of $\s{C}$ are convex, but if $F\neq \emptyset$, they will not be strongly convex.} $\s{C}$ in $M_{\bb{R}}$, called \textbf{the cluster complex}, such that
\begin{enumerate}
    \item $\s{C}$ is a sub cone-complex of the cone-complex induced by the supports of the walls of $\f{D}^{\s{A}_q}$.
    \item The chambers of $\s{C}$ (i.e., the top-dimensional cones of $\s{C}$) are in bijective correspondence with the clusters of $\s{A}_q^{\up}$.  Furthermore, the chamber $\s{C}_{S_{\jj}}$ corresponding to the seed $S_{\jj}$ is equal to $T_{\jj}^{-1}(C_{S_{\jj}}^+)$, for $C_{S_{\jj}}^+$ defined as in \eqref{CS+}.
    \item Each chamber of $\s{C}$ has exactly $\#(I\setminus F)$ facets.  For two seeds $S_{\jj}$ and $S_{\kkk}$ mutation equivalent to $S$, the chambers $\s{C}_{S_{\jj}}$ and $\s{C}_{S_{\kkk}}$ in $\s{C}$ intersect along a facet $\f{d}$ if and 
    only if the clusters $\s{A}_q^{\jj}$ and $\s{A}_q^{\kkk}$ are related by a mutation $\mu^{\s{A}}_j$ for some $j\in I\setminus F$, 
    i.e., if $\mu_{S_{\jj},j}^{\s{A}}(\s{A}_q^{\jj} )=\s{A}_q^{\kkk}$.  Let $v= \pm \psi_{\jj}^{-1}(B_1(e_{S_{\jj},j}))$, with the sign chosen so that $v$ lies in $\sigma_{S,\s{A}}$ (using $\kkk$ instead of $\jj$ would result in the opposite sign).
     Then $\f{d}\subset v^{\Lambda\perp}$, and the scattering function along $\f{d}$ is $\Psi_t(z^{v})$. 
     \item For generic $\sQ_{\jj}\in \s{C}_{S_{\jj}}$, denote
     \begin{align}\label{psij}
         \psi_{\jj}\coloneqq \psi_{\jj,\sQ_{\jj}}.
     \end{align}
     Let $\sQ$ be another generic point, this time in $\s{C}_S$, and let $\gamma$ be a path from $\sQ$ to $\sQ_{\jj}$ for which the path-ordered product $\f{p}=\theta_{\gamma,\f{D}^{\s{A}_q}}$ is well-defined.  Let
     \begin{align*}
         \iota_{\jj}\coloneqq \Ad_{\f{p}}.
     \end{align*}
     Then on $\s{A}_q^{\up}$, we have
     \begin{align*}
         \psi_{\jj}\circ \iota_{\jj} = \mu_{\jj}^{\s{A}}.
     \end{align*}
\end{enumerate}
\end{prop}
The classical version of Proposition \ref{Chambers}(1)--(3) follows from applying \cite[Thm. 1.24]{GHKK} inductively, cf. \cite[Construction 1.30]{GHKK}.  The quantum analog of \cite[Thm. 1.24]{GHKK} is Proposition \ref{mut-inv}, and Proposition \ref{Chambers}(1)--(3) similarly follows from applying this inductively.

The classical version of Proposition \ref{Chambers}(4) is equivalent to \cite[Thm. 4.4]{GHKK}, specifically, cf. the commutative diagram (4.6) of loc. cit in the case $v=v'$ there.  The arguments of loc. cit. apply to yield the quantum version in the same way, i.e., as a consequence of Proposition \ref{mut-inv}.

While deducing (1)--(3) from Proposition \ref{mut-inv} is fairly simple, the proof of (4) is more involved.\footnote{The fact that two seeds correspond to the same cluster \textit{only} if they correspond to the same chamber of $\s{C}$ is actually not immediate from Proposition \ref{mut-inv}.  This instead follows from (4).}  We therefore present a different argument which uses positivity to deduce Proposition \ref{Chambers} directly from its classical analog.  We will need the following Lemmas:

\begin{lem}\label{TjAdLem}
For any linear automorphism $\phi\colon M\risom M$, any $v,m\in M$, and $\epsilon=\pm 1$, we have
\begin{align}\label{TjAdEqn}
    \phi\circ \Ad^{\epsilon}_{\Psi_t(z^v)} \circ \phi^{-1}(z^m) = \Ad^{\epsilon}_{\Psi_t(z^{\phi(v)})}(z^m)
\end{align}
in $\mr{\s{A}}_q^S$.  Here, $\Psi_t(z^v)$ is viewed as an element of $\kk_{t,\sigma}^{\Lambda}\llb M\rrb$ for some cone $\sigma$ containing $v$, and $\Psi_t(z^{\phi(v)})$ is viewed as an element of $\kk_{t,\sigma'}^{\Lambda}\llb M\rrb$ for some cone $\sigma'$ containing $\phi(v)$.  As usual, we define $\phi$ on $\mr{\s{A}}_q^S$ via $\phi(z^m)\coloneqq z^{\phi(m)}$, and similarly for $\phi^{-1}$.
\end{lem}

\begin{proof}
Note that $\phi$ induces a $\kt$-algebra isomorphism $\phi\colon\kk_t^{\Lambda}[M]\otimes_{\kk_t[z^v]} \kk_t\llb z^v\rrb\rar \kk_t^{\Lambda}[M]\otimes_{\kk_t[z^v]} \kk_t\llb z^{\phi(v)}\rrb$, so the left-hand side of \eqref{TjAdEqn} becomes \begin{align*}
    \phi(\Psi_t(z^v)^{\epsilon} \phi^{-1}(z^m) \Psi_t(z^v)^{-\epsilon}) = \Psi_t(z^{\phi(v)})^{\epsilon} (z^m) \Psi_t(z^{\phi(v)})^{-\epsilon}.
\end{align*}
This is the right-hand side of \eqref{TjAdEqn}, as desired.
\end{proof}

\begin{lem}\label{AdAd}
If 
$n\coloneqq \Lambda(v,p)$ and $z^p\in \kk_t^{\Lambda}[M]$, then
\begin{align*}
   \Ad_{\Psi_t(z^{-v})} \circ \Ad_{\Psi_t(z^v)}(z^p)=z^{p+nv}.
\end{align*}
\end{lem}
In order to make sense of the operators on the left hand side, we proceed as in the definition of $\mu_j^{\s{A}}$.  First we embed $\kk_t^{\Lambda}[M]\hookrightarrow \mr{\s{A}}_q^S$.  Then $\Ad_{\Psi_t(z^v)}(z^p)$ is an automorphism of (some) completion of $\kk_t^{\Lambda}[M]$ containing $\Psi_t(z^v)$, which induces an automorphism of $\mr{\s{A}}_q^S$, while $\Ad_{\Psi_t(z^{-v})}$ is an automorphism of some (other) completion of $\kk_t^{\Lambda}[M]$, inducing another automorphism of $\mr{\s{A}}_q^S$.
\begin{proof}
Suppose $n\geq 0$.  Then 
\begin{align*}
   \Ad_{\Psi_t(z^{-v})} \circ \Ad_{\Psi_t(z^v)}(z^p)=&z^p(1+z^vt)(1+z^vt^3)\cdots (1+z^vt^{2n-1})(1+z^{-v}t^{1-2n})^{-1}\cdots (1+z^{-v}t^{-1})^{-1} \\
    =&z^p z^{nv}t^{n^2} \\
    =&z^{p+nv}.
\end{align*}
A similar calculation applies for $n\leq 0$.
\end{proof}

\begin{myproof}[Proof of Proposition \ref{Chambers}]
The existence of the fan $\s{C}$ satisfying Properties (1) and (2) follows from the classical case \cite[Construction 1.30]{GHKK} because positivity implies that the set of walls is unchanged when we take the $t\mapsto 1$ limit.  Consider a wall $(\f{d},f)$ of $\s{C}$ with associated vector $v = \pm \psi_{\jj}^{-1}(B_1(e_{S_{\jj},j}))$ as in the statement of (3).  The classical analog of (3) says that
\begin{align}\label{limf}
    \lim_{t\rar 1} f = \lim_{t\rar 1} \Psi_t(z^{v}).
\end{align}  
Theorem \ref{PosScat} tells us that $f=\bb{E}(-p(t)z^{v})$ for $p(t)$ a bar-invariant Laurent polynomial in $t$ with positive integer coefficients.  By \eqref{limf}, $p(1)$ must equal $1$, and so (using positivity and bar-invariance) the only possibility is that $p(t)=1$, as desired.

For (4), let $\jj=(j_1,\ldots,j_s)$, and suppose inductively that $\psi_{\jj_s}\circ \iota_{\jj_s}=\mu^{\s{A}}_{\jj_s}$.  By (3), we know that the wall-crossing automorphism for crossing from $\s{C}_{\jj_s}$ to $\s{C}_{\jj}$ is given by $\Ad^{\epsilon}_{ \Psi_t(z^{\psi_{\jj}^{-1}(-\epsilon v)})}$, where $v=B_1(e_{S_{\jj_s},j_s})$, so $\mu^{\s{A}}_{S_{\jj_s},j_s}=\Ad^{-1}_{\Psi_t(z^v)}$.  The sign $\epsilon=\pm 1$ is, on one hand, determined by the requirement that $\epsilon v\in \sigma_{S,\s{A}}$.  On the other hand, by the definition of the wall-crossing automorphism \eqref{WallCross}, the exponent for $\Ad$ should be $\sign(\Lambda(\psi_{\jj_s}^{-1}(-\epsilon v),p))$ for $p$ in the interior of $\s{C}_{\jj_s}$.  To see that this is $\epsilon$, we compute
\begin{align*}
    \sign(\Lambda(\psi_{\jj_s}^{-1}(-\epsilon v),p))=-\epsilon \sign(\Lambda(v,\psi_{\jj_s}(p))) = \epsilon \sign(\langle e_{S_{\jj_s},j_s}, \psi_{\jj_s}(p)\rangle)=\epsilon,
\end{align*}
where the last equality uses part (2) to say that $p\in C_{S_{\jj_s}}^+$.

Now, we have
\begin{align*}
    \psi_{\jj}\circ \iota_{\jj}&=\psi_{\jj}\circ \Ad^{\epsilon}_{\Psi_t(z^{\psi_{\jj}^{-1}(-\epsilon v)})} \circ \iota_{\jj_s} \\
    &= \psi_{\jj}\circ \Ad^{\epsilon}_{\Psi_t(z^{\psi_{\jj}^{-1}(-\epsilon v)})} \circ \psi_{\jj_s}^{-1}\circ \mu^{\s{A}}_{S,\jj_s} \\
    &= \Ad^{\epsilon}_{\Psi_t(z^{-\epsilon v})} \circ \psi_{\jj} \circ \psi_{\jj_s}^{-1}\circ \mu^{\s{A}}_{S,\jj_s},
\end{align*}
where the second equality uses the inductive assumption and the last equality uses Lemma \ref{TjAdLem}.  Since $\mu^{\s{A}}_{S,\jj}=\mu^{\s{A}}_{S_{\jj_s},j_s} \circ \mu^{\s{A}}_{S,\jj_s}$, applying $(\mu^{\s{A}}_{\jj_s})^{-1}$ on the right and applying $\Ad^{-\epsilon}_{\Psi_t(z^{-\epsilon v})}$ on the left reduces the desired result to
\begin{align*}
    \psi_{\jj}\circ \psi_{\jj_s}^{-1} = \Ad^{-\epsilon}_{\Psi_t(z^{-\epsilon v})} \circ \Ad^{-1}_{\Psi_t(z^v)}.
\end{align*}
For $\epsilon=-1$, the right-hand side is the identity, and for $\epsilon=1$, Lemma \ref{AdAd} tells us that the right-hand side is given by $z^p\mapsto z^{p-\Lambda(v,p)v}$.  So it suffices to check that $\psi_{\jj}\circ \psi_{\jj_s}^{-1}$ is determined by the appropriate linear map.  Since the above arguments all apply in the classical limit in the same way, and we know the result holds in the classical limit, it follows that $\psi_{\jj}\circ \psi_{\jj_s}^{-1}$ is indeed determined by this linear map, and the result follows.
\end{myproof}

\begin{rmk}
\cite{Mou} proves the analogous chamber decomposition for Hall algebra scattering diagrams (as in \cite{Bridge}) associated to quivers with non-degenerate potential.  It is not always the case that the associated quantum stability scattering diagram is equivalent to $\f{D}^{\s{A}_q}$ (unless the potential is genteel as in Def. \ref{genteel}), but \cite[Thm. 1.1]{Mou} together with Proposition \ref{Chambers} above imply that, if the potential is non-degenerate, then the two will at least agree on the support of $\s{C}$. \end{rmk}

\subsection{Theta functions as elements of the quantum cluster algebras}\label{thetas-in-clusters}

We now use Proposition \ref{Chambers} to deduce our results on the ordinary and middle quantum cluster algebras.  Recall the linear maps $\psi_{\jj}$ as in \eqref{psij}, as well as the induced automorphisms of the quantum torus algebras, also denoted $\psi_{\jj}$.

\begin{cor}\label{UpTheta}
Consider $p\in M$ and generic $\sQ,\sQ_{\jj}$ in $\s{C}_S$ and $\s{C}_{S_{\jj}}$, respectively.  Suppose $\vartheta_{p,\sQ}\in \s{A}_q^S$, i.e., $\vartheta_{p,\sQ}$ is a Laurent polynomial, not just a formal Laurent series.  Then $\vartheta_{p,\sQ_{\jj}}$ is also a Laurent polynomial.  Furthermore, $\psi_{\jj}(\vartheta_{p,\sQ_{\jj}})= \mu^{\s{A}}_{S,\jj}(\vartheta_{p,\sQ})\in \s{A}_q^{S_{\jj}}$, so $\vartheta_{p}$ determines an element of $\s{A}_q^{\up}$.

If $p\in \s{C}_{S_{\jj}}^+\cap M$ and $\sQ$ is a generic point in the interior of $\s{C}_{S_{\jj}}^+$, then $\vartheta_{p,\sQ}=z^p$, and so $\vartheta_{p,\sQ}$ is a quantum cluster monomial.  On the other hand, all quantum cluster monomials are quantum theta functions.
\end{cor}
\begin{proof}
The statement that each $\vartheta_{p,\sQ_{\jj}}$ is a Laurent polynomial follows from its classical analog \cite[Prop. 7.1]{GHKK} and the positivity of broken lines --- since the classical limit has contributions from only finitely many broken lines, and applying $\lim_{t\rar 1}$ does not cause any broken lines to vanish, the quantum theta function must consist of contributions from only finitely many broken lines as well.  The claim that $\psi_{\jj}(\vartheta_{p,\sQ_{\jj}}) = \mu^{\s{A}}_{S,\jj}(\vartheta_{p,\sQ})\in \s{A}_q^{S_{\jj}}$ then follows from Theorem \ref{CPS} and Proposition \ref{Chambers}(4).

For the statements on cluster monomials, \cite[Prop. 3.6 and 3.8]{GHKK} together show that in the classical limit, the only broken line contributing to $\vartheta_{p,\sQ}$ is the straight line with attached monomial $z^p$.  By positivity, this is also the only broken line for the quantum analog (any others would have non-trivial classical limit), and since it has no kinks, the attached monomial must again be the initial monomial $z^p$.
\end{proof}

Let $\Theta_{\s{A}}^{\midd}$ denote the set of $p\in M$ such that $\vartheta_{p,\sQ}$ is a Laurent polynomial for some generic $\sQ\in \s{C}$ (hence for all generic $\sQ\in \s{C}$ by Corollary \ref{UpTheta}).  I.e., $\Theta_{\s{A}}^{\midd}$ is the set of $p$ such that $\vartheta_p\in \s{A}_q^{\up}$ rather than just $\wh{\s{A}}_q^S$.  In particular, $$\s{C}\cap M\subset \Theta_{\s{A}}^{\midd},$$ because for $p\in \s{C}\cap M$ and $\sQ$ a generic point of $M_{\bb{R}}$ sharing a maximal cone of $\s{C}$ with $p$, Corollary \ref{UpTheta} tells us that $\vartheta_{p,\sQ}=z^p$, which in particular is a Laurent polynomial (this yields a new proof of the quantum Laurent phenomenon \eqref{qLaurentPhen}).  In fact, using positivity again, we find that $\Theta_{\s{A}}^{\midd}$ is the same in the quantum setup as in the classical setup, and furthermore, we find that all of \cite[Theorem 7.5]{GHKK} carries over to the quantum setup (cf. Lemma \ref{Thetaprime}).  In particular, we have the following:
\begin{prop}\label{MidAll}
Let $\s{A}_q^{\midd}$ denote the $\kt$-submodule of $\s{A}_q^{\up}$ spanned by $\{\vartheta_p\}_{p\in \Theta_{\s{A}}^{\midd}}$.  Then $\s{A}_q^{\midd}$ is naturally a subalgebra of $\s{A}_q^{\up}$, and we have inclusions of algebras
\begin{align*}
    \s{A}_q^{\ord} \subset \s{A}_q^{\midd} \subset \s{A}_q^{\up} \subset \wh{\s{A}}_q.
\end{align*}
Furthermore, the classical $t\mapsto 1$ limit of $\s{A}_q^{\midd}$ is the classical middle cluster algebra, called $\midd(\s{A})$ in \cite{GHKK}.
\end{prop}

For the quantum $\s{X}$-algebras, recall from the proof of Proposition \ref{q2classicalTheta} that the theta functions in $\wh{\s{X}}_q^S$ can be obtained by applying $\xi^{-1}$ to the $M$-degree $0$ theta functions of $\wh{\s{A}}_q^{\prin,S}$, for $M$-degree as in \eqref{Mprime}.  Let $$\Theta^{\midd}_{\s{X}}\coloneqq \Theta_{\s{A}^{\prin}}^{\midd}\cap N',$$
viewed as a subset of $N$ using $\xi^{-1}$.  Here, $\Theta_{\s{A}^{\prin}}^{\midd}$ is defined in the same way as $\Theta_{\s{A}}^{\midd}$, but for the seed $S^{\prin}$ in place of $S$.  By \cite[Lem. 7.10(3)]{GHKK}, the $t\mapsto 1$ limit of any $\vartheta_p$ with $p\in \Theta^{\midd}_{\s{X}}$ is an element of $\s{X}^{\up}$ (as opposed to just $\wh{\s{X}}^S$).  This together with positivity implies that $\vartheta_p$ must have been in $\s{X}_q^{\up}$ (as opposed to $\wh{\s{X}}_q^S$).  In summary, we have the following:
\begin{prop}\label{midX}
Let $\s{X}_q^{\midd}$ denote the $\kt$-submodule of $\wh{\s{X}}_q^S$ spanned by $\{\vartheta_p\}_{p\in \Theta^{\midd}_{\s{X}}}$.  Then  $\s{X}_q^{\midd}$ is naturally a subalgebra of $\s{X}^{\up}_q$, and we have the inclusions of algebras
\begin{align*}
   \s{X}_q^{\ord} \subset \s{X}^{\midd}_q \subset \s{X}^{\up}_q \subset \wh{\s{X}}_q.
\end{align*}
Furthermore, the classical $t\mapsto 1$ limit of $\s{X}_q^{\midd}$ is the algebra $\midd(\s{X})$ of \cite{GHKK}.
\end{prop}
Here, the fact that $\s{X}_q^{\midd}$ is closed under multiplication follows from the observation that it is the intersection of $\s{A}_q^{\prin,\midd}$ and $\xi(\wh{\s{X}}_q)$, and each of these is closed under multiplication.  The fact that $\s{X}_q^{\ord}\subset \s{X}_q^{\midd}$ follows from noting that global monomials in $\s{X}_q^{\up}$ are just global monomials in $\s{A}_q^{\prin,\up}$ having degree $0$ with respect to the $M$-grading (precisely as in the classical analog, \cite[Lem. 7.10(3)]{GHKK}), and these are theta functions in $\s{A}_q^{\prin,\midd}$.

\subsection{Atlases}\label{iotaQrmk}

We will now describe various atlases --- i.e., sets of morphisms to (formal) quantum torus algebras --- on our quantum cluster algebras.

\subsubsection{Atlases for $\s{A}$}

Let $\sQ,\sQ_0\in M_{\bb{R}}$ be generic with $\sQ_0$ in the interior of $C_S^+$.  Let $\gamma$ be a smooth path from $\sQ_0$ to $\sQ$ which is transverse to the walls of $\f{D}^{\s{A}_q}$.  Then the isomorphism $\iota_{\sQ}\colon \wh{\s{A}}_q^S\risom \kk_{\sigma_{S,\s{A}},t}^{\Lambda}\llb M\rrb$ of \eqref{iotaQintro} is given by the path-ordered product
\begin{align*}
    \iota_{\sQ}\coloneqq \theta_{\gamma,\f{D}^{\s{A}_q}}.
\end{align*}
These charts $\iota_{\sQ}$ for generic $\sQ\in M_{\bb{R}}$ form the scattering atlas for $\wh{\s{A}}^S_q$.

The subset of charts $\iota_{\sQ}$ for generic $\sQ\in \s{C}$ is what we call the cluster atlas.  For $\sQ\in C_{S_{\jj}}^+\subset \s{C}$, denote the corresponding chart $\iota_{\sQ}$ by $\iota_{\jj}$.  If $\sQ\in C_{S_{\jj}}^+$, then Proposition \ref{Chambers}(4) implies that $\psi_{\jj} \circ
\iota_{\sQ}=\mu_{S,\jj}^{\s{A}}$ on $\s{A}_q^{\up}$.  Since $\psi_{\jj}$ is a $\kt$-algebra automorphism of $\kk_t^{\Lambda}[M]$, we may view the restriction of the cluster atlas to $\s{A}_q^{\up}$ as agreeing with the usual set of quantum clusters.

\subsubsection{Atlases for $\s{X}$}\label{AtlasX}
To construct $\iota_{\sQ}$ for the $\s{X}$-spaces in terms of path-ordered products, we use the injections $\xi\colon \wh{\s{X}}_q^S\hookrightarrow \wh{\s{A}}_q^{\prin,S}$ and $\xi\colon N_{\bb{R}}\hookrightarrow M_{\bb{R}}^{\prin}$.  Let $\sQ$ be a generic point in $N_{\bb{R}}$, and let $\sQ_0$ be a generic point in the interior of $C_{S^{\prin}}^+\subset M_{\bb{R}}^{\prin}$.  Let $\gamma$ be a smooth path from $\sQ_0$ to $\xi(\sQ)$ in $M_{\bb{R}}^{\prin}$, transverse to the walls of $\f{D}^{\s{A}_q^{\prin}}$.  Then $\iota_{\sQ}\colon \wh{\s{X}}_q^S\risom \kk_{\sigma_{S,\s{X}},t}^{B}\llb M\rrb$ is given by
\begin{align*}
    \iota_{\sQ}\coloneqq \xi^{-1}|_{\xi(\wh{\s{X}}_q^S)}\circ \theta_{\gamma,\f{D}^{\s{A}_q^{\prin}}}\circ \xi.
\end{align*}
Here we use from the proof of Proposition \ref{q2classicalTheta} that scattering with respect to $\f{D}^{\s{A}_q^{\prin}}$ preserves the $M$-degree and that $\xi(\wh{\s{X}}_q^S)$ can be identified with the subalgebra of degree $0$ elements of $\wh{\s{A}}_q^{\prin,S}$.  This gives the scattering atlas for $\wh{\s{X}}_q^S$.

In fact, we could more generally allow $\gamma$ to end at any generic $\sQ\in M_{\bb{R}}^{\prin}$, not just restricting to $\sQ\in \xi(N_{\bb{R}})$, giving what we call the principal coefficients scattering atlas.  Theorems \ref{qPosLoc} and \ref{qAtomic} still hold for $\wh{\s{X}}_q^S$ when we allow these more general choices of $\sQ$, by the same arguments. 

Suppose $\sQ$ is in the cluster complex for $\s{A}_q^{\prin}$, say in the chamber corresponding to a sequence of mutations $\jj$, and denote the corresponding $\iota_{\sQ}$ by $\iota_{\jj}^{\prin}$.  Let $\psi_{\jj}^{\prin}$ denote the linear automorphism of $M^{\prin}$ defined as in \eqref{psij} but using $S^{\prin}$.  Then $\psi_{\jj}^{\prin}\circ\iota^{\prin}_{\jj}$ agrees with $\mu_{S^{\prin},\jj}^{\s{A}^{\prin}}$ on $\s{A}_q^{\prin,\up}$, and by further restricting, they agree on $\s{X}_q^{\up}$.  Thus, as in the $\s{A}$-case, we can view the cluster atlas on $\s{X}_q^{\up}$ as agreeing with the usual set of quantum clusters, as considered in \cite{FG1}.

\subsection{Enough global monomials}\label{EGM}

Consider the subalgebras $\s{A}_q^{\can}\subset \wh{\s{A}}_q^S$ and $\s{X}_q^{\can} \subset \wh{\s{X}}_q^S$ generated by the theta functions.  By Proposition \ref{TopBasis}, the theta functions form topological bases for these algebras.  The theta functions will form an ordinary basis precisely when the multiplication rule is polynomial, i.e., if for each $p_1,p_2$, there are only finitely many $p$ such that the structure constant $\alpha(p_1,p_2;p)$ is nonzero.  By strong positivity, nonzero structure constants never vanish when applying $\lim_{t\rar 1}$, so this polynomiality property is equivalent to the analogous property in the classical limit (alternatively, this equivalence follows from Lemma \ref{Thetaprime}, with $\Theta$ and $\Theta'$ both equal to the full lattice).  Thus,
\begin{prop}\label{AXcan}
The theta functions form a $\kt$-module basis for $\s{A}_q^{\can}$ (resp. $\s{X}_q^{\can}$) if and only if their classical limits form a $\kk$-module basis for the corresponding classical algebra $\s{A}^{\can}$ (resp. $\s{X}^{\can}$).
\end{prop}
Conditions implying polynomiality in the classical setup are explored in \cite{GHKK}.  In particular, \cite[Thm. 0.12(1)]{GHKK} states that the multiplication of theta functions in $\s{A}^{\can}$ (resp. $\s{X}^{\can}$) is polynomial whenever $\s{A}^{\can}$ (resp. $\s{X}^{\can}$) has ``Enough Global Monomials,'' as we will now define.

Consider a Laurent polynomial $g=\sum_{v\in L} c_vz^v$ in a quantum torus algebra $\kk_t^{\omega}[L]$ or in the classical limit $\kk[L]$.  Let $u\in L^*$.  Then $g^{\trop}(u)$ is defined by
\begin{align*}
    g^{\trop}(u)\coloneqq \min_{v\in L|c_v\neq 0} \langle v,u\rangle.
\end{align*}
Geometrically (at least in the classical setup), this should be viewed as the valuation of $g$ along the toric boundary divisor corresponding to the direction $u$.

\begin{dfn}\label{EGMdef}
One says that $\s{A}_q$ (resp., $\s{X}_q$) has \textbf{enough global monomials (EGM)} if for every $u\in N_S$ (resp., every $u\in M_S$), there exists a global monomial $g\in \s{A}_q^{\ord}\subset \s{A}_q^S= \kk_t^{\Lambda}[M]$ (resp. $g\in \s{X}_q^{\ord}\subset \s{X}_q^S=\kk_t^{B}[N]$) such that $g^{\trop}(u)< 0$.
\end{dfn}
The analogous definition applies in the classical setup.  Of course, having EGM in the quantum setup is equivalent to having EGM in the classical setup by our usual argument: positivity implies that no terms vanish when applying $\lim_{t\rar 1}$.  See \cite[Prop. 0.14]{GHKK} for conditions implying that $\s{A}^{\prin}$ has EGM.

\begin{prop}\label{canEGM}
Suppose $\s{A}_q$ (resp. $\s{X}_q$) has enough global monomials.  Then $\s{A}_q^{\can}$ (resp., $\s{X}^{\can}_q$) is finitely generated over $\kt$.  Furthermore, each element of $\s{A}_q^{\up}$ (resp. $\s{X}_q^{\up}$) is a finite $\kk_t$-linear combination of theta functions, so in particular we have  $\s{A}_q^{\up} \subset \s{A}_q^{\can}$ (resp., $\s{X}_q^{\up}\subset \s{X}_q^{\can}$).
\end{prop}
\begin{proof}
The classical analog of the first statement for $\s{A}^{\prin}$ and $\s{X}$ is \cite[Cor. 8.21]{GHKK}, 
while the classical analog of the second statement for $\s{A}^{\prin}$ is \cite[Prop. 8.22]{GHKK}.  The same argument applies to prove the second statement for $\s{X}$, and if the injectivity assumption holds, then the same arguments also apply to $\s{A}$.  The proof of the quantum version is the same.  Indeed, the arguments of \cite[\S 8.1-\S 8.3]{GHKK} are based entirely on properties of convex piecewise-linear functions and on the directions of broken lines/exponents of the associated monomials, and this data is identical in the classical and quantum setups.
\end{proof}

\subsection{Proofs of the main theorems}\label{MainProof}
\begin{myproof}[Proof of Theorem \ref{MainThm}]
We now put our results together to complete the proof of Theorem \ref{MainThm}.  First, note that all broken lines, mutations, and actions of path-ordered products could be defined over $\bb{Z}_t$ rather than $\kk_t$, so we may work  over $\bb{Z}_t$ as in \S \ref{MainResultsSection}.  The statement that the theta functions form a topological basis for $\wh{\s{V}}_q^S$ is part of Proposition \ref{TopBasis}.  The statement that the global monomials are elements of the theta basis is Corollary \ref{UpTheta} in the $\s{A}$ case and the proof of Proposition \ref{midX} in the $\s{X}$ case.

The fact that $\iota_{\sQ}(\vartheta_p)$ has coefficients in $\Z[t^{\pm 1}]$ instead of $\Z[t^{\pm 1/D}]$ was already implicit in our definition of the theta functions in \S \ref{theta-def} --- the fundamental point here is the assumption that $\omega(L_0,L)$ is contained in $\bb{Z}$ instead of just $\bb{Q}$.  That $\vartheta_p$ is $p$-pointed is given by \eqref{pointed}.  Next we note that the functions on the walls of $\f{D}^{\s{V}_q}$ are bar-invariant and pre-Lefschetz by the corresponding statements in Theorem \ref{PosScat}.  That $\vartheta_p$ is bar-invariant and pre-Lefschetz then follows as a special case of Proposition \ref{bar-pL}.  The claim regarding the parity of the coefficients is a special case of Proposition \ref{ParitProp}.

Next, positivity of the coefficients $\alpha(p_1,p_2;p)$ is given by Theorem \ref{qPosStrong}.  Finally, atomicity is a special case of Theorem \ref{qAtomic}, with the extension to the principal coefficients scattering atlas following as in \S \ref{AtlasX}.
\end{myproof}

\begin{myproof}[Proof of Theorem \ref{MainThm2}]

(1) is just Proposition \ref{q2classicalTheta}.  (2) follows from our usual positivity arguments --- the cardinality of the set of broken lines with nonzero attached monomial is unchanged when we apply $\lim_{t\rar 1}$.  The statement that $\Theta_{\s{V}}^{\midd}=\Theta^{\midd}_{\s{V},\bb{R}}\cap L_{\s{V}}$ for $\Theta^{\midd}_{\s{V},\bb{R}}$ a convex cone then follows from \cite[Thm. 0.3(4)]{GHKK}.  That this remains a convex cone 
under the application of any $T_{\jj}^{\s{V}}$ is then a consequence of (6), which itself follows from Corollaries \ref{TjCor} and \ref{FormalLaurent}.

(3) is Proposition \ref{MidAll} for $\s{V}=\s{A}$ and Proposition \ref{midX} for $\s{V}=\s{X}$.  (4) is the combination of Propositions \ref{AXcan} and \ref{canEGM}.  The first statement of (5) is a special case of (2), and the second statement then follows from the last part of (4).  (7) follows from Proposition \ref{homprop}.

Next, (8) follows from Lemmas  \ref{pi1} and \ref{pi1Theta}.  The $*=\up$ case is part of Lemma \ref{pi1}, while the $*=\can$ case follows from the fact that $B_1$ takes theta functions to theta functions, which is the statement of Lemma \ref{pi1Theta}.  Combining these, a theta function in $\s{X}_q^{\up}$ maps to a theta function and also to an element of $\s{A}_q^{\up}$, so the $*=\midd$ case follows.  The $*=\ord$ case follows from noting that $B_1$ takes global monomials to global monomials.  The fact that the image has degree $0$ is immediate from the definition of the $K_{\s{A}}$--grading and from the homogeneity of the theta functions as in Proposition \ref{homprop}.

Finally, (9) is a special case of Proposition \ref{HV}.
\end{myproof}

\begin{myproof}[Proof of Theorem \ref{AcyclicLefschetz}]
By definition, $S$ being acyclic means that $(B(e_i,e_j))_{i,j\in I\setminus F}$ is the signed adjacency matrix of an acyclic quiver.  If there is a compatible $\Lambda$, then \eqref{Lambda} implies that $$\left(\Lambda(B_1(e_i),B_1(e_j))\right)_{i,j\in I\setminus F}$$ is also the signed adjacency matrix of an acyclic quiver.   Theorem \ref{acyclic_scat_thm} thus applies to by $\f{D}^{\s{X}_q}$ and $\f{D}^{\s{A}_q}$.  Finally, the claim for the theta functions follows using Lemma \ref{Lefschetz_prop}.
\end{myproof}

\begin{myproof}[Proof of Theorem \ref{flat}]
It follows from Theorem \ref{MainThm2}(2,3) for $*=\midd,\can$, or more directly for $*=\ord$, that the images of these restrictions are the corresponding classical analogs, so in these three cases they indeed give surjective homomorphisms $\lim_{t\rar 1}\colon \s{V}_q^*\rar \s{V}^*$.  Now, since $\langle t^{1/D}-1\rangle\wh{\s{V}}_q^S$ is the kernel of $\lim_{t\rar 1}\colon \wh{\s{V}}_q^S\rar \wh{\s{V}}^S\coloneqq \kk_{\sigma}\llb L\rrb$, the injectivity problem amounts to showing that
\begin{align*}
    \left(\langle t^{1/D}-1\rangle\wh{\s{V}}_q^S\right)\cap \s{V}_q^* \subset \langle t^{1/D}-1\rangle\s{V}_q^*
\end{align*}
(the reverse containment being obvious).

Suppose $*=\up$.  If $f\in \wh{\s{V}}_q^S\setminus \s{V}_q^{\up}$, then there is some cluster where $f$ is an infinite Laurent series (not a Laurent polynomial), and then $(t^{1/D}-1)f$ is also an infinite series in this cluster.  So then $(t^{1/D}-1)f\notin \s{V}_q^{\up}$.  I.e., $(t^{1/D}-1)f$ can be in $\s{V}_q^{\up}$ only if $f\in \s{V}_q^{\up}$. The desired containment follows.

Now suppose $*=\midd$ or $*=\can$.  Let $f=\sum_p a_p\vartheta_p\in \s{V}_q^*$  lie in the kernel of $\lim_{t\rightarrow 1}$.  Let $\?{\vartheta}_p$ denote the classical theta function corresponding to $p$ as in \cite{GHKK}, so $\?{\vartheta}_p=\lim_{t\rar 1}(\vartheta_p)$ by Theorem \ref{MainThm}(4).  Then 
\begin{align*}
    \lim_{t\rar 1} (f)=\sum_p \lim_{t\rar 1}(a_p) \?{\vartheta}_p,
\end{align*}
and since the classical theta functions are linearly independent, this implies that $\lim_{t\rar 1}(a_p)=0$ for each $p$.  Hence, each $a_p$ is divisible by $(t^{1/D}-1)$, and so $f$ must be divisible by $(t^{1/D}-1)$ in $\s{V}_q^{*}$, as desired.

Finally, suppose that $\s{V}^{\ord}=\s{V}^{\midd}$.  Since we have proven the claim for $*=\midd$, the $*=\ord$ case follows as long as $\s{V}_q^{\ord}=\s{V}_q^{\midd}$.  Let $\Theta$ denote the set of $p\in L_{\s{V}}$ for which $\vartheta_p$ is a global monomial (alternatively, a quantum cluster variable), and let $\Theta'=\Theta_{\s{V}}^{\midd}$.  Then the claim follows from Lemma \ref{Thetaprime}, because $\s{V}^{\ord}=\s{V}^{\midd}$ exactly means that $\{\lim_{t\rar 1}\vartheta_p\}_{p\in \Theta'}$ spans $A_{\Theta}=\s{V}^{\ord}$ as in the hypotheses there.
\end{myproof}

\section{Reduction of quantum positivity to the two-wall cases}\label{PosSect}

Our goal for the remainder of the paper is to prove Theorem \ref{PosScat}.  In this section, we will use the arguments of \cite[\S C.3]{GHKK} to reduce to the case of two initial walls, i.e., we will reduce to the case
\begin{align}\label{2In}
    \f{D}_{\In}=\{(v_1^{\omega\perp},\EE(-t^{a_1}z^{v_1})), (v_2^{\omega\perp},\EE(-t^{a_2}z^{v_2}))\}
\end{align}
for $a_1,a_2\in\mathbb{Z}$ to prove the positivity statement in Theorem \ref{PosScat}, and the case
\begin{align}\label{pLIn}
    \f{D}_{\In}=\{(v_1^{\omega\perp},\EE(-\!\pl_{a_1}(t)z^{v_1})), (v_2^{\omega\perp},\EE(-\!\pl_{a_2}(t)z^{v_2}))\}
\end{align}
to prove the statement regarding pL type polynomials.  Furthermore, we will be able to assume that $v_1$ and $v_2$ form part of a basis for $L$.  From here, \cite{GHKK} uses their ``change of lattice trick'' to further reduce to the case where $|\omega(v_1,v_2)|=1$.  However, this final trick does not work in the quantum setting!  We will therefore use different methods, developed in \S \ref{pos_DT_sec} and applied in \S \ref{2wallSection}, to tackle the cases of \eqref{2In} and \eqref{pLIn}.

\subsection{The change of monoid trick}\label{MonoidTrick}

For any $v\in L_0^+$, we say $v$ has \textbf{order $k$} if $v\in kL_0^+\setminus (k+1)L_0^+$.  The change of monoid trick from \cite[\S C.3]{GHKK} is essentially a way to reduce to the case where the vectors $v_1,\ldots,v_s$ in \eqref{DIn} all have order $1$ and form part of a basis for $L$.  We review this here.
 
Let $u_1,\ldots, u_r$ be a basis for $L$.  For $\f{D}_{\In}$ as in \eqref{DIn}, consider the lattice
$$\wt{L}\coloneqq \bb{Z}\langle 
\wt{v}_1,\ldots,\wt{v}_s,\wt{u}_1,\ldots,\wt{u}_r\rangle$$ together with the 
map 
\begin{align*}
\pi\colon &\wt{L}\rar L, \\ 
&\sum_{j=1}^{s} a_j\wt{v}_j+\sum_{i=1}^rb_i\wt{u}_i \mapsto 
\sum_{j=1}^s a_jv_j+\sum_{i=1}^rb_iu_i.
\end{align*}
The form $\omega$ on $L$ pulls back to give a new rational skew-symmetric form $\wt{\omega}\coloneqq \pi^*\omega$ on $\wt{L}$.  Let $\wt{\sigma}$ be the cone in $\wt{L}$ spanned by $\wt{v}_1,\ldots,\wt{v}_s$.  Note that $\pi(\wt{\sigma})\subset \sigma$.  Define $\wt{L}_0=\bb{Z}\langle\wt{v}_1\ldots,\wt{v}_s\rangle$ and $\wt{L}_0^+\coloneqq (\wt{\sigma}\cap \wt{L}_0)\setminus \{0\}$. Consider the scattering diagram $\wt{\f{D}}_{\In}$ in $\wt{L}_{\bb{R}}$ over $\f{g}_{\wt{L}_0^+,\wt{\omega}}$ given by
\begin{align*}
\wt{\f{D}}_{\In}\coloneqq \left\{(\pi^{-1}(\f{d}_i),\EE(-p_i(t)z^{\wt{v}_i}))\colon i=1,\ldots,s\right\}.
\end{align*}

Define the Lie algebra homomorphism
\begin{align*}
\varpi\colon&\f{g}_{\wt{L}_0^+,\wt{\omega}}\rightarrow \f{g}_{L_0^+,\omega}  \\
&\hat{z}^v\mapsto \frac{z^{\pi(v)}}{t^{\lvert v\lvert}-t^{-\lvert v\lvert}}.
\end{align*}
Then Lemma \ref{QuotientScat} gives us a well-defined consistent scattering diagram in $L_{\bb{R}}$ over $\f{g}_{L_0^+,\omega}$, namely,
$(\pi,\varpi)_*(\scat(\wt{\f{D}}_{\In}))\coloneqq \{(\pi(\f{d}),\varpi(f_{\f{d}})\colon (\f{d},f_{\f{d}})\in \scat(\wt{\f{D}}_{\In})\}$, and the incoming walls of $(\pi,\varpi)_*(\scat(\wt{\f{D}}_{\In}))$ correspond to the incoming walls of $\scat(\wt{\f{D}}_{\In})$, which are given by $\wt{\f{D}}_{\In}$.  It follows that the incoming walls of $(\pi,\varpi)_*(\scat(\wt{\f{D}}_{\In}))$ are precisely the walls of $\f{D}_{\In}$, so by the uniqueness part of Theorem \ref{KSGS}, we have that up to equivalence, 
\begin{align*}
    \scat(\f{D}_{\In})=(\pi,\varpi)_{*}(\scat(\wt{\f{D}}_{\In})).
\end{align*}

So to prove that each scattering function of $\scat(\f{D}_{\In})$ can be chosen to be of the form $\EE(-t^az^v)$ or $\EE(-\pl_{a}(t)z^v)$, it suffices to prove the analogous statement for $\scat(\wt{\f{D}}_{\In})$ in $\wt{L}_{\bb{R}}$ over $\f{g}_{\wt{L}_0^+,\wt{\omega}}$.  Note that each $\wt{v}_i$ has order $1$ in $\wt{L}_0^+$.  We thus reduce to the case where each $v_i$ has order $1$ and they in fact form a basis for $L^{\oplus}$ and for $L_0$.  

\subsection{The perturbation trick}\label{PertTrick}

Firstly, using \eqref{plus-prod} and Example \ref{EquivEx}(1), if $\f{D}_{\In}$ is a scattering diagram for which all the walls have the form given in (\ref{DIn}), we may replace $\f{D}_{\In}$ with an equivalent scattering diagram for which all of the walls are of the form
\begin{align}\label{basic_pos_wall}
(v_i^{\omega\perp},\EE(-t^{a}z^{v_i}))
\end{align}
where the number of such walls appearing in the equivalent scattering diagram is the coefficient of $t^a$ in $p_i(t)$.  Similarly, if all of the polynomials $p_i(t)$ are of pL type, we may replace $\f{D}_{\In}$ with an equivalent diagram in which all walls are of the form
\begin{equation}\label{basic_PL_wall}
(v_i^{\omega\perp},\EE(-\!\pl_j(t)z^{v_i}))
\end{equation}
where the number of copies of this wall is given by $\lambda_j$ in the decomposition
\begin{equation*}
p_i(t)=\sum_{j\in\mathbb{Z}}\lambda_j\pl_j(t).
\end{equation*}

For any scattering diagram $\f{D}$, the \textbf{asymptotic scattering diagram} $\f{D}_{\as}$ of $\f{D}$ is defined as follows: every wall $(v+\f{d},f_{\f{d}})\in \f{D}$, with $\f{d}$ denoting a rational polyhedral cone (apex at the origin) and $v \in L_{\bb{R}}$ translating this cone, is replaced by the wall $(\f{d},f_{\f{d}})$.  Note that consistency of $\f{D}$ implies consistency of $\f{D}_{\as}$.  Indeed, for each $k>0$ and each closed path $\gamma\colon [0,1]\rar L_{\bb{R}}$ crossing the walls of $\f{D}_{\as}$ transversely,  we can find $R\gg 0$ such that
\begin{align*}
    \theta_{\gamma,(\f{D}_k)_{\as}} = \theta_{R\gamma,(\f{D}_k)_{\as}} = \theta_{R\gamma,\f{D}_k}  = \id \in G_k,
\end{align*}
where $R\gamma$ is the path mapping $t$ to $R\cdot \gamma(t)$ in $L_{\bb{R}}$. Since this holds for every $k$, we deduce that $\theta_{\gamma,
\f{D}_{\as}}=\Id\in G$.

We deal with the positivity statement in Theorem \ref{PosScat} first.  So assume that we have replaced $\f{D}_{\In}$ with an equivalent scattering diagram in which all walls are of the form (\ref{basic_pos_wall}).  Now suppose that we deform $\f{D}_{\In}$ by translating each wall by some generic vector in $L_{\bb{R}}$, yielding a new scattering diagram $\f{D}'_{\In}$.  Let $\f{D}'\coloneqq \scat(\f{D}'_{\In})$.  Since $\f{D}'_{\as}=\scat(\f{D}_{\In})$, it suffices to check that all scattering functions of $\f{D}'$ (up to equivalence) have the form $\EE(-t^az^v)$ for various $v\in L_0^+$ and $a\in\mathbb{Z}$.

Suppose inductively that every wall of any such $\f{D}'_k$ (up to equivalence) is of the form $(\f{d},f_{\f{d}})$ for some $f_{\f{d}}=\EE(-t^az^v)$ in $G_k$.   Since $(\f{D}'_{\In})_1$ is already consistent over $\f{g}_1$ (because $\f{g}_1$ is Abelian and all walls of $\f{D}'_{\In}$ are full affine hyperplanes), all walls of $\f{D}'\setminus \f{D}'_{\In}$ have order $>1$ (the order of a non-trivial wall $(\f{d},f_{\f{d}})$ is the smallest integer $k$ for which the projection of $f_{\f{d}}$ to $G_k$ is non-trivial).  Hence, $\f{D}'_1=(\f{D}'_{\In})_1$ (up to equivalence), and since all walls of $\f{D}'_{\In}$ have the desired form, this yields the base case for the induction.

Now consider $\f{D}'_{k+1}$.  Since this scattering diagram is finite, we can first apply the equivalences of Example \ref{EquivEx}(1,2) to replace $\f{D}'_{k+1}$ with an equivalent scattering diagram for which any two walls intersect in codimension at least 2.  Then we again use Example \ref{EquivEx}(1) to replace $\f{D}'_{k+1}$ with an equivalent scattering diagram for which all functions on walls are of the form $\EE(\epsilon t^az^v)$ for $\epsilon= \pm 1$ and various $a\in\bb{Z}$ and $v\in L^{\oplus}$, and any two walls either have identical support, or intersect in codimension at least two, i.e., only along a joint.  We pick this $\f{D}'_{k+1}$ to be minimal, in the following sense: if two walls $(\f{d}_1,\EE(\epsilon_1 t^{a_1}z^{v_1}))$ and $(\f{d}_2,\EE(\epsilon_2 t^{a_2}z^{v_2}))$ satisfy $\f{d}_1=\f{d}_2$, $a_1=a_2$ and $v_1=v_2$, then $\epsilon_1=\epsilon_2$.  Since $(\f{D}'_{k+1})_k$ is equivalent to $\f{D}'_k$, it follows from the inductive assumption that if the order of $v$ is less than $k+1$, then any wall of the form $(\f{d},\EE(\epsilon t^{a}z^{v}))$ in $\f{D}'_{k+1}$ has $\epsilon=-1$, as required.  

So to complete the inductive step we consider a wall
$(\f{d},\EE(\epsilon t^{a}z^{v}))$
of $\f{D}'_{k+1}$ in which $v$ has order $k+1$.   If the wall is incoming, then $\epsilon=-1$ by assumption.  If it is not incoming, then $\f{d}$ contains a joint $\f{j}$, i.e, a codimension $2$ stratum of the affine polyhedral complex formed by $\f{D}'_{k+1}$.  Consider a point $p\in \f{j}$ that is not in any other joint of $\f{D}'_{k+1}$ (i.e., not in a codimension $3$ stratum of this affine polyhedral complex).  We define a new scattering diagram $\f{D}_{k+1,p}$ by taking only the walls of $\f{D}'_{k+1}$ that contain $p$ and extending them to infinity.  Precisely, for each wall $(\f{d},\EE(\epsilon' t^{a'}z^{v'}))$ of $\f{D}'_{k+1}$, with $p\in\f{d}$, the scattering diagram $\f{D}_{k+1,p}$ contains the wall 
\[
\left(\{p+\delta(x-p)\in L_{\bb{R}}\colon x\in \f{d},\delta\in \bb{R}_{\geq 0}\},\EE(\epsilon' t^{a'}z^{v'})\right).
\]
This scattering diagram is consistent, and contains as its unique joint the codimension 2 subspace of $L_{\bb{R}}$ containing $\f{j}$.  Also, since walls of $\f{D}_{\In}'$ were generically translated, and since all walls of $\f{D}'\setminus \f{D}'_{\In}$ have order $>1$, $\f{D}'_{k+1}$ has at most two order $1$ walls containing $\f{j}$, and so $\f{D}_{k+1,p}$ has at most two order $1$ walls.  Let us label the walls of $\f{D}_{k+1,p}$ as $\{(\f{d}_j,\EE(\epsilon_jt^{a_j}z^{v_j})\colon j=1,\ldots,s\}$.  We may assume that any order $1$ walls have $j=1$ or $2$.

We claim that, up to equivalence, we can assume that every incoming wall $\left(\f{d},\EE(\epsilon t^{a}z^{v})\right)$ of $\f{D}_{k+1,p}$ has $\epsilon=-1$.  If the order of $v$ is less than $k+1$, this follows from the inductive hypothesis.  If the order of $v$ equals $k+1$, since $\EE(-t^{a}z^{v})$ is central in $G_{k+1}$, we can add the wall $(v^{\omega\perp}+p,\EE(-t^{a}z^{v}))$ to $\f{D}_{k+1,p}$ without changing its equivalence class, and then by Examples \ref{EquivEx}(1,2) we can remove any incoming wall for which $\epsilon=1$.

Now we apply the change of monoid trick to $\f{D}_{k+1,p}$.  Let $\wt{\f{D}}_{k+1,p,\In}$ denote the lifted  scattering diagram in $\wt{L}_{\bb{R}}$ as in \S \ref{MonoidTrick}, and consider a wall  $(\f{d},f_{\f{d}})\in \scat(\wt{\f{D}}_{k+1,p,\In})$ whose projection to $L_{\bb{R}}$ has order $k+1$.    Since the order of $\wt{v}_j$ is less than the order of $v_j$ for $j\geq 3$, the order of $(\f{d},f_{\f{d}})$ is $\leq k$ in $\wt{L}_0^+$ unless the wall is in $\scat(\{(\wt{\f{d}}_j,\EE(-t^{a_j}z^{\wt{v}_j}))\colon j=1,2\})$.  If the order is $\leq k$, then the inductive assumption applies to show that $(\f{d},f_{\f{d}})$ has the desired form.  Thus, it suffices to check the claim for $\scat(\{(\wt{\f{d}}_j,\EE(-t^{a_j}z^{\wt{v}_j}))\colon j=1,2\})$.

We have thus reduced the positivity statement in Theorem \ref{PosScat} to the case of two incoming walls.  Relabelling and applying the change of monoid trick, we can assume that our initial scattering diagram is
\begin{align}\label{two_wall_SD}
    \f{D}_{\In}=\{(v_1^{\omega^{\perp}},\EE(-t^{a_1}z^{v_1})),(v_2^{\omega^{\perp}},\EE(-t^{a_2}z^{v_2}))\}
\end{align}
for $v_1,v_2$ forming a basis for $L^{\oplus}$.  This is exactly the special case of the positivity statement of Theorem \ref{PosScat} that is considered in Corollary \ref{2WallPos}.

Exactly the same argument, starting with replacing $\f{D}_{\In}$ by a scattering diagram in which all walls are of the form given in \eqref{basic_PL_wall}, and continuing to replace instances of $t^a$ for various $a$ in the proof by $\pl_a(t)$, reduces the statement regarding pL type polynomials in Theorem \ref{PosScat} to the case of two walls as in \eqref{pLIn}.

\section{Positivity for Donaldson--Thomas invariants}
\label{pos_DT_sec}
After a reminder of some notation and constructions regarding quiver representations in \S\ref{quivers_sec}, in \S\ref{coh_DT_sec} we prove Theorem \ref{DT_pos_thm}, our main positivity result for the refined DT invariants of the category of right modules for the path algebra $\mathbb{C} Q$ of a quiver $Q$.  This generalizes a positivity result in \cite{MeiRei} which applied only to the case of acyclic quivers.  The proof of this theorem is within the framework of \textit{cohomological} Donaldson--Thomas theory.  The reader who is more interested in cluster algebras than Donaldson--Thomas theory is advised to focus on this positivity theorem, as opposed to the one that follows it in \S\ref{ref_DT_sec}, since it is sufficient to prove the positivity of quantum theta functions and involves fewer unfamiliar concepts, so it should be easier to digest for the uninitiated.

On the other hand, since with a little work, the original positivity result of Meinhardt and Reineke \cite{MeiRei}, written in the language of \textit{refined} DT theory, can be used to prove the pL preservation statement of Theorem \ref{PosScat}, we present this version of the theory in \S\ref{ref_DT_sec}.  Since the pL property is stronger than positivity, this provides an alternative proof of the (strong) positivity of quantum theta functions.

\subsection{Quiver representations}
\label{quivers_sec}
We introduce the background notation for talking about either flavour of Donaldson--Thomas theory for quivers.  We will work with right modules throughout, and so our sign conventions will agree with \cite{DavPos} and differ slightly from \cite{DavMei}.  A \textbf{quiver} $Q$ is determined by two sets $Q_1$ and $Q_0$, the arrows and vertices respectively, along with a pair of maps $s,t\colon Q_1\rightarrow Q_0$ taking an arrow to its source and target, respectively.  We will always assume that $Q_0$ and $Q_1$ are finite.  In contrast with the usual situation in cluster algebras, we allow $Q$ to contain oriented cyclic paths of length $1$ and $2$.

We denote by $K=\mathbb{Z}_{\geq 0}^{Q_0}$ the semigroup of dimension vectors for $Q$.  For $i\in Q_0$ we denote by $\g{i}\in K$ the generator corresponding to $i$.  
We denote by $\mathbb{C}Q$ the free path algebra of $Q$ over $\mathbb{C}$, and for each $i\in Q_0$ we denote by $\lp{i}$ the ``lazy path'' of length $0$ beginning and ending at $i$.  Let $v\in K$.  A $v$-dimensional right $\mathbb{C}Q$-module $\rho$ is determined by a $Q_0$-tuple of complex vector spaces $\{\rho_i\coloneqq \rho\cdot \lp{i}\}_{i\in Q_0}$ with $\dim_{\mathbb{C}}(\rho_i)=v_i$
and a linear map $\rho(a)\colon \rho_{t(a)}\rightarrow \rho_{s(a)}$ for each $a\in Q_1$.  We denote by 
\begin{equation}
    \label{Qstackdef}
\mathfrak{M}_{v}(Q)\coloneqq \prod_{a\in Q_1}\Hom(\mathbb{C}^{v_{t(a)}},\mathbb{C}^{v_{s(a)}})/\GL_{v}
\end{equation}
the stack of $v$-dimensional right $\mathbb{C}Q$-modules, where
\[
\GL_{v}\coloneqq \prod_{i\in Q_0} \GL_{v_i}
\]
is the gauge group, acting by change of basis on each of the factors $\bb{C}^{v_i}$.  This stack is smooth, and so all open substacks of it are also smooth.  We denote by $\udim(\rho)=(\dim(\rho_i))_{i\in Q_0}\in K$ the \textbf{dimension vector} of a $\mathbb{C}Q$-module $\rho$.  Then
\[
\dim(\rho)=\sum_{i\in Q_0}\udim(\rho)_i.
\]
We define the Euler pairing on dimension vectors via
\begin{align*}
    \chi_Q\!\colon &\,\,\mathbb{Z}^{Q_0}\times\mathbb{Z}^{Q_0}\rightarrow \mathbb{Z}\\
    &(\gamma',\gamma'')\mapsto \sum_{i\in Q_0}\gamma'_i\gamma''_i-\sum_{a\in Q_1}\gamma'_{t(a)}\gamma''_{s(a)},
\end{align*}
and we define 
\begin{align}\label{BQ}
     B_Q( \gamma',\gamma'')=\chi_Q(\gamma'',\gamma')-\chi_Q(\gamma',\gamma'').
\end{align}
Where there is no possibility of confusion we will omit the superscript $Q$.  Note that $\dim \f{M}_v(Q)=-\chi_Q(v,v)$.

We denote by $\mathscr{B}_Q$ the $\mathbb{Z}\lcbs t\rcbs$-algebra freely generated as a $\mathbb{Z}\lcbs t\rcbs$-module by symbols $x^{v}$, for $v\in K$.  The multiplication is determined by
\[
x^{v}\cdot x^{v'}=t^{B( v,v')}x^{v+v'}.
\]
We complete with respect to the two sided ideal $I_Q\subset \mathscr{B}_Q$ spanned as a $\mathbb{Z}\lcbs t\rcbs$-module by symbols $x^v$ with $0\neq v\in K$, to obtain a ring that we denote $\hat{\mathscr{B}}_Q$.  For $i\in Q_0$ we write $x_i\coloneqq x^{\g{i}}$.  The \textbf{quantum tropical vertex group} $G^{\qtrop}_Q$ is defined to be the closure of the subgroup of the group of units $\hat{\mathscr{B}}_Q^{\times}$ generated by elements $\EE(t^nx^v)$ for $v\in K\setminus \{0\}$ and $n\in\mathbb{Z}$.  As in \S \ref{qTorLie}, this group is obtained by exponentiation from a Lie subalgebra (defined over $\mathbb{Z}$) $\f{g}_Q^{\qtrop}$ of a Lie algebra $\hat{\mathfrak{g}}_{Q}$.  The algebra $\hat{\mathfrak{g}}_{Q}$ is in turn  obtained by completing a Lie algebra $\mathfrak{g}_{Q}$ with respect to the sequence of Lie ideals $\mathfrak{g}_{Q}\cap I_Q^n$.  The Lie algebra $\mathfrak{g}_{Q}$ here is the Lie subalgebra of the commutator Lie algebra of $\mathscr{B}_Q$ generated over $\kt$ by the elements 
\[
\hat{x}^v\coloneqq x^v/(\lvert v\lvert)_t
\]
for $0\neq v\in K$, where $\lvert v\lvert$ is the index of $v$, i.e. the largest positive integer such that $v/\lvert v\lvert\in K$.  Then $\f{g}_Q^{\qtrop}$ is the closure of the $\mathbb{Z}$-span of the quantum dilogarithms $-\Li(-t^az^v;t)$.

\subsubsection{Stability conditions}
Let $\zeta\in \mathbb{R}^{Q_0}$ be a  \textbf{stability condition}.  We define the \textbf{slope} of a non-trivial $\bb{C}Q$-module $\rho$ by
\[
\mu_{\zeta}(\rho)=\frac{\zeta\cdot\udim(\rho)}{\dim(\rho)}.
\]
We will omit the subscript $\zeta$ when the choice of stability condition is clear.  We denote by $S_{\theta}^{\zeta}\subset \bb{Z}_{\geq 0}^{Q_0}$ the submonoid of dimension vectors of slope $\theta$ with respect to $\zeta$.  We say that a stability condition $\zeta$ is $\theta$-\textbf{generic} if $B(v,v')=0$ for all $v,v'\in S_{\theta}^{\zeta}$.  We say a stability condition is simply \textbf{generic} if it is $\theta$-generic for every $\theta$.

A $\mathbb{C}Q$-module is called $\zeta$-\textbf{semistable} if for all non-trivial proper submodules $\rho'\subset \rho$ we have the inequality of slopes $\mu_{\zeta}(\rho')\leq \mu_{\zeta}(\rho)$, and $\rho$ is called $\zeta$-stable if this inequality is strict for all non-trivial proper submodules.  We denote by 
\[
\Mst^{\zeta\sst}_v(Q)\subset \Mst_v(Q)
\]
the open substack of $v$-dimensional $\zeta$-semistable $\mathbb{C}Q$-modules.  We denote by $\Msp^{\zeta\sst}_v(Q)$ the coarse moduli space of $\zeta$-semistable $v$-dimensional $\mathbb{C}Q$-modules defined by King \cite{King}, and by $\Msp^{\zeta\stable}_v(Q)\subset \Msp^{\zeta\sst}_v(Q)$ the smooth open subvariety corresponding to $\zeta$-stable modules.\footnote{Strictly speaking, King defines these spaces as GIT quotients under the assumption that $\zeta\in\bb{Q}^{Q_0}$.  So we define these moduli spaces by taking a perturbation $\zeta'_v\in\bb{Q}^{Q_0}$ of $\zeta$, which will depend on $v$ in general, such that $\zeta$-(semi)stability for $v$-dimensional $\bb{C}Q$-modules is equivalent to $\zeta'$-(semi)stability. 
}  As it is constructed as a GIT quotient, the map from $\Msp^{\zeta\sst}_v(Q)$ to its affinization is projective.  On the other hand, by \cite[Thm. 1]{LeBPro} the functions on the affinization are generated by taking traces of evaluations of modules on oriented cycles in $Q$.  It follows that $\Msp^{\zeta\sst}_v(Q)$ is projective if and only if the support of $v$ does not contain an oriented cycle in $Q$, and $\Msp^{\zeta\sst}_v(Q)$ is projective for all $v\in K$ if $Q$ is acyclic.

Let $\rho$ be a non-trivial $\mathbb{C} Q$-module.  Then $\rho$ admits a unique \textbf{Harder--Narasimhan filtration}
\[
0=\rho_{(0)}\subset \rho_{(1)}\subset\ldots\rho_{(r)}=\rho
\]
for some $r\in\mathbb{Z}_{\geq 1}$, such that
\begin{enumerate}
    \item Each subquotient $\rho_{(i)}/\rho_{(i-1)}$ for $1\leq i\leq r$ is a non-trivial $\zeta$-semistable $\mathbb{C}Q$-module.
    \item There is a strict inequality of slopes $\mu_{\zeta}(\rho_{(i)}/\rho_{(i-1)})>\mu_{\zeta}(\rho_{(i+1)}/\rho_{(i)})$ for $1\leq i\leq r-1$.
\end{enumerate}
Let $I\subset (-\infty,\infty)$ be an interval.  We denote by $\Mst^{\zeta}_I(Q)\subset \Mst(Q)$ the open substack, the points of which correspond to modules $\rho$ satisfying the condition that each subquotient $\rho_{(i)}/\rho_{(i-1)}$ appearing in the Harder--Narasimhan filtration of $\rho$ has slope contained in $I$.  We define $\Mst_{I,v}^{\zeta}(Q)=\Mst_I^{\zeta}(Q)\cap \Mst_v(Q)$.
\subsection{Cohomological DT theory}
\label{coh_DT_sec}
We denote by $\Vect^+$ the Abelian category of cohomologically graded vector spaces $V$ such that $\HO^i(V)=0$ for $i\ll 0$, and for which $\dim(\HO^i(V))< \infty$ for all $i\in \mathbb{Z}$.  We denote by $\Vect^+_{K}$ the category of formal direct sums
\[
\bigoplus_{v\in K}V_{v}
\]
with $V_{v}\in \Vect^+$. 
 We denote by $\Vect^{\zeta,+}_{\theta}$ the full subcategory of $\Vect_{K}^+$ containing those $V$ such that $V_v=0$ for $v\notin S_{\theta}^{\zeta}$.  We give $\Vect_K^+$ the twisted monoidal structure, defined by
\begin{equation}
\label{mon_prod}    
(V'\otimes^{\tw}V'')_{v}=\bigoplus_{v'+v''=v}V'_{v'}\otimes V''_{v''}[B( v'',v')].
\end{equation}
The associator for this monoidal product is as defined in \cite[Sec 3.2]{DavMei}.  For $\zeta$ a $\theta$-generic stability condition, the shift by $B(v'',v')$ is trivial on $\Vect^{\zeta,+}_{\theta}$ and (\ref{mon_prod}) is part of the usual \textit{symmetric} monoidal structure on graded vector spaces, incorporating the Koszul sign rule with respect to the cohomological degree.  The monoidal product (\ref{mon_prod}) induces a (noncommutative) ring structure on $\KK(\Vect^+_{K})$.  The characteristic function
\begin{align}
\label{chit_def}
\chi_{K,t}\colon \KK(\Vect^+_{K})\rightarrow &\hat{\mathscr{B}}_Q\\
[V]\mapsto &\sum_{i\in\mathbb{Z}}\sum_{v\in K}\dim(\HO^i(V_{v}))t^ix^v \nonumber
\end{align}
is an isomorphism satisfying
\[
\chi_{K,t}(V'\otimes^{\tw} V'')=\chi_{K,t}(V')\chi_{K,t}(V'').
\]
Furthermore, (\ref{chit_def}) is an isomorphism of $\mathbb{Z}\lcbs t\rcbs$-algebras, where $t$ acts on $\KK(\Vect^{+}_K)$ by increasing cohomological degree by one.

For $v\in K$ we define $\s{A}_v^{\zeta}(Q)\coloneqq \HO\left(\Mst^{\zeta\sst}_v(Q),\mathbb{Q}\right)_{\vir}$.  Let $\zeta\in \mathbb{R}^{Q_0}$, and let $\theta\in (-\infty,\infty)$ be a slope.  We define
\[
\mathcal{A}_{\theta}^{\zeta}(Q)\coloneqq \bigoplus_{v\in S_{\theta}^{\zeta}}\HO\left(\Mst^{\zeta\sst}_v(Q),\mathbb{Q}\right)_{\vir}\in \ob(\Vect^{\zeta,+}_{\theta}).
\]
Here and elsewhere, if $X$ is a disjoint union of connected irreducible stacks $X=\coprod_{i\in I} X_i$ we set 
\[
\HO(X,\mathbb{Q})_{\vir}\coloneqq \bigoplus_{i\in I}\HO\left(X_i,\mathbb{Q}\right)[\dim(X_i)].
\]
Via the inclusion of categories $\Vect_{\theta}^{\zeta,+}\subset \Vect_K^+$ we can consider $\mathcal{A}_{\theta}^{\zeta}(Q)$ as an object in $\Vect_K^+$.  
\subsubsection{Cohomological DT invariants}
As in \cite{MeiRei} we define 
\begin{equation}
    \label{BPS_def}
\BPS^{\zeta}_v(Q)\coloneqq \begin{cases} \IC\left(\Msp_v^{\zeta\sst}(Q),\mathbb{Q}\right)[1-\chi_Q(v,v)] &\textrm{if }\Msp^{\zeta\stable}_v(Q)\neq \emptyset\\
0& \textrm{otherwise},\end{cases}
\end{equation}
where in the first line we have taken the total hypercohomology of the intersection complex on the (possibly singular) irreducible variety $\Msp^{\zeta\sst}_v(Q)$ --- i.e. the usual intersection cohomology of $\Msp^{\zeta\sst}_v(Q)$.  
The vector space $\BPS_v^{\zeta}(Q)$ is a finite-dimensional cohomologically graded vector space, and so is in particular an element of $\Vect^+$.  

Let $\mathbb{C}^*\subset \GL_v$ denote the subgroup containing $Q_0$-tuples of the form $(\lambda\cdot \Id_{v_{i}\times v_i})_{i\in Q_0}$ for $\lambda\in\mathbb{C}^*$.  Assuming that $\Msp^{\zeta\stable}_v(Q)\neq 0$, it is the quotient of a free action of $\GL_{v}/\mathbb{C}^*$ on an open subspace of $\prod_{a\in Q_1}\Hom(\mathbb{C}^{v_{t(a)}},\mathbb{C}^{v_{s(a)}})$.  Moreover, $\Msp_v^{\zeta\stable}(Q)$ is a dense open subvariety of $\Msp_v^{\zeta\sst}(Q)$.  It follows that
\begin{equation}
    \label{dim_count}
\dim(\Msp_v^{\zeta\stable}(Q))=\dim(\Msp_v^{\zeta\sst}(Q))=1-\chi_Q(v,v).
\end{equation}

\begin{rem}
Assume that the support of $v$ contains no oriented cycles in $Q$, so that $\Msp^{\zeta\sst}_v(Q)$ is projective.  Since the cohomological shift in the definition of $\BPS^{\zeta}_v$ is equal to 
the complex dimension of $\Msp^{\zeta\sst}_v(Q)$, it follows from Poincar\'e duality for intersection cohomology that $\chi_t(\BPS^{\zeta}_v(Q))$ is bar-invariant.  By the hard Lefschetz theorem for intersection cohomology \cite[Chap. 6]{BBD}, we deduce moreover that $\chi_t(\BPS^{\zeta}_v(Q)))$ is of Lefschetz type.
\end{rem}

\begin{lem}
\label{dimension_lemma}
For a quiver $Q$, dimension vector $v$ and stability condition $\zeta\in\bb{R}^{Q_0}$, $\BPS^{\zeta}_v(Q)$ is nonzero if and only if there exists a $\zeta$-stable $v$-dimensional $\mathbb{C}Q$-representation, and so in particular $1-\chi_Q(v,v)\geq 0$.  If there exists such a stable module, there are equalities
\[
\dim(\HO^{j}(\BPS^{\zeta}_v(Q)))=\begin{cases} 1 &\textrm{if }j=\chi_Q(v,v)-1\\
0&\textrm{if }\lvert j\lvert>1- \chi_Q(v,v).
\end{cases}
\]
\end{lem}
\begin{proof}
First we claim that for any quasiprojective variety $X$ the degree zero intersection cohomology $\IC^0(X,\bb{Q})$ has a natural basis indexed by the irreducible components of $X$, which we label $X_p$ for $p$ in an indexing set $P_X$.  This follows from definitions and foundational results in intersection homology, see \cite[Chapter 4]{KiWo}.  

In a little more detail:  Firstly, the morphism from the normalization $\wt{X}\rightarrow X$ is small, so $\IC(X,\bb{Q})\cong \IC(\wt{X},\bb{Q})\cong \oplus_{p\in P_{\tilde{X}}} \IC(\wt{X_p},\bb{Q})$, and so it is sufficient to prove the claim under the assumption that $X$ is irreducible.  By results of Whitney \cite{Whitney}, \L ojasiewicz \cite{Loj} and Goresky \cite{Gor}, there is a Whitney stratification of $X$ by complex subvarieties, along with a triangulation of $X$ respecting the stratification.  The sum of the top degree simplices of this triangulation gives a spanning top degree class in intersection homology, i.e. a degree zero class in intersection cohomology, by Poincar\'e duality.  This proves the claim.  

Then the first equality follows from the fact that $\Msp_v^{\zeta\stable}(Q)$ is a dense open subvariety of $\Msp_v^{\zeta\sstable}(Q)$, and is irreducible, as it is a free quotient of an irreducible variety.  The second equality follows from the fact that the degree of $\IC\left(\Msp_v^{\zeta\sst}(Q),\mathbb{Q}\right)$ is bounded above by $-2(\chi_Q(v,v)-1)$, since by \eqref{dim_count} this is the real dimension of $\Msp_v^{\zeta\sstable}(Q)$.
\end{proof}

\subsubsection{Cohomological wall crossing}
The following ``categorification'' of the wall crossing formula from Donaldson--Thomas theory is a very special case of \cite[Thm B]{DavMei} (see also \cite{FraRei} and the proof of \cite[Thm 3.21]{DavPos}).
\begin{thm}\label{coh_WCF_thm}
Let $\zeta\in\mathbb{R}^{Q_0}$ be a stability condition (not necessarily generic).  Then there is an isomorphism in $\Vect^+_{K}$:
\begin{equation}
\label{CWCE}
\HO(\Mst(Q),\mathbb{Q})_{\vir}\cong \bigotimes^{\tw}_{\infty \xrightarrow{\theta}-\infty}\mathcal{A}_{\theta}^{\zeta}(Q).
\end{equation}
\end{thm}
Since the monoidal product is not symmetric, it is important to note the order in which we take it (i.e., moving from left to right as the slope decreases).  Also, in order to define the above infinite product, note that for each $v\in K$ there are only finitely many ways of decomposing $v=v_1+\ldots +v_l$ into nonzero dimension vectors in $K$.  In addition, the $v=0$ graded piece of every $\mathcal{A}_{\theta}^{\zeta}(Q)$ is equal to $\mathbb{Q}$, the one dimensional vector space concentrated in cohomological degree zero.  This is the monoidal unit in $\Vect^+$.  So
we may write the $v$-graded piece of the right-hand side of (\ref{CWCE}) as
\[
\bigoplus_{\substack{v_1,\ldots v_l\neq 0\\v_1+\ldots+v_l=v\\ \mu_{\zeta}(v_1)\geq \ldots\geq \mu_{\zeta}(v_l)}}\mathcal{A}^{\zeta}_{v_1}(Q)\otimes^{\tw}\cdots\otimes^{\tw}\mathcal{A}^{\zeta}_{v_l}(Q),
\]
which is, in particular, a cohomologically graded vector space which is finite-dimensional in every cohomological degree, as required.

Applying $\chi_{K,t}$ to both sides of (\ref{CWCE}), we deduce that there is an equality
\begin{equation}
\label{WCF}
\chi_{K,t}\left(\HO(\Mst(Q),\mathbb{Q})_{\vir}\right)=\prod_{\infty\xrightarrow{\theta} -\infty}\chi_{K,t}(\mathcal{A}_{\theta}^{\zeta}(Q))\in\hat{\mathscr{B}}_Q.
\end{equation}

\begin{cor}\label{par_cor}
The cohomologically graded vector space $\mathcal{A}_v^{\zeta}$ is entirely concentrated in even or odd degrees, depending on whether $\chi_Q(v,v)$ is even or odd, respectively. 
\end{cor}
\begin{proof}
Taking the monoidal unit in every term in the infinite tensor product in \eqref{CWCE} except the term corresponding to $\theta$, there is an inclusion
\[
\mathcal{A}_{\theta}^{\zeta}(Q)\subset \HO(\Mst(Q),\mathbb{Q})_{\vir},
\]
and so, since \eqref{CWCE} respects the $K$-grading, an inclusion  $\mathcal{A}_{v}^{\zeta}(Q)\subset \HO(\Mst_v(Q),\mathbb{Q})_{\vir}$ for $v\in S^{\zeta}_{\theta}$.  Now the result follows from the fact that 
\[
\HO(\Mst_v(Q),\mathbb{Q})\cong \HO(\pt/\GL_v,\mathbb{Q})
\]
is entirely concentrated in even degrees, and the dimension of $\Mst_v(Q)$ is $-\chi_Q(v,v)$.
\end{proof}
\begin{rmk}
A more elementary proof of the vanishing of even cohomology of $\HO(\Mst^{\zeta\sst}_v(Q),\QQ)$ avoiding e.g. the use of the decomposition theorem in \cite{DavMei} has more recently been provided in \cite[Thm 5.1]{FraRei}. 
\end{rmk}

\subsubsection{Integrality theorem}
The second main ingredient we use from cohomological DT theory is a  version of the integrality conjecture.  The conditions on $\zeta$ for this theorem are slightly stronger than in Theorem \ref{coh_WCF_thm}. Recall that if $\zeta$ is $\theta$-generic, the twist in the twisted monoidal product becomes trivial on $\Vect^{\zeta,+}_{\theta}$ and the monoidal product on $\Vect^{\zeta,+}_{\theta}$ can be made symmetric, with symmetrizing morphism incorporating the Koszul sign rule with respect to the cohomological degree.  In particular, under this genericity assumption, given $V\in \Vect^{\zeta,+}_{\theta}$ we may form
\begin{equation}
\label{SymDef}
\Sym(V)=\bigoplus_{n\geq 0}\Sym^n(V),
\end{equation}
the underlying graded object of the free \textit{super}commutative algebra generated by $V$.  Recall that this is isomorphic to 
\[
\CSym\left(\HO^{\textrm{even}}(V)\right)\otimes \bigwedge \HO^{\textrm{odd}}(V),
\]
the tensor product of the free commutative algebra generated by $\HO^{\textrm{even}}(V)$ and the free exterior algebra generated by $\HO^{\textrm{odd}}(V)$.  We ignore the extra grading on this vector space determined by the decomposition on the right hand side of \eqref{SymDef}. 
The object $\Sym(V)$ will again be an object of $\Vect^{\zeta,+}_{\theta}$ if $V_0=0$.
\begin{thm}\label{coh_int_thm}\cite[Thm. A]{DavMei}
Let $\zeta\in \mathbb{R}^{Q_0}$ be a $\theta$-generic stability condition.  There is an isomorphism in $\Vect^{\zeta,+}_{\theta}$
\begin{equation}
\label{CIE}
\mathcal{A}_{\theta}^{\zeta}(Q)\cong \Sym\left(\bigoplus_{0\neq v\in S_{\theta}^{\zeta}}\BPS_v^{\zeta}(Q)\otimes \HO(\pt/\mathbb{C}^*,\mathbb{Q})_{\vir}\right).
\end{equation}
\end{thm}
Theorem \ref{coh_int_thm}, along with \eqref{cat_Expb}, implies that $\chi_{K,t}(\mathcal{A}_{\theta}^{\zeta}(Q))\in G^{\qtrop}_Q$.  Combining with \eqref{WCF} we deduce moreover that $\chi_{K,t}\left(\HO(\Mst(Q),\mathbb{Q})_{\vir}\right)\in G^{\qtrop}_Q$.
\begin{rem}
If $Q$ is symmetric, in the sense that between any two distinct vertices $i$ and $j$ there are as many arrows from $i$ to $j$ as from $j$ to $i$, then the degenerate stability condition $\zeta=(0,\ldots,0)$ is generic, and there is an equality
\[
\mathcal{A}_{\theta=0}^{\zeta}(Q)=\HO(\Mst(Q),\QQ)_{\vir}.
\]
Under these conditions, Efimov \cite{Efi} proved that $\mathcal{A}_{\theta=0}^{\zeta}(Q)\cong \Sym\left(V_{\mathrm{prim}}\otimes \HO(\pt/\mathbb{C}^*,\mathbb{Q})_{\vir}\right)$ for some $V_{\mathrm{prim}}\in \Vect^+_K$ such that each $V_{\mathrm{prim},v}$ has finite-dimensional total cohomology.  Theorem \ref{coh_int_thm} gives a precise definition of $V_{\mathrm{prim}}$ in this case.
\end{rem}

\subsubsection{Examples}
We let $\LQ{r}$ denote the quiver with $1$ vertex and $r$ loops.  The quiver $\LQ{1}$ is known as the Jordan quiver.  Whichever stability condition $\zeta\in\mathbb{R}$ we choose, all $\CLQ{r}$-modules are semistable, and they are stable if and only if they are simple.  Thus $\zeta$ is not relevant, and we drop it from the notation when considering the quivers $\LQ{r}$.
\begin{eg}\label{coh_ferm}
The variety $\Msp_1(\LQ{0})=\pt$ is smooth, and so 
\begin{align*}
\IC(\Msp_1(\LQ{0}),\mathbb{Q})=&\HO\left(\Msp_1(\LQ{0}),\mathbb{Q}\right)\\
=&\mathbb{Q}.
\end{align*}
On the other hand, there are no simple $n$-dimensional $\CLQ{0}$-modules for $n\geq 2$.  So
\[
\BPS_n(\LQ{0})=\begin{cases} \mathbb{Q} &\textrm{if }n=1\\
0&\textrm{otherwise.}\end{cases}
\]
\end{eg}
\begin{eg}\label{coh_bos}
Now we consider the one-loop, or Jordan quiver.  Again, $\Msp_1(\LQ{1})\cong\mathbb{A}^1$ is smooth, and now $\chi_{\LQ{1}}(1,1)=0$.  Also, there are still no simple modules of dimension greater than $1$.  So we deduce that
\[
\BPS_n(\LQ{1})=\begin{cases} \mathbb{Q}[1] &\textrm{if }n=1\\
0&\textrm{otherwise.}\end{cases}
\]
\end{eg}

\begin{eg}\label{coh_K2}
Let $\Kr{2}$ be the Kronecker quiver, with vertex set $\{1,2\}$ and with $2$ arrows from $2$ to $1$.  We choose the (generic) stability condition $\zeta=(1,-1)$, and consider $\zeta$-semistable $\CKr{2}$-modules of slope $0$.  
Under Beilinson's derived equivalence \cite{Beil} between $\CKr{2}$-modules and coherent sheaves on $\mathbb{P}^1$, $\zeta$-semistable $(n,n)$-dimensional $\CKr{2}$-modules correspond to coherent sheaves on $\mathbb{P}^1$ with zero-dimensional support, of length $n$.  Since such sheaves are simple precisely if $n=1$, we deduce 
\[
\BPS_{(n,n)}^{\zeta}(\Kr{2})=\begin{cases}\HO(\mathbb{P}^1,\mathbb{Q})_{\vir}=\mathbb{Q}[-1]\oplus \mathbb{Q}[1]& \textrm{if }n=1\\
0&\textrm{otherwise.}
\end{cases}
\]
\end{eg}
More generally, let $\Kr{r}$ denote the $r$-Kronecker quiver, with vertices $\{1,2\}$ and $r$ arrows, all oriented from $2$ to $1$.  
\begin{eg}
\label{coh_Kplus}
Let $\zeta=(1,-1)$ as above.  There are precisely three $\zeta$-stable $\CKr{1}$-modules, of dimension vectors $(1,0)$, $(0,1)$ and $(1,1)$.  As such, we deduce that
\[
\BPS_{(n,n)}^{\zeta}(\Kr{1})=\begin{cases}\HO(\pt,\mathbb{Q})_{\vir}=\mathbb{Q}& \textrm{if }n=1\\
0&\textrm{otherwise.}
\end{cases}
\]
It is easy to check that for $r\geq 3$ there are $\zeta$-stable $\CKr{r}$-modules of dimension $(n,n)$ for all $n\in\mathbb{Z}_{>0}$.  There is an equality $\BPS_{(1,1)}^{\zeta}(\Kr{r})=\HO(\mathbb{P}^{r-1},\QQ)_{\vir}$, and in contrast with the cases $r\leq 2$, the higher order terms $\BPS_{(n,n)}^{\zeta}(\Kr{r})$ for $n\geq 2$ are nontrivial by Lemma \ref{dimension_lemma}.
\end{eg}
\subsubsection{Positivity}
Combining Theorem \ref{coh_int_thm} with (\ref{WCF}), we deduce the following.
\begin{prop}\label{WCF2}
Let $\zeta\in\mathbb{R}^{Q_0}$ be a generic stability condition.  Then there is an equality
\[
\chi_{K,t}(\HO(\Mst(Q),\mathbb{Q})_{\vir})=\prod_{\infty\xrightarrow{\theta} -\infty}\EE\left(\sum_{0\neq v\in S_{\theta}^{\zeta}}\chi_t(\BPS^{\zeta}_v(Q))x^v\right)
\]
where $\chi_t$ is as in \eqref{chit}.
\end{prop}
\begin{proof}
From Theorem \ref{coh_int_thm}, we deduce that
\begin{align} \label{deCatA}
    \chi_{K,t}(\mathcal{A}_{\theta}^{\zeta}(Q))&=  \chi_{K,t}\left(\Sym\left(\bigoplus_{0\neq v\in S_{\theta}^{\zeta}}\BPS_v^{\zeta}(Q)\otimes \HO(\pt/\mathbb{C}^*,\mathbb{Q})[-1]\right)\right) \\ \nonumber
    &=\EE\left(\chi_{K,t}\left(\bigoplus_{0\neq v\in S_{\theta}^{\zeta}}\BPS_v^{\zeta}(Q)\right)\right) \\ \nonumber
    &=\EE\left(\sum_{0\neq v\in S_{\theta}^{\zeta}} \chi_t\left(\BPS_v^{\zeta}(Q)\right)x^v\right),
\end{align}
where the second equality follows from \eqref{cat_Expb}.  Then the result follows from (\ref{WCF}).
\end{proof}
We will use the following lemma:
\begin{lem}\label{BPS_cor}
Let $\zeta\in\mathbb{R}^{Q_0}$ be a $\theta$-generic stability condition, and assume that $\mu_{\zeta}(v)=\theta$.  Then the cohomologically graded vector space $\BPS_v^{\zeta}(Q)$ is entirely supported in even or odd degree, depending on whether $\chi_Q(v,v)$ is odd or even, respectively.
\end{lem}
\begin{proof}
From Theorem \ref{coh_int_thm} we deduce that there is an inclusion
\begin{equation}\label{BPSin}
\BPS_v^{\zeta}(Q)\otimes\HO(\pt/\mathbb{C}^*)_{\vir}\subset \mathcal{A}_{v}^{\zeta}.
\end{equation}
On the other hand, $\HO(\pt/\mathbb{C}^*)_{\vir}$ is supported in odd cohomological degree, since $\dim(\pt/\mathbb{C}^*)=-1$, and by Corollary \ref{par_cor}, $\mathcal{A}_v^{\zeta}$ is supported in even or odd cohomological degree, depending on whether $\chi_Q(v,v)$ is even or odd, respectively.
\end{proof}

We now come to our main positivity result regarding Donaldson--Thomas invariants.  It will be a key ingredient in proving preservation of positivity in Corollary \ref{2WallPos}.  It generalizes the positivity statement of \cite[Cor. 1.2]{MeiRei}, which applies to the case in which $Q$ is acyclic.  The statement that the refined DT invariant is of Lefschetz type, or even bar-invariant, on the other hand, fails outside of the non-acyclic case.  This is a result of the fact that the moduli spaces we consider outside of the acyclic case are not proper, and so their intersection cohomology does not satisfy Poincar\'e duality or the hard Lefschetz theorem.

\begin{prop}\label{Pos1}
Let $\zeta\in\mathbb{R}^{Q_0}$ be a $\theta$-generic stability condition for a quiver $Q$.  Then
\[
\chi_{K,t}(\mathcal{A}^{\zeta}_{\theta}(Q))=\prod_{0\neq v\in S_{\theta}^{\zeta}}\EE(f_v(t)x^v)
\]
where each $f_v(t)\in\mathbb{Z}_{\geq 0}[t^{\pm 1}]$, and $\parit(f_v(t))\equiv \chi_Q(v,v)+1$ (mod $2$).  Moreover, $f_v(t)\neq 0$ if and only if there exists a $\zeta$-stable $\bb{C}Q$-module of dimension $v$, and if $f_v(t)\neq 0$ it can be written
\[
f_v(t)=t^{
\chi_Q(v,v)-1}(1+g_v(t))
\]
where $g_v(t)\in t\mathbb{Z}_{\geq 0}[t]$ has degree less than or equal to $2(1-\chi_Q(v,v))$.
\end{prop}
\begin{proof}
Combine \eqref{deCatA} and Lemmas \ref{dimension_lemma} and \ref{BPS_cor}.
\end{proof}
Combining Proposition \ref{Pos1} with (\ref{WCF}), we deduce positivity for refined Donaldson--Thomas invariants.
\begin{thm}[Positivity of DT invariants]
\label{DT_pos_thm}
Let $\zeta\in \mathbb{R}^{Q_0}$ be a generic stability condition.  There is an equality of generating series 
\begin{equation*}
\chi_{K,t}(\HO(\Mst(Q),\mathbb{Q})_{\vir})=\prod_{\infty\xrightarrow{\theta} -\infty}\prod_{0\neq v\in S_{\theta}^{\zeta}}\EE(f_v(t)x^v)
\end{equation*}
where the $f_v(t)\in\mathbb{Z}_{\geq 0}[t^{\pm 1}]$ have positive coefficients, $\parit(f_v(t))\equiv \chi_Q(v,v)+1$ (mod $2$) and $f_v(t)\neq 0$ if and only if there exists a $\zeta$-stable $\bb{C}Q$-module of dimension $v$, in which case we can write
\[
f_v(t)=t^{
\chi_Q(v,v)-1}(1+g_v(t))
\]
where $g_v(t)\in t\mathbb{Z}_{\geq 0}[t]$ has degree less than or equal to $2(1-\chi_Q(v,v))$.
\end{thm}

\subsection{Refined Donaldson--Thomas theory}
\label{ref_DT_sec}
In this section we recall some results from refined Donaldson--Thomas theory, culminating in the Meinhardt--Reineke result \cite{MeiRei} on the positivity and Lefschetz type property of refined DT invariants of acyclic quivers.  As per the warnings at the beginning of \S\ref{pos_DT_sec}, we will need to recall some extra definitions first.  
\subsubsection{Graded mixed Hodge structures}
Recall that a \textbf{(rational) mixed Hodge structure} is given by the data of
\begin{itemize}
    \item a finite-dimensional vector space $V$ over $\mathbb{Q}$
    \item an ascending filtration $W_{\bullet}$ of $V$
    \item a descending filtration $F^{\bullet}$ of $V\otimes_{\mathbb{Q}} \mathbb{C}$
\end{itemize}
such that the filtration induced by $F^{\bullet}$ on $W_{n}\otimes_{\mathbb{Q}}\mathbb{C}/W_{n-1}\otimes_{\mathbb{Q}}\mathbb{C}$ determines a pure Hodge structure of weight $n$.  A \textbf{cohomologically graded mixed Hodge structure} is a possibly infinite formal direct sum $V=\bigoplus_{i\in\mathbb{Z}}V^i[-i]$ with each $V^i$ a rational mixed Hodge structure, i.e. a cohomologically graded object with $V^i$ in cohomological degree $i$.  We then write $\HO^i(V)=V^i$.

We denote by $\MHS$ the category of rational mixed Hodge structures.  We denote by $\MHS^+$ the category of partially bounded cohomologically graded mixed Hodge structures, i.e. those satisfying  
\begin{itemize}
    \item if $n\ll 0$ then $\Gr^W_n(\HO^i(V))=0$ for all $i$
    \item for all $n\in \mathbb{Z}$ the underlying vector space of $\bigoplus_{i\in\mathbb{Z}}\Gr^W_n(\HO^i(V))$ is finite-dimensional, and so it is naturally a pure weight $n$ Hodge structure.
\end{itemize}
The category $\MHS^+$ is a symmetric monoidal category, where the symmetrizing morphism incorporates the Koszul sign rule.  We call an object of $\MHS^+$ \textbf{pure} if $\Gr^W_i(\HO^j(V))=0$ for all $i\neq j$.  We denote by $\LL$ the pure object of $\MHS^+$ given by $\HO_c(\mathbb{A}^1,\mathbb{Q})$.  It is a one dimensional pure weight $2$ Hodge structure concentrated in cohomological degree $2$.  Given an element $V\in \MHS^+$ we define
\[
\chi_{\Wt}(V)=\sum_{i,n\in \mathbb{Z}}(-1)^i\dim_{\bb{Q}}(\Gr^W_n(\HO^i(V)))(-t)^n\in\mathbb{Z}\lcbs t \rcbs.
\]

The category $\MHS$ is a full subcategory of the category $\MMHS$ of monodromic mixed Hodge structures defined in \cite[Sec.7]{CoHA}, which we refer to for the facts about the extension of $\chi_{\Wt}$ recalled below.  We define the category of partially bounded monodromic mixed Hodge structures $\MMHS^+$ via the same conditions as for $\MHS^+$.  We make almost no use of this larger category, except for the following fact: there is a tensor square root $\mathbb{L}^{1/2}$ for $\mathbb{L}$ inside $\MMHS^+$, and we denote by $\overline{\MHS}^+$ the smallest full tensor subcategory of $\MMHS^+$ containing $\MHS^+$ and this tensor square root.  Since $\LL$ is concentrated in cohomological degree $2$, the object $\LL^{1/2}$ is concentrated in degree $1$.  Alternatively, as in \cite[Sec 3.4]{CoHA}, one can adjoin a formal tensor square root of $\mathbb{L}$ to $\MHS^+$ in order to define $\overline{\MHS}^+$.  For the purposes of this paper, the only important thing to note is that the relation
\[
\chi_{\Wt}(V\otimes V')=\chi_{\Wt}(V)\chi_{\Wt}(V')
\]
extends to the Grothendieck ring of objects in $\overline{\MHS}^+$, and we have
\[
\chi_{\Wt}(\mathbb{L}^{1/2})=t.
\]
Furthermore, if $V\in \overline{\MHS}^+$ is pure then $\chi_{\Wt}(V)\in\mathbb{Z}_{\geq 0}\lcbs t\rcbs$, and if $V$ is pure and carries a Lefschetz operator then $\chi_{\Wt}(V)$ is a Lefschetz type Laurent polynomial.

  We will slightly abuse notation and refer to objects in $\mMHS^+$ as cohomologically graded mixed Hodge structures, despite the fact that $\MHS^+$ is only a full subcategry of $\mMHS^+$.  We say an object of $\mMHS^+$ is of \textbf{Tate type} if each cohomologically graded piece is a direct sum of tensor powers of $\LL^{1/2}[1]$.  We denote by $\mMHS_{ K}^+$ the category of formal $K$-graded direct sums in $\mMHS^+$.  We make this into a monoidal category by setting
\[
(V'\otimes^{\tw}V'')_{v}=\bigoplus_{v'+v''=v}V'_{v'}\otimes V''_{v''}\otimes\LL^{B(v',v'')/2}.
\]
For $V\in \ob(\mMHS_{K}^+)$ we define
\[
\chi_{\Wt,K}(V)=\sum_{v\in K}\chi_{\Wt}(V_v)x^v\in \hat{\mathscr{B}}_Q.
\]
Then $\chi_{\Wt,K}(V'\otimes^{\tw}V'')=\chi_{\Wt,K}(V')\chi_{\Wt,K}(V'')$.

For $X$ an irreducible finite type global quotient Artin stack, we define the cohomologically graded mixed Hodge structures
\begin{align*}
\HO(X,\mathbb{Q})_{\vir}=&\HO(X,\mathbb{Q})\otimes \LL^{-\dim(X)/2}\\
\HO_c(X,\mathbb{Q})_{\vir}=&\HO_c(X,\mathbb{Q})\otimes \LL^{-\dim(X)/2}.
\end{align*}
The latter is a cohomologically graded mixed Hodge structure which may have nontrivial pieces of weight $n$ for $n\ll 0$, and so it is not an object of $\mMHS^+$.  As such we will only consider its (cohomologically graded) dual $\HO_c(X,\mathbb{Q})_{\vir}^*$ in the category of cohomologically graded mixed Hodge structures, which \textit{is} an object of $\mMHS^+$.

\subsubsection{Factorization and integrality}

We recall the following result of Reineke \cite[Sec. 6]{ReinHN}.
\begin{thm}\label{refWCF}
There is an equality of generating series
\[
\chi_{\Wt,K}\left(\bigoplus_{v\in K}\HO_c(\Mst_v(Q),\mathbb{Q})^*_{\vir}\right)=\prod_{\infty\xrightarrow{\theta} -\infty}\chi_{\Wt,K}\left(\bigoplus_{v\in S^{\zeta}_{\theta}}\HO_c(\Mst^{\zeta\sstable}_{v}(Q),\bb{Q})_{\vir}^*\right).
\]
\end{thm}
Reineke's original result was written in terms of compactly supported cohomology (not its dual), and so one should make the substitution $t\mapsto t^{-1}$ to translate between his result and its statement above.  Note that since all of the stacks $\Mst$ in the above expression are smooth, by Poincar\'e duality we have 
\[
\HO_c(\Mst,\mathbb{Q})^*_{\vir}\cong \HO(\Mst,\mathbb{Q})_{\vir},
\]
and so we can rewrite Theorem \ref{refWCF} as
\[
\chi_{\Wt,K}\left(\bigoplus_{v\in K}\HO(\Mst_v(Q),\mathbb{Q})_{\vir}\right)=\prod_{\infty\xrightarrow{\theta} -\infty}\chi_{\Wt,K}\left(\mathcal{A}^{\zeta}_{\theta,\Hodge}(Q)\right),
\]
where
\[
\mathcal{A}^{\zeta}_{\theta,\Hodge}(Q)\coloneqq \bigoplus_{v\in S^{\zeta}_{\theta}}\HO(\Mst^{\zeta\sstable}_v(Q),\QQ)_{\vir}
\]
is the Hodge-theoretic upgrade of $\mathcal{A}_{\theta}^{\zeta}(Q)$ from \S \ref{coh_DT_sec}.
Later we will use a slight variant of this result.  The proof is the same, following essentially from the existence of the Harder--Narasimhan stratification of $\Mst(Q)$ and the long exact sequence in compactly supported cohomology.  
\begin{prop}\label{I_WCF}
Let $I=[a,b]\subset (-\infty,\infty)$ be an interval, and denote by $\Mst^{\zeta}_{I,v}(Q)\subset \Mst_v(Q)$ the open substack of modules $\rho$ such that the slope of every subquotient in the Harder--Narasimhan filtration of $\rho$ has slope contained in $I$.  Then there is an equality of generating series in $\hat{\mathscr{B}}_Q$:
\begin{equation}
    \label{WCFwt}
\chi_{\Wt,K}\left(\bigoplus_{v\in K}\HO_c(\Mst^{\zeta}_{I,v}(Q),\mathbb{Q})_{\vir}^*\right)=\prod_{b\xrightarrow{\theta} a}\chi_{\Wt,K}\left(\bigoplus_{v\in S^{\zeta}_{\theta}}\HO_c(\Mst^{\zeta\sstable}_{v}(Q),\QQ)_{\vir}^*\right).
\end{equation}
\end{prop}
\subsubsection{Integrality}
The main result of \cite{MeiRei} establishes that if the stability condition is generic, each term in the right hand side of (\ref{WCFwt}) is the plethystic exponential of a more manageable power series.  Denote by $\overline{x}$ the set $\{x^{\g{i}}\}_{i\in Q_0}$ of degree one monomials in $\hat{\mathscr{B}}_Q$.  To state their result in something close to its original form, for $\zeta$ a $\theta$-generic stability condition, we use the isomorphism \eqref{Log_def} to define elements $\Omega_{Q,v}^{\zeta}\in \mathbb{Z}\lcbs t\rcbs$ via
\[
\sum_{0\neq v\in S^{\zeta}_{\theta}} \Omega_{Q,v}^{\zeta}x^v=(t^{-1}-t)\Log_{-t,\overline{x}}\left(\chi_{\Wt,K}\left(\bigoplus_{v\in S^{\zeta}_{\theta}}\HO_c(\Mst^{\zeta\sstable}_{v}(Q),\QQ)_{\vir}^*\right)\right)
\]
so that we have the equation
\begin{equation}\label{OmegaDef}
\chi_{\Wt,K}\left(\bigoplus_{v\in S^{\zeta}_{\theta}}\HO_c(\Mst^{\zeta\sstable}_{v}(Q),\QQ)_{\vir}^*\right)=\EE\left(\sum_{0\neq v\in S^{\zeta}_{\theta}} \Omega_{Q,v}^{\zeta}x^v\right).
\end{equation}
These $\Omega_{Q,v}^{\zeta}$ are, by definition, the refined DT invariants.  We define 
\[
\BPS^{\zeta}_{v,\Hodge}(Q)\coloneqq \begin{cases} \IC(\Msp_v^{\zeta\sst}(Q),\underline{\mathbb{Q}}_{\textrm{sm}})\otimes\LL^{(\chi_Q(v,v)-1)/2} &\textrm{if }\Msp^{\zeta\stable}_v\neq \emptyset\\
0& \textrm{otherwise},\end{cases}
\]
where on the right we have taken the total hypercohomology of the intersection complex for the constant (shifted) mixed Hodge module $\underline{\QQ}_{\textrm{sm}}$ on the smooth locus of $\Msp_v^{\zeta\sstable}$.  Note that the underlying cohomologically graded vector space of $\BPS^{\zeta}_{v,\Hodge}(Q)$ is $\BPS^{\zeta}_{v}(Q)$.

\begin{thm}\cite{MeiRei}\label{MRtheorem}
Assume that $\zeta$ is $\theta$-generic.  Then there is an equality
\[
\Omega^{\zeta}_{Q,v}=\chi_{\Wt}(\BPS_{v,\Hodge}^{\zeta}(Q)).
\]
In particular, $\Omega_{Q,v}^{\zeta}\in\mathbb{Z}[t^{\pm 1}]$.
\end{thm}
As in the statement of Theorem \ref{refWCF} we are working with the graded duals of the cohomology groups considered by Meinhardt and Reineke.  In addition, they consider the compactly supported cohomology with coefficients in the IC complex to define $\BPS_{v,\Hodge}^{\zeta}(Q)$, whereas we have taken cohomology.  These two differences cancel out, by Poincar\'e duality for intersection cohomology groups.
\begin{eg}
\label{ref_examples}
For all of the quivers $Q$ considered in Examples \ref{coh_ferm}, \ref{coh_bos} and \ref{coh_K2} from \S \ref{coh_DT_sec}, it is easy to verify directly that $\BPS_{v,\Hodge}^{\zeta}(Q)$ is pure, so that $\chi_{\Wt}(\BPS_{v,\Hodge}^{\zeta}(Q))=\chi_t(\BPS_{v}^{\zeta}(Q))$.  So in particular, from the calculations given in these examples, we deduce that 
\begin{align}
    \Omega_{\LQ{0},n}=&\delta_{n1}\label{ref_ferm} \\
    \Omega_{\LQ{1},n}=&\delta_{n1}t^{-1}\label{ref_bos} \\ 
    \Omega_{\Kr{2},(n,n)}^{(1,-1)}=&\delta_{n1}[2]_t,\label{ref_Kron}
\end{align}
where we have used the delta function 
\[
\delta_{ab}=\begin{cases} 1& \textrm{ if }a=b\\
0& \textrm{otherwise.}
\end{cases}
\]
\end{eg}
Combining Proposition \ref{I_WCF} and Theorem \ref{MRtheorem} we deduce that
\begin{equation}\label{GrandForm}
\chi_{\Wt,K}\left(\bigoplus_{v\in K}\HO_c\left(\Mst^{\zeta}_{I,v}(Q),\mathbb{Q}\right)_{\vir}^*\right)=\prod_{b\xrightarrow{\theta} a}\EE\left(\sum_{0\neq v\in S^{\zeta}_{\theta}} \chi_{\Wt}\left(\BPS_{v,\Hodge}^{\zeta}(Q)\right)x^v\right).
\end{equation}

The following corollary of Theorem \ref{MRtheorem} is \cite[Cor 1.2]{MeiRei}.  It follows from purity and the existence of a Lefschetz operator on $\IC(X,\underline{\mathbb{Q}}_{\textrm{sm}})$ for $X$ a projective variety.  The part of the statement concerning parity is a consequence of the fact that $\Omega_{Q,v}^{\zeta}$ arises from a formal series  in the class $[\LL^{1/2}]$ in the Grothendieck ring of mixed Hodge structures, but on the other hand $\IC(\Msp_v^{\zeta\sstable},\underline{\mathbb{Q}}_{\textrm{sm}})\in \MHS^+$, a category in which $\LL$ has no square root. 
\begin{cor}\cite{MeiRei}\label{MRcor}
Let $Q$ be acyclic, and let $\zeta$ be $\theta$-generic, where $\mu(v)=\theta$.  Then $\Omega_{Q,v}^{\zeta}$ is of Lefschetz type and is concentrated in even or odd degrees, depending on whether $\chi_Q(v,v)$ is odd or even, respectively.  In particular, $\Omega_{Q,v}^{\zeta}$ has positive integer coefficients.
\end{cor}

\subsubsection{Purity}
In this section we explain how it comes about that both versions of Donaldson--Thomas theory that we have presented can produce the same generating functions.  This coincidence is a byproduct of \textit{purity}.  Note that for pure objects of $\mMHS^+$ there is an equality
\[
\chi_{\Wt}(V)=\chi_t(V).
\]

Purity for cohomological DT invariants of quivers (possibly containing oriented cycles) without potential is in fact a direct consequence of the main theorems of \cite{DavMei}; we briefly explain how.

\begin{prop}\label{Apurity}
The object $\mathcal{A}_{\theta,\Hodge}^{\zeta}(Q)$ is pure of Tate type, and is moreover concentrated entirely in even or odd degrees, depending on whether $\chi_Q(v,v)$ is even or odd, respectively.
\end{prop}
\begin{proof}
Theorem B in \cite{DavMei} gives the following Hodge theoretic upgrade of (\ref{CWCE}):
\[
\bigoplus_{v\in K}\HO(\Mst_v(Q),\QQ)_{\vir}\cong \bigotimes^{\tw}_{\infty \xrightarrow{\theta}-\infty}\mathcal{A}^{\zeta}_{\theta,\Hodge}(Q).
\]
In particular there is an embedding
\[
\mathcal{A}_{\theta,\Hodge}^{\zeta}(Q)\subset \bigoplus_{v\in S_{\theta}^{\zeta}}\HO(\Mst_v(Q),\QQ)\otimes \LL^{\chi_Q(v,v)/2}.
\]
Now the result follows from the fact that 
\[
\HO(\Mst_v(Q),\mathbb{Q})\cong \HO(\pt/\GL_v,\QQ)
\]
is pure, of Tate type, and concentrated in even degrees.
\end{proof}
\begin{prop}
Let $\zeta$ be a $\theta$-generic stability condition.  Then the mixed Hodge structure $\BPS_{v,\Hodge}^{\zeta}(Q)$ is pure, of Tate type, and concentrated entirely in even or odd degrees depending on whether $\chi_Q(v,v)$ is odd or even, respectively. 
\end{prop}
\begin{proof}
Theorem A in \cite{DavMei} upgrades (\ref{CIE}) to the category of mixed Hodge structures.  In particular, there is an inclusion
\[
\BPS_{v,\Hodge}^{\zeta}(Q)\otimes\LL^{1/2}\subset \mathcal{A}_{v,\Hodge}^{\zeta}(Q).
\]
A sub mixed Hodge structure of a pure Tate type Hodge structure is pure of Tate type, and the result follows from Proposition \ref{Apurity}.
\end{proof}

In particular, it follows that for any $Q$, and for $\zeta$ a $\theta$-generic stability condition with $\mu(v)=\theta$, there are equalities.
\[
\chi_t(\BPS_{v}^{\zeta}(Q))=\chi_{\Wt}(\BPS_{v,\Hodge}^{\zeta}(Q))=\Omega_{Q,v}^{\zeta}.
\]

\section{Stability scattering diagrams and the basic two-wall cases}\label{2wallSection}

\subsection{Stability scattering diagrams}
\label{acyclic_sec}

In this section we recall the stability scattering diagrams of \cite{Bridge}.  Since that paper is written in terms of $Q$-representations, or equivalently \textit{left} $\mathbb{C}Q$-modules, when passing from our right $\mathbb{C}Q$-modules to Bridgeland's representations, one should replace $Q$ with $Q^{\opp}$.  

Given $\zeta\in \mathbb{R}^{Q_0}$, a $\mathbb{C}Q$-module $\rho$ is said to be $\zeta$ King-semistable if $\zeta\cdot \udim(\rho)=0$, and for every proper nonzero submodule $\rho'\subset \rho$ we have $\zeta\cdot\udim(\rho')\leq 0$.  We let 
\[
\Mst^{\zeta \kss}_v(Q)\subset \Mst_v(Q)
\]
denote the open substack of $\zeta$ King-semistable $\mathbb{C}Q$-modules.

Following the alternative scattering diagram convention mentioned in Footnote \ref{ConventionsFootnote}, it is shown in \cite{Bridge} that there is a consistent scattering diagram $\f{D}_{\Hall}$ in $\mathbb{R}^{Q_0}$
over a Lie algebra $\mathfrak{g}_{\Hall}$ such that the support of $\f{D}_{\Hall}$ is precisely the set of $\zeta$ for which there exists a $\zeta$ King-semistable $\mathbb{C}Q$-module, and functions at general points $\zeta$ of the walls are elements $1_{\mathrm{ss}}(\zeta)$ which lie in the pro-unipotent algebraic group $\hat{G}_{\Hall}$.   This group is defined to be a subgroup of the group of units of the (completed) motivic Hall algebra of the category of $\mathbb{C}Q$-modules.  We refer the reader to \cite[\S 5--6]{Bridge} for the original treatment of $\f{D}_{\Hall}$ and for the definition of $\hat{G}_{\Hall}$.

Denoting $B_1\colon \bb{R}^{Q_0}\rar (\bb{R}^{Q_0})^* = \bb{R}^{Q_0}$ (identifying $(\bb{R}^{Q_0})^*$ with $\bb{R}^{Q_0}$ using the dot product) for $B_1(u)\coloneqq B(u,\cdot)$ with $B$ as in \eqref{BQ}, we view $\f{D}_{\Hall}$ as a scattering diagram under our conventions by applying $B_1^{-1}$ to the support of each wall --- i.e., for us, the scattering function of $\f{D}_{\Hall}$ at a general point $u\in \bb{R}^{Q_0}$ is $1_{\mathrm{ss}}(B_1(u))$.  Here we must assume that $B\neq 0$ since otherwise $1_{\mathrm{ss}}(B_1(u))=\f{M}(Q)$ for all $u\in \bb{R}^{Q_0}$.  For $B\neq 0$, we have $\codim (\ker(B))\geq 2$, and then no general points will lie in $\ker(B)$,  so we get a well-defined scattering diagram.

We forego recalling the details of the definition of $\hat{G}_{\Hall}$ ourselves, instead remarking that by \cite[Thm 6.1]{joyce2007configurations} there is a group homomorphism 
\[
\int\colon \hat{G}_{\Hall}\rightarrow \hat{\mathscr{B}}_{Q}^{\times}
\]
where, by definition, there is an identity
\begin{align}\label{int1}
\int 1_{\mathrm{ss}}(\zeta)=\chi_{\Wt,K}\left(\bigoplus_{v\in K}\HO_c (\Mst^{\zeta \kss}_{v}(Q),\mathbb{Q})_{\vir}^*\right).
\end{align}

By the definition of King semistability, there is an equality
\[
\bigoplus_{v\in K}\HO_c(\Mst^{\zeta \kss}_{v}(Q),\mathbb{Q})_{\vir}=\bigoplus_{v\in S_0^{\zeta}}\HO_c(\Mst^{\zeta \sstable}_{v}(Q),\mathbb{Q})_{\vir},
\]
i.e. the notion of King semistability is the same as the one used throughout the paper, but with the added requirement that the slope is zero.

Let $\log \int\colon \hat{\f{g}}_{\Hall} \rar \hat{\f{g}}_Q$ denote the morphism of Lie algebras associated to $\int$.  We define the \textbf{stability scattering diagram} to be the scattering diagram $\f{D}_{\Stab}$ in $\bb{R}^{Q_0}$ over $\f{g}_Q$ given by 
\begin{align*}
    \f{D}_{\Stab}\coloneqq \left[\left(\log \int\right)_* \f{D}_{\Hall}\right]^-
\end{align*}
where $(\log \int)_*$ is defined as at the start of \S \ref{ScatOperations}, and $(\bullet)^-$ is defined as in Corollary \ref{reverse_cor}.

\begin{rem}
As noted above, in translating between this paper and \cite{Bridge} one should replace $Q$ with $Q^{\opp}$.  The result is that the skew-symmetric form $\langle \cdot,\cdot\rangle$ of \cite{Bridge} and the skew-symmetric form $B_Q(\cdot,\cdot)$ of this paper differ by a sign.
\end{rem}

For generic $\zeta$ satisfying $\zeta\cdot \g{i}=0$, we have
\[
\bigoplus_{v\in K}\HO_c(\Mst^{\zeta \kss}_{v}(Q),\mathbb{Q})^*_{\vir}=\bigoplus_{n\in\mathbb{Z}_{\geq 0}}\HO_c(\Mst_{n\g{i}}(Q),\mathbb{Q})^*_{\vir}.
\]
In particular, for each $i\in Q_0$ there is an incoming wall
\begin{equation}
    \label{B_walls_in}
(\g{i}^{B\perp},F_i);\quad\quad F_i=\chi_{\Wt,K}\left(\bigoplus_{n\in\mathbb{Z}_{\geq 0}}\HO_c(\Mst_{n\g{i}}(Q),\mathbb{Q})^*_{\vir}\right)^{-1}
\end{equation}
in $\f{D}_{\Stab}$ (assuming $1_i\notin \ker(B)$ --- otherwise our version of $\f{D}_{\Hall}$ has no wall associated to $1_i$). 

\begin{eg}\label{InEx}
Let $i\in Q_0$ satisfy $1_i\notin \ker(B_1)$.  If the vertex $i\in Q_0$ does not support a loop, then by Example \ref{coh_ferm} or \eqref{ref_ferm}, the corresponding incoming wall of $\f{D}_{\Stab}$ is 
\begin{align*}
    (\g{i}^{B\perp},\EE(-x_i)).
\end{align*}
Similarly, if there is a single loop based at $i$, then by Example \ref{coh_bos} or \eqref{ref_bos}, the corresponding incoming wall is
\begin{align*}
    (\g{i}^{B\perp},\EE(-t^{-1}x_i)).
\end{align*}
The reason we consider quivers with loops is precisely because they allow us to consider scattering functions $\EE(-p(t)x)$ where $p(t)$ has odd parity.
\end{eg}

\begin{dfn}\label{genteel}
A quiver $Q$ is called \textbf{genteel} if, up to equivalence, the only incoming walls of $\f{D}_{\Stab}$ are those of \eqref{B_walls_in} for $i \in Q_0$. 
\end{dfn}

A $\bb{C}Q$-module $\rho$ is called \textbf{self-(semi)stable} if it is King-(semi)stable for the stability condition $B(\udim(\rho),\cdot)$. In \cite[Def. 11.3]{Bridge}, a quiver (with potential) is called genteel if the only self-stable modules are supported at a single vertex.  As pointed out to us by Lang Mou and recently acknowledged in \cite[arXiv v4]{Bridge}, it is not clear that this implies the condition of Definition \ref{genteel}.  On the other hand, if every self-\textit{semi}stable object is supported at a single vertex, then the condition of Definition \ref{genteel} in fact follows easily.  Unfortunately, this self-semistable version of the definition turns out to be too strong, e.g., Lemma \ref{acyclic_genteel_new} would fail (as we shall see in the lemma's proof).  Hence our decision to modify the definition as above.

 \begin{eg}
 \label{J_genteel}
 Each of the quivers $\LQ{r}$ are genteel.  Indeed, $\LQ{r}$ has only one vertex, so the genteel condition is trivially satisfied.
 \end{eg}
In fact we can generalise this:

\begin{lem}
\label{acyclic_genteel_new}
Let $Q$ be a quiver.  Then $Q$ is genteel if and only if every cycle in $Q$ is composed of loops.
\end{lem}
The proof is partially motivated by that of \cite[Lem 11.5]{Bridge}, which covers one of the implications in the case in which $Q$ is acyclic, but for the self-stable version of genteelness.  Also cf. \cite[Cor. 1.2(i)]{Mou} and \cite[Thm. 1.2.2]{Qin2} for genteelness results in the quantum and classical cases, respectively, for quivers $Q$ without loops and with a green-to-red sequence and non-degenerate potential.
\begin{proof}
Let $a_n\cdots a_1$ be a cycle in $Q$, and let $P\subset Q_0$ be the set of vertices occurring as $s(a_l)$ or $t(a_l)$ for $1\leq l\leq n$.  Set $N=\lvert P\lvert$.   Then consider any $N$-dimensional $\bb{C}Q$-module $\rho$, with dimension vector $v$ satisfying
\[
v_i=\begin{cases} 1 &\textrm{if }i\in P \\
0 &\textrm{otherwise}\end{cases}
\]
and for which the action of each $a_1,\ldots,a_n$ is an isomorphism.  Clearly this module is simple, and so in particular it is King-stable for any stability condition in $\udim(\rho)^{\perp}$.  

Let $\zeta\in \udim(\rho)^\perp$ be a general stability condition.  Then $\rho$ is $\zeta$-stable, and so $\Mst^{\zeta \kss}_{v}(Q)$ is nonempty.  The top degree compactly supported cohomology of this stack has, as a basis, the set of top-dimensional irreducible components; since the stack is irreducible, this cohomology is one-dimensional, and in particular nontrivial.  It follows that
\[
\bigoplus_{v\in K}\HO_c (\Mst^{\zeta \kss}_{v}(Q),\mathbb{Q})_{\vir}\neq 0
\]
and so by Proposition \ref{Apurity}
\[
\int 1_{\mathrm{ss}}(\zeta)\neq 0.
\]
It follows that $\udim(\rho)^{\perp}$ is covered by walls of $\f{D}_{\Stab}$, and in particular contains an incoming wall.  So if there exists a cycle as above with $N>1$, then $Q$ is not genteel.

On the other hand, assume that $Q$ contains no cycles that are not entirely composed of loops.  Equivalently, we can label the vertices $Q_0$ with numbers $\{1,\ldots,r\}$ in such a way that for $i>j$ there are no arrows from $i$ to $j$.

Given a $\bb{C}Q$-module $\rho$, let $Q_0^{\rho}$ denote the vertices of $Q$ which lie in the support of $\rho$, and let $Q^{\rho}$ be the subquiver of $Q$ consisting of these vertices and all arrows between them.  Given any connected component $Q'$ of $Q^{\rho}$, let $\alpha\in \{1,\ldots,r\}$ be the minimal number for which $\rho$ has non-trivial support $\rho_{\alpha}$ at the vertex $\alpha$.  We extend $\rho_{\alpha}$ by zero to obtain a submodule $\rho'\subset \rho$.  Then $B(\udim(\rho),\udim(\rho'))\geq 0$, with equality if and only if $\alpha$ is the only vertex of $Q'$.  Thus, $\rho$ cannot be self-semistable unless every component of $Q^{\rho}$ has only a single vertex (possibly with loops).

So if $\rho$ is self-semistable, it must be a direct sum of representations which are supported on disjoint vertices.  Suppose $\rho$ is semistable for some other stability condition $\zeta$.  Then each summand $\rho_{\alpha}$ of $\rho$ must be $\zeta$-semistable, thus forcing $\zeta$ to be in $\udim(\rho_{\alpha})^{\perp}$ for each $\rho_{\alpha}$.  So the locus of such $\zeta$ lives in codimension at least two (and thus does not form a wall) unless $\rho$ is supported at a single vertex, as desired.
\end{proof}

One may consider a possibly stronger notion of genteelness which requires that the only incoming walls of $\f{D}_{\Hall}$ (as opposed to $\f{D}_{\Stab}$) are $(\g{i}^{B\perp},F_i)$ with \begin{equation*}
F_i^{-1}=1_{\mathrm{ss}}(B_1(1_i)),
\end{equation*}
i.e., without application of $\chi_{\Wt,K}$.  We note that Lemma \ref{acyclic_genteel_new} holds for this stronger notion as well by the same proof.

\subsection{Positivity and parity proofs}\label{PosPar}

As usual, let $L_0$ be a lattice with a skew-symmetric integer valued form $\omega(\cdot,\cdot)$, let $\sigma$ be a strongly convex rational polyhedral cone in $L_{0,\bb{R}}$, and denote $L_0^{\oplus}\coloneqq L_0\cap \sigma$ and $L_0^+\coloneqq L_0^{\oplus}\setminus \{0\}$. Let $(V,P)$ be a pair consisting of a tuple of elements $V=(v_1,\ldots,v_s)$ of $L_0^+$, and nonzero Laurent polynomials $P=(p_1(t),\ldots,p_s(t))$ with each $p_i(t)\in\mathbb{Z}_{\geq 0}[t^{\pm 1}]$.  We will assume that there is some $1\leq s'\leq s$ such that $p_1(t),\ldots,p_{s'}(t)$ are even, and $p_{s'+1}(t),\ldots,p_s(t)$ are odd.  

We define a quiver $Q=Q(V,P)$ depending on this data as follows.  First define 
\[
Q_0=\{(i,\alpha)\colon i\in \{1,\ldots, s\} \textrm{ and }\alpha\in\{1,\ldots,p_i(1)\}\}.
\]
For each pair of vertices $(i,\alpha),(j,\beta)\in Q_0$ with $(i,\alpha)\neq (j,\beta)$, we add $\max(0,-\omega(v_i,v_j))$ arrows with source $(i,\alpha)$ and target $(j,\beta)$.  Finally, we add a loop to each $(i,\alpha)$ with $i>s'$. 
The following is a consequence of Lemma \ref{acyclic_genteel_new}.
\begin{cor}\label{GenteelCor}
Let $Q(V)$ be the quiver with vertices $Q(V)_0=V=\{v_1,\ldots,v_s\}$ and $\max(0,-\omega(v_i,v_j))$ arrows from $v_i$ to $v_j$ for all $i,j\leq s$.  Then $\BQ{V,P}$ is genteel if and only if $Q(V)$ is acyclic.  In particular, if $s=2$ then $\BQ{V,P}$ is genteel. 
\end{cor}
\begin{proof}
Note that $Q(V)$ is isomorphic to $Q(V,(1,\ldots,1))$.  It is straightforward to see that $Q(V)$ being acyclic is equivalent to the condition that the only oriented cycles of $Q(V,P)$ are composed of loops.  The claim then follows from Lemma \ref{acyclic_genteel_new}.
\end{proof}

\begin{prop}
\label{Bridge_scat}
Let $L$ be a lattice with skew-symmetric rational bilinear form $\omega(\cdot,\cdot)$, and let $L_0\subset L$ be a sublattice on which $\omega(\cdot,\cdot)$ is integer valued, containing elements $v_1,\ldots,v_{s'},v_{s'+1},\ldots v_{s}$. For $a\in \mathbb{Z}^{s}$ write $v_a=\sum_i a_i v_i$.  Let $p_1(t),\ldots,p_{s}(t)\in\mathbb{Z}_{\geq 0}[t^{\pm 1}]$ satisfy 
\[
\parit(p_i(t))=\begin{cases} 0& 1\leq i\leq s'\\
1& s'<i\leq s.
\end{cases}
\]
We define the bilinear form $\lambda(\cdot,\cdot)$ on $\mathbb{Z}^{s}$ as in \eqref{lambda_def}.

Then there is a consistent scattering diagram $\f{D}$ in $L_{\bb{R}}$ over $\f{g}=\f{g}_{L_0^+,\omega}$ such that 
\begin{enumerate}
    \item The set of incoming walls contains the set of walls
    \begin{equation}\label{DWallsProp}
    \{(v_i^{\omega\perp},\EE(-p_i(t)z^{v_i}))\colon i=1,\ldots,s\}
    \end{equation}
    \item
    All walls have the form 
\[
    (\f{d},\EE(-p_a(t)z^{v_a}))\]
    for $a\in \mathbb{Z}^{s}_{\geq 0}$, where $p_a(t)\in\mathbb{Z}_{\geq 0}[t^{\pm 1}]$ and $\f{d}\subset v_a^{\omega\perp}$
    \item
    Each $p_a(t)$ is either odd or even, with parity given by $\lambda(a,a)+1$.
\end{enumerate}
If the quiver $Q(V)$ with vertices $v_1,\ldots,v_s$ and with $\max(0,-\omega(v_i,v_j))$ arrows from $v_i$ to $v_j$ is acyclic, then we can choose $\f{D}$ so that the incoming walls are precisely given by \eqref{DWallsProp}.  If $p_i(t)=1$ for each $i=1,\ldots,s$ (so $s'=s$) in addition to $Q(V)$ being acyclic, then we can further impose the condition that each $p_a(t)$ as in condition (2) above has Lefschetz type.
\end{prop}
Note that the final statement above yields Theorem \ref{acyclic_scat_thm}.
\begin{proof}
It will be convenient to assume that the vectors $v_i$ span $L_{\bb{R}}$.  We can always assume this after possibly extending $L_0$ to some finite-index sublattice $L'_0$ of $L$, enlarging $\sigma$ to have full rank in $L_{\bb{R}}$, and then enlarging $V$ with a minimal set of additional vectors $v_i$ so that the resulting set $V'$ spans $L_{\bb{R}}$.   The polynomials $p_i(t)$ associated to $v_i\in V'\setminus V$ are taken to be, say, $1$.  After proving the results for this setup, one can then take the quotient by $\bigoplus_{v\in (L'_0)^+\setminus L_0^+} \f{g}_v$ to obtain the results in general.  So we now assume $V$ spans $L_{\bb{R}}$.

Recall the quiver $Q=Q(V,P)$ defined above.  
Let $K=\mathbb{Z}_{\geq 0}^{Q_0}$ be the monoid of dimension vectors for $\BQ{V,P}$.  There is a homomorphism of algebras
\[
F\colon \kk_t[K]\rightarrow \kk_t[L_0^{\oplus}]
\]
defined as follows.  Write $p_i(t)=\sum_l p_{i,l}t^l$.  For $(i,\alpha)\in Q_0$ define 
\[
\eta'_i(\alpha)=\max\left\{ r~\colon~ \sum_{r'< r}p_{i,r'} < \alpha,~ \parit(r)= \parit(p_i(t))\right\}.  
\]
For example if $p_i(t)=3t^{-2}+4+3t^2$, 
\[
(\eta'_i(1),\ldots ,\eta'_i(10))=(-2,-2,-2,0,0,0,0,2,2,2).
\]
We define 
\[
\eta_i(\alpha)=\begin{cases} \eta'_i(\alpha) & \textrm{if }i\leq s'\\
\eta'_i(\alpha)+1 &\textrm{if }s'<i\leq s.
\end{cases}
\]
Note that for all $(i,\alpha)\in Q_0$ there is an equality 
\begin{equation}\label{eta_parity}
    \eta_i(\alpha)\equiv 0 \quad\textrm{ (mod $2$).} 
\end{equation}
We set 
\begin{equation*}
F(x_{i,\alpha})=t^{\eta_i(\alpha)}z^{v_i}.
\end{equation*}

Passing to completions, $F$ induces a morphism $G^{\qtrop}_Q\rightarrow G^{\qtrop}$ which we denote by $\hat{F}$.  From \eqref{bosvfer} and \eqref{eta_parity} it follows that
\[
\hat{F}(\EE(p(t)x^v)=\EE(p(t)F(x^v))
\]
for all $0\neq v\in K$ and all $p(t)\in\mathbb{Z}[t^{\pm 1}]$.

We define a $L_0^{+}$-grading on $\f{g}_Q$ by pulling back the $L_0^{+}$-grading on the image of $F$.  Then the scattering diagram $\f{D}_{\Stab}$ for the category of $\mathbb{C}Q$-modules satisfies
\begin{enumerate}
    \item For each $i\leq s'$ and $\alpha\leq p_i(1)$ there is an incoming wall $(\g{i,\alpha}^{\perp},\EE(-x_{i,\alpha}))$ \item
    For each $s'<i\leq s$ and $\alpha\leq p_i(1)$ there is an incoming wall $(\g{i,\alpha}^{\perp},\EE(-t^{-1}x_{i,\alpha}))$
    \item
    All walls in $\f{D}_{\Stab}$ are of the form $(\f{d},\EE(-p_v(t)x^v))$ with $p_v(t)\in\bb{Z}_{\geq 0}[t^{\pm 1}]$ and $\f{d}\subset v^{\perp}$ for $v\in K$. 
    \item
    $\parit(p_v(t))\equiv \chi_Q(v,v)+1$ (mod $2$).
\end{enumerate}
The first two parts of the claim follow from Example \ref{InEx}.  The last two parts are a special case of Proposition \ref{Pos1}.  The morphism
\begin{align*}
    R\colon \mathbb{R}^{Q_0}\rightarrow &L_{0,\mathbb{R}}\\
    \g{i,\alpha}\mapsto &v_i
\end{align*}
induces a surjection of vector spaces, with $R^*\omega=B_Q$.  As in Lemma \ref{QuotientScat}, $R$ and $F$ induce a consistent scattering diagram $(R,F)_*\f{D}_{\Stab}$ in $L_{\bb{R}}$ over $\f{g}$, giving the desired scattering diagram.  For the statement regarding the parity of $\f{D}$, identify $\bb{Z}^s$ with $\bb{Z}^{Q(V)_0}$, and let $\wt{R}$ denote the projection $\bb{Z}^{Q_0}\rar \bb{Z}^{Q(V)_0}$, $\wt{R}(1_{i,\alpha})=1_i$.  Then the parity statement follows from \eqref{eta_parity}, property (4) above, and the observation that $\chi_Q=\wt{R}^*\lambda$.

If $Q(V)$ is acyclic, then $Q(V,P)$ is genteel by Corollary \ref{GenteelCor}, so the incoming walls listed in Properties (1) and (2) of $\f{D}_{\Stab}$ above are the only incoming walls of $\f{D}_{\Stab}$.  The claim that the only incoming walls of $(R,F)_*\f{D}_{\Stab}$ are as in \eqref{DWallsProp} follows using Lemma \ref{QuotientScat} (assuming we choose the representative of $\f{D}_{\Stab}$ to be saturated).

If we additionally have $p_i(t)=1$ for each $i$, then by Corollary \ref{MRcor}, we can represent $\f{D}_{\Stab}$ so that each $p_v(t)$ as in Property (3) above has Lefschetz type.  Since in this case $F$ maps $x_{i,\alpha}$ to $z^{v_i}$, all scattering functions of $(R,F)_*\f{D}_{\Stab}$ will still have Lefschetz type, as desired. 
\end{proof}

\subsubsection{Proof of positivity for the basic two-wall cases}

We now prove the positivity statement of Theorem \ref{PosScat} in the cases in which the initial scattering diagram has only two walls.  The general case then follows by \S \ref{PosSect}.

\begin{cor}\label{2WallPos}
Consider the scattering diagram
\begin{align}\label{2WallPosIn}
    \f{D}_{\In}=\{(v_1^{\omega\perp},\EE(-t^{m_1}z^{v_1})), (v_2^{\omega\perp},\EE(-t^{m_2}z^{v_2}))\}
\end{align}
where $v_1,v_2\in L_0^{+}$.  Then there is some $\f{D}$ representing the equivalence class of $\scat(\f{D}_{\In})$ such that all walls of $\f{D}$ are either one of the two incoming walls, or outgoing walls of the form $(\f{d}_{v},\EE(-t^mz^v))$ for $
\f{d}_v\subset v^{\omega \perp}
$ and $v=\alpha  v_1+\beta v_2$, with $\alpha,\beta\in\mathbb{Z}_{\geq 1}$.
\end{cor}

\begin{proof}
We apply Proposition \ref{Bridge_scat} in the case where $V=(v_1,v_2)$ and $P=(t^{m_1},t^{m_2})$.  Since $Q(V)=\Kr{|\omega(v_1,v_2)|}$ is acyclic, the resulting incoming walls are precisely as in \eqref{2WallPosIn}, and then property (2) from Proposition \ref{Bridge_scat} yields the claim. 
\end{proof}

\subsubsection{Dense regions}\label{denseSec}
We next consider some examples that make precise our claim from the introduction that ``almost all'' two wall examples have dense regions.

\begin{eg}\label{dense_example}
In Corollary \ref{2WallPos}, set $\omega(v_1,v_2)=1$, $m_1=0$ and $m_2=-1$.  Then the scattering diagram $\f{D}$ is the stability scattering diagram for the quiver with vertices $\{1,2\}$, a loop $c$  at $2$, and an arrow $a$ from $2$ to $1$.  Let $\alpha\leq\beta$.  Fix numbers $\lambda_1,\ldots,\lambda_{\alpha}\in\bb{C}$.  We make $\bb{C}^{\alpha}\oplus\bb{C}^{\beta}$ into a $\bb{C}Q$-module as follows: firstly $\cdot \lp{i}$ acts on $\bb{C}^{\alpha}\oplus\bb{C}^{\beta}$ via projection onto the $i$th factor.  Secondly, the linear operator $\cdot a\colon \bb{C}^{\alpha}\rightarrow\bb{C}^{\beta}$ acts on row vectors from the right via the matrix 
\[
\left(\begin{array}{cccc} \Id_{\alpha\times\alpha} & 0& \ldots & 0\end{array}\right)
\]
while $\cdot c$ acts  via the linear map $\bb{C}^{\beta}\rightarrow \bb{C}^{\beta}$ defined via the matrix 
\[
\left(\begin{array}{ccccc}  0&1&0&0&\ldots\\
 0& 0& 1 &0&\ldots\\
 &&\ldots&&\\
0&0&\ldots&0&1\\
\lambda_1&\lambda_2&\ldots&&\lambda_{\beta}
\end{array}\right).
\]
If $V$ is a submodule of $\rho$ of dimension vector $(x,y)$, the fact that $V\cdot \lp{2}$ is preserved by $\cdot c$ implies that $y\geq x+\beta-\alpha$ whenever $x>0$.  Since $-B(\udim(\rho),(x,y))=-\alpha y +\beta x$, we have $-B(\udim(\rho),\udim(V))=-\alpha y\leq 0$ if $x=0$, and if $x> 0$, we have \begin{align*}
    -B(\udim(\rho),(x,y))  & \leq -\alpha x -\alpha\beta+\alpha^2+\beta x=(\beta-\alpha)(x-\alpha)\leq 0.
\end{align*} 
Hence, $\rho$ is $-B(\udim(\rho),\cdot)$-semistable.

In fact, if $x=\alpha$, then $y=\beta$ and so $V=\rho$. Hence, if $V$ is a non-trivial proper submodule of $\rho$, then we must have $x<\alpha$ and so if $\alpha<\beta$ we find $-B(\udim(\rho),\udim(V))<0$.  So $\rho$ is in fact $-B(\udim(\rho),\cdot)$-stable.  Moreover
\begin{enumerate}
    \item 
    If $\alpha=\beta$, all indecomposable $-B(\udim(\rho),\cdot)$ King-semistable modules arise this way.  Since $\rho$ contains a $(1,1)$-dimensional representation spanned by the vectors $v'$ and $v$, where $v$ is any eigenvector of $\cdot c$ and $v'$ satisfies $v'\cdot a=v$, there are no $-B((1,1),\cdot)$ King-stable $(n,n)$-dimensional representations for $n\geq 2$.
    \item
    If $\alpha=1$, all $-B(\udim(\rho),\cdot)$ King-semistable (and hence stable) $(\alpha,\beta)$-dimensional $\bb{C}Q$-modules arise this way.  As such, there is an isomorphism $\Msp_{(1,\beta)}^{-B(\udim(\rho),\cdot)\sst}(Q)\cong \bb{A}^{\beta}$. Geometrically, this also follows from the fact that the moduli space $\Msp_{(1,\beta)}^{-B(\udim(\rho),\cdot)\sst}(Q)$ can be identified with the Hilbert scheme of $\beta$ points on $\bb{A}^1$, which is identified with $\CSym^{\beta}\bb{A}^1\cong\bb{A}^{\beta}$ via the Hilbert--Chow morphism. 
\end{enumerate}
We have drawn the diagram $\f{D}$ of Corollary \ref{2WallPos} in Figure \ref{scat_fig}.  One ellipsis in the diagram indicates that for every $n\in \bb{Z}_{\geq 1}$ there is a wall with function starting $\EE(t^{-n}z^{(1,n)}+\cdots)$, a fact that we deduce from (2).  From (1) we deduce that there are no higher order terms on the wall through $(-1,-1)$, i.e. the $z^{(2,2)}$ coefficient is zero.  The grey region indicates that, because of the existence of the above stable modules, by Lemma \ref{dimension_lemma} there is a dense region of nonzero walls in the scattering diagram --- in fact, the existence of these stable modules implies that $z^{(\alpha,\beta)}$ has a nonzero coefficient in the scattering function in which it appears for every $0<\alpha<\beta$.  
\end{eg}
\begin{figure}[ht]
\bigskip{\center 
\begin{tikzpicture}
\draw [black, fill=gray] (0,0) -- (2.7,-2.7) -- (2.7,0)--(0,0);
\draw[->](-3,0) --(3,0);
\draw[->](0,3) --(0,-3);
\draw[->](0,0) --(3,-3);
\draw[->](0,0) --(3,-1.5);
\draw[->](0,0) --(3,-1);

\node at (0.3,3.25){$\EE(-z^{(1,0)})$};
\node at (-3.2,-0.25){$\EE(-t^{-1}z^{(0,1)})$};
\node at (3.7,-3.3){$\EE(-t^{-1}z^{(1,1)})$};
\node at (4.5,-1.7){$\EE(-t^{-2}z^{(1,2)}+\ldots)$};
\node at (4.5,-1.1){$\EE(-t^{-3}z^{(1,3)}+\ldots)$};
\node at (3.5,-.5){$\ldots$};

\end{tikzpicture}}
\caption[]{Scattering diagram for $\omega=\left(\begin{array}{cc}0&1\\-1&0\end{array}\right)$ with initial scattering functions $\EE(-t^{-1}z_2)$ and $\EE(-z_1)$.}
\label{scat_fig}
\end{figure} 

\begin{eg}[Density of walls for Kronecker quivers]
\label{Kronecker_example}
We now consider a situation like that of \cite[Ex. 1.15]{GHKK}. I.e., consider the two-dimensional scattering diagram $\scat(\f{D}_{\In})$ for $\f{D}_{\In}=\{(v_i^{\Lambda \perp},\EE(-z^{v_i}))\colon i=1,2\}$ with $\Lambda(v_1,v_2)=n$, i.e., the scattering diagram associated to the $n$-Kronecker quiver $\Kr{n}$.  For convenience, take $v_1=(0,1)$, $v_2=(-1,0)$.  It follows from the mutation invariance of $\f{D}$ (Proposition \ref{mut-inv}) that $\f{D}\setminus \f{D}_{\In}$ is preserved by the action of $\mu\coloneqq \left(\begin{matrix} 0 & -1 \\ 1 & n 
\end{matrix}\right)$.  

For $n\geq 2$, it follows that the rays of the cluster complex accumulate along the eigenrays of $\mu$ generated by $(1,-\lambda)$ for $\lambda=\frac{n\pm \sqrt{n^2-4}}{2}$.   Let $\sigma_{\bad}$ denote the interior of the cone generated by these two rays.  It has often been suggested (e.g., in \cite[Ex. 1.6]{GPS}, \cite[Ex. 1.4]{GP} and \cite[Ex. 1.15]{GHKK}) that every ray of rational slope in $\sigma_{\bad}$ appears to support a nontrivial wall of $\f{D}$.  Using Example \ref{dense_example}, we can see that this is indeed true.

\begin{figure}
\bigskip{\center 
\begin{tikzpicture}
\draw [black, fill=gray] (0,0) -- (3,-1.3) -- (3,-3) -- (1.3,-3)--(0,0);
\draw[->](-3,0) --(3,0);
\draw[->](0,3) --(0,-3);
\draw[->](0,0) --(1,-3);
\draw[->](0,0) --(3,-.7);
\draw[->](0,0) --(3,-1.0);
\draw[->](0,0) --(3,-1.15);
\draw[->](0,0) --(3,-1.23);
\draw[->](0,0) --(.7,-3);
\draw[->](0,0) --(1,-3);
\draw[->](0,0) --(1.15,-3);
\draw[->](0,0) --(1.23,-3);
\node at (-3.5,0.3){$\EE(-z^{v_2})$};
\node at (0.45,3.2){$\EE(-z^{v_1})$};
\node at (2,-2){$\sigma_{\bad}$};
\end{tikzpicture}}
\caption[]{Sketch of the scattering diagram for the $n$-Kronecker quiver $\Kr{n}$, $n\geq 3$.}
\label{scat_fig_Kr}
\end{figure}
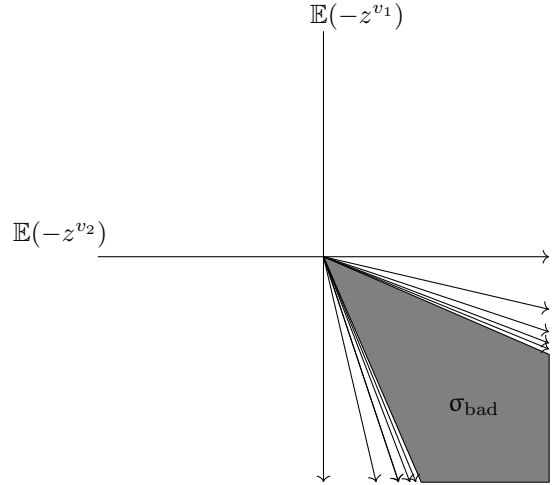 

It suffices to check this for the classical version of $\f{D}$.  In this case, we can apply the change of lattice trick of \cite[\S C.3, Step IV]{GHKK} to replace our initial walls with the pair of walls $\lim_{t\rar 1} \EE(-nz^{u_1})$ and $\lim_{t\rar 1} \EE(-nz^{u_2})$ with $\Lambda(u_1,u_2)=1$.  We choose a basis so that $u_1=(0,1)$, $u_2=(0,-1)$.

We can factor these initial walls into $2n-2$ walls $\{\f{d}_{ij}\colon  i=1,2, j=1,\ldots,n-1\}$, where $f_{\f{d}_{i1}}= \lim_{t\rar 1} \EE(-2z^{u_i})$, and $f_{\f{d}_{ij}}= \lim_{t\rar 1}  \EE(-z^{u_i})$ for $j=2,\ldots,n-2$.  We perturb these walls by translating each of $\f{d}_{i1}$ down and towards the left, and translating the other walls by small generic amounts.  Then the interaction of the $\f{d}_{i1}$'s produces a limiting ray with function $\lim_{t\rar 1}\EE(-(t+t^{-1})z^{u_1+u_2})$ (this follows from, say, Example \ref{coh_K2}).  When this interacts with $\lim_{t\rar 1} \EE(-z^{u_i})$, Example \ref{dense_example} implies that the resulting outgoing rays include every ray of rational slope in the cone spanned by $u_1+u_2$ and $u_1+u_2+u_i$.  Then by letting ray $\bb{R}_{\geq 0}(u_1+u_2+u_i)$ interact with another $\EE(-z^{u_i})$-wall (if there are others), we extend this to $u_1+u_2+2u_i$, and so on.  This shows that every ray of rational slope in the cone $C\coloneqq \bb{R}_{\geq 0}\langle(n-1,-1),(1,1-n)\rangle$ has non-trivial attached scattering function.  Finally, noting that $\mu((1,-1))=(1,1-n)$, and so $C$ overlaps $\mu(C)$, we see that the cones $\mu^k(C)$ for $k\in \bb{Z}$ completely cover the interior of $\sigma_{\bad}$, and the claim follows.

In fact, we can see a bit more than density of the walls --- we see that $z^{v}$ has a nonzero-coefficient along $\bb{R}_{\geq 0}(-v)$ for every nonzero $v\in -\sigma_{\bad}$.  Indeed, using the above arguments, the analogous statement about $z^{(\alpha,\beta)}$ having nonzero coefficients for all $0<\alpha<\beta$ in Example \ref{dense_example} implies the claim for $-v\in \bb{R}_{>0} \langle(1,-1),(1,1-n)\rangle$.  We extend this to include the ray $\bb{R}_{>0} (1,-1)$ using Reineke's description of this ray \cite[Thm. 6.4]{ReinCompos} (cf. \cite[Ex. 1.15]{GHKK}).  Then applying $\mu^k$  for all $k\in \bb{Z}$ as before yields the claim for the full region.  In terms of quiver moduli, this can be viewed as saying that there exist stable representations of $\Kr{n}$ ($n\geq 3$) of dimension vector $v$ for all $v\in \bb{R}_{> 0} \langle (1,\frac{n-\sqrt{n^2-4}}{2}),(1,\frac{n+\sqrt{n^2-4}}{2}) \rangle\cap \bb{Z}_{\geq 0}^2$.
\end{eg}

\begin{rmk}
Via Theorem \ref{MRtheorem} and Lemma \ref{dimension_lemma}, the statement above regarding nonvanishing of the coefficients for $z^{(\alpha,\beta)}$ is equivalent to the existence of stable modules of dimension vector $(\alpha,\beta)$.  So rather than using the above example to deduce the existence of such stable modules, one could hope to directly show the existence of the stable modules, and then deduce the nonvanishing of the coefficients from this.  Indeed, as pointed out to us by Pierrick Bousseau, there is an argument taking as its inputs only results from the quiver literature \cite{Kac} and \cite{Scho} which demonstrates the existence of such stable modules, without invoking any DT theory or scattering diagram techniques.
\end{rmk}

Putting the above examples together, we see that if $\lvert \omega(v_1,v_2)\lvert \neq 0$ in Corollary \ref{2WallPos}, the only two choices of $(n,m_1,m_2)$ that do not give rise to regions in $L_{\bb{R}}$ of nonzero measure in which walls of $\f{D}$ are dense are $(1,0,0)$ and $(2,0,0)$.

\subsubsection{Parities: Proof of Theorem \ref{parity_cor}}\label{parity_proof}
Consider $\f{D}_{\In}$ given by \eqref{DIn} as in Theorem \ref{PosScat}, where the polynomials $p_i(t)$ are assumed to satisfy the positivity and parity assumptions of Theorem \ref{parity_cor}.  
Let $\f{D}^{\circ}$ be the consistent scattering diagram constructed in Proposition \ref{Bridge_scat} starting with the elements $v_1,\ldots v_s$ and polynomials $p_1(t),\ldots,p_s(t)$.  Let $\f{D}^{\circ}_{\In}$ be the set of incoming walls of $\f{D}^{\circ}$.  Then by construction, $\f{D}_{\In}\subset \f{D}^{\circ}_{\In}$, and $\f{D}^{\circ}$ satisfies the parity conditions in the statement of Proposition \ref{Bridge_scat}.

It is not clear that the set of walls in $\f{D}_{\In}^{\circ}$ is finite, so we work to arbitrary finite order $k$, for $k$ large enough to ensure that no walls of $\f{D}_{\In}$ are trivial in $\f{g}_k$.  We will add the subscript $k$ to a scattering diagram to indicate the subset of walls which are nontrivial in $\f{g}_k$ (but still working over $\f{g}$ rather than projecting to $\f{g}_k$).  By extending all initial walls to be full hyperplanes as in \S \ref{ComparisonTricks}, we can let $(\f{D}^{\circ}_{\In})_k=\{(v_i^{\omega\perp},-p_{i}(t)z^{v_{i}})\}_{i\in I}$ for $I$ some finite index-set.  Then $\f{D}_{\In}$ is the subset of walls corresponding to $i\in I'$ for some $I'\subset I$ which we identify with $\{1,\ldots,s\}$.

Similarly to the change of monoid trick (cf. \S \ref{MonoidTrick}), let $\wt{L}\coloneqq \bb{Z}\langle \wt{u}_1,\ldots,\wt{u}_r,\wt{v}_i\colon i\in I\rangle$ and let $\pi\colon \wt{L}\rar L$ be the map sending $\wt{v}_i\mapsto v_i$ and $\wt{u}_i\mapsto u_i$ for $\{u_i\colon i=1,\ldots,r\}$ any basis for $L$. Let $\wt{\omega}\coloneqq \pi^* \omega$.  Define 
\begin{align*}
p\colon &\wt{L}\rightarrow \bb{Z}\langle e_1,\ldots e_s\rangle\\
&\wt{v}_i\mapsto e_i& i\in I'\\
&\wt{v}_i\mapsto 0 &i\in I\setminus I'\\
&\wt{u}_i\mapsto 0&i=1,\ldots,r
\end{align*}
and set $\wt{\lambda}=p^*\lambda$.  Let $\wt{\sigma}\subset \wt{L}_{\bb{R}}$ be the cone spanned by $\wt{v}_i$ for $i\in I$ and let $\wt{L}^+=\wt{L}\cap \wt{\sigma}\setminus\{0\}$.  We define a homomorphism of Lie algebras
\begin{align*}
    \varpi\colon& \f{g}_{\wt{L}^+,\wt{\omega}}\rightarrow \f{g}_{L^+,\omega}\\
    &\hat{z}^{\tilde{v}}\mapsto z^{\pi(\tilde{v})}/(1)_t.
\end{align*}

We then consider 
$$(\wt{\f{D}}^{\circ}_{\In})_k\coloneqq \{(\wt{v}_i^{\wt{\omega}\perp},-p_{i}(t)z^{\wt{v}_{i}})\}_{i\in I},$$
with $\wt{\f{D}}_{\In}$ denoting the subset consisting of the walls with $i\in I'$.  Then $\wt{\f{D}}\coloneqq \scat(\wt{\f{D}}_{\In})$ is a subset of the walls of $\wt{\f{D}}_k^{\circ}\coloneqq \scat((\wt{\f{D}}^{\circ}_{\In})_k)$, specifically, the walls $(\f{d},\EE(-p_{\f{d}}(t)z^{\wt{v}_{\f{d}}}))$ with $\wt{v}_{\f{d}}$ in the span of $\{\wt{v}_i\}_{i\in I'}$.  

By Lemma \ref{QuotientScat}, $\pi$ and $\varpi$ induce consistent scattering diagrams $(\pi,\varpi)_*(\wt{\f{D}}_k^{\circ})$ and $(\pi,\varpi)_*(\wt{\f{D}})$ in $L_{\bb{R}}$, defined as in \eqref{piD}.  By the positivity statement from Theorem \ref{PosScat}, up to equivalence, all walls of  $\wt{\f{D}}_k^{\circ}$ are of the form $(\f{d},\EE(-p_{\f{d}}(t)z^{\wt{v}_{\f{d}}}))$ for $p_{\f{d}}(t)$ a Laurent polynomial with positive coefficients.  Up to equivalence, we have $\f{D}=(\pi,\varpi)_*(\wt{\f{D}})$, so then for all general $x\in L_{\bb{R}}$, we can write (in the notation of Example \ref{EquivEx}(3))
\begin{align*}
f_{x,\f{D}}=\EE\left(-\sum_{v\in x^{\omega\perp}} a_{x,v}(t)z^v\right)
\end{align*}
    and
\begin{align*}
    f_{x,(\pi,\varpi)_*(\wt{\f{D}}_k^{\circ})}=\EE\left(-\sum_{v\in x^{\omega\perp}} [a_{x,v}(t)+a^{\circ}_{x,v}(t)]z^v \right)
\end{align*}
where each $a_{x,v}(t)$ and $a^{\circ}_{x,v}(t)$ is a Laurent polynomial in $t$ with positive integer coefficients.  Since $((\pi,\varpi)_*(\wt{\f{D}}_k^{\circ}))_k$ is equivalent to $\f{D}_k^{\circ}$, and since the latter satisfies the desired parity condition by Proposition \ref{Bridge_scat}, we have that $\parit(a_{x,v}(t)+a^{\circ}_{x,v}(t))=\wt{\lambda}(\wt{v},\wt{v})+1$ for all $v\in L^+\setminus (k+1)L^+$, where $\wt{v}\in \pi^{-1}(v)$.  The positivity then implies that $\parit(a_{x,v}(t))=\lambda(p(\wt{v}),p(\wt{v}))+1$ for $v\in L^+\setminus (k+1)L^+$ as well.  Since $k$ was arbitrary, the claim follows.
\qed

\subsection{Pre-Lefschetz for the basic two-wall cases}

We now apply the constructions of \S \ref{ref_DT_sec} to prove the stronger pL version of Corollary \ref{2WallPos}.  The general case of the pL statement of Theorem \ref{PosScat} then follows by \S \ref{PosSect}.

Recall from (\ref{plPolys}) the definition of the basic polynomials $\pl_m(t)$ of pL type. 

\begin{prop}\label{2WallpL}
Consider the scattering diagram
\begin{align}
    \f{D}_{\In}=\{\left(v_1^{\omega\perp},\EE(-\!\pl_{m_1}(t)z^{v_1})\right), \left(v_2^{\omega\perp},\EE(-\!\pl_{m_2}(t)z^{v_2})\right)\}
\end{align}
with $v_1,v_2\in L_0^+$ and $m_1,m_2\in\mathbb{Z}$.   Then there is some $\f{D}$ representing the equivalence class of $\scat(\f{D}_{\In})$ such that all walls of $\f{D}$ are either one of the two incoming walls, or outgoing walls of the form $(\f{d}_{v},\EE(-\!\pl_m(t)z^v))$ for $\f{d}_v\subset v^{\omega\perp}$ with $v=\alpha v_1+\beta v_2$, $\alpha,\beta\geq 1$.
\end{prop}
\begin{proof}
As in the proof of Proposition \ref{Bridge_scat}, after an appropriate substitution of $z^{v_1}$ and $z^{v_2}$ we may assume that $m_1,m_2\in \{0,1\}$.  Re-ordering the walls if necessary, we assume $\omega(v_1,v_2)=n\in\mathbb{Z}_{\geq 0}$.  We also assume that $n>0$, since otherwise $\f{D}_{\In}$ is already consistent.  We construct a quiver $Q$ via the following recipe.
\begin{enumerate}
    \item First, we define quivers $Q^{(i)}$, for $i=1,2$ by the rule:
    \begin{itemize}
        \item If $m_i=1$ then $Q^{(i)}$ is isomorphic to the Kronecker quiver $\Kr{2}$, with vertices $1^{(i)}$ and $2^{(i)}$ and 2 arrows from $2^{(i)}$ to $1^{(i)}$.
        \item If $m_i=0$ then $Q^{(i)}$ is isomorphic to the quiver $\LQ{0}$, with one vertex, denoted $1^{(i)}$, and no arrows.
    \end{itemize}
    \item We define $Q'$ to be the disjoint union of $Q^{(1)}$ and $Q^{(2)}$.
    \item We adjoin $n$ arrows from $1^{(2)}$ to $1^{(1)}$ to $Q'$, to form the quiver $Q$.
\end{enumerate}

\begin{figure}[ht]
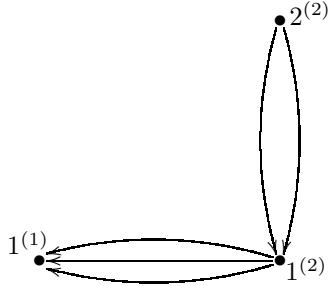

\[
\xy
{\ar (15,0); (-15,0)};
{\ar@/^.50pc/ (15,-0.5); (-15,-1)};
{\ar@/_.50pc/ (15,0.5); (-15,1)};
{\ar@/_.50pc/ (15.5,31); (15.5,1)};
{\ar@/^.50pc/ (16.5,31); (16.5,1)};
(-16,0)*{\bullet};
(16,0)*{\bullet};
(16,32)*{\bullet};
(19.5,-1)*{1^{(2)}};
(20,33)*{2^{(2)}};
(-17.5,2.5)*{1^{(1)}}
\endxy
\]
\caption[]{An example of the quiver $Q$ arising in the proof of Proposition \ref{2WallpL} for $m_1=0$, $m_2=1$ and $n=3$.}
\label{example_quiver}
\end{figure}

In Figure \ref{example_quiver} we have drawn the quiver $Q$ for the case $m_1=0$, $m_2=1$, $n=3$.  Note that regardless of the values of $m_1,m_2,n$, the quiver $Q$ is acyclic.  

Fix an integer $k\geq 0$, and denote by $\pi\colon \widehat{G}_Q\rightarrow G_{Q,k}$ the projection defined as in \S\ref{qTorLie}.  Next, fix numbers $0\leq \epsilon\ll \delta\ll 1\ll M\ll N$, which will depend on $k$.  We define the stability conditions $\zeta(\pm)$ for $Q$ as follows.
\begin{itemize}
    \item
    If $m_1=1$, we set $\zeta(\pm)(1^{(1)})=M\pm\epsilon$ and $\zeta(\pm)(2^{(1)})=-M$.
    \item
    If $m_1=0$ we set $\zeta(\pm)(1^{(1)})=\pm \epsilon$.
    \item 
    If $m_2=1$, we set $\zeta(\pm)(1^{(2)})=N$ and $\zeta(\pm)(2^{(2)})=-N$.
    \item
    If $m_2=0$ we set $\zeta(\pm)(1^{(2)})=0$.
\end{itemize}
Let $\zeta$ be one of $\zeta(\pm)$.  Let $\rho$ be a $\mathbb{C}Q$-module, with $\dim(\rho)\leq k$, satisfying the condition that each of the subquotients appearing in the Harder--Narasimhan filtration of $\rho$ with respect to the stability condition $\zeta$ has slope contained in $I\coloneqq [-\delta,\delta]$.  Then in particular, $\rho$ has slope contained in $[-\delta,\delta]$.  It follows that if $Q^{(2)}$ is isomorphic to the Kronecker quiver, the restriction of the dimension vector of $\rho$ to $Q^{(2)}$ must be $(a,a)$ for some $a\in \mathbb{Z}_{\geq 0}$. Then the contribution of $Q^{(2)}$ to the slope of $\rho$ is zero, whichever quiver $Q^{(2)}$ is, and so if $Q^{(1)}$ is isomorphic to the Kronecker quiver, the restriction of the dimension vector of $\rho$ to $Q^{(1)}$ must be of the form $(b,b)$ for some $b\in\mathbb{Z}_{\geq 0}$.  

Let $\rho$ be a $\mathbb{C}Q$-module with $\dim(\rho)\leq k$, and with the dimension vector of $\rho$ constant after restriction to $Q^{(1)}$, and also constant after restriction to $Q^{(2)}$. We claim that the following are equivalent
\begin{enumerate}
    \item All of the subquotients appearing in the Harder--Narasimhan filtration of $\rho$ with respect to $\zeta(-)$ have slope in $I$
    \item The restrictions $\rho^{(i)}$ of $\rho$ to $Q^{(i)}$ are $\zeta(+)$ semistable for $i=1,2$
    \item All of the subquotients appearing in the Harder--Narasimhan filtration of $\rho$ with respect to $\zeta(+)$ have slope in $I$.
\end{enumerate}
\begin{itemize}
    \item[$(1)\Leftrightarrow (2)$] The equivalence is trivial if one of $\rho^{(1)}$ or $\rho^{(2)}$ is trivial, so assume that they are both nontrivial.  So $\rho^{(2)}$ is a proper nontrivial submodule of $\rho$, and destabilizes it (for the stability condition $\zeta(-)$).  In particular, picking $\rho'\subset\rho^{(2)}$ to be the maximal submodule to attain the maximal slope of a submodule of $\rho^{(2)}$, $\rho'$ is the maximal destabilizing submodule of $\rho$ and it has slope in $I$ if and only if $\rho^{(2)}$ is semistable (with respect to both stability conditions).  Similarly, the maximal destabilizing quotient of $\rho$ with respect to $\zeta(-)$ is a quotient of $\rho^{(1)}$, and it has slope in $I$ if and only if $\rho^{(1)}$ is semistable with respect to both stability conditions.
    \item[$(3)\Rightarrow (2)$] If $\rho'\subset \rho^{(2)}$ was a destabilizing submodule with respect to $\zeta(+)$, it would also be a submodule of $\rho$ with $\zeta(+)$-slope greater than $\delta$, which is impossible by the condition on $\rho$.  Similarly, if $\rho'\subset \rho^{(1)}$ was destabilizing, then $\rho'\oplus \rho^{(2)}$, with its natural $\mathbb{C}Q$-module structure, would have slope greater than $\delta$.  So we deduce that both $\rho^{(1)}$ and $\rho^{(2)}$ are semistable.
    \item[$(2)\Rightarrow (3)$] Say that $\rho^{(1)}$ and $\rho^{(2)}$ are semistable.  Let $\rho'\subset \rho$ be the maximal destabilizing submodule with respect to the stability condition $\zeta(+)$, i.e. the first non-trivial term in the Harder--Narasimhan filtration of $\rho$.  If $Q^{(2)}$ is isomorphic to the Kronecker quiver, let $(a,b)$ be the dimension vector of $\rho'$ restricted to $Q^{(2)}$.  Since this restriction is a submodule of $\rho^{(2)}$ we must have $a\leq b$.  On the other hand, since $\rho'$ destabilizes $\rho$ we must have $a\geq b$.  We deduce that $a=b$, and so whichever quiver $Q^{(2)}$ is, the contribution of $Q^{(2)}$ to the slope of $\rho'$ is zero.  Now assume that $Q^{(1)}$ is isomorphic to the Kronecker quiver.  By the same argument, the dimension vector of $\rho'$, restricted to $Q^{(1)}$, must be constant.  We deduce that the slope of $\rho'$ lies within $I$.  Applying the same argument to $\rho/\rho'$, we see that the slope of all the subquotients in the Harder--Narasimhan filtration of $\rho$ belong to $I$.
\end{itemize}

From $(1)\Leftrightarrow (3)$ we deduce that
\begin{equation}
\label{zeta_ind}
\Mst^{\zeta(+)}_I(Q)=\Mst^{\zeta(-)}_I(Q).
\end{equation}

Consider the case in which $Q^{(1)}$ is the single vertex quiver, and write $e\in K$ for the generator corresponding to the vertex $1^{(1)}$.  We have
\[
\chi_{\Wt,K}\left(\bigoplus_{v\in \mathbb{Z}_{\geq 0}\cdot e}\HO_c(\Mst^{\zeta(\pm)\sstable}_{v}(Q))_{\vir}^*\right)=\EE(x^{e})
\]
by (\ref{ref_ferm}).

Now consider the case in which $Q^{(1)}$ is isomorphic to the Kronecker quiver.  Write $e$ for the sum of the generators in $K$ corresponding to the vertices $1^{(1)}$ and $2^{(1)}$.  Then by (\ref{ref_Kron}) we have
\[
\chi_{\Wt,K}\left(\bigoplus_{v\in \mathbb{Z}_{\geq 0}\cdot e}\HO_c(\Mst^{\zeta(\pm)\sstable}_{v}(Q))_{\vir}^*\right)=\EE(x^{e}[2]_t).
\]
The same results hold for $Q^{(2)}$; we let $f$ denote the element in $K$ defined analogously to $e$ but for $Q^{(2)}$.  Applying (\ref{GrandForm}) with $\zeta=\zeta(-)$ we deduce that
\begin{equation}
\label{half_one}    
\pi\left(\chi_{\Wt,K}\left(\bigoplus_{v\in K}\HO_c(\Mst^{\zeta(-)}_{I,v}(Q),\mathbb{Q})^*_{\vir}\right)\right)=\pi\left(\EE([m_2+1]_tx^f)\EE([m_1+1]_tx^e)\right).
\end{equation}
Now for the stability condition $\zeta(+)$ there may be more semistable modules with slope contained in $I$, but they must have dimension vector of the form $\alpha e+\beta f$ for $\alpha,\beta\in \mathbb{Z}_{>0}$.  Applying (\ref{GrandForm}) again, we deduce that 
\begin{align}
\label{half_two}
&\pi\left(\chi_{\Wt,K}\left(\bigoplus_{v\in K}\HO_c(\Mst^{\zeta(+)}_{I,v}(Q),\mathbb{Q})^*_{\vir}\right)\right)=\\&\pi\left(\EE([m_1+1]_tx^e)\prod_{\substack{0\xrightarrow{\beta/\alpha}\infty\\\alpha\neq 0,\beta\neq 0}}\EE(\chi_{\Wt}(\BPS^{\zeta(+)}_{\alpha e+\beta f}(Q))x^{\alpha e+\beta f}) \EE([m_2+1]_tx^f)\right).\nonumber
\end{align}
Combining (\ref{zeta_ind}), (\ref{half_one}), and (\ref{half_two}), we deduce that 
\begin{align}
\label{two_halves}
&\pi\left(\EE([m_2+1]_tx^f)\EE([m_1+1]_tx^e)\right)=\\&\pi\left(\EE([m_1+1]_tx^e)\prod_{\substack{0\xrightarrow{\beta/\alpha}\infty\\\alpha\neq 0,\beta\neq 0}}\EE(\chi_{\Wt}(\BPS^{\zeta(+)}_{\alpha e+\beta f}(Q))x^{\alpha e+\beta f})\nonumber \EE([m_2+1]_tx^f)\right).
\end{align}
Note that $\BPS_{\alpha e+\beta f}^{\zeta(+)}$ is independent of the large $M,N$ we pick in the above argument.  We denote this mixed Hodge structure instead by $\BPS_{\alpha e+\beta f}^{\lim}$.  In particular, it does not depend on $k$.  Letting $k$ tend to infinity, in the limit (\ref{two_halves}) becomes
\begin{align}
\label{two_halves_two}
\EE([m_2+1]_tx^f)\EE([m_1+1]_tx^e)=\EE([m_1+1]_tx^e)\prod_{\substack{0\xrightarrow{\beta/\alpha}\infty\\\alpha\neq 0,\beta\neq 0}}\EE(\chi_{\Wt}(\BPS^{\lim}_{\alpha e+\beta f}(Q))x^{\alpha e+\beta f}) \EE([m_2+1]_tx^f).
\end{align}

By Corollary \ref{MRcor}, each term $\EE(\chi_{\Wt}(\BPS^{\lim}_{\alpha e+\beta f}(Q))x^{\alpha e+\beta f})$ is of Lefschetz type, and so is of pL type.  Note that $x^ex^f=t^nx^{e+f}$.  Let $\hat{\mathfrak{g}}^{\circ}\subset \hat{\mathfrak{g}}_{Q}$ be the sub Lie algebra obtained by taking the closure of the span of the elements $\hat{x}^{g}$ for $g=\alpha e+ \beta f$, with $\alpha,\beta\in\mathbb{Z}_{\geq 0}$ and $\alpha+\beta>0$.  Denote by $L'$ the lattice spanned by $v_1$ and $v_2$.  There is an isomorphism of Lie algebras $\hat{\mathfrak{g}}^{\circ}\rightarrow \mathfrak{g}_{L'^+,\omega}$ sending $\hat{x}^{\alpha e+\beta f}$ to $\hat{z}^{\alpha v_1+ \beta v_2}$.  Exponentiating, we obtain an isomorphism of pro-unipotent algebraic groups
\[
\Phi\colon \hat{G}^{\circ}\xrightarrow{\cong}\hat{G}_{L'}.
\]

Finally, we use the identity \eqref{two_halves_two} to construct $\f{D}$.  Let $V=v_1^{\omega\perp}\cap v_2^{\omega\perp}$.  From $\omega(v_1,v_2)\neq 0$ and $v_{i}\in v_{i}^{\omega\perp}$ it follows that $V$ has codimension 2.  We define $\f{D}$ to be the scattering diagram containing $\f{D}_{\In}$ as well as a wall
\begin{align}\label{Dout}
(V-\bb{R}_{\geq 0}\cdot (\alpha v_1+\beta v_2),\EE(-\chi_{\Wt}(\BPS^{\lim}_{\alpha e+\beta f}(Q))z^{\alpha v_1+\beta v_2}))
\end{align}
for each $\alpha,\beta\in\mathbb{Z}_{>0}$.  Then $V$ is the only joint in $\f{D}$, and consistency around this joint amounts to the equality
\begin{align}\label{PL_eq}
\EE(-[m_1+1]_tz^{v_{1}})&\EE(-[m_2+1]_tz^{v_{2}})\\\nonumber=&\EE(-[m_2+1]_tz^{v_2})\prod_{\substack{\infty \xrightarrow{\beta/\alpha}0\\\alpha\neq 0,\beta\neq 0}}\EE(-\chi_{\Wt}(\BPS^{\lim}_{\alpha e+\beta f}(Q))z^{\alpha v_1+\beta v_2}) \EE(-[m_1+1]_tz^{v_1}),
\end{align}
cf. Figure \ref{PL_fig} where this corresponds to the equality of $\theta_{\gamma_1,\f{D}}$ and $\theta_{\gamma_2,\f{D}}$.  Finally, \eqref{PL_eq} follows from applying $\Phi$ to \eqref{two_halves_two} and inverting each side. 
\begin{figure}[ht]
\begin{tikzpicture}
\draw[->](-3,0) --(3.1,0);
\draw[->](0,3) --(0,-3.1);
\draw [black, fill=lightgray] (0,0) -- (2.9,0) -- (2.9,-2.9) -- (0,-2.9)--(0,0);
\node at (0.1,3.25){$\EE(-t^{m_1}z^{v_1})$};
\node at (-3.2,-0.25){$\EE(-t^{m_2}z^{v_2})$};
\node at (1.8,2.195){$\cdot$};
\node at (2.195,1.8){$\cdot$};
\draw [dashed,->] (1.8,2.2)  to [out=-180, in=90]  (-2.2,-1.8);
\draw [dashed,->] (2.2,1.8)  to [out=-90, in=0]  (-1.8,-2.2);
\node at (1,1.8){$\gamma_1$};
\node at (1.8,1.0){$\gamma_2$};
\end{tikzpicture}
\caption[]{Sketch of $\f{D}$, with outgoing walls as in \eqref{Dout} living in the grey region.}
\label{PL_fig}
\end{figure}
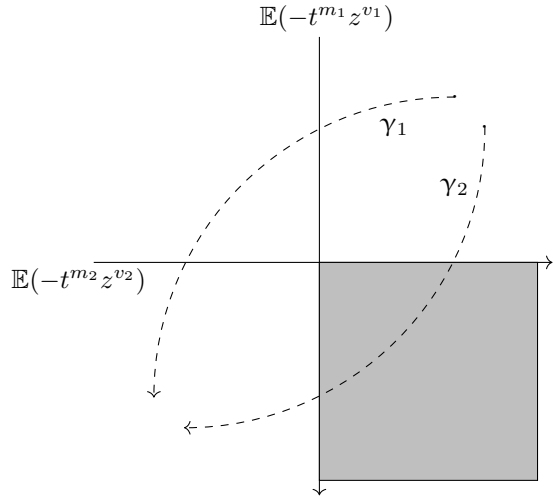 
\end{proof}

The substitution we used in the above proof in order to assume that $m_1,m_2\in \{0,1\}$ preserves the pL property but not the Lefschetz property, explaining why this argument does not yield preservation of Lefschetz type.  To prove Conjecture \ref{Lefschetz_conjecture}, it would be sufficient to have a version of Proposition \ref{2WallpL} stating that if the two incoming walls carry elements of the form $\EE(-[r_1]_tz^{v_1})$ and $\EE(-[r_2]_tz^{v_2})$ for $r_1,r_2\in \mathbb{Z}_{>0}$, then all outgoing walls can be chosen to carry elements of the form $\EE(-[r]_tz^{\alpha v_1+\beta v_2})$ for $r,\alpha,\beta\in\mathbb{Z}_{>0}$.  The reason we cannot easily adapt the proof of Proposition \ref{2WallpL} to prove this statement is that for $r$-Kronecker quivers with $r>2$ there are non-primitive DT invariants, i.e. we can have $\Omega_{\Kr{r},(n,n)}^{(1,-1)}\neq 0$ for $n\geq 2$ (see Examples \ref{coh_K2} and \ref{coh_Kplus}).
\appendix

\section{Mutation invariance}

Here we prove the mutation invariance properties of the scattering diagrams, broken lines, and theta functions.  The arguments are based on those of \cite{GHKK} in the classical setting.

\subsection{Mutation invariance of the scattering diagram}\label{MutInvSect}

Let $(S,\Lambda)$ be a compatible pair.  Recall the map $B_{1}\colon N\rightarrow M$ sending $e\mapsto B_S(e,\cdot)$.  Throughout the appendix we denote $v_k\coloneqq B_{1}(e_{k})$ for $k\in I$.  For $k\in I\setminus F$, recall the piecewise-linear maps $T_k\colon M_{\bb{R}}\rar M_{\bb{R}}$ from \eqref{Tj}, i.e.,
\begin{align*}
    T_k(m)\coloneqq 
    \begin{cases} m+\langle e_k,m\rangle v_k & \mbox{if $\langle e_k,m\rangle \geq 0$} \\ m & \mbox{otherwise.}
    \end{cases}
\end{align*}
Let
\begin{align*}
    \s{H}_{k,+}\coloneqq \{m\in M_{\bb{R}}\colon\langle e_k,m\rangle \geq 0\} \quad \mbox{and} \quad  \s{H}_{k,-}\coloneqq \{m\in M_{\bb{R}}\colon\langle e_k,m\rangle \leq 0\},
\end{align*}
i.e., $\s{H}_{k,\pm}$ are the domains of linearity for $T_k$.  
Let $T_{k,+}$ and $T_{k,-}$ denote the integral linear functions on $M_{\bb{R}}$ obtained by extending $T_k|_{\s{H}_{k,+}}$ and $T_k|_{\s{H}_{k,-}}$, respectively.  In particular, $T_{k,-}=\Id_{M_{\mathbb{R}}}$.

Note that each of $T_{k,\pm}$ induces a $\kk_t$-module automorphism of $\Pi(M)\coloneqq \prod_{m\in M} \kk(t^{1/D})$  
by taking the $m$-factor to the $T_{k,\pm}(m)$-factor.  Each $M$-graded algebra, Lie algebra, or pro-unipotent algebraic group we consider naturally embeds as an $M$-graded $\kk_t$-module, $\kk[t^{\pm 1}]$-module, or set, respectively, into $\prod_{m\in M} \kk(t^{1/D})$ via $\sum_{m\in M} a_mz^m \mapsto (a_m)_m$,  i.e., the coefficient of $z^m$ gives the $m$-factor of $\prod_{m\in M} \kk(t^{1/D})$.  We can then define $T_{k,\pm}(\sum_{m\in M} a_mz^m)\coloneqq \sum_{m\in M} a_mz^{T_{k,\pm}(m)}$, where the right-hand side is viewed as the element $(a_m)_{T_{k,\pm}(m)}\in \prod_{m\in M} \kk(t^{1/D})$.  In particular, since $T_{k,\pm}$ both respect $\Lambda$, they induce $\kk_t$-algebra isomorphisms $T_{k,\pm}\colon \kk_t^{\Lambda}[M]\rar \kk_t^{\Lambda}[M]$.

Let $\f{d}_k$ denote the wall $(e_k^{\perp},\Psi_t(z^{v_k}))$ in $M_{\bb{R}}$.   
 Let $\f{D}_S$ denote the consistent scattering diagram $\f{D}^{\s{A}_{q}}$ defined in \S\ref{Initial} for the compatible pair $(S,\Lambda)$, and similarly, let $\f{D}_{\mu_k(S)}$ denote the analogous scattering diagram for the compatible pair $(\mu_k(S),\Lambda)$.  Note that while both scattering diagrams are in $M_{\bb{R}}$, $\f{D}_S$ is defined over the Lie algebra $\f{g}_{M_{S,0}^+,\Lambda}$ for $M_{S,0}^+\coloneqq (B_1(N_{\uf})\cap \sigma)\setminus\{0\}$, where $\sigma\subset M_{\bb{R}}$ is the cone spanned by $\{v_i=B_1(e_i)\}_{i\in I\setminus F}$, while $\f{D}_{\mu_k(S)}$ is defined over $\f{g}_{M_{\mu_k(S),0}^+,\Lambda}$ for $M_{\mu_k(S),0}^{+}\coloneqq (B_1(N_{\uf})\cap \mu_k(\sigma))\setminus \{0\}$, 
 where $\mu_k(\sigma)\subset M_{\bb{R}}$ is the cone spanned by $\{B_1(\mu_k(e_i))\}_{i\in I\setminus F}$ for $\mu_k$ as in \eqref{eprime}.  The following is the quantum analog of \cite[Thm. 1.24]{GHKK}.

\begin{prop}\label{mut-inv}
Up to equivalence, the walls of $\f{D}_{\mu_k(S)}$ can be given as follows.  For each $\f{d}\in \f{D}_{S}\setminus \{\f{d}_k\}$, we have one or two walls, given as
\begin{align*}
    (T_k(\f{d}\cap \s{H}_{k,-}),T_{k,-}(f_{\f{d}})), \hspace{.25 in} (T_k(\f{d}\cap \s{H}_{k,+}),T_{k,+}(f_{\f{d}})),
\end{align*}
throwing out one of these if the support is of codimension at least $2$.  The only other wall is 
\begin{align*}
    \f{d}'_k\coloneqq (e_k^{\perp},\Psi_t(z^{-v_k})).
\end{align*}
\end{prop}

Here, if $f_{\f{d}}=1+\sum_{r=1}^{\infty} a_rz^{rv_{\f{d}}}\in \hat{G}_S\coloneqq \exp (\hat{\f{g}}_{M_{S,0}^+,\Lambda})\subset \prod_{m\in M} \kk(t^{1/D})$, then by definition, $T_{k,\pm}(f_{\f{d}})=1+\sum_{r=1}^{\infty} a_rz^{T_{k,\pm}(rv_{\f{d}})}\in \prod_{m\in M} \kk(t^{1/D})$.  The fact that this latter sum is actually contained in $\hat{G}_{\mu_k(S)}\coloneqq \exp (\hat{\f{g}}_{M_{\mu_k(S),0}^+,\Lambda})$ is not at all obvious.  This technical issue is dealt with in \cite{GHKK} using their Theorem 1.28.  Fortunately, we can circumvent the need to generalize this theorem to the quantum setup, instead using positivity to recover the following directly from its classical analog:
\begin{lem}\label{Tkscat}
The walls described in Proposition \ref{mut-inv} are well-defined and form a scattering diagram $T_k(\f{D}_S)$ in $M_{\bb{R}}$ over $\f{g}_{M_{\mu_k(S),0}^+,\Lambda}$.
\end{lem}
\begin{proof}
By Theorem \ref{PosScat}, all walls of $\f{D}_S$ specialize to non-trivial walls of the corresponding classical scattering diagram with no cancellation being possible.  Furthermore, this classical specialization clearly preserves the $z$-exponents of all terms of the scattering functions. Thus, the fact that the classical specialization of $T_k(\f{D}_S)$ is a scattering diagram over $\lim_{t\rar 1} (\f{g}_{M_{\mu_k(S),0}^+,\Lambda})$ \cite[Thm 1.24]{GHKK} implies the desired statement for the quantum scattering diagram.
\end{proof}

\begin{myproof}[Proof of Proposition \ref{mut-inv}]
We know from Lemma \ref{Tkscat} that $T_k(\f{D}_S)$ is a scattering diagram over $\f{g}_{M_{\mu_k(S),0}^+,\Lambda}$.  We will now show that the incoming walls of $T_k(\f{D}_S)$ are the same as the incoming walls of $\f{D}_{\mu_k(S)}$, and we will then show that $T_k(\f{D}_S)$ is consistent.  Proposition \ref{mut-inv} then follows from the uniqueness statement of Theorem \ref{KSGS}.

\textit{Step 1, incoming walls coincide:}  Using the equivalence of Example \ref{EquivEx}(2), let us split up all walls $(\f{d},f_{\f{d}})$ under consideration into their intersections $(\f{d}\cap \s{H}_{k,+},f_{\f{d}})$ and $(\f{d}\cap \s{H}_{k,-},f_{\f{d}})$, throwing away any empty intersections (and keeping only one copy of $\f{d}_k=\f{d}_k\cap \s{H}_{k,\pm}$).  Now since $T_{k,\pm}$ are linear and injective, it is clear that a wall of the form $(\f{d}\cap \s{H}_{k,\pm},f_{\f{d}})$ contains its direction $-v_{\f{d}}$ if and only if $T_k(\f{d}\cap \s{H}_{k,\pm})$ contains its direction $T_{k,\pm}(-v_{\f{d}})$.  That is, $T_k$ maps incoming walls to incoming walls and maps outgoing walls to outgoing walls.

Next we check that the functions on the incoming walls of $T_k(\f{D}_S)$ match those on the incoming walls of $\f{D}_{\mu_k(S)}$.  Consider an incoming wall $(\f{d}_i\cap \s{H}_{k,\pm},f_{\f{d}_i})$ of the modified $\f{D}_{S,\In}$, where $\f{d}_i=e_i^{\perp}$ and $f_{\f{d}_i}=\Psi_t( z^{v_i})$.  The corresponding wall in $T_k(\f{D}_{S,\In})$ is $(T_k(\f{d}_i\cap \s{H}_{k,\pm}),\Psi_t(z^{T_k(v_i)}))$ if $i\neq k$, or $(\f{d}_k,\Psi_t(z^{-v_i}))$ if $i=k$.  In any case, \eqref{Tkei} implies that this is the same as $(T_k(\f{d}_i\cap \s{H}_{k,\pm}),\Psi_t( z^{B_1(\mu_k(e_i))}))$, which is the corresponding initial wall of $\f{D}_{\mu_k(S),\In}$.  We have thus shown that the incoming walls of $T_k(\f{D}_S)$ and $\f{D}_{\mu_k(S)}$ coincide.

\textit{Step 2, $T_k(\f{D}_S)$ is consistent:} Here we follow the strategy of \cite[Proof of Thm. 1.24, Step II]{GHKK}.  
It suffices to check consistency of $T_k(\f{D}_S)$ around joints contained in $e_k^{\perp}$ because this is the only locus where $T_k$ fails to be linear.  For non-initial walls $(\f{d},f_{\f{d}})\in \f{D}_{S}\setminus \f{D}_{S,\In}$, the vector $v_{\f{d}}\in \sigma$ is not contained in an extremal ray of $\sigma$, i.e., is not equal to a multiple of any $e_i$.  So $\f{d}=\f{d}_k$ is the only wall with $v_{\f{d}}^{\Lambda\perp}=e_k^{\perp}$.  I.e., $\f{D}_S$ has no walls contained in $e_k^{\perp}$ other than $\f{d}_k$.

Now, let $\f{j}$ be one such joint in $e_k^{\perp}$. Let $\gamma$ denote a small closed path around $\f{j}$ based in $\s{H}_{k,-}$.  To any finite order, we can consider a path $\gamma=\gamma_4\circ \gamma_3\circ \gamma_2\circ \gamma_1$ such that:
\begin{itemize}
    \item $\gamma_1$ is a path starting in $\s{H}_{k,-}$ and immediately crossing $\f{d}_k$, crossing no other walls,
    \item $\gamma_2$ crosses every wall containing $\f{j}$ with interior in $\s{H}_{k,+}$, but crosses no other walls,
    \item $\gamma_3$ is a path starting in $\s{H}_{k,+}$ and immediately crossing $\f{d}_k$, crossing no other walls, and    
    \item $\gamma_4$ crosses every wall containing $\f{j}$ with interior in $\s{H}_{k,-}$, returning to the initial point of $\gamma_1$.
\end{itemize}

It is sufficient to check that 
\begin{align}\label{thetaEq}
    \theta_{\gamma,T_k(\f{D}_S)}=\theta_{\gamma,\f{D}_S},
\end{align} since we know the latter is the identity because $\f{D}_S$ is consistent. 
We have
\begin{align*}
    {\theta_{\gamma_1,T_k(\f{D}_S)}}&={\theta_{\gamma_1,\f{D}_{\mu_k(S)}}} \\
    {\theta_{\gamma_2,T_{k,+}(\f{D}_S)}}&=T_{k,+}(\theta_{\gamma_2,\f{D}_S})  \\
    \theta_{\gamma_3,T_k(\f{D}_S)}&=\theta_{\gamma_3,\f{D}_{\mu_k(S)}} \\
    \theta_{\gamma_4,T_k(\f{D}_S)}&=\theta_{\gamma_4,\f{D}_S}.
\end{align*}

Recall that the existence of a compatible $\Lambda$ implies that $\ker(B_{S,1})\cap M_{S,0}^+=\emptyset$, and similarly $\ker(B_{\mu_k(S),1})\cap M_{\mu_k(S),0}^+=\emptyset$.  Hence, $\hat{G}_S$ and $\hat{G}_{\mu_k(S)}$ have trivial center, thus are isomorphic to their Adjoint representation.  It follows that elements of $\hat{G}_S$ and $\hat{G}_{\mu_k(S)}$ are determined by their conjugation actions on $\kk_{t,\sigma}^{\Lambda}\llb M\rrb$ and $\kk_{t,\mu_k(\sigma)}^{\Lambda}\llb M\rrb$, respectively, and in fact, these actions are already determined by their restrictions to monomials $z^m$ with $m\in \s{H}_{k,-}$.

Suppose $f\in \kk_t^{\Lambda}[M]$.  Then $T_{k,+}^{-1}(f)$ is also in $\kk_t^{\Lambda}[M]\subset \kk_{t,\sigma}^{\Lambda}\llb M\rrb$, hence $\Ad_{\theta_{\gamma_2,\f{D}_S}}(T_{k,+}^{-1}(f))\in \kk_{t,\sigma}^{\Lambda}\llb M\rrb$.  We thus see that we can factor the conjugation homomorphism 
$$\left.{\Ad_{T_{k,+}(\theta_{\gamma_2,\f{D}_S})}}\right|_{\kk_t^{\Lambda}[M]}\colon \kk_t^{\Lambda}[M]\hookrightarrow \kk_{t,\mu_k(\sigma)}^{\Lambda}\llb M\rrb$$
into the composition $$\left.T_{k,+}\right|_{T_{k,+}^{-1}(\kk_{t,\mu_k(\sigma)}^{\Lambda}\llb M\rrb)}\circ {\theta_{\gamma_2,\f{D}_S}} \circ \left.T_{k,+}^{-1}\right|_{\kk_t^{\Lambda}[M]}.$$

Let $\theta_{\f{d}_k}\coloneqq \theta_{\gamma_1,\f{D}_S}=\Psi_t(z^{v_k})$, and let $\theta_{\f{d}'_k}\coloneqq  \theta_{\gamma_1,\f{D}_{\mu_k(S)}}=\Psi_t(z^{-v_k})^{-1}$.  Note that $\theta_{\gamma_3,\f{D}_S}=\theta_{\f{d}_k}^{-1}$, and similarly, $\theta_{\gamma_3,\f{D}_{\mu_k(S)}}=\theta_{\f{d}'_k}^{-1}$.  Since $T_{k,-}=\Id$, we have $\theta_{\gamma_4,\f{D}_{S}}=\theta_{\gamma_4,\f{D}_{\mu_k(S)}}$.  In summary, we have on $\kk_t^{\Lambda}[M]$ that
\begin{align*}
    \Ad_{\theta_{\gamma,\f{D}_{\mu_k(S)}}} =  \Ad_{\theta_{\gamma_4,\f{D}_{S}}}\circ \Ad_{\theta_{\f{d}_{k}'}^{-1}} \circ (T_{k,+}\circ \Ad_{\theta_{\gamma_2,\f{D}_S}} \circ \left.T_{k,+}^{-1}\right)\circ \Ad_{\theta_{\f{d}'_k}}.
\end{align*}

Hence, \eqref{thetaEq} will follow if we show that 
\begin{align*}
    T_{k,+}^{-1} \circ \Ad_{\theta_{\f{d}'_k}}(z^m) = \Ad_{\theta_{\f{d}_k}}(z^m)
\end{align*}
for all $m\in \s{H}_{k,-}$, i.e., $0<\langle -e_k,m\rangle = \Lambda(v_k,m)$, cf.  \eqref{Lambda}.   
Applying \eqref{AdPsi}, we have
\begin{align*}
    T_{k,+}^{-1} \circ \Ad_{\theta_{\f{d}'_k}}(z^m) &= T_{k,+}^{-1}\left(z^m \prod_{k=1}^{|\Lambda(-v_k,m)|} \left(1+t^{\sign(\Lambda(-v_k,m))(2k-1)}z^{-v_k}\right)\right) \\
    &= z^{m+\Lambda(v_k,m) v_k} \prod_{k=1}^{|\Lambda(v_k,m)|} \left(1+t^{-(2k-1)}z^{-v_k}\right) \\
     &= z^{m}t^{\Lambda(v_k,m)^2} \prod_{k=1}^{|\Lambda(v_k,m)|}  \left(z^{v_k}+t^{-(2k-1)}\right).  
\end{align*}
In the second equality, we applied $T_{k,+}^{-1}$ and used $-\langle e_k,m\rangle = \Lambda(v_k,m)$  again.   We obtained the third equality by factoring $$z^{m+\Lambda(v_k,m)v_k}=t^{\Lambda(v_k,m)^2} z^mz^{\Lambda(v_k,m)v_k}$$ and then distributing the last monomial.  Noting that $t^{\Lambda(v_k,m)^2}=\prod_{k=1}^{|\Lambda(v_k,m)|} t^{2k-1}$, we see that the last expression can be rewritten as
\begin{align}\label{thetadk}
     z^{m}\prod_{k=1}^{|\Lambda(v_k,m)|} \left(t^{2k-1}z^{v_k}+1\right).
\end{align}
Since $\sign(\Lambda(v_k,m))=1$, we see using \eqref{AdPsi} again that this is indeed equivalent to $\Ad_{\theta_{\f{d}_k}}(z^m)$, as desired.
\end{myproof}

\subsection{Mutation invariance of broken lines}\label{MutInvBL}

We continue to work with the map $T_k\colon M_{\bb{R}}\rar M_{\bb{R}}$ and the induced maps $T_{k,\pm}\colon \kk_t[M]\rar \kk_t[M]$ as in \S \ref{MutInvSect}.  We write $\vartheta_{p,\sQ}^S$ to indicate the theta function $\vartheta_{p,\sQ}$ associated to the consistent scattering diagram $\f{D}_S$ defined at the start of \S \ref{MutInvSect}, and similarly we write $\vartheta_{p,\sQ}^{\mu_k(S)}$ for the theta function defined with respect to $\f{D}_{\mu_k(S)}$.

\begin{prop}\label{MutInvBrokenLines}
For any $p\in M$ and generic $\sQ\in M_{\bb{R}}$, $T_k$ defines a bijection between broken lines with respect to $\f{D}_S$ with ends $(p,\sQ)$ and broken lines with respect to $\f{D}_{\mu_k(S)}$ with ends $(T_k(p),T_k(\sQ))$.  If $z^m$ is the final monomial of one such broken line for $\f{D}_S$, then the final monomial of the corresponding broken line for $\f{D}_{\mu_k(S)}$ is $z^{T_{k,+}(m)}$ if $\sQ\in \s{H}_{k,+}$ and $z^{T_{k,-}(m)}$ if $\sQ\in \s{H}_{k,-}$.  Hence,
\begin{align}\label{TkTheta}
    T_{k,\pm}(\vartheta_{p,\sQ}^S)=\vartheta_{T_k(p),T_{k}(\sQ)}^{\mu_k(S)},
\end{align} where the sign in the subscript of $T_{k,\pm}$ on the left-hand side is determined by $\sQ\in \s{H}_{k,\pm}$.
\end{prop}
We follow the proof of the analogous classical result, \cite[Prop. 3.6]{GHKK}.
\begin{proof}
Given a broken line $\gamma$ with ends $(p,\sQ)$ whose underlying map is $\gamma\colon (-\infty,0]\rar M_{\bb{R}}$, the underlying map for $T_k(\gamma)$ is of course taken to be $T_k\circ \gamma$.  Subdivide the domains of linearity of $\gamma$ so that each lies entirely in either $\s{H}_{k,+}$ or $\s{H}_{k,-}$.  For a domain of linearity $L\subset \s{H}_{k,\pm}$, $T_k$ takes the attached monomial $c_Lz^{m_L}$ to $T_{k,\pm}(c_Lz^{m_L})$.  We wish to check that this gives a broken line.

Let us assume that the representative of $\f{D}_S$ is chosen so that each scattering function is of the form $\Psi_t(t^{2a}z^v)$ or $\Psi_{t^{2^r}}(t^{2^r(2a+1)}z^{2^rv})$ for various $a\in \bb{Z}$, $r\in \bb{Z}_{\geq 0}$, and $v\in M_{S,0}^+$, as in Theorem \ref{PosScat}.  Let $\hat{G}_S^{\circ}$ denote the subgroup of $\hat{G}_S$ generated by the scattering functions of $\f{D}_S$, so Lemma \ref{Tkscat} ensures that $T_{k,\pm}$ take $\hat{G}_S^{\circ}$ into $\hat{G}_{\mu_k(S)}$.  For $(\f{d},f_{\f{d}})$ a wall of $\f{D}_S$, $p\in M$, and $\epsilon=\sign(\Lambda(v_{\f{d}},m))$, we see that
\begin{align}\label{TkAd}
    T_{k,\pm}(\Ad_{f_{\f{d}}^{\epsilon}
    }(z^p)) = \Ad_{T_{k,\pm}(f_{\f{d}})^{\epsilon}}(T_{k,\pm}(z^p)).
\end{align}
Indeed, by \eqref{tAdBinomial}, $\Ad_{f_{\f{d}}^{\epsilon}}(z^p)$ and $\Ad_{T_{k,\pm}(f_{\f{d}})^{\epsilon}}(z^p)$ are both contained in $\kk_t^{\Lambda}[M]$, so the equality follows from the fact that, since $T_{k,\pm}$ respects $\Lambda$, the maps $T_{k,\pm}\colon \hat{G}_S^{\circ}\rar \hat{G}_{\mu_k(S)}$ and $T_{k,\pm}\colon \kk_t^{\Lambda}[M]\rar \kk_t^{\Lambda}[M]$ are group and algebra homomorphisms which intertwine the actions of the groups on the algebras (cf. the proof of Lemma \ref{TjAdLem}).  As a result, to show that $T_k(\gamma)$ is a broken line for $\f{D}_{\mu_k(S)}$, it remains to check what happens when $\gamma$ crosses $e_k^{\perp}$.

Let $L_1,L_2$ be consecutive refined domains of linearity of $\gamma$ which lie on opposite sides of $e_k^{\perp}$, and let $c_iz^{m_i}$ denote the monomial attached to $L_i$.  For $L_1\subset \s{H}_{k,-}$, we have $\Lambda(v_k,m_1)>0$, and so $c_2z^{m_2}$ is any term in
\begin{align}\label{c2m2}
    \Ad_{\Psi_t(z^{v_k})} (c_1z^{m_1}) = c_1z^{m_1} \prod_{n=1}^{|\Lambda(v_k,m_1)|}  \left(1+t^{\sign(\Lambda(v_k,m_1))(2n-1)}z^{v_k}\right)
\end{align}
where the equality is a special case of \eqref{AdPsi}, with $\epsilon=1$.  Applying $T_{k,+}$, we have that $T_{k,+}(c_2z^{m_2})$ is a term in
\begin{align*}
    c_1z^{m_1-\Lambda(v_k,m_1)v_k} \prod_{n=1}^{|\Lambda(v_k,m_1)|}  \left(1+t^{\sign(\Lambda(v_k,m_1))(2n-1)}z^{v_k}\right).
\end{align*}
Applying the same manipulations which lead up to \eqref{thetadk}, we find that this is equal to
\begin{align*}
    c_1z^{m_1} \prod_{n=1}^{|\Lambda(-v_k,m_1)|}  \left(t^{\sign(\Lambda(-v_k,m_1))(2n-1)}z^{-v_k}+1\right)=\Ad_{\Psi_t(z^{-v_k})}(c_1z^{m_1}).
\end{align*}
Since the right-hand side is the corresponding wall-crossing for $\f{D}_{\mu_k(S)}=T_k(\f{D}_S)$, it follows that $T_k(\gamma)$ does in fact satisfy the correct bending rule.

Similarly, if we instead had $L_1\subset \s{H}_{k,+}$, then $T_{k,-}(c_2z^{m_2})=c_2z^{m_2}$ could be any term in the right-hand side of \eqref{c2m2} (although the sign of $\Lambda(v_k,m_1)$ is now $-1$).  On the other hand, if we apply $T_k$ before the wall-crossing, we would have to apply $T_{k,+}$ to $c_1z^{m_1}$, yielding $c_1z^{m_1-\Lambda(v_k,m_1)v_k}$.  One then computes that
\begin{align*}
\Ad_{\Psi_t(z^{-v_k})}(c_1z^{m_1-\Lambda(v_k,m_1)v_k}) =c_1z^{m_1-\Lambda(v_k,m_1)v_k} \prod_{k=1}^{|\Lambda(-v_k,m)|}  \left(1+t^{\sign(\Lambda(-v_k,m))(2k-1)}z^{-v_k}\right),
    \end{align*}
and applying similar manipulations as before, noting that now $\Lambda(v_k,m)<0$ so  $z^{-\Lambda(v_k,m_1)v_k}=z^{|\Lambda(v_k,m_1)|v_k}$, we find that this indeed equals the right-hand side of \eqref{c2m2}, as desired.

We have thus shown that $T_k$ takes broken lines to broken lines.  That this is a bijection follows from noting that $T_k^{-1}\colon M_{\bb{R}}\rar M_{\bb{R}}$ induces the inverse map on broken lines.
\end{proof}

\subsection{Mutation invariance of theta functions}\label{MutInvTheta}

Recall the linear maps $\psi_{\jj,\sQ}$ (depending on generic $\sQ\in M_{\bb{R}}$) and $\psi_{\jj}\coloneqq \psi_{\jj,\sQ_{\jj}}$ as defined in \eqref{psijQ} and \eqref{psij}, respectively.  From now on, we will denote these by $\psi_{\jj,\sQ}^{\s{A}}$ and $\psi_{\jj}^{\s{A}}$.

Applying this to $S^{\prin}$, we obtain $\psi^{\s{A}^{\prin}}_{\jj,\sQ}$ and $\psi_{\jj}^{\s{A}^{\prin}}$ mapping $M^{\prin}_{\bb{R}}\rar M^{\prin}_{\bb{R}}$ (in \S \ref{AtlasX}, $\psi_{\jj}^{\s{A}^{\prin}}$ was denoted $\psi_{\jj}^{\prin}$).  By restriction to $N'_{\bb{R}}=\xi(N_{\bb{R}})\subset M^{\prin}_{\bb{R}}$, we obtain linear automorphisms $\psi^{\s{X}}_{\jj,\sQ}$ and $\psi_{\jj}^{\s{X}}\colon N_{\bb{R}}\rar N_{\bb{R}}$ which induce maps on $\Pi(N)$ and $\kt$-algebra automorphisms of $\kk_t^{B}[N]$.

More directly, let \begin{align*}
    T_{j}^{S,\s{X}}(n)\coloneqq \begin{cases}
n+B(n,e_j) e_j & \mbox{if $B(n,e_j) \geq 0$},\\
n & \mbox{otherwise},
\end{cases}
\end{align*}
and then let
\begin{align}\label{TjX}
    T_{\jj}^{\s{X}}=T^{S_{\jj_s},\s{X}}_{j_s}\circ \cdots \circ T^{S_{\jj_1},\s{X}}_{j_1}\colon N_{\bb{R}}\rar N_{\bb{R}}.
\end{align}
Then $\psi^{\s{X}}_{\jj,\sQ}$ is obtained by restricting $T_{\jj}^{\s{X}}$ to a neighborhood of $\s{Q}$ and extending linearly.  We note that the definition as the restriction of $\psi^{\s{A}^{\prin}}_{\jj,\sQ}$ has the advantage of easily applying to generic $\sQ\in M_{\bb{R}}^{\prin}$, not just generic $\sQ\in N_{\bb{R}}$.

\begin{cor}\label{TjCor}
Let $\s{V}$ denote $\s{A}$ or $\s{X}$.  Let $S$ be a seed, and if $\s{V}=\s{A}$, pick a compatible $\Lambda$.  Then for any tuple $\jj\in (I\setminus F)^s$ and any generic $\sQ\in L_{\s{V},\bb{R}}$, $\psi_{\jj,\sQ}^{\s{V}}$ induces a $\kt$-algebra isomorphism
\begin{align*}
    \psi_{\jj,\sQ}^{\s{V}}\colon \iota_{\sQ}(\s{V}_q^{\can,S})\risom \iota_{T_{\jj}^{\s{V}}(\sQ)}(\s{V}_q^{\can,S_{\jj}})
\end{align*}
with $\psi_{\jj,\sQ}^{\s{V}}(\vartheta^S_{p,\sQ})=\vartheta^{S_{\jj}}_{T_{\jj}^{\s{V}}(p),T_{\jj}^{\s{V}}(\sQ)}$.  In particular, $\psi_{\jj}^{\s{V}}(\vartheta_p)\coloneqq \vartheta_{T_{\jj}^{\s{V}}(p)}$ defines a $\kt$-algebra isomorphism $\psi_{\jj}^{\s{V}}\colon \s{V}_q^{\can,S}\risom \s{V}_q^{\can,S_{\jj}}$.  
Similarly if we replace each instance of $\can$ with $\midd$.
\end{cor}
\begin{proof}
Take $\s{V}=\s{A}$.  The claim that $\psi_{\jj,\sQ}^{\s{A}}(\vartheta_{m,\sQ})$ and $\vartheta_{T_{\jj}^{\s{A}}(m),T_{\jj}^{\s{A}}(\sQ)}$ are equal in $\prod_{n\in N} \kk_t$ follows from applying \eqref{TkTheta} inductively. The fact that the isomorphism respects the structure constants follows from the mutation invariance of broken lines in Proposition \ref{MutInvBrokenLines} because, by Proposition \ref{StructureConstants}, the broken lines directly determine the structure constants.

However, the topological structures of $\s{A}_q^{\can,S}$ and $\s{A}_q^{\can,S_{\jj}}$ are different, so a bijection between the formal bases respecting the structure constants is not sufficient to conclude that this gives an isomorphism.  To prove this, it suffices to consider the case $\jj=(j)$ since the isomorphisms can then be applied inductively.  We have \begin{align*}
    \sigma_{S,\s{A}}=&\bb{R}_{\geq 0} \langle B_1(e_i)\colon i\in I\setminus F\rangle\subset M_{\bb{R}}\\
    \sigma_{S_{j},\s{A}}=&\bb{R}_{\geq 0}\langle B_1(\mu_j(e_i))\colon i\in I\setminus F\rangle\subset M_{\bb{R}}.
\end{align*}
Note that both cones are contained in the convex (but not strongly convex) cone $$\sigma_{S,S_j}\coloneqq \bb{R}_{\geq 0}\langle B_1(e_i)\colon i\in I\setminus (F\cup \{j\})\rangle+\bb{R}B_1(e_j).$$

We can define $\kk_{\sigma_{S,S_j},t}^{\Lambda}\llb M\rrb\coloneqq \kk_t^{\Lambda}[M]\otimes_{\kk_t^{\Lambda}[M^{\oplus_{S,S_{j}}}]}\kk_t^{\Lambda}\llb M^{\oplus_{_{S,S_{j}}}}\rrb$, where $M^{\oplus_{S,S_j}}\coloneqq M\cap \sigma_{S,S_j}$.  Here, $\kk_t^{\Lambda}\llb M^{\oplus_{_{S,S_{j}}}}\rrb$ is the completion of $\kk_t^{\Lambda}[M^{\oplus_{S,S_j}}]$ with respect to the maximal monomial ideal generated by $z^m$ with $m\in M^{+_{S,S_j}}\coloneqq M^{\oplus_{S,S_j}}\setminus M^{\times_{S,S_j}}$, where $M^{\times_{S,S_j}}$ is the monoid of invertible elements of $M^{\oplus_{S,S_j}}$. 

We see that $\s{A}_q^{\can,S}$ and $\s{A}_q^{\can,S_j}$ are both naturally contained in $\kk_{\sigma_{S,S_j},t}^{\Lambda}\llb M\rrb$ --- between any two instances of a broken line $\gamma$ crossing $e_j^{\perp}$, there must be an instance of $\gamma$ crossing a wall $\f{d}$ with $v_{\f{d}}\in M^{+_{S,S_j}}$, so the completion in the $\pm e_j$ direction is not necessary.  Furthermore, since $T_j$ preserves $M^{+_{S,S_j}}$, it acts via homeomorphism and automorphism on $\kk_{\sigma_{S,S_j},t}^{\Lambda}\llb M\rrb$ viewed as a topological $\kk_t$-module.  Since this action takes $\iota_{\sQ_j}(\s{A}_q^{\can,S})$ to $\iota_{T_j(\sQ_j)}(\s{A}_q^{\can,S_j})$ (for $\sQ_j\in \s{C}_j$), it identifies $\s{A}_q^{\can,S}$ and $\s{A}_q^{\can,S_j}$ as $\kk_t$-modules which now have the same topology.   The claim follows. 

For the $\s{X}$-version, we first consider the compatible pair $(S^{\prin},\Lambda^{\prin})$ for $\Lambda^{\prin}$ as in \eqref{LambdaPrin} and recall the map $\xi$ of Lemma \ref{xirho}.  By construction, $\psi_{\jj,\sQ}^{\s{A}^{\prin}}\circ \xi = \xi \circ \psi_{\jj,\sQ}^{\s{X}}$.  The $\s{X}$-version of the result thus follows from the $\s{A}^{\prin}$-version, and this is a special case of the $\s{A}$-version above.

For the final statement, we just have to check that $\Theta_{\s{V}}^{\midd}$ is preserved by $T_{\jj}^{\s{V}}$.  This again follows from Proposition \ref{MutInvBrokenLines} since $\Theta_{\s{V}}^{\midd}$ is defined in terms of the finiteness of certain sets of broken lines.
\end{proof}

\subsubsection{The formal quantum upper cluster algebra}\label{FormalUp}
Let $\s{V}$ denote $\s{A}$ or $\s{X}$.  For $\sQ\in C_{S_{\jj}}^+\subset \s{C}$ (in $M_{\bb{R}}$ for $\s{V}=\s{A}$ or $M_{\bb{R}}^{\prin}$ for $\s{V}=\s{X}$), denote $\iota_{\s{Q}}$ by $\iota_{\jj}$.  Recall that $\psi_{\jj}^{\s{V}}\circ \iota_{\jj}$ agrees with $\mu^{\s{V}}_{\jj}$ on $\s{V}_q^{\up}$.  Motivated by this, we say that an element $f\in \wh{\s{V}}_q^S$ satisfies the \textbf{formal Laurent phenomenon} if $\psi_{\jj}^{\s{V}}\circ \iota_{\jj}(f)\in \wh{\s{V}}_q^{S_{\jj}}$ for each $\jj$.  The algebra of such $f\in \wh{\s{V}}_q^S$ is what we call the \textbf{formal quantum upper cluster algebra}, denoted $\wh{\s{V}}_q^{\up}$.  Note that the cluster atlas $\{\iota_{\jj}\}_{\jj}$ can naturally be viewed as an atlas on $\wh{\s{V}}_q^{\up}$.

 Since $(\psi_{\jj}^{\s{V}})^{-1}(\s{V}_q^{S_{\jj}})=\s{V}_q^{S_{\jj}}$, we can view $\s{V}_q^{\up}$ as 
\begin{align*}
    \s{V}_q^{\up}=\bigcap_{\jj} \iota_{\jj}^{-1}(\s{V}_q^{S_{\jj}})\subset \wh{\s{V}}_q^S.
\end{align*}
Thus,
\begin{align*}
    \s{V}_q^{\up}\subset \wh{\s{V}}_q^{\up}.
\end{align*}

Let $p\in L$ be arbitrary.  For $\sQ_0\in C_S^+$, we have $\vartheta^S_{p,\sQ_0}\in \wh{\s{V}}_q^S$ by construction.  If $\sQ_{\jj}$ is instead in some other $C_{S_{\jj}}^+$, then it follows from Corollary \ref{TjCor} that $\psi_{\jj}^{\s{V}}(\vartheta^S_{p,\sQ_{\jj}})$ is in $\wh{\s{V}}_q^{S_{\jj}}$.  Since $\vartheta^S_{p,\sQ_{\jj}} = \iota_{\jj}(\vartheta^S_{p,\sQ_{0}})$, it follows that $\vartheta_{p,\sQ_0} \in \wh{\s{V}}_q^{\up}$.  I.e., while the Laurent phenomenon holds for all cluster monomials, this formal version of the Laurent phenomenon holds for all theta functions.
\begin{cor}[The formal quantum Laurent phenomenon]\label{FormalLaurent}
$\s{V}_q^{\can} \subset \wh{\s{V}}_q^{\up}$.
\end{cor}

\bibliographystyle{alpha}  
\bibliography{biblio}        
\index{Bibliography@\emph{Bibliography}}%

\end{document}